\documentclass[11pt,oneside,reqno]{amsart}
\usepackage[top=25mm, bottom=25mm, left=25mm, right=25mm]{geometry}
\usepackage{amsfonts,amssymb,amsmath,amsthm,amsxtra}
\usepackage[T1]{fontenc}
\usepackage[utf8]{inputenc}
\usepackage{bm}
\usepackage{xparse}
\usepackage{mathtools}
\usepackage{dsfont}
\usepackage{bbm}
\usepackage{mathrsfs}
\usepackage{enumitem}
\usepackage{tikz}
\usepackage{comment}
\usepackage{stmaryrd}

\usepackage{multirow}

\usepackage{graphics}
\usepackage{color}
\usepackage[unicode,hidelinks,bookmarks]{hyperref}
\hypersetup{
colorlinks=true,
citecolor=[rgb]{0,0,0.8},
linkcolor=[rgb]{0.89,0.02,0.17},
urlcolor=[rgb]{0,0,0.75},
pdfpagemode=UseNone,
pdfstartview=FitH,
pdfdisplaydoctitle=true,
pdftitle={Discrete analogues in harmonic analysis: Multi-parameter Radon averages},
pdfauthor={Hejna, Kim, Langowski, Mirek, Song, Wright},
pdflang=en-US
}

\numberwithin{equation}{section}
\newtheorem{theorem}[equation]{Theorem}
\newtheorem{corollary}[equation]{Corollary}
\newtheorem{lemma}[equation]{Lemma}

\newtheorem{proposition}[equation]{Proposition}

\newtheorem*{BF}{Bellow--Furstenberg problem}
\newtheorem*{MBF}{Multi-parameter Bellow--Furstenberg problem}

\theoremstyle{definition}
\newtheorem{remark}[equation]{Remark}
\newtheorem{definition}[equation]{Definition}
\newtheorem{example}[equation]{Example}

\newcommand{\BB}{\mathbb{B}}
\newcommand{\CC}{\mathbb{C}}
\newcommand{\DD}{\mathbb{D}}
\newcommand{\EE}{\mathbb{E}}

\newcommand{\GG}{\mathbb{G}}

\newcommand{\II}{\mathbb{I}}
\newcommand{\JJ}{\mathbb{J}}
\newcommand{\KK}{\mathbb{K}}

\newcommand{\NN}{\mathbb{N}}

\newcommand{\QQ}{\mathbb{Q}}
\newcommand{\RR}{\mathbb{R}}

\newcommand{\TT}{\mathbb{T}}

\newcommand{\XX}{\mathbb{X}}

\newcommand{\ZZ}{\mathbb{Z}}

\usepackage{stackengine,scalerel}

\newcommand{\calF}{\mathcal{F}}
\newcommand{\calB}{\mathcal{B}}

\newcommand{\lo}{\mathrm{l}}

\newcommand{\dif}{\mathrm{d}}
\newcommand{\ex}{\bm{e}}

\newcommand{\proj}{\mathrm{p}}

\newcommand{\ind}[1]{\mathds{1}_{{#1}}}

\usepackage{scalerel,stackengine}
\stackMath
\newcommand\reallywidehat[1]{%
\savestack{\tmpbox}{\stretchto{%
  \scaleto{%
    \scalerel*[\widthof{\ensuremath{#1}}]{\kern-.7pt\bigwedge\kern-.7pt}%
    {\rule[-\textheight/2]{1.5ex}{\textheight}}
  }{\textheight}%
}{0.9ex}}%
\stackon[2.25pt]{#1}{\tmpbox}%
}

\newcommand*{\DMO}[1]{\expandafter\DeclareMathOperator\csname #1\endcsname {#1}}
\DMO{ord}
\DMO{supp}
\DMO{Sym}
\DMO{cl}
\DMO{lcm}
\DMO{dist}
\DMO{tr}
\DMO{diam}
\DMO{Log}

\DeclarePairedDelimiter\abs{\lvert}{\rvert}

\DeclarePairedDelimiter\norm{\lVert}{\rVert}

\DeclarePairedDelimiter\floor{\lfloor}{\rfloor}

\DeclarePairedDelimiterX\spr[2]{\langle}{\rangle}{#1,#2}
\newcommand{\ipr}[2]{#1\cdot#2}
\DeclarePairedDelimiterX\Set[2]{\{}{\}}{#1\colon #2}
\DeclarePairedDelimiterX\Seq[1]{(}{)}{#1}

\dedicatory{Dedicated to the memory of Alexandra Bellow (1935--2025).}

\begin{document}
\title[Discrete analogues in harmonic analysis: Multi-parameter Radon averages]{Discrete analogues in harmonic analysis:\\ Multi-parameter Radon averages}

\author{Agnieszka Hejna-Łyżwa}
\address[Agnieszka Hejna-Łyżwa]{
Instytut Matematyczny,
Uniwersytet Wroc{\l}awski,
Plac Grunwaldzki 2,
50-384 Wroc{\l}aw
Poland}
\email{agnieszka.hejna@math.uni.wroc.pl}

\author{Joonil Kim}
\address[Joonil Kim]{
Department of Mathematics,
Yonsei University,
Seoul 120-729, Korea}
\email{jikim7030@yonsei.ac.kr}

\author{Bartosz Langowski}
\address[Bartosz Langowski]
{
Department of Mathematics and Physical Sciences, 
Franciscan University of Steubenville, 
1235 University Blvd., Steubenville, OH 43952, USA
}
\email{blangowski@franciscan.edu}

\author{Mariusz Mirek }
\address[Mariusz Mirek]{
Department of Mathematics,
Rutgers University,
Piscataway, NJ 08854-8019, USA \\
\&
Instytut Matematyczny,
Uniwersytet Wroc{\l}awski,
Plac Grunwaldzki 2/4,
50-384 Wroc{\l}aw
Poland}
\email{mariusz.mirek@rutgers.edu}

\author{Hoyoung Song}
\address[Hoyoung Song]{
Department of Mathematics,
Yonsei University,
Seoul 120-729, Korea}
\email{nonspin0070@gmail.com}

\author{James Wright}
\address[James Wright]{
James Clerk Maxwell Building,
The King's Buildings,
Peter Guthrie Tait Road,
City Edinburgh,
EH9 3FD}
\email{J.R.Wright@ed.ac.uk}

\begin{abstract}
In this paper we study maximal and oscillation inequalities for
multi-parameter discrete Radon averaging operators. We develop a
robust variant of the multi-parameter circle method within the
framework of \textit{Discrete Analogues in Harmonic Analysis}.  In
particular, this gives quantitative estimates for these averages and
their underlying Fourier multipliers which reveals an interesting
major-arcs rigidity phenomenon. As a consequence, we completely
resolve in the affirmative the multi-parameter Bellow--Furstenberg
problem in pointwise ergodic theory.
\end{abstract}

\date{\today}

\thanks{Joonil Kim and Hoyoung Song were supported by the National Research Foundation of Korea (NRF) under Grant No. RS-2026-25482410. Bartosz Langowski was supported by the AMS–Simons Research
Enhancement Grant for Primarily Undergraduate Institution (PUI)
Faculty and by the National Science Centre of Poland under the OPUS
grant no. 2025/59/B/ST1/01786.  Mariusz Mirek was partially supported
by the NSF CAREER grant DMS-2236493.  James Wright was partially
supported by a Leverhulme Research Fellowship
RF-2023-709$\backslash$9}
\maketitle

\setcounter{tocdepth}{1}

\tableofcontents

\section{Introduction}\label{section:1}

\subsection{Motivations and statement of the main results}
Multi-parameter problems have been both a plague and a challenge in
classical harmonic analysis, from the failure of Lebesgue's
differentiation theorem on \(L^1(\mathbb{R}^d)\) for averages over
rectangles with sides parallel to the coordinate axes \cite{Saks1, Saks}
 to C. Fefferman's counterexample for the pointwise convergence of 
double Fourier series \cite{Fef}. These challenges and counterexamples have produced new insights and significant advancements from A. C\'ordoba and R. Fefferman's geometric proof of the strong maximal function \cite{CF} to J. Journ\'e's multi-parameter theory for classical Calder\'on--Zygmund singular integral operators \cite{journe}.

The theory of {\it Singular and Maximal Radon Transforms} (operators too singular to fall within the scope of classical Calder\'on--Zygmund theory) has been a central topic in Euclidean harmonic analysis since the early 1970s. The theory in the one-parameter setting was developed by Stein and his collaborators and culminated in work of Christ, Nagel, Stein and Wainger from the early 1990s \cite{CNSW}. More recently, it
has been extended to the multi-parameter setting by Ricci, Stein and Street; see for example \cite{RS} and \cite{SS}. This is all in the continuous setting.

From early on, operators on Euclidean spaces ${\mathbb R}^d$ and their discrete counterparts on $\mathbb Z^d$ have been studied hand in hand. For instance, when Hardy and Littlewood \cite{HLmax} established in 1930 the basic mapping properties of the classical maximal operator
$$
\sup_{N\in\RR_+} \bigg|\frac{1}{N} \int_0^N f(x-t) dt\bigg| \quad \text{for} \quad f \in L^p({\mathbb R}),
$$
 they did so by establishing these properties for the discrete maximal operator
$$
\sup_{N\in\ZZ_+} \bigg|\frac{1}{N} \sum_{n=1}^N f(x-n)\bigg|, \quad \text{for} \quad f \in \ell^p({\mathbb Z}),
$$
and then {\it transferring} the bounds from ${\mathbb Z}$ to ${\mathbb R}$. 

When the underlying operators are not too singular, then a simple
transference argument allows us to pass back and forth from the
continuous to the discrete settings; see for example the sampling
principle \cite[Proposition 2.1]{MSW}, which is also 
formulated in Proposition \ref{prop:msw}.  However this argument
completely breaks down for singular Radon averaging operators.  Over
the last 35 years, a successful theory has been established to treat
discrete Radon averaging operators in the one-parameter setting by
several authors. We will describe these developments in more detail
below.

This paper aims to develop quantitative tools to study multi-parameter
maximal and oscillation inequalities for discrete multi-parameter
Radon averaging operators. As an application, we solve the
multi-parameter Bellow--Furstenberg problem in pointwise ergodic
theory.

For
$d, k\in \ZZ_+:=\{1, 2,\ldots\}$, let
$P=(P_1,\ldots, P_d):\ZZ^k\to\ZZ^d$ be a polynomial mapping, where 
$P_1,\ldots, P_d$ are $k$-variate polynomials with integer coefficients.
Then for any function $f\in\ell^1(\ZZ^d)$ and $x\in\ZZ^d$, we define the
\textit{discrete multi-parameter polynomial Radon averaging operator} by
\begin{align}
\label{eq:38}
A_{M; \ZZ^d}^{P}f(x):=\EE_{m\in Q_M^k}f(x-P(m)), 
\end{align}
where $Q_M^k:=[M_1]\times\cdots\times [M_k]$ with $M=(M_1,\ldots, M_k)\in[1, \infty)^k$.
Here and throughout the paper we use the notation
$\EE_{y\in Y}f(y):=\frac{1}{|Y|}\sum_{y\in Y}f(y)$ for any finite set
$Y\neq\emptyset$ and any function $f:Y\to\CC$, where $|Y|$ denotes the
cardinality of the set $Y$, and $[M]:=(0, M]\cap\ZZ$ for any real
number $M\ge1$.

Our main result is the following  theorem for the discrete Radon averages from \eqref{eq:38}.

\begin{theorem}
\label{thm:1}
Let $d, k\in\ZZ_+$ be given and let
$P=(P_1,\ldots, P_d):\ZZ^k\to\ZZ^d$ be a polynomial mapping as above.
Define $\DD_{\tau}:=\{\tau^n:n\in\NN\}$ for any $\tau\in(1, \infty)$. Let
$A_{M; \ZZ^d}^{P}$ be the multi-parameter Radon average defined in
\eqref{eq:38} with $M=(M_1,\ldots, M_k)\in[1, \infty)^k$. Then the
following inequalities hold.

\begin{enumerate}
\item[(i)] \textit{(Multi-parameter maximal inequality).}  For every $p\in(1, \infty]$ there exists  $C_{d, k, p}\in\RR_+$ such that for every function $f\in \ell^p(\ZZ^d)$, one has
\begin{align}
\label{eq:80}
\big\|\sup_{M\in\ZZ_+^k}|A_{M; \ZZ^d}^{P}f|\big\|_{\ell^p(\ZZ^d)}\le C_{d, k, p}\|f\|_{\ell^p(\ZZ^d)}.
\end{align}
\item[(ii)] \textit{(Multi-parameter oscillation inequality).}  For every $p\in(1, \infty)$ and $\tau\in(1, \infty)$ there exists  $C_{d, k, p, \tau}\in\RR_+$ such that  for every $J\in\ZZ_+$ and for any $k$ strictly increasing sequences $(I_{j, 1})_{j\in\NN},\ldots, (I_{j, k})_{j\in\NN}\subseteq \DD_{\tau}$ and for every function $f\in \ell^p(\ZZ^d)$, one has
\begin{align}
\label{eq:88}
\bigg\|\Big(\sum_{j=0}^{J-1}\sup_{M\in \BB[I_j]\cap\DD_{\tau}^k}
\big|A_{M;\ZZ^d}^{P}f - A_{I_j;\ZZ^d}^{P}f\big|^2\Big)^{1/2}\bigg\|_{\ell^p(\ZZ^d)}\le C_{d, k, p, \tau}\|f\|_{\ell^p(\ZZ^d)},
\end{align}
where $I_j:=(I_{j,1}, \ldots, I_{j,k})\in\DD_{\tau}^k$, and
$\BB[I_j]:=[I_{j, 1}, I_{j+1, 1})\times\cdots\times[I_{j, k}, I_{j+1, k})$ for $j\in \mathbb N$.
\end{enumerate}
The constants $C_{d, k, p}$ and $C_{d, k, p,\tau}$,  respectively in \eqref{eq:80} and \eqref{eq:88}, may depend on the degree
of the underlying polynomial mapping $P$, but are independent of its
coefficients. Importantly, the constant $C_{d, k, p,\tau}$ is independent of the sequence $I := (I_j)_{j\in {\mathbb N}}$ and $J\in {\mathbb Z}_{+}$. Moreover, the parameter ranges $p$ for which inequalities \eqref{eq:80} and
\eqref{eq:88} hold are both sharp.
\end{theorem}

The proof of Theorem \ref{thm:1} is reduced to proving the oscillation
inequality \eqref{eq:88}, which implies the maximal inequality from
\eqref{eq:80} in view of the estimate stated in \eqref{eq:135} below.  The
proof of the multi-parameter oscillation inequality \eqref{eq:88} is a
serious challenge requiring substantial new ideas that we
systematically develop in this paper. The key tool is a
\textit{multi-parameter circle method} developed in the context of
maximal and oscillation estimates, with the ``\textit{major arcs
rigidity}'' as a new phenomenon which will be described below.

Since the discrete averaging Radon operators
\eqref{eq:38} live in a very close symbiosis with their ergodic
counterparts, inequalities \eqref{eq:80} and \eqref{eq:88} have direct
applications to pointwise convergence problems in ergodic
theory. Specifically, the multi-parameter oscillation inequality \eqref{eq:88} will be used to give a
complete solution of the multi-parameter Bellow--Furstenberg problem.

In order to see this, let the triple $(X, \mathcal{B}(X), \mu)$ denote
a $\sigma$-finite measure space throughout this paper.  
For $d, k \in\ZZ_+$ let $P=(P_1,\ldots, P_d):\ZZ^k\to\ZZ^d$ be a polynomial mapping as above. Given  
a family ${\mathcal T} = \{T_1,\ldots, T_d\}$ of invertible commuting
measure-preserving transformations on $X$, a function $f\in L^1(X)$,  a
$k$-tuple $M = (M_1,\ldots, M_k)\in[1, \infty)^k$,  we define for any $x\in X$, the \textit{multi-parameter polynomial ergodic
average} by
\begin{align}
\label{eq:107}
A_{{M}; X, {\mathcal T}}^{P}f(x):= \EE_{m\in Q_{M}^k}f(T_1^{P_1(m)}\cdots T_d^{P_d(m)}x),
\end{align}
where $Q_{M}^k:=[M_1]\times\ldots\times[M_k]$ is a rectangle as in \eqref{eq:38}.

We now make the connection between the Radon averages \eqref{eq:38} and the ergodic averages \eqref{eq:107}.

\begin{example}[Shift integer system]
\label{ex:1}
Fix $d, k \in\ZZ_+$, and let $P=(P_1,\ldots, P_d):\ZZ^k\to\ZZ^d$ be a
polynomial mapping as above.  Consider the
$d$-dimensional lattice $(\ZZ^d, \mathcal B(\ZZ^d), \mu_{\ZZ^d})$
equipped with a family of shifts $S_1,\ldots, S_d:\ZZ^d\to\ZZ^d$,
where $\mathcal B(\ZZ^d)$ denotes the $\sigma$-algebra of all subsets
of $\ZZ^d$, $\mu_{\ZZ^d}$ denotes counting measure on $\ZZ^d$, and
$S_j(x)=x-e_j$ for every $x\in\ZZ^d$ (here $e_j$ is the $j$-th basis
vector from the standard basis in $\ZZ^d$ for each $j\in[d]$). Then the
average $A_{M; X, {\mathcal T}}^{{P}}$ from \eqref{eq:107} with
${\mathcal T} = (T_1,\ldots, T_d)=(S_1,\ldots, S_d)$
becomes the Radon average $A_{M; \ZZ^d}^{P}$ from \eqref{eq:38}.

In some situations, depending on how explicit we need to be, we will
write out the averages
$A_{{M}; X, {\mathcal T}}^{P}=A_{M_1,\ldots, M_k;X, \mathcal T}^{P_1,\ldots, P_d}$
or
$A_{{M}; X, {\mathcal T}}^{P}= A_{M_1,\ldots, M_k;X,T_1,\ldots, T_d}^{P_1,\ldots, P_d}$.
We will also often abbreviate $A_{M; X, {\mathcal T}}^{P}$ to
$A_{M; X}^{P}$ when the transformations are understood. In particular, 
in the integer shift system we will write out the
averages
$A_{M; \ZZ^d}^{P}=A_{M_1,\ldots, M_k; \ZZ^d}^{P}$ or
$A_{M; \ZZ^d}^{P}=A_{M; \ZZ^d}^{P_1,\ldots, P_d}$ or even
$A_{M; \ZZ^d}^{P}=A_{M_1,\ldots, M_k; \ZZ^d}^{P_1,\ldots, P_d}$.

In ergodic applications, finite measure spaces are of primary importance. However, by including $\sigma$-finite spaces, the connection between the discrete averaging Radon operators in \eqref{eq:38} and their ergodic counterparts in \eqref{eq:107}, as illustrated in Example \ref{ex:1}, becomes evident.

\end{example}

Having introduced some basic ergodic terminology and the multi-parameter polynomial ergodic averages in \eqref{eq:107}, we can formulate the multi-parameter Bellow--Furstenberg problem, which reads as follows.

\begin{MBF}\label{BMSW-conj}
For $d, k \in\ZZ_+$ let $P=(P_1,\ldots, P_d):\ZZ^k\to\ZZ^d$ be
a polynomial mapping as above.  Let $\mathcal T=\{T_1, \ldots, T_d\}$
be a family of invertible commuting measure-preserving transformations
on a probability space $(X, \mathcal{B}(X), \mu)$.  Is it true that
for any $f\in L^{\infty}(X)$, the ergodic averages defined in \eqref{eq:107} for any
$M = (M_1,\ldots, M_k)\in[1, \infty)^k$, i.e.
\begin{align*}
A_{{M}; X, {\mathcal T}}^{P}f(x):= \EE_{m\in Q_{M}^k}f(T_1^{P_1(m)}\cdots T_d^{P_d(m)}x),
\end{align*}
 converge pointwise for 
$\mu$-almost every $x\in X$, as $\min(M_1,\ldots, M_k)\to\infty$?
\end{MBF}

For \(d=k=1\) this was a famous problem of Bellow \cite{Bel} and
Furstenberg \cite{F}, solved in the affirmative by Bourgain \cite{B1,
B2, B3}. For \(d=1\) and arbitrary \(k\in\mathbb{Z}_+\), this problem
was established only recently  in \cite{BMSW} by the fourth and last authors, together with
Bourgain and Stein. Our main multi-parameter ergodic theorem of this
paper resolves this problem completely.
\begin{theorem}\label{main-thm}
The multi-parameter Bellow--Furstenberg problem is true for all $d, k \in\ZZ_+$.
\end{theorem}

Theorem \ref{main-thm} is a consequence of the following quantitative ergodic theorem, see item (ii) below.

\begin{theorem}
\label{thm:main}
Let $d, k\in\ZZ_+$ be given and let
$P=(P_1,\ldots, P_d):\ZZ^k\to\ZZ^d$ be a polynomial mapping as above.
Let $(X, \mathcal B(X), \mu)$ be a $\sigma$-finite measure space
equipped with a family of commuting invertible measure-preserving
transformations ${\mathcal T} = \{T_1,\ldots, T_d\}$ on $X$. Define
$\DD_{\tau}:=\{\tau^n:n\in\NN\}$ for any $\tau\in(1, \infty)$. Let
$A_{M;X,{\mathcal T}}^{P}$ be the ergodic
average defined in \eqref{eq:107} with $M=(M_1,\ldots, M_k)\in[1, \infty)^k$. Then the
following statements hold.

\begin{itemize}
\item[(i)] \textit{(Multi-parameter mean ergodic theorem).} For any $p\in(1, \infty)$ and $f\in L^p(X)$  the averages
$A_{M;X,{\mathcal T}}^{P}f$ converge in $L^p(X)$ norm as $\min(M_1,\ldots, M_k)\to\infty$.
\smallskip
\item[(ii)] \textit{(Multi-parameter pointwise ergodic theorem).} For any $p\in(1, \infty)$ and $f\in L^p(X)$  the averages
$A_{M;X,{\mathcal T}}^{P}f$ converge pointwise $\mu$-almost everywhere on $X$ as $\min(M_1,\ldots, M_k)\to\infty$.
\smallskip

\item[(iii)] \textit{(Multi-parameter maximal ergodic theorem).}
For every $p\in(1, \infty]$ there exists  $C_{d, k, p}\in\RR_+$ such that for every function $f\in L^p(X)$, one has
\begin{align}
\label{eq:230}
\big\|\sup_{M\in \ZZ_+^k}|A_{M;X,{\mathcal T}}^{P}f|\big\|_{L^p(X)}\le C_{d, k, p}\|f\|_{L^p(X)}.
\end{align}
\item[(iv)] \textit{(Multi-parameter oscillation ergodic theorem).}
For every $p\in(1, \infty)$ and $\tau\in(1, \infty)$ there exists  $C_{d, k, p, \tau}\in\RR_+$ such that  for every $J\in\ZZ_+$ and for any $k$ strictly increasing sequences $(I_{j, 1})_{j\in\NN},\ldots, (I_{j, k})_{j\in\NN}\subseteq \DD_{\tau}$ and for every function $f\in L^p(X)$, one has
\begin{align}
\label{eq:284}
\bigg\|\Big(\sum_{j=0}^{J-1}\sup_{M\in \BB[I_j]\cap\DD_{\tau}^k}
\big|A_{M;X,\mathcal T}^{P}f - A_{I_j;X,\mathcal T}^{P}f\big|^2\Big)^{1/2}\bigg\|_{L^p(X)}\le C_{d, k, p, \tau}\|f\|_{L^p(X)},
\end{align}
where $I_j:=(I_{j,1}, \ldots, I_{j,k})\in\DD_{\tau}^k$, and
$\BB[I_j]:=[I_{j, 1}, I_{j+1, 1})\times\cdots\times[I_{j, k}, I_{j+1, k})$ for $j\in\mathbb N$.
\end{itemize}
The implied constants $C_{d, k, p}$ and $C_{d, k, p,\tau}$,
respectively in \eqref{eq:230} and \eqref{eq:284}, may depend on the
degree of the underlying polynomial mapping $P$, but are independent
of its coefficients. Importantly, the constant $C_{d, k, p,\tau}$ is independent of the sequence $I := (I_j)_{j\in {\mathbb N}}$ and $J\in {\mathbb Z}_{+}$. Moreover, the parameter ranges $p$ for which
inequalities \eqref{eq:230} and \eqref{eq:284} hold are both sharp.
\end{theorem}

Inequality \eqref{eq:284} implies inequality \eqref{eq:230} due to the
estimate \eqref{eq:135} below.  Once inequality \eqref{eq:284} is
established for any $\tau\in(1, \infty)$, then
$A_{\tau^{n_1},\ldots, \tau^{n_k};X,\mathcal T}^{P}f$ converge
pointwise $\mu$-almost everywhere on $X$ as
$\min(n_1,\ldots, n_k)\to\infty$.  Now arguing as in \cite[Lemma
1.5]{RW}, the conclusion from (ii) follows. Using 
(ii) and (iii) and the dominated convergence theorem, we obtain
(i). Since parts (i)-(iii) in Theorem \ref{thm:main} follow from (iv),
it suffices to prove \eqref{eq:284} or, more specifically, its
counterpart \eqref{eq:88} from Theorem \ref{thm:1} on the integer shift system, due to the
Calder{\'o}n transference principle \cite{Cald}.  Indeed, in our case, oscillation inequalities
\eqref{eq:88} from the integer shift system are transferred by the
Calder{\'o}n transference principle \cite{Cald} to the corresponding
bounds in \eqref{eq:284}  on an arbitrary
$\sigma$-finite measure space $(X, \mathcal B(X), \mu)$. Since
inequality \eqref{eq:88} is a special case of inequality
\eqref{eq:284} for the integer shift system, the inequalities from
\eqref{eq:284} and \eqref{eq:88}, in view of the
Calder{\'o}n transference principle \cite{Cald}, are in fact equivalent.

Here we emphasize that the Calder{\'o}n transference
principle \cite{Cald} does not transfer pointwise almost everywhere
convergence itself! For instance, in the integer shift system,
pointwise convergence is implied by norm convergence, since the
$\ell^\infty(\mathbb Z^d)$ norm is dominated by the
$\ell^p(\mathbb Z^d)$ norm for any $p\in[1, \infty)$, but this cannot
be transferred by the Calder{\'o}n transference principle. The
Calder{\'o}n transference principle \cite{Cald} transfers only
quantitative bounds (such as the inequalities in \eqref{eq:80} and
\eqref{eq:88}) from the integer shift system to the corresponding
bounds in a $\sigma$-finite measure-preserving system in question,
where they are further used to deduce pointwise almost everywhere
convergence on that space.  Hence, from now on we focus on discussing
and proving quantitative bounds on the integer shift systems
(specifically inequality \eqref{eq:88} from Theorem \ref{thm:1}),
which further will underscore the symbiosis between the ergodic world
and discrete analogues in harmonic analysis.

This reduction has both advantages and disadvantages. The disadvantage
is that we completely lose information about the original
measure-preserving system, which a priori may possess properties that
could help to establish pointwise convergence without relying on
quantitative bounds such as maximal or oscillation inequalities, which
are typically challenging to establish. The advantage is that we pass
from an abstract measure-preserving system to the integer shift
system, which is algebraically well structured. The set of
integers forms an abelian group, and thanks to that reduction the
averaging operators from \eqref{eq:38} become convolution operators
with discrete probability measures, making tools from harmonic
analysis available.

Before discussing the results from Theorem \ref{thm:1}  in detail and
explaining the key ideas of our multi-parameter circle method and other
novelties of the paper, we provide a brief historical background on
discrete analogues in harmonic analysis, highlighting one-parameter
developments that will be critical going forward in the paper and
justifying our motivation to pursue the multi-parameter direction in
the fields of discrete analogues in harmonic analysis and 
ergodic theory.

\subsection{A brief history of discrete analogues in harmonic
analysis} The study of discrete analogues in Euclidean harmonic
analysis has a long history that goes back to the 1920s foundational
papers of Riesz \cite{R2} and Hardy and Littlewood \cite{HLmax}. In
\cite{R2}, Riesz established $L^p(\mathbb{R})$ boundedness of the
Hilbert transform and, by a simple comparison argument, (based on the
comparison of sums with integrals), deduced $\ell^p(\mathbb{Z})$
boundedness of its discrete analogue. A few years later, Hardy and
Littlewood \cite{HLmax} established $L^p(\mathbb{R})$ boundedness of
the Hardy--Littlewood maximal function on the real line.
Interestingly, they worked entirely in the discrete setting and solved
a maximal problem stated for a finite set of positive numbers
(interpreting it in the language of cricket or any other game in which
a player compiles a series of scores whose average is recorded), and
used it to deduce classical maximal results in the continuous
setting \cite{HLmax}. Since that time the Hardy--Littlewood maximal
inequality has been viewed as a quantitative form of the Lebesgue
differentiation theorem, which in its simplest form asserts that for
every  $f\in L^1(\RR)$, the limit
\begin{align}
\label{eq:113}
\lim_{M\to 0}\frac{1}{M}\int_0^Mf(x-t)dt=f(x)
\end{align}
exists 
for almost every $x\in\RR$, see also \cite[Section 1.3]{bigs} for its higher-dimensional analogues.

Essentially, at the same time von Neumann \cite{vN} and Birkhoff
\cite{BI} proved their famous ergodic theorems --- von Neumann in the
norm and Birkhoff pointwise almost everywhere, respectively. Their two
ergodic theorems can be summarized as follows. If
$(X,\mathcal B(X), \mu)$ is a probability measure space equipped with
a measure-preserving transformation $T:X\to X$, then for every
$p\in[1, \infty)$ and every $f\in L^p(X)$ the classical ergodic
averages $A_{M;X, T}^{\mathrm m}f(x):=\mathbb E_{m\in [M]}f(T^{m}x)$
converge pointwise for $\mu$-almost every $x\in X$ and in $L^p(X)$
norm as $M\to\infty$. Moreover, if $T$ is ergodic on $X$, then 
\begin{align}
\label{eq:114}
\lim_{M\to \infty}A_{M;X, T}^{\mathrm m}f=\int_Xf(x)d\mu(x)
\end{align}
holds pointwise almost everywhere and in norm for every $f\in L^1(X)$.
This simply states that the time averages converge pointwise almost
everywhere and in mean to the space average, which is a key feature of von Neumann’s and Birkhoff’s theorems. At the time when Birkhoff  established his
pointwise ergodic theorem \cite{BI}, its connections with the Hardy--Littlewood
maximal functions \cite{HLmax} were not yet obvious; these connections
became clearer later, and, from a broader perspective, even more so
after Calder{\'o}n published his transference principle
\cite{Cald}. Calder{\'o}n's transference principle revealed the bridge
between ergodic theory and discrete analogues in harmonic analysis.

In the early 1950s Calder{\'o}n--Zygmund theory for singular integral
operators emerged \cite{CZ}, and boundedness of discrete analogues of
Calder{\'o}n--Zygmund type singular integral operators followed
easily, as observed in \cite{CZ}, from their continuous counterparts
either by a simple comparison argument as in the Hilbert transform
case, or by mimicking continuous arguments in the discrete setting.

Calder{\'o}n--Zygmund theory \cite{CZ} became central in
Euclidean harmonic analysis, mainly by reinterpreting
classical Littlewood--Paley theory \cite{LP} in settings where the
Calder{\'o}n--Zygmund framework makes sense. The latter serves as an
important tool to detect orthogonalities in $L^p(\mathbb{R}^d)$ spaces
for $p \neq 2$, which, combined with Fourier methods or 
almost-orthogonality methods, was successfully developed to study the
so-called generalized Radon operators in $\mathbb{R}^d$, even in situations where
Calder{\'o}n--Zygmund theory cannot be directly applied, see
\cite{bigs, CNSW}.

For the
purpose of our discussion the most representative will be polynomial Radon
operators. Namely, for $d, k\in\ZZ_+$ let
$P=(P_1,\ldots, P_d):\RR^k\to\RR^d$ be a polynomial mapping, where
$P_1,\ldots, P_d$ are $k$-variate polynomials with real coefficients.
Then for any $f\in L^1(\RR^d)$ a vector $M=(M_1,\ldots, M_k)\in\RR_+^k$, and $x\in\RR^d$, define the
\textit{multi-parameter polynomial Radon averaging operator}
by
\begin{align}
\label{eq:109}
A_{M; \RR^d}^{P}f(x):=\frac{1}{M_1\cdots M_k}\int_0^{M_1}\ldots \int_0^{M_k}f(x-P(t))dt.
\end{align}
We see that $A_{M; \RR^d}^{P}$ is the continuous, integral counterpart
of the operator $A_{M; \ZZ^d}^{P}$ from \eqref{eq:38}.

If
$d=k\in\ZZ_+$, $P(t)=(t_1,\ldots, t_d)$ and
$M_1=\ldots=M_k=M\in\RR_+$, then
$A_{M; \RR^d}^{{\rm t}_1,\ldots, {\rm t}_d}$ is the one-parameter
Hardy--Littlewood averaging operator over cubes $[0, M]^d$. It is well
known that for every $f\in L^p(\RR^d)$ the maximal function
$\sup_{M\in\RR_+}|A_{M; \RR^d}^{{\rm t}_1,\ldots, {\rm t}_d}f|$ is
bounded on $L^p(\RR^d)$ for all $p\in(1, \infty]$ and weak-type
$(1,1)$ at the endpoint $p=1$ as a consequence of the Vitali covering
lemma, see \cite[Section 1.3]{bigs}. In this case, as for $d=k=1$ in  \eqref{eq:113}, the $L^p(\RR^d)$
and weak-type $(1,1)$ bounds are quantitative forms of
the Lebesgue differentiation theorem  in $\RR^d$, see \cite[Section 1.3]{bigs}.

It was very surprising, in 1932, when Saks \cite{Saks, Saks1} and
independently Busemann and Feller \cite{BF} showed that the Lebesgue
differentiation theorem may fail for genuinely multi-parameter
averages $A_{M_1,\ldots, M_d; \RR^d}^{{\rm t}_1,\ldots, {\rm t}_d}f$
as $\max(M_1,\ldots, M_d)\to 0$ for some $f\in L^1(\RR^d)$
whenever $d\ge2$. At the same time Zygmund \cite{Zyg34} showed that
the Lebesgue differentiation theorem remains true for
the multi-parameter averages
$A_{M_1,\ldots, M_d; \RR^d}^{{\rm t}_1,\ldots, {\rm t}_d}f$ as
$\max(M_1,\ldots, M_d)\to 0$, whenever $f\in L^p(\RR^d)$ for
every $p\in(1, \infty)$.  The latter result was further extended to
Orlicz spaces by Jessen, Marcinkiewicz and Zygmund \cite{JMZ}.

The $L^p(\RR^d)$ bounds for $p\in(1, \infty]$ of maximal Radon
functions $\sup_{M\in\RR_+}|A_{M; \RR^d}^{P}f|$ corresponding to the
one-parameter Radon averages from \eqref{eq:109} with
$M_1=\ldots=M_k=M\in\RR_+$ for arbitrary polynomial mappings
$P=(P_1,\ldots, P_d):\RR^k\to\RR^d$ were extensively studied by many authors, including
Stein and Wainger \cite{SW} and later  by Duoandikoetxea and Rubio de Francia
\cite{DR}, which are classical references, see also \cite[Section
11.2]{bigs}. Interestingly, the question whether weak-type $(1,1)$
bounds hold for $\sup_{M\in\RR_+}|A_{M; \RR^d}^{P}f|$ is still open
\cite{STW}, except the case $d=k=1$, which follows from
\cite{CRWRev}.  The maximal functions
$\sup_{M\in\RR_+^k}|A_{M; \RR^d}^{P}f|$ for genuinely multi-parameter
Radon averages \eqref{eq:109} with arbitrary polynomial mappings were
studied in a foundational paper by Ricci and Stein
\cite{RS}. Recently, the third and fourth authors in collaboration
with Kosz and Plewa \cite{KLMP} established a multi-parameter
oscillation inequality in the spirit of \eqref{eq:88} for the multi-parameter Radon averages from
\eqref{eq:109}.

In view of these developments for Radon operators in the Euclidean
setting, it was natural to ask whether similar results hold in the
discrete setting for Radon averaging operators from \eqref{eq:38}. However, it quickly became clear that discrete
one-parameter Radon-type operators (see averages \eqref{eq:38} with
$M_1=\ldots=M_k=M\in[1, \infty)$) present difficulties that are 
unapproachable from the continuous perspective; comparison
arguments and covering lemmas fail in the discrete setting. This was a serious challenge that, in contrast to  discrete analogues in classical Calder{\'o}n--Zygmund theory, stagnated for almost four decades.

This situation dramatically changed in the mid-1980s when Bourgain, in
groundbreaking papers \cite{B1,B2,B3}, established the pointwise
ergodic theorem for ergodic averages with polynomial iterates, solving
a problem posed by Bellow \cite{Bel} and Furstenberg \cite{F}, which
reads as follows.

\begin{BF}
Let $(X, \mathcal B(X), \mu)$ be a probability space endowed with an
invertible measure-preserving map $T:X\to X$. Is it true that  for any polynomial
$P\in\ZZ[{\rm m}]$ and for every $f\in L^{\infty}(X)$, the limit of the polynomial ergodic averages
\begin{align}
\label{eq:98}
A_{M; T, X}^{P}f(x):=\EE_{m\in[M]} f(T^{P(m)} x), \quad x\in X, \quad M\in\ZZ_+,
\end{align}
exists pointwise  $\mu$-almost everywhere on $X$, as
$M\to \infty$? 
\end{BF}

Bourgain not only solved the Bellow--Furstenberg problem and obtained
a far-reaching generalization of Birkhoff's ergodic theorem \cite{BI},
but  introduced many important tools to study
pointwise convergence problems and more importantly established the Calder{\'o}n
transference principle \cite{Cald} as the key tool linking the world
of quantitative estimates in ergodic theory with the world of discrete
analogues in harmonic analysis, with the integer shift system, as a
universal measure-preserving system for studying quantitative bounds, like maximal or oscillation inequalities among others.

Bourgain's papers \cite{B1, B2, B3} have had an important impact on
the development of discrete analogues in harmonic analysis. However,
before we get to this point, it will be important to understand the
ergodic perspective, which is not only interesting in its own right,
but also brings number-theoretic tools such as Weyl's inequality
for exponential sums \cite{Weyl} and the Hardy--Littlewood circle method
\cite{IK} into play, which transformed the field of discrete analogues
in harmonic analysis.

\subsubsection{\textbf{Bellow--Furstenberg problem and Bourgain's era}}
In 1910, Bohl \cite{Bohl},
Sierpi\'nski \cite{S} and Weyl \cite{Weyl-10}, proved for any irrational number $\theta \in\RR\setminus\QQ$ that
for every $0\le a<b\le 1$, we have
\begin{align}
\label{eq:108}
\lim_{M\to \infty}M^{-1}|\{m\in[M]: \{\theta m\}\in[a, b]\}|=b-a,
\end{align}
where $\{x\}$ denotes the fractional part of $x\in\RR$.  In other
words, \eqref{eq:108} says that the sequence
$(\{\theta m\})_{m\in\NN}$ is equidistributed on $[0,1]$. This is a
classical equidistribution theorem that has two far-reaching
generalizations. First, in 1916, Weyl \cite{Weyl} proved that the
equidistribution theorem remains true for any polynomial sequence
$(\{P(m)\})_{m\in\NN}$ in place of $(\{\theta m\})_{m\in\NN}$ with any
polynomial $P\in\mathbb R[{\rm m}]$ having at least one irrational
coefficient. A second generalization arose from the question whether the
intervals in \eqref{eq:108} can be replaced by Lebesgue measurable
sets.

While it is not difficult to show that \eqref{eq:108} remains
true if we replace intervals with Borel sets whose boundaries have
Lebesgue measure zero, arbitrary Lebesgue measurable sets are not
allowed in \eqref{eq:108} in place of the intervals. Indeed, fixing
$\theta \in\RR\setminus\QQ$, the sequence $(\{\theta m\})_{m\in\NN}$
is equidistributed, and considering
$E=\{\{\theta m\}: m\in\NN\}\subseteq [0, 1]$, we see
\[
0=\lim_{M\to \infty}M^{-1}|\{m\in[M]: \{\theta m\}\in E^c\}|=1.
\]

However, after a bit of reflection one sees that this question is
ill-posed due to the delicate nature of Lebesgue measure. Khintchin
\cite{Khin} observed that Birkhoff's pointwise ergodic theorem
\cite{BI} gives a strict generalization of the classical
equidistribution theorem for all Lebesgue-measurable sets. Namely, for
every $\theta \in\RR\setminus\QQ$ and every Lebesgue measurable set
$E\subseteq [0, 1]$, we have
\begin{align}
\label{eq:5}
\lim_{M\to\infty} M^{-1}|\{ m\in[M]: \{x + \theta m\}  \in E \}|  =  |E| 
\end{align}
for almost every $x\in[0, 1]$. The latter follows from applying
Birkhoff's pointwise ergodic theorem \eqref{eq:114} to $f=\ind{E}$ and
$T=R_{\theta}:\mathbb{T}\to \mathbb{T}$, the rotation about an
irrational angle \(\theta \in \mathbb{R}\setminus\mathbb{Q}\) on
$\TT$, which is given by $R_{\theta}(x)=\{x+\theta\}$ for any
$x\in\TT$. Here it is important that $R_{\theta}$ is ergodic on $\TT$.

Therefore, it was natural to ask whether an extension of Weyl's
equidistribution theorem in the spirit of Khintchin holds.  This
question was formally posed by Bellow \cite{Bel} and Furstenberg
\cite{F} in the early 1980s, and, using ergodic language, it has been stated
above.

In the mid-1980s, Bourgain, in three papers \cite{B1,B2,B3} that
represent a far-reaching common generalization of Birkhoff's pointwise
ergodic theorem and Weyl's equidistribution theorem, answered
affirmatively the Bellow--Furstenberg problem by establishing
pointwise almost everywhere convergence for the averages from
\eqref{eq:98} for all $f \in L^p(X)$ with $p\in (1,\infty)$.

Bourgain had the good insight to transfer the Bellow--Furstenberg
problem from an abstract measure-preserving system to the integer
shift system by using the Calder{\'o}n transference principle
\cite{Cald}. Then, for the Radon averages
$A_{M; \ZZ}^{P}f:=\EE_{m\in [M]}f(x-P(m))$ corresponding to
\eqref{eq:38} with $d=k=1$, he showed that for every
$p\in(1, \infty]$ and every polynomial $P\in\ZZ[{\rm m}]$ there exists
$C_{p}\in\RR_+$ such that for every function $f\in \ell^p(\ZZ)$,
one has
\begin{align}
\label{eq:111}
\big\|\sup_{M\in\ZZ_+}|A_{M; \ZZ}^{P}f|\big\|_{\ell^p(\ZZ)}\le C_{p}\|f\|_{\ell^p(\ZZ)},
\end{align}
leaving open the question of weak-type $(1,1)$ estimates.  Interestingly
it was shown much later by Buczolich and Mauldin \cite{BM} for
$P(m) = m^2$, and by LaVictoire \cite{LaV1} for $P(m) = m^k$ with
$k\ge2$ that the above weak-type $(1,1)$ estimates fail. This sharply
contrasts with the continuous setting as the one-dimensional maximal
functions $\sup_{M\in\RR_+}|A_{M, \RR}^Pf|$ satisfy weak-type $(1,1)$
estimates for any $f\in L^1(\RR)$ by \cite{CRWRev}.  This also illustrates that any intuition we
develop in Euclidean harmonic analysis --- where sums are replaced by
integrals --- can fail dramatically in discrete problems, making the
field of discrete analogues in harmonic analysis both challenging and interesting.

Besides proving \eqref{eq:111}, Bourgain 
also proved in \cite{B1, B2, B3} a non-uniform variant of oscillation inequality on
$\ell^2(\ZZ)$, which asserts that for every $\tau\in(1, \infty)$ and
 $J\in\ZZ_+$ there exists $C_{d,\tau}(J)\in\RR_+$ such that for
every strictly increasing sequence $(I_j)_{j\in\NN}\subseteq \DD_{\tau}$ and for every
$f\in\ell^2(\ZZ)$, one has
\begin{align}
\label{eq:112}
\bigg\|\Big(\sum_{j=0}^{J-1}\sup_{M\in [I_j, I_{j+1})\cap\DD_{\tau}}
\big|A_{M;\ZZ}^{P}f - A_{I_j;\ZZ}^{P}f\big|^2\Big)^{1/2}\bigg\|_{\ell^2(\ZZ)}\le C_{d,\tau}(J)\|f\|_{\ell^2(\ZZ)}.
\end{align}

An important aspect of Bourgain's proof is the quantified bound
$C_{d,\tau}(J)=o_{d,\tau}(J^{1/2})$, which, after transferring
inequality \eqref{eq:112} to an arbitrary $\sigma$-finite
measure-preserving system, solves the Bellow--Furstenberg problem for
all square-integrable functions, see \cite[Proposition 2.8,
p.2267]{MSW-survey} for the details. This, in turn, combined with
\eqref{eq:111} transferred to the measure-preserving system in
question solves the Bellow--Furstenberg problem for all $p$-integrable
functions with $p\in(1, \infty)$. Ultimately, Bourgain \cite{B1,B2,B3} showed, as a consequence of quantitative bounds, a somewhat stronger result: the Bellow--Furstenberg problem holds for any \(f \in L^p(X)\) with \(p \in (1, \infty)\) on any \(\sigma\)-finite measure space \((X, \mathcal{B}(X), \mu)\) endowed with an invertible measure-preserving map \(T: X \to X\).

Another important aspect of Bourgain's inequality \eqref{eq:112} with
$C_{d,\tau}(J)=o_{d,\tau}(J^{1/2})$ is that inequality \eqref{eq:112}
trivially follows from \eqref{eq:111} for $p=2$ with
$C_{d,\tau}(J)=O_{d}(J^{1/2})$ by the triangle inequality. Therefore,
inequality \eqref{eq:112} can be thought of as the minimal
quantitative requirement that ensures pointwise convergence, since the supremum
$\sup_{M\in [I_j, I_{j+1})\cap\DD_{\tau}}$ cannot be removed from the
definition of the oscillation semi-norms, see for instance \cite[Remark 1.2, p.892]{jkrw}.

Working over the integers allows one to exploit the additive structure of $\ZZ$ and Fourier methods, which Bourgain \cite{B1,B2,B3} combined for the first time with the Hardy--Littlewood circle method from analytic number theory to understand the maximal inequality
\eqref{eq:111} and oscillation estimates \eqref{eq:112}.
It was a turning point not only in ergodic theory but also for the
area of discrete analogues in harmonic analysis, as the
Hardy--Littlewood circle method combined with the Fourier methods
developed \cite{DR, SW} in the context of continuous Radon operators
\eqref{eq:109} with \(M_1=\dots=M_k=M\) made progress on problems that classical Calder{\'o}n--Zygmund theory alone could not treat.

\subsubsection{\textbf{Magyar--Stein--Wainger and Ionescu--Wainger era}}

Bourgain's papers \cite{B1, B2, B3} initiated a systematic study of
discrete operators where classical Calder{\'o}n--Zygmund theory alone is not sufficient.
This resulted in many important papers in the field of
discrete analogues in harmonic analysis, including results for
discrete averaging operators \cite{BM, jkrw, Kr, LaV1, MSW, Mes, MSW-survey, MT,
Slomian, Trojan}, for discrete singular integrals \cite{IW, M10,
Pierce0, SW0}, as well as for the operators in the noncommutative
discrete settings \cite{IMMS,IMSW, IMW, MSW1, SW0}. Besides
Bourgain's papers \cite{B1, B2, B3}, there are two papers that have played an important role in the field of discrete analogues in harmonic
analysis. The first is the paper by Magyar, Stein, and Wainger
\cite{MSW}, where the authors established $\ell^p(\ZZ^d)$ boundedness
of discrete spherical maximal functions for spherical averages
over discrete spheres in $\ZZ^d$. This is a striking result
showing that the behavior in the discrete setting, such as the range
of $\ell^p(\ZZ^d)$ boundedness, may be different from its continuous counterpart. More importantly, the authors in
\cite{MSW} secured the sampling principle as a key tool in the field; it
allows us to compare norms of discrete operators to their
continuous analogues as long as the frequencies are sampled from
sufficiently ``narrow'' sets. The sampling principle is 
stated in Proposition \ref{prop:msw} in Section \ref{sec:IW}.

The second paper is by Ionescu and Wainger \cite{IW}, who established
$\ell^p(\ZZ^d)$ boundedness for all $p\in(1, \infty)$ of discrete
singular Radon transforms $R_N^P$; that is, inequalities of the form
\begin{align}
\label{eq:37}
\sup_{N\in\ZZ_+}\|R_N^Pf\|_{\ell^p(\ZZ^d)}\lesssim _{d, k, P}\|f\|_{\ell^p(\ZZ^d)}, \qquad f\in\ell^p(\ZZ^d),
\end{align}
where
\begin{align}
\label{eq:105}
R_N^Pf(x):=\sum_{m\in[-N, N]^k\setminus\{0\}}f(x-P(m))K(m), \qquad x\in\ZZ^d\qquad N\in\ZZ_+,
\end{align}
 where $P:\ZZ^k\to\ZZ^d$ is a
polynomial mapping with integer coefficients and $K$ is a
Calder{\'o}n--Zygmund kernel.  The authors in \cite{IW} developed a
deep multi-frequency multiplier theorem that allowed them to handle
discrete singular Radon transforms for the full range
$p\in(1, \infty)$, which was out of reach previously. Although
\cite{IW} was a  breakthrough at the time and shared the same
starting point as \cite{B1, B2, B3} --- the circle method --- the full
strength of its methods was not yet fully understood.  Moreover,
Stein's question --- asking about the existence of a discrete
Littlewood--Paley theory that would allow treating discrete averaging
Radon operators as well as discrete singular Radon transforms in a
unified way, as in the continuous setting --- remained open.

\subsubsection{\textbf{Discrete Littlewood--Paley theory}}
This situation changed when the fourth author \cite{M10} provided a
new proof of inequality \eqref{eq:37}, highlighting that the
Ionescu--Wainger multi-frequency multiplier theorem, as stated in
Theorem \ref{thm:iw-semi} in Section \ref{sec:IW}, can be in fact
interpreted as a multi-frequency Littlewood--Paley theory that captures
all the Diophantine features arising from the Hardy--Littlewood circle
method.  This, on the one hand, gave a positive answer to Stein's
question, about the existence of a discrete Littlewood--Paley
theory. On the other hand, it resulted in many papers on the
$r$-variational \cite{MST1, MST2}, jump \cite{MSZ3}, and uniform
oscillation \cite{MSS} inequalities for the Radon averaging operators
\eqref{eq:98} and truncated variants of discrete singular Radon
transforms \eqref{eq:105}. The most important outcome of the efforts undertaken in
\cite{MST1, MST2, MSZ3, MSS} is that all these problems --- whether
for the Radon averaging operators or for truncated discrete singular
Radon transforms --- can be handled in a unified framework, with the
Rademacher--Menshov inequality, see for instance inequality
\eqref{eq:164} with its analogues for $r$-variations or jumps
\cite{MT, MSZ3}, as a key tool to handle the frequencies where the
sampling principle from \cite{MSW} becomes ineffective.

The Ionescu--Wainger multi-frequency multiplier theorem \cite{IW}, 
a key tool in the field of discrete analogues in harmonic analysis,
itself went through a period of considerable change and
development. Originally, it was proved for the so-called
Ionescu--Wainger fractions, a set of reduced rational fractions with
denominators having certain intricate factorization properties, which
exploited strong orthogonality phenomena in the proof and
ultimately controlled the norm of the inequality in the statement of
the Ionescu--Wainger multiplier theorem in terms of the size of the
family of Ionescu--Wainger fractions. Improvements on the size of the
norm and simpler proofs of the Ionescu--Wainger multiplier theorem
were subsequently given in \cite{M10, MSZ3} and
\cite{Pierce}. Finally, Tao \cite{TaoIW} provided a proof with
a norm independent of the size of the family of Ionescu--Wainger
fractions.

Although the Ionescu--Wainger multiplier theorem turned out to be an
invaluable tool for handling discrete operators with arithmetic
features, its formulation based on the family of Ionescu--Wainger
fractions is a disadvantage due to its involved
construction and misalignment with the so-called set of canonical
fractions (see definition \eqref{IWeq:372}), which naturally arises in
applications of the classical circle method of Hardy and Littlewood.
Fortunately, this has recently changed due to the collaboration of the
fourth and last authors with Kosz, Peluse, and Wan \cite{KMPWW}, who
answered a question by Ionescu and
Wainger in \cite{IW}, establishing
a variant of the Ionescu--Wainger multiplier theorem for canonical fractions --- this is
Theorem \ref{thm:iw-semi} in Section \ref{sec:IW}.   Theorem \ref{thm:iw-semi} has  also had
an important impact on proving a multilinear variant of Weyl's
inequality, which was key to Bergelson's conjecture \cite{KMPWW} for
multilinear ergodic averages with commuting transformations and
polynomial iterates, where the underlying polynomials have distinct
degrees. The details of how Theorem \ref{thm:iw-semi} can be used to
interpret the classical Weyl inequality for exponential sums (see
Theorem \ref{thm:weyl1par}) in $\ell^p(\mathbb{Z})$ spaces for
$p \in (1, \infty)$ can be found in \cite{M-ICM}. Surprisingly, the dependence of the norm
in \eqref{eq:376} on the size of the set of canonical fractions cannot
be removed, which stands in sharp contrast to Tao's result
\cite{TaoIW}. We refer the reader to \cite{KMSZ} for more details and discussion.

Now, keeping in mind this brief outline of the one-parameter theory of
discrete analogues in harmonic analysis, we turn to the
multi-parameter setting, which is the subject of the paper.

\subsubsection{\textbf{Multi-parameter Bellow--Furstenberg problem}}
After completing \cite{B1, B2, B3}, Bourgain immediately
observed that his maximal \eqref{eq:111} and oscillation \eqref{eq:112} estimates imply that for any polynomials $P_1,\ldots, P_d\in\ZZ[{\rm m}]$ and any $f\in L^p(X)$ with $p\in(1, \infty)$, the limit
\begin{align}
\label{eq:4}
\lim_{\min(M_1,\ldots,M_d) \to \infty}\EE_{(m_1,\ldots, m_d)\in Q_M^d} f(T_1^{P_1(m_1)}  \cdots  T_d^{P_d(m_d)} x) 
\end{align}
exists pointwise $\mu$-almost everywhere on any $\sigma$-finite
measure space \((X, \mathcal{B}(X), \mu)\) endowed with commuting
invertible measure-preserving maps \(T_1,\ldots, T_d\). This is a
generalization to polynomial iterates of the Dunford \cite{D} and
Zygmund \cite{Z} result who independently proved \eqref{eq:4} whenever
$P_j(m_j)=m_j$ for all $j\in[d]$, simultaneously demonstrating that when, $d\ge 2$, the
pointwise convergence result is manifestly false for general
$f\in L^1(X)$ like in Saks' observation \cite{Saks} for $A_{M; \RR^d}^{{\rm t}_1,\ldots, {\rm t}_d}f$. Interestingly, the result in Dunford \cite{D}
and Zygmund \cite{Z} 
does not require the underlying measure-preserving
transformations to commute. While Bourgain's argument relied on the maximal \eqref{eq:111} and oscillation \eqref{eq:112} estimates, which required the underlying measure-preserving transformations to commute, Wierdl \cite{Wierdl-priv} pointed out that by a simple argument (relying on the one-parameter maximal estimates and dominated convergence theorem) this assumption can be relaxed and \eqref{eq:4} holds for arbitrary invertible measure-preserving transformations $T_1,\ldots, T_d$, not necessarily commuting.
It is striking that
commutativity is not needed in \eqref{eq:4} and contrasts sharply with
Weiss's observation in \cite{B1, BMSW}; namely that the one-parameter averages
$\mathbb{E}_{m\in [M]} f(T_1^{P_1(m)} \cdots T_d^{P_d(m)} x)$ in \eqref{eq:107} with $k=1$ and $d\ge 2$, may fail to
converge if \(T_1,\ldots, T_d\) do not commute.  Bourgain has never
published \eqref{eq:4}, although the key points of his proof are
presented in \cite{MSW-survey}.

Bourgain's interest in multi-parameter phenomena arose after completing \cite{B1}, where he proved \eqref{eq:111} for $p=2$. Initially, Bourgain had hoped that the technique of doubling variables in exponential sum estimates, combined with Vinogradov's method, would yield \eqref{eq:111} for all $p>1$. This turned out to be only partially true, as shown in \cite{B2}. The remaining limitations were ultimately overcome in \cite{B3} by relying only on classical exponential sum estimates (that give a tiny amount of decay) together with the introduction of his multi-frequency logarithmic lemma.
Nevertheless, the technique of doubling variables in exponential sums, which gives rise to multi-parameter phenomena, became an important catalyst for the profound 1979 work of Arkhipov, Chubarikov, and Karatsuba \cite{ACK-fractional} and \cite[Chapter~6]{ACK}. In that work, Weyl's equidistribution theorem \cite{Weyl} was extended to families of multivariate polynomials
$P_1,\ldots,P_d\in\ZZ[{\rm m}_1,\ldots,{\rm m}_k]$
considered over unrestricted regions
$(m_1,\ldots,m_k)\in [M_1]\times\cdots\times[M_k]$
under the assumption that
$\min(M_1,\ldots,M_k)\to\infty$. The latter led Bourgain to the observation in \eqref{eq:4}. However,  Bourgain was hoping that a
somewhat iterative application of his methods \cite{B1, B2, B3} would
result in a generalization of Arkhipov, Chubarikov and Karatsuba's
equidistribution theorem in the spirit of Khintchin's result from
\eqref{eq:5}, but it had resisted his methods for three decades and led
to the multi-parameter Bellow--Furstenberg problem that has been
formulated above for the averages \eqref{eq:107}.

Even for the averages
$\mathbb{E}_{(m_1, m_2)\in [M_1]\times[M_2]} f(T^{{m}_1^2 {m}_2^3} x)$,
the multi-parameter Bellow--Furstenberg problem (i.e. $d=1$, $k=2$ and
$P(m_1, m_2)=m_1^2m_2^3$) becomes very challenging, though it does not
fully reveal the difficulties that arise for genuinely multi-parameter
averages \eqref{eq:107}.  Surprisingly, in contrast to the continuous
setting, it seems that there is no simple way (like changing variables
or interpreting the average
$\mathbb{E}_{(m_1, m_2)\in [M_1]\times[M_2]} f(T^{{m}_1^2{m}_2^3}x)$
as a composition of simpler one-parameter averages as in \eqref{eq:4})
that would help reduce the matter to the setup where pointwise
convergence is known.  Ultimately, the multi-parameter Bellow--Furstenberg problem
for $d=1$ and $k\in\ZZ_+$ 
was 
solved in \cite{BMSW}. The methods from \cite{BMSW} are limited to the
case \(d=1\); here we propose a new, more
efficient approach to handle the general case, which we describe
below.

\subsection{Multi-parameter circle method: overview of the proof}
From now on, it is convenient to restrict attention to truncated
versions of the average \(A_{M; \mathbb{Z}^d}^{P}\) from
\eqref{eq:38}, which for any polynomial mapping $P:\ZZ^k\to\ZZ^d$, any
$f\in\ell^1(\ZZ^d)$ and $x\in\ZZ^d$, is given by
\begin{align}
\label{eq:3}
\tilde{A}_{M; \ZZ^d}^{P}f(x):=\EE_{m\in R_M^k}f(x-P(m)),
\end{align}
where
\[
R_{M}^k:= Q_{\tau M}^k\setminus Q_M^k=([\tau M_1]\setminus[M_1])\times\cdots\times([\tau M_k]\setminus[M_k]).
\]
is a truncated rectangle with scale $M=(M_1,\ldots, M_k)\in[1, \infty)^k$  and  $\tau\in(1, \infty)$.

Instead of working with the averages \(A_{M; \mathbb{Z}^d}^{P}\), we
will work with their truncations \(\tilde{A}_{M; \mathbb{Z}^d}^{P}\);
this is solely a technical reason that will simplify many arguments in
the paper. 
We will prove the following oscillation inequality for the discrete
truncated Radon averages from \eqref{eq:3}.
\begin{theorem}
\label{thm:osc}
Let $d, k\in\ZZ_+$ be given  and let $P=(P_1,\ldots, P_d):\ZZ^k\to\ZZ^d$ be a polynomial mapping
with $P_j\in\ZZ[{\rm m}_1, \ldots, {\rm m}_k]$ for $j\in[d]$.  Let  $\tilde{A}_{M; \ZZ^d}^{P}f$ be the truncated multi-parameter Radon average defined in
\eqref{eq:3} with $M=(M_1,\ldots, M_k)\in[1, \infty)^k$. Define $\DD_{\tau}:=\{\tau^n:n\in\NN\}$ for any $\tau\in(1, \infty)$.
Then for  every $p\in(1, \infty)$ there exists  $C_{d, k, p, \tau}\in\RR_+$ such that for every $f\in \ell^p(\ZZ^d)$, one has
\begin{align}
\label{eq:103}
\sup_{J\in\ZZ_+}\sup_{I\in\mathfrak S_J(\DD_{\tau}^k) }\big\|O_{I, J}(\tilde{A}_{M;\ZZ^d}^{P}f: M \in\DD_{\tau}^k)\big\|_{\ell^p(\ZZ^d)}\le C_{d, k, p, \tau}\|f\|_{\ell^p(\ZZ^d)}.
\end{align}
We refer to Section \ref{section:2} for definitions of the oscillation seminorm $O_{I, J}$ in \eqref{eq:102} and the family $\mathfrak S_J(\DD_{\tau}^k)$ in \eqref{eq:96}. 
Moreover, the implied constant $C_{d, k, p,\tau}$ in \eqref{eq:103}, may depend on the degree
of the underlying polynomial mapping $P$, but is independent of its
coefficients.
\end{theorem}

For ease of exposition, we prove Theorem \ref{thm:osc} in the
two-parameter setting $k=2$. However, a careful reader will readily
see that all two-parameter arguments can be adapted to the general
multi-parameter setting for arbitrary $k \ge 2$, albeit at the cost of
introducing cumbersome notation that would make the exposition rather
difficult to read and the arguments harder to digest.

The proof of Theorem \ref{thm:osc} will take up the bulk of this paper, so a few remarks are in order.

\begin{enumerate}[label*={\arabic*}.]

\item The range of parameters $p\in(1, \infty)$ for which inequality
\eqref{eq:103} holds is sharp.

\smallskip

\item Theorem \ref{thm:osc} is the main result of the paper and
implies Theorem \ref{thm:1}. Namely, inequality \eqref{eq:103} implies
inequality \eqref{eq:88}. This follows from \cite[Proposition 3.7,
p. 19]{BMSW}, which asserts that there exists a constant
$C:=C_{d, k, p,\tau}\in\RR_+$ such that for every $f\in \ell^p(\ZZ^d)$, we have 
\begin{align}
\label{eq:104}
\begin{split}
\sup_{J\in\ZZ_+}&\sup_{I\in \mathfrak S_J(\DD_{\tau}^k)}
\big\|O_{I, J}(A_{M; \ZZ^d}^{P}f:M\in\DD_{\tau}^k)\big\|_{\ell^p(\ZZ^d)}\\
&\le C\sup_{J\in\ZZ_+}\sup_{I\in \mathfrak S_J(\DD_{\tau}^k)}
\big\|O_{I, J}(\tilde{A}_{M; \ZZ^d}^{P}f:M\in\DD_{\tau}^k)\big\|_{\ell^p(\ZZ^d)}+C\|f\|_{\ell^p(\ZZ^d)}.
\end{split}
\end{align}
Now, if we combine \eqref{eq:103} with \eqref{eq:104}, we readily
obtain \eqref{eq:88}. It is also not difficult to see that inequality
\eqref{eq:88} implies inequality \eqref{eq:103}. Therefore,
inequalities \eqref{eq:88} and \eqref{eq:103} are equivalent in the
sense that the bounds for one imply the bounds for the other. Hence, in view of inequality \eqref{eq:104} it suffices to
prove the oscillation inequality \eqref{eq:103}.

\smallskip
\item Similar to the remark  after Theorem \ref{thm:1}, the
oscillation inequality \eqref{eq:103} implies the corresponding
maximal bound. Indeed, combining inequality \eqref{eq:103} with
inequality \eqref{eq:135} below (with $\II=\DD_{\tau}^k$) we
deduce that for every $p\in(1, \infty]$ there exists
$C_{d, k, p, \tau}\in\RR_+$ such that for every function
$f\in \ell^p(\ZZ^d)$, we have
\begin{align}
\label{eq:101}
\big\|\sup_{M\in\ZZ_+^k}|\tilde{A}_{M; \ZZ^d}^{P}f|\big\|_{\ell^p(\ZZ^d)}\le C_{d, k, p, \tau}\|f\|_{\ell^p(\ZZ^d)}.
\end{align}
By \cite[Proposition
3.7, p. 19]{BMSW} an analogue of \eqref{eq:104} for maximal functions in
place of oscillations also holds. Hence, inequality \eqref{eq:101}, with say
$\tau=2$, implies inequality \eqref{eq:88}.
\smallskip

\item Inequality \eqref{eq:103} is sometimes called the \textit{long
multi-parameter oscillation inequality} due to the fact that the
parameters are taken along lacunary sequences. If inequality
\eqref{eq:103} holds with \(\mathbb{Z}_+^k\) in place of \(\DD_{\tau}^k\)
it is called the \textit{full
multi-parameter oscillation inequality}.

\smallskip

\item In view of the previous remark, we emphasize that, at the
expense of additional work, one can prove a full variant of the oscillation
inequality \eqref{eq:103}, i.e. \eqref{eq:103} with \(\mathbb{Z}_+^k\)
in place of \(\DD_{\tau}^k\). This, in particular, implies that
inequality \eqref{eq:103} can be proved with a constant independent of
\(\tau>1\). Although this is interesting, we will not address it in this paper, as establishing
inequality \eqref{eq:103} itself is a fairly challenging task and will
entirely suffice for our applications in answering in the affirmative
the multi-parameter Bellow--Furstenberg problem, see Problem
\ref{BMSW-conj} and the discussion below.

\smallskip

\item In \cite{KLMP}, the third and fourth authors, together with Kosz
and Plewa, proved a continuous analogue of Theorem \ref{thm:osc},
where a multi-parameter oscillation inequality similar to
\eqref{eq:103} was established with $\tilde{A}_{M;\mathbb{Z}^d}^{P}$
replaced by $A_{M;\mathbb{R}^d}^{P}$ from \eqref{eq:109}, and with the
$L^p(\mathbb{R}^d)$ spaces with Lebesgue measure in place of the
discrete $\ell^p(\mathbb{Z}^d)$ spaces with counting measure. The methods developed in
\cite{KLMP} allow us to handle full multi-parameter oscillation
inequalities in the continuous setting, which is very important in
view of the previous remark.

\smallskip

\item Finally, we remark that it suffices to prove Theorem
\ref{thm:osc} for arbitrary family of monomials
$P_1,\ldots, P_d \in \mathbb{Z}[{\rm m}_1,\ldots, {\rm m}_k]$ all
having coefficient $1$. This is a simple consequence of the lifting
procedure, which can, for instance, be adapted from \cite[Section
1.4]{MSZ3}. This remark explains why the implied constant
$C_{d, k, p,\tau}$ in \eqref{eq:103}, depends on the degree of the
underlying polynomial mapping $P$, but is independent of its
coefficients.
\end{enumerate}

We now give a brief overview of the paper and its methods,
highlighting novelties and challenges.  In Section \ref{section:2}
we introduce notation and some technical results needed in the
paper. Most of the results there have already been developed in \cite{BMSW}
and \cite{MSW-survey}.  Section \ref{sec:exp} focuses on estimates for
exponential sums. These estimates rely on the results from
\cite[Section 5]{BMSW}, and no fundamentally new results will be required.
In the remaining sections, we develop what we call the
\textit{multi-parameter circle method} aligned with quantitative
estimates (maximal and oscillation inequalities) for the discrete
multi-parameter Radon averaging operators \eqref{eq:38}.

 The multi-parameter circle method is not new and has
been developed for more than four decades, mainly by Arkhipov,
Chubarikov and Karatsuba \cite{ACK-fractional}, in the context of
Vinogradov's mean value type theorems, generalizations of the
equidistribution problems for multivariate polynomials, or
multi-parameter techniques to estimate exponential sums. However, our
context, though  inspired and influenced by these
achievements in \cite{ACK-fractional}, requires new ideas that, to the
best of our knowledge, have not been apparent before, even in
\cite{BMSW}.

There are two important new pillars of the paper that we now highlight:
\begin{itemize}
\item[1.] In Section \ref{sec:IW}, we develop the multi-parameter
Ionescu--Wainger multiplier theorem in the context of multi-parameter
maximal and oscillation inequalities, see Theorem \ref{thm:IW1'},
which handles products of one-dimensional multipliers, and Theorem
\ref{thm:osc_IW_2par}, which handles the multipliers corresponding to
the multi-parameter continuous Radon averages from \eqref{eq:109}.
These results are crucial for interpreting the circle method in the
operator-theoretic setting.

\item[2.] In Sections \ref{sec:circle}-\ref{sec:l2hard}, we
systematically build our \textit{multi-parameter circle method} in the
context of maximal and oscillation inequalities. Our argument, like in
\cite{ACK}, is based on an iterative application of the classical
circle method, which is the most natural approach to take, but things
become more complicated as the circle method must be interpreted and
implemented in the semi-norm setting of maximal and oscillation
inequalities, which is not the most natural environment to implement the
circle-method. This is where our methods diverge from
the approaches developed in \cite{ACK}, since we cannot work directly
with multi-parameter exponential sums as in equidistribution
problems. The objects we need to handle will be partly defined on the
operator side and partly on the frequency side (reminiscent of
exponential sums), which leads to several new complications. The
remedy to overcome these complications are exponential sums estimates
developed  in \cite[Section 5]{BMSW}. Although
these estimates are very effective in some situations, such as in
\cite{BMSW}, they are not sufficient when one deals with genuinely
multi-parameter cases as in \eqref{eq:38}. The general case requires
developing combinatorial tools like those in Lemma \ref{lem:comb} that
reveal a new, interesting phenomenon which we call the ``\textit{major-arcs
rigidity}'', and which we will describe below.
\end{itemize}

\subsubsection{\textbf{Multi-parameter
Ionescu--Wainger multiplier theorem}}
As we noted above, the one-parameter Ionescu--Wainger theory has been
successful on many fronts and so it is natural to seek an
extension to the discrete multi-parameter context. Section
\ref{sec:IW} establishes two results: Theorem \ref{thm:IW1'} and Theorem
\ref{thm:osc_IW_2par}. The first is more technical and serves as an
important tool for establishing the second, which is the main result
of that section. Both results can be viewed as multi-parameter
semi-norm variants of the Ionescu--Wainger theorem \cite{IW}, or more
precisely of Theorem \ref{thm:iw-semi} and Theorem \ref{thm:IW1}. Here
we explain how these two difficult multi-frequency and multi-parameter
ideas come together to yield Theorem \ref{thm:osc_IW_2par}.

Theorem \ref{thm:osc_IW_2par} is an important ingredient of our proof,
which is applied in the last stage when the final approximating
operator of $\tilde{A}_{M; \ZZ^d}^{P}$ is constructed. The latter
arises from applying the circle method, and on the frequency side
corresponds to the major-arcs periodization of the continuous
multiplier associated with truncations of \eqref{eq:109}; see
\eqref{eq:1060} for this multiplier, with \(\Sigma\) being the
rational fractions with small denominators that are the centers of the
corresponding major arcs. This explains how the multi-frequency and
multi-parameter concepts arise together.

According to the last item of the remark below Theorem \ref{thm:osc}
we may assume that $P$ is a canonical polynomial mapping, i.e. its
entries are monomials all having coefficient one, see \eqref{eq:84} in
Section \ref{section:2} and the corresponding continuous multiplier
$\mathfrak{m}_{M}^{\langle\Gamma\rangle}$ defined in
\eqref{eq:PhiM1M2_def_kpar}; see Section \ref{sec:IW}.  Then the proof
of Theorem \ref{thm:osc_IW_2par} is by induction on the number of
components of $P$ or size of $\Gamma$.

By the inclusion-exclusion formula, we decompose
$\mathfrak{m}_{M}^{\langle\Gamma\rangle} = {\rm main}(\mathfrak{m}_{M}^{\langle\Gamma\rangle}) + {\rm error}(\mathfrak{m}_{M}^{\langle\Gamma\rangle})$,
where the main term
consists of many
terms that are tensor products of lower-dimensional multipliers
\eqref{eq:PhiM1M2_def_kpar} and smooth bumps on the remaining
variables. These terms are handled by induction. There is also one
term consisting purely of bumps, which is handled by Theorem
\ref{thm:IW1'}. The most complicated part is to estimate the error
term
${\rm error}(\mathfrak{m}_{M}^{\langle\Gamma\rangle})$. Fortunately,
the error term possesses cancellation and can be handled by 
multi-frequency Littlewood--Paley theory.

The basic idea is to control the oscillations corresponding to the
error term by a square function as in \eqref{eq:91}. It is also
important that we work with \(\DD_\tau\) not the entire \(\mathbb{Z}_+\).
In practice, this is not so easy, as we may encounter the so-called
``\textit{parameters-gluing}'' phenomenon, a certain obstruction which
prevents a naive application of the Littlewood--Paley theory to
control oscillations. To be more precise, since we are working
with a normalized monomial polynomial mapping \(P\) with \(d\) components, they determine certain
``\textit{directions}''. A naive thought suggests that for each
direction one should perform an appropriate Littlewood--Paley
decomposition to control the oscillations by a square function. This
idea is efficient if we have \(d\) nondegenerate directions, but fails
if the number of directions is less than \(d\).

To overcome this difficulty we introduce an auxiliary parameter \(r\),
the rank of a matrix of size \(d\times k\) whose rows are the
exponents of the monomials from \(P\). There is a one-to-one
correspondence between these exponents and the
``\textit{directions}'', and the number of nondegenerate directions is
\(r\le \min(d,k)\). Hence, \(r\) measures how many exponents or
directions have been glued together and explains the
``\textit{parameters-gluing}'' phenomenon. This is precisely stated in
Lemma \ref{lem:long_only}. The idea of the proof of Theorem
\ref{thm:osc_IW_2par} comes from \cite{KLMP}, but here it is slightly
simpler as we only work with long oscillations along
\(\DD_\tau\). Some of these ideas are also apparent in Ricci and
Stein's paper \cite{RS}, but only in the context of multi-parameter
maximal functions, and they are not applicable to the multi-parameter
oscillation semi-norm setting that we discuss here.

It is also clear from the proof that the
method presented here is not only limited to the multipliers
associated with the multi-parameter averaging Radon operators
\eqref{eq:109}; it can be applied to all multipliers that impose some
regularity that can be abstracted from the proof.

\subsubsection{\textbf{Multi-parameter
circle method: challenges in discrete settings}}
The proof of the multi-parameter oscillation inequality from
\eqref{eq:103} is quite intricate. However, serious difficulties arise
only in case \(d\ge 2\) and \(k\ge 2\). The case \(d=1\) and
\(k\ge 2\) was understood in \cite{BMSW}, but at the end of the day it
had many features similar to the one-parameter case
even though it does not directly follow from the one-parameter
case. The reason for this lies behind the special form of the
multiplier corresponding to the average in \eqref{eq:38} for \(d=1\)
and \(k\ge 2\), which is a normalized exponential sum of the form
\(\mathbb{E}_{m\in Q_M^k} e(\xi P(m))\), where
\(P\in \mathbb{Z}[{\rm m}_1,\ldots, {\rm m}_k]\) and
\(\xi\in \mathbb{T}\). Here the key feature is  that the frequency variable \(\xi \in \mathbb{T}\) is one-dimensional. Then, using the geometry of the Newton diagrams corresponding to \(P\), we can split the time-parameter space \(\DD_{\tau}^k \ni M=(M_1,\ldots, M_k)\) into finitely many cones, each associated with a leading monomial that dominates \(P\) in this cone. Then in each fixed cone we perform the major/minor arc analysis with the scales determined by the leading monomial in that cone.

The advantage of this approach is twofold:
\begin{itemize}
\item[1.] The minor-arcs as well as the major-arcs approximations'
contributions in that cone always produce a decay that is summable
with respect to the largest parameter of the time-parameter vector
\(M=(M_1,\ldots, M_k)\). Therefore, the corresponding terms can be easily
handled by simple square-function arguments involving \eqref{eq:91}.

\item[2.] Finally, we end up with a major-arcs periodization of the
continuous multiplier associated with \eqref{eq:109} (for \(d=1\) and
\(k\ge 2\), whose phase function is a leading monomial in that
cone). By a simple change of variables, the problem becomes
one-parameter.
\end{itemize}

This is the approach in the case \(d=1\) and \(k\ge 2\) that was
undertaken and successful in \cite{BMSW}. If \(d\ge 2\) and
\(k\ge 2\), this approach immediately fails. To see this, suppose that
\(d=k=2\) and consider the following polynomial mapping
\begin{align}
\label{eq:99}
P(m_1, m_2) := (m_1^2 m_2^4, m_1^3 m_2^4).
\end{align}
 In this case the
corresponding multiplier is a normalized exponential sum of the form
\begin{align}
\label{eq:8}
\EE_{(m_1, m_2)\in[M_1]\times [M_2]}e(\xi_1 m_1^2 m_2^4+\xi_2 m_1^3 m_2^4), \qquad (\xi_1,\xi_2)\in\TT^2.
\end{align}
The first complication arises on minor arcs, since the multi-parameter
(two-parameter) Weyl inequality, when \((\xi_1,\xi_2)\) lies in the two-parameter 
minor arcs, yields only a negative power of \(\min(M_1,M_2)\),
regardless of how good the major arcs are chosen, see for instance
inequality \eqref{eq:86} below. This prevents us from handling minor arcs by
square-function techniques using inequality \eqref{eq:91}. The reason is that the frequency variable
\((\xi_1,\xi_2)\in \mathbb{T}^2\) is two-dimensional, so the idea of
extracting the leading monomial by studying the Newton diagrams is inefficient
here. Instead of relying on the geometry of Newton
diagrams as in \cite{BMSW}, we will apply the classical circle method
iteratively, similar to Arkhipov, Chubarikov and Karatsuba in
\cite{ACK}.

This also illustrates the difference between the exponential sums
\eqref{eq:8} and their oscillatory integrals analogues, namely by
Proposition \ref{thm:CCW} applied to the polynomial in
\eqref{eq:99}, we have
\begin{align}
\label{eq:69}
\Big|\int_{[0, 1]^2}e((\xi_1, \xi_2)\cdot P(M_1t_1, M_2t_2))dt_1dt_2\Big|\le C_{2,2}\big(M_1^2M_2^4|\xi_1|+M_1^3M_2^4|\xi_2|\big)^{-1/7}.
\end{align}
If \((\xi_1,\xi_2)\in \mathbb{T}^2\) satisfies
\(M_1^2 M_2^4 |\xi_1| \ge \max(M_1, M_2)^{\delta}\) or
\(M_1^3 M_2^4 |\xi_2| \ge \max(M_1, M_2)^{\delta}\) for some small
\(\delta \in (0,1)\), (i.e. $(\xi_1,\xi_2)$ belongs to the continuous
minor arc), then we gain a negative power of
\(\max(M_1, M_2)^{\delta}\), allowing us to efficiently run
square-function arguments. This was crucial to derive the
multi-parameter oscillation inequalities for the continuous Radon
averages from \(\eqref{eq:109}\) in \cite{KLMP}. This is a serious
difference, which makes the discrete multi-parameter analogues much more
difficult.

\subsubsection{\textbf{Multi-parameter circle method: major-arcs rigidity}}
The remedy to compensate for the lack of
decay in terms of \(\max(M_1,M_2)\) in Weyl's inequality for \eqref{eq:8} are
counting arguments (see for example, Lemma \ref{lem:comb} below) that lead to
the ``\textit{major-arcs rigidity}'' phenomena.

To see this, we continue to analyze the polynomial from \eqref{eq:99} and the obstacles encountered above. This  suggests a strategy for proving \eqref{eq:103}: for a small \(\varrho \in (0,1)\), we consider two cases:
\begin{itemize}
\item[1.] The large-scale regime, when  \(M_2^{\varrho} \le M_1 \le M_2\).
\item[2.] The small-scale regime, when  \(M_1 \le M_2^{\varrho}\).
\end{itemize}

In the
first case, the proof of inequality \eqref{eq:103} is somewhat
straightforward, and goes much the same way as arguments in
\cite{BMSW}. The large-scale regime
\((M_2^{\varrho} \le M_1 \le M_2)\), which, in fact, is a regime where
the scales are comparable upon taking logarithms ensures that Weyl's
inequality from Theorem \ref{thm:weyl2par} is very efficient. In this
case, the contribution from minor arcs, as well as the approximation
on major arcs, produce a decay with respect to the larger parameter
\(M_2\), which is easily handled by simple two-parameter square
function arguments. We end up with the approximating operator corresponding to the major-arcs periodization of the
continuous multiplier 
\begin{align}
\label{eq:100}
\int_{[0, 1]^2} e\big((\xi_1, \xi_2)\cdot P(M_1 t_1, M_2 t_2)\big) \, dt_1 dt_2.
\end{align}
 This, in turn, can be handled by Theorem \ref{thm:osc_IW_2par}, as in Section \ref{sec:circle}, and we are done.

In the second case, the proof of inequality \eqref{eq:103} is quite
intricate. Working in the small-scale regime
\((M_1 \le M_2^{\varrho})\), the parameters are no longer comparable,
and one has to apply the classical circle method iteratively, variable
by variable, which is a technically demanding process. Here instead of aiming
at square-function arguments summable in two parameters, we aim (using \eqref{eq:95}) at a
square-function argument with respect to the smaller parameter and a
maximal-function argument with respect to the larger parameter, with
the norm summable with respect to the smaller one.

The starting point is to apply Weyl's inequality to \eqref{eq:8} with
respect to the largest parameter \(M_2\). Then one must understand the
behavior of the multiplier along $m_2^4$ whose coefficient is
$\xi_1 m_1^2 +\xi_2 m_1^3$, i.e.  a polynomial 
corresponding to the smaller scale \(M_1\). Then we encounter  the ``\textit{major arcs rigidity}'' phenomena, which can be
formulated as the following dichotomy: 
\begin{itemize}
\item[(i)] The coefficient of the polynomial $m_2^4$ corresponding to the largest scale \(M_2\) --- which is $\xi_1 m_1^2 +\xi_2 m_1^3$, a polynomial with variables corresponding to the smaller scale \(M_1\) --- belongs to the major arcs very rarely, and we gain a power saving with respect to the smaller scales \(M_1\), which is stated in Lemma \ref{lem:comb}. 

\smallskip

\item[(ii)] Otherwise, the original  polynomial mapping
\(P(m_1, m_2) = (m_1^2 m_2^4, m_1^3 m_2^4)\) is equidistributed in the
sense that its frequencies \((\xi_1, \xi_2)\) on the Fourier transform
side are all well-structured and must belong to the two-parameter
major arcs.  This  also shows that the multiplier from
\eqref{eq:8} can be approximated by the major-arcs periodization of
\eqref{eq:100} as desired.
\end{itemize}

We have only presented a simplified model involving \eqref{eq:99} to illustrate the ``\textit{major arcs rigidity}'' phenomena. In general it is much more complicated, and all details of this procedure are given in Sections \ref{sec:circle}-\ref{sec:l2hard}. This kind of phenomenon was not apparent in \cite{BMSW} since the case considered there, as we explained above, was closer to the one-parameter theory. To be more precise, the small-scale regime \((M_1 \le M_2^{\varrho})\) did not occur in \cite{BMSW}. This only arises when one studies a genuinely multi-parameter case \(d\ge 2\) and \(k\ge 2\), as we highlighted above.


\section{Notation and useful tools}\label{section:2}
We now set up notation and terminology that will be used throughout the paper. 

\subsection{Basic notation}  The set of positive integers and nonnegative
integers will be denoted respectively by $\ZZ_+:=\{1, 2, \ldots\}$ and
$\NN:=\{0,1,2,\ldots\}$. 
For $d\in\ZZ_+$ the sets $\ZZ^d$, $\QQ^d$, $\RR^d$, $\CC^d$ and $\TT^d:=\RR^d/\ZZ^d$
have standard meaning. The cardinality of a set $A$ is denoted by $|A|$.  Moreover, for a finite set $A$, we will write $\XX^A=\XX^{|A|}$, whenever $\XX\in\{\ZZ, \QQ, \RR, \CC, \TT\}$.

For any $x\in\RR$ we will use the integer and fractional part functions
\begin{align*}
\lfloor x \rfloor := \max\{ n \in \ZZ : n \le x \},
\qquad \text{ and } \qquad
\{x\}:=x-\lfloor x \rfloor.
\end{align*}
We denote $\RR_+:=(0, \infty)$ and
for every $N\in\RR_+$ we set
\[
[N]:=(0, N]\cap\ZZ=\{1, \ldots, \lfloor N\rfloor\},
\]
and we will also write
\begin{align}\nonumber
\NN_{\le N}:= [0, N]\cap\NN,\ \: \quad &\text{ and } \quad
\NN_{< N}:= [0, N)\cap\NN,\\
\NN_{\ge N}:= [N, \infty)\cap\NN, \quad &\text{ and } \quad
\NN_{> N}:= (N, \infty)\cap\NN. \label{eq:Nless}
\end{align}
For any real number $\tau>1$ we define the set
\begin{align*}
\DD_{\tau}:=\{\tau^n: n\in\NN\}. 
\end{align*}

We use $\ind{A}$ to denote the indicator function of a set $A$. If $S$ 
is a statement we write $\ind{S}$ to denote its indicator, equal to $1$
if $S$ is true and $0$ if $S$ is false. For instance $\ind{A}(x)=\ind{x\in A}$.

Throughout the paper $C\in\RR_+$ is an absolute constant which may
change from occurrence to occurrence. For two nonnegative quantities
$A, B$ we write $A \lesssim B$ if there is an absolute constant $C\in\RR_+$
such that $A\le CB$. We will write $A \simeq B$ when
$A \lesssim B\lesssim A$.  We will write $\lesssim_{\delta}$ or
$\simeq_{\delta}$ to emphasize that the implicit constant depends on
$\delta$. 

For a function $f:X\to \CC$ and positive-valued function
$g:X\to (0, \infty)$, we write $f = O(g)$ if there exists a constant
$C\in\RR_+$ such that $|f(x)| \le C g(x)$ for all $x\in X$. We will also
write $f = O_{\delta}(g)$ if the implicit constant depends on
$\delta$.

If $\KK$ is either $\ZZ$ or $\RR$, then
$\KK[{\rm m}_1, \ldots, {\rm m}_k]$ denotes the space of all formal
$k$-variate polynomials $P({\rm m}_1, \ldots, {\rm m}_k)$ with
indeterminates ${\rm m}_1, \ldots, {\rm m}_k$ and coefficients in
$\KK$. Consequently, each polynomial
$P\in\KK[{\rm m}_1, \ldots, {\rm m}_k]$ will be identified with a map
$\KK^k\ni(m_1,\ldots, m_k)\mapsto P(m_1,\ldots, m_k)\in\KK$.

\subsection{Basic number theory}
For a vector $a = (a_1,\ldots, a_n) \in \ZZ^n$ and an integer $q\in\ZZ_+$, we denote by $(a,q)$ the greatest
common divisor of $a$ and $q$; that is, the largest integer $d\in\ZZ_+$ that divides $q$ and all the components
$a_1, \ldots, a_n$ of the vector $a$. Clearly any vector in $\QQ^n$ has a unique representation as $a/q$ with $q\in \ZZ_{+}$,
$a \in \ZZ^n$ and $(a,q)=1$.

For any sequences $(a_m:m\in\NN)\subseteq\CC$ and
$(b_m:m\in\NN)\subseteq\CC$ and any integers $U,V\in\ZZ$ such
that $U< V$ we will use the following version of the summation by
parts formula
\begin{align}
\label{eq:sbpformula}
\sum_{m=U+1}^{V}a_mb_m=S_{V}b_{V+1}-\sum_{m=U+1}^{V}S_m(b_{m+1}-b_m),
\end{align}
where $S_m=\sum_{l=U+1}^ma_l$. 

We will also use a simple deviation of the Dirichlet's principle from \cite[Lemma A.15]{MSZ3}.

\begin{lemma}[]
\label{lem:3}
Let $\theta\in\RR$ and $Q\in\ZZ_+$ be given.
Suppose that $|\theta - {a}/{q}|\le q^{-2}$
with $(a,q)=1$ and $0\le a<q \le M$ for some $M\in\RR_+$.
Then there is a fraction $a'/q'\in\QQ$ so that $(a', q') = 1$ and
\[
\Big|Q \theta - \frac{a'}{q'}\Big|
\le
\frac{1}{2q'M},
\]
with $q/(2Q) \le q' \le 2 M$.
\end{lemma}

\subsection{Euclidean spaces} For  $d\in\ZZ_+$ the set $\Set{e_i\in\RR^d}{i\in[d]}$ denotes the standard basis in
$\RR^d$. The standard inner product and the corresponding Euclidean norm on $\RR^d$ are denoted by 
\begin{align*}
x\cdot\xi:=\sum_{k=1}^dx_k\xi_k, \quad \text{ and } \quad
\abs{x}:=\abs{x}_2:=\sqrt{\ipr{x}{x}}
\end{align*}
for every $x=(x_1,\ldots, x_d)$ and $\xi=(\xi_1, \ldots, \xi_d)\in\RR^d$. 

Throughout the paper the $d$-dimensional torus $\TT^d = (\RR/\ZZ)^d$, which unless otherwise stated will be  identified with
$[-1/2, 1/2)^d$, is a priori endowed with the   periodic norm
\begin{align*}
\norm{\xi}:=\Big(\sum_{k=1}^d \norm{\xi_k}^2\Big)^{1/2}
\qquad \text{for}\qquad
\xi=(\xi_1,\ldots,\xi_d)\in\TT^d,
\end{align*}
where $\norm{\xi_k}=\dist(\xi_k, \ZZ)$ for all $\xi_k\in\TT$ and
$k\in[d]$.  However, identifying $\TT^d$ with $[-1/2, 1/2)^d$, we
see that the norm $\norm{\:\cdot\:}$ coincides with the Euclidean norm
$\abs{\:\cdot\:}$ restricted to $[-1/2, 1/2)^d$.

Finally, we will use a convenient convention: for $t\in \RR$  and  $n=(n_1,\ldots, n_d) \in \ZZ^d$ we  define  $t^n \coloneqq  (t^{n_1}, \dots, t^{n_d})\in \RR^d$. This will be clear and cause no confusion throughout. 
We shall also  abbreviate the $d$-tuples $(0,\dots,0)$ and $(1,\dots,1)$ to $\bm 0$ and $\bm 1$ respectively.

\subsection{Smooth functions} The partial derivative of a differentiable function $f:\RR^d\to\CC$ with respect to the
$j$-th variable $x_j$ will be denoted by
$\partial_{x_j}f$ or $\partial_j f$, while for any multi-index
$\alpha\in\NN^d$ let $\partial^{\alpha}f$ denote the derivative operator 
$\partial^{\alpha_1}_{x_1}\cdots \partial^{\alpha_d}_{x_d}f=\partial^{\alpha_1}_1\cdots \partial^{\alpha_d}_df$
of total order $|\alpha|:=\alpha_1+\ldots+\alpha_d$.

Let  $\eta:\RR\to[0, 1]$ be a smooth and even cutoff function such that
\begin{align}
\label{eq:eta}
\ind{[-1/4, 1/4]}\le\eta\le \ind{[-1/2, 1/2]}.
\end{align}
For any $n, \xi\in\RR$, we set $\eta_{\le n}(\xi):=\eta(2^n \xi)$.
For any $\xi=(\xi_1,\ldots, \xi_d)\in\RR^d$ and
$i\in[d]$, we also define
\begin{align}\label{eq:etai_def}
\eta_{\le n}^{\{i\}}(\xi):=\eta_{\le n}(\xi_i).
\end{align}
More generally, let $m\in[d]$, and for any
$A=\{i_1,\ldots, i_m\}\subseteq [d]$, and a vector
$n=(n_{i_1},\ldots, n_{i_m})\in\RR^m$ corresponding to the set $A$, we
will write
\begin{align}
\label{eq:789}
\eta_{\le n}^{A}(\xi):=\prod_{j\in[m]}\eta_{\le n_{i_j}}^{\{i_j\}}(\xi)=\prod_{j\in[m]}\eta_{\le n_{i_j}}(\xi_{i_j}).
\end{align}
We will abuse the notation and abbreviate  $\eta_{\le n}^{A}$  to $\eta_{\le k}^{A}$ whenever $n=(n_{i_1},\ldots, n_{i_m})=(k,\ldots, k)$.

\subsection{Function spaces}
All vector spaces in this paper will be defined over the field of complex numbers $\CC$. 
The triple $(X, \mathcal B(X), \mu)$
is a measure space $X$ with $\sigma$-algebra $\mathcal B(X)$ and
$\sigma$-finite measure $\mu$.  The space of all $\mu$-measurable
complex-valued functions defined on $X$ will be denoted by $L^0(X)$.
The space of all functions in $L^0(X)$ whose modulus is integrable
with $p$-th power is denoted by $L^p(X)$ for $p\in(0, \infty)$,
whereas $L^{\infty}(X)$ denotes the space of all essentially bounded
functions in $L^0(X)$.
These notions can be extended to functions taking values in a finite
dimensional normed vector space $(B, \|\cdot\|_B)$, for instance
\begin{align*}
L^{p}(X;B)
:=\big\{F\in L^0(X;B):\|F\|_{L^{p}(X;B)} \coloneqq \left\|\|F\|_B\right\|_{L^{p}(X)}<\infty\big\},
\end{align*}
where $L^0(X;B)$ denotes the space of measurable functions from $X$ to
$B$ (up to almost everywhere equivalence). Of course, if $B$ is
separable, these notions can be extended to infinite-dimensional
$B$. In this paper, we will always be able to work in
finite-dimensional settings by appealing to  standard approximation arguments.
In our case we will usually take  $X=\RR^d$ or
$X=\TT^d$ equipped with Lebesgue measure, and   $X=\ZZ^d$ endowed with 
counting measure. If $X$ is endowed with counting measure we will
abbreviate $L^p(X)$ to $\ell^p(X)$ and $L^p(X; B)$ to $\ell^p(X; B)$.

If $T : B_1 \to B_2$ is a continuous linear  map between two normed
vector spaces $B_1$ and  $B_2$, we use $\|T\|_{B_1 \to B_2}$ to denote its
operator norm.

The following extension of the Marcinkiewicz--Zygmund inequality to
the Hilbert space setting, whose proof can be found in \cite{MST1},  will be very useful in Section \ref{sec:IW}.
\begin{lemma}
\label{lem:10}
Let $(X, \mathcal B(X), \mu)$ be a $\sigma$-finite measure space endowed with
a family $T=(T_m: m\in\NN)$ of bounded linear operators
$T_m:L^p(X)\to L^p(X)$ for some $p\in(0, \infty)$.
 Suppose that 
\begin{align*}
A_p(T):=\sup_{(\omega_m:m\in\NN)\in\{-1, 1\}^{\NN}}\norm[\Big]{\sum_{m\in\NN}\omega_mT_m}_{L^p\to L^p}<\infty.
\end{align*}
Then there is a constant $C_p\in\RR_+$ such that for every sequence
$(f_j: j\in\NN)\in L^p(X;\ell^2(\NN))$ we have
\begin{align}
\label{eq:175}
\norm[\Big]{\big(\sum_{j\in\NN}\sum_{m\in\NN}\abs{T_mf_j}^2\big)^{1/2}}_{L^p(X)}
\le C_pA_p(T)
\norm[\Big]{\big(\sum_{j\in\NN}\abs{f_j}^2\big)^{1/2}}_{L^p(X)}.
\end{align}
The index set $\NN$ in  \eqref{eq:175} can be
replaced by any other countable set and the result  remains valid.
\end{lemma}

\subsection{Convolutions and the Fourier transform}
Let $(\GG, +)$ be a locally compact abelian group (LCA group) equipped
with a Haar measure $\lambda_{\GG}$. A convolution of  two functions
$f, g\in L^1(\GG)$ is defined  by
\begin{align}
\label{eq:2100}
f*_{\GG}g(x) \coloneqq f*g(x) \coloneqq\int_{\GG}f(x-y)g(y)d\lambda_{\GG}(y).
\end{align}
For the sake of simplicity, we will abbreviate $*_{\GG}$ to
$*$, it will be clear from the context and there should be no
confusion.  Since $\GG$ is abelian we readily see that $f*g=g*f$.

Although we will use Fourier analysis only on $\RR^d$, $\TT^d$ or
$\ZZ^d$, it will be convenient to set out some abstract harmonic
analysis notation to perform this analysis in a unified fashion and to
avoid repetition. We shall write $e(z)=e^{2\pi {\bm i} z}$ for every
$z\in\CC$, where ${\bm i}^2=-1$.

It is well known (see for instance \cite{Rudin}), that every LCA group $\GG$ has a Pontryagin dual $\hat \GG = (\hat \GG,+)$, 
an LCA group  with a Haar measure $\lambda_{\hat \GG}$ and a pairing, i.e.~a continuous bihomomorphism $\GG \times \hat \GG\ni (x,\xi) \mapsto \langle x , \xi\rangle\in\TT$, such that the Fourier transform $\mathcal F_{\GG} \colon L^1(\GG) \to C(\hat \GG)$ given by
\[
 \mathcal F_{\GG} f(\xi) \coloneqq \int_{\GG} f(x) e(\langle x , \xi\rangle)d\lambda_{\GG}(x), \qquad \xi\in\hat \GG,
\]
extends to a unitary map from $L^2(\GG)$ to $L^2(\hat \GG)$; in particular we have Plancherel's identity
\[
\|\mathcal F_\GG f\|_{L^2(\hat \GG)}=\|f\|_{L^2(\GG)}, \qquad f\in L^2(\GG).
\]
Moreover, the inverse Fourier transform $\mathcal F_\GG^{-1} \colon L^2(\hat \GG) \to L^2(\GG)$ is given by the formula
\[
 \mathcal F_\GG^{-1} f(x) = \int_{\hat \GG} f(\xi) e(-\langle x , \xi\rangle)d\lambda_{\hat \GG}(\xi), \qquad f\in L^1(\hat \GG) \cap L^2(\hat \GG),\quad x\in\GG.
\]
We will mainly work with concrete pairs $(\GG,\hat \GG)$ of Pontryagin dual LCA groups:
\begin{itemize}
\item [(i)]  If $\GG = \RR^d$ with Lebesgue measure $\lambda_{\RR^d} = dx$, then $\hat \GG = \RR^d$ with Lebesgue measure $\lambda_{\RR^d} = d\xi$ is a Pontryagin dual, with pairing $\langle x , \xi\rangle \coloneqq x \cdot \xi$. Also, for any $f \in L^1(\RR^d)$, we have 
\begin{align*}
\mathcal F_{\RR^d} f(\xi)  \coloneqq  \int_{\RR^d} f(x) e(x\cdot\xi) d x, \qquad \xi\in\RR^d.
\end{align*}

\item[(ii)]  If $\GG = \ZZ^d$ with counting measure $\lambda_{\ZZ^d}$, then $\hat \GG = \TT^d$ with Lebesgue measure $\lambda_{\TT^d} = d\xi$ is a Pontryagin dual, with pairing $\langle x , \xi\rangle \coloneqq x \cdot \xi$. Also, for any $f \in \ell^1(\ZZ^d)$, we have 
\begin{align*}
\mathcal F_{\ZZ^d}f(\xi) \coloneqq \sum_{x \in \ZZ^d} f(x) e(x\cdot\xi), \qquad \xi\in \TT^d.
\end{align*}
Sometimes we shall abbreviate $\calF_{\ZZ^d}f$ to $\hat{f}$.
\end{itemize}

For any bounded function $\mathfrak m \colon \hat \GG\to\CC$ and a test function $f \colon \GG\to\CC$ we define the Fourier multiplier operator  by 
\begin{align}
\label{eq:1000}
T_{\GG}[\mathfrak m]f(x) \coloneqq \int_{\hat \GG}e(-\langle x , \xi\rangle)\mathfrak m(\xi)\mathcal F_{\GG}f(\xi)d\lambda_{\hat \GG}(\xi), \qquad  x\in\GG.
\end{align}
One may think that $f \colon \GG\to\CC$ is a compactly supported function on $\GG$ (and smooth if $\GG=\RR^d$) or any other function for which \eqref{eq:1000} makes sense.
Finally for any finite set  $\Sigma\subseteq \hat \GG$ we define
\begin{align}
\label{eq:1060}
T_{\GG}^{\Sigma}[\mathfrak m]f(x) \coloneqq T_{\GG}\Big[\sum_{\theta\in\Sigma}\tau_{\theta}\mathfrak m\Big]f(x), \qquad  x\in\GG,
\end{align}
where $\tau_{\theta}\mathfrak m(\xi) \coloneqq \mathfrak m(\xi-\theta)$ for $\xi\in\hat \GG$. Notation from \eqref{eq:1060} will be especially useful in Section \ref{sec:IW}.

\subsection{Oscillatory integrals and canonical polynomial mappings} 
Let $\RR_{\le d}[{\rm x}_1,\ldots, {\rm x}_k]$ be the vector space of all polynomials in $\RR[{\rm x}_1,\ldots, {\rm x}_k]$ of degree at most $d\in\ZZ_+$, which is equipped with the norm $\|P\|:=\sum_{0\le|\gamma|\le d}|c_{\gamma}|$ whenever
\begin{align*}
P(x)=\sum_{0\le|\gamma|\le d}c_{\gamma}x_1^{\gamma_1}\cdots x_k^{\gamma_k} \quad \text{ for } \quad x=(x_1,\ldots, x_k)\in\RR^k.
\end{align*}
A variant of the van der Corput oscillatory integral lemma for polynomials on $\RR^k$ reads as follows.
\begin{proposition}\label{thm:CCW}
For each $d, k\in\ZZ_+$, there exists a constant $C_{d, k}\in\RR_+$ such that
for any  $P\in \RR_{\le d}[{\rm x}_1,\ldots, {\rm x}_k]$ with $P(0) = 0$, one has
\begin{align*}
  \bigg|\int_{[0, 1]^k}\ex(P(x))dx\bigg|\le C_{d, k}\|P\|^{-1/d}.
\end{align*}
\end{proposition}
The proof of Proposition \ref{thm:CCW} can be found in \cite[Corollary 7.3., p. 1008]{CCW}, see also \cite[Section 1]{ACK}. This proposition will be applied mainly to canonical polynomial mappings.

For $\gamma=(\gamma_1,\ldots, \gamma_k)\in\NN^k$ and $t=(t_1,\ldots, t_k)\in\RR^k$ define 
\begin{align}
\label{eq:85}
t^{\gamma}:= t_1^{\gamma_1}\cdots\:  t_k^{\gamma_k}.
\end{align}
In other words, $t^{\gamma}$ is a  monomial with variables $t_1,\ldots, t_k$ and exponents $\gamma_1,\ldots, \gamma_k$.
Now, let a finite and nonempty set $\Gamma\subset\NN^k\setminus\{{\bm 0}\}$ be given. Using
\eqref{eq:85} we define a polynomial mapping  induced by the set
$\Gamma$ by setting
\begin{align}
\label{eq:84}
\RR^k\ni t\mapsto (t)^{\Gamma}:=(t^{\gamma}: \gamma\in\Gamma)\in\RR^{\Gamma}.
\end{align}
Here, the polynomial mapping $(t)^{\Gamma}$, (sometimes called a \textit{canonical polynomial mapping}), is simply a vector in the $|\Gamma|$-dimensional vector space whose entries are the monomials $t^{\gamma}$ indexed by the set $\Gamma$. 

An important feature of the polynomials in \eqref{eq:84} is their scale homogeneity. 
For two vectors $s=(s_1,\ldots, s_k), u=(u_1,\ldots, u_k) \in \RR^k$, we define their coordinatewise product by
\begin{equation}\label{eq:otimes}
s \otimes u \coloneqq   (s_1 u_1, \dots, s_k u_k)\in\RR^k
\end{equation} 
and observe that $(s \otimes u)^{\Gamma}=(s )^{\Gamma}\otimes (u)^{\Gamma}$ for any $s, u\in\RR^k$. This property will be crucial in the proof of Theorem \ref{thm:osc_IW_2par}, see also Lemma \ref{lem:long_only}.

Finally,  Proposition \ref{thm:CCW} applied to  canonical polynomial mappings of the form
\[
\xi\cdot (t)^\Gamma=\sum_{\gamma\in\Gamma}\xi_\gamma t^\gamma,
\]
for $\xi=(\xi_\gamma\colon\gamma\in \Gamma)\in\RR^\Gamma$,
 yields the conclusion in terms of the size of $|\xi|$.

\subsection{Coordinatewise order $\preceq$}  For any $x=(x_1,\ldots, x_k)\in\RR^k$ and $y=(y_1,\ldots, y_k)\in\RR^k$ 
we say  $x\preceq y$ if and only if $x_i\le y_i$ for each $i\in[k]$.
We also write $x\prec y$ if and only if $x\preceq y$ and
$x\neq y$, and $x\prec_{\rm s} y$ if and only if $x_i< y_i$ for each
$i\in[k]$. Let $\II\subseteq \RR^k$ be an index set
such that $|\II|\ge2$ and 
for every $J\in\ZZ_+\cup\{\infty\}$
define the set
\begin{equation}\label{eq:96}
\mathfrak S_J(\II):
=
\Set[\big]{(t_i:i\in\NN_{\le J})\subseteq \II}{t_{0}\prec_{\rm s}
t_{1}\prec_{\rm s}\ldots \prec_{\rm s}t_{J}},
\end{equation}
where $\NN_{\le \infty}:=\NN$.
In other
words, $\mathfrak S_J(\II)$ is a family of all strictly increasing
sequences (with respect to the coordinatewise order) of length $J+1$
taking their values in the set $\II$.

\subsection{Oscillation seminorms}
Let $\II\subseteq \RR^k$ be an index set such that $|\II|\ge2$. Let $(\mathfrak a_{t}: t\in\II)\subseteq\CC$ be a $k$-parameter
family of complex numbers. For any
$\JJ\subseteq \II$ and a sequence
$I=(I_i : i\in\NN_{\le J}) \in \mathfrak S_J(\II)$ the multi-parameter
oscillation seminorm is defined by
\begin{align}
\label{eq:102}
O_{I, J}(\mathfrak a_{t}: t \in \JJ):=
\Big(\sum_{j=0}^{J-1}\sup_{t\in \BB[I_j]\cap\JJ}
\abs{\mathfrak a_{t} - \mathfrak a_{I_j}}^2\Big)^{1/2},
\end{align}
where
$\BB[I_i]:=[I_{i, 1}, I_{i+1, 1})\times\ldots\times[I_{i, k}, I_{i+1, k})$
is a box determined by the element $I_i=(I_{i,1}, \ldots, I_{i,k})$ of
the sequence $I\in \mathfrak S_J(\II)$. We often use oscillation seminorms with respect to $k$-parameter families of measurable functions
$(\mathfrak a_{t}(x): t\in\II)$ on $X$.
In order to avoid problems with measurability we always assume that
$\II\ni t\mapsto \mathfrak a_{t}(x)\in\CC$ is continuous for
$\mu$-almost every $x\in X$, or $\JJ$ is countable. We also use the
convention that the supremum taken over the empty set is zero.

We make a few remarks about \eqref{eq:102}.

\begin{enumerate}[label*={\arabic*}.]
\item Clearly, $O_{I, J}(\mathfrak a_{t}: t \in \JJ)$ defines a semi-norm.

\item  Let $\II\subseteq \RR^k$ be an index set such that $|\II|\ge2$, and let $\JJ_1, \JJ_2\subseteq \II$ be disjoint. Then for any family $(\mathfrak a_t:t\in\II)\subseteq \CC$,  any $J\in\ZZ_+$ and any $I\in\mathfrak S_J(\II)$ one has
\begin{align*}
O_{I, J}(\mathfrak a_{t}: t\in\JJ_1\cup\JJ_2)\le O_{I, J}(\mathfrak a_{t}: t\in\JJ_1)+O_{I, J}(\mathfrak a_{t}: t\in\JJ_2).
\end{align*}

\item Let $\II\subseteq \RR^k$ be a countable index set such that $|\II|\ge2$ and $\JJ\subseteq \II$. 
Then 
 for any family $(\mathfrak a_t:t\in\II)\subseteq \CC$, any  $J\in\ZZ_+$, any  $I\in\mathfrak S_J(\II)$  one has
 \begin{align}
 \label{eq:91}
 O_{I, J}(\mathfrak a_{t}: t \in \JJ)\lesssim \Big(\sum_{t\in\II}|\mathfrak a_{t}|^2\Big)^{1/2}.
 \end{align}
 \item Let $\II\subseteq \RR^k$ be a countable index set such that
 $\#{\II}\ge2$.  For $l\in[k]$, let $\proj_l:\RR^k \to \RR$ be the
 projection onto the  $l$-th coordinate of $\RR^k$.  Note that for any family
 $(\mathfrak a_t:t\in\II)\subseteq \CC$, any $J\in\ZZ_+$, any
 $I\in\mathfrak S_J(\II)$ and any $l\in[k]$ one has
\begin{align}
\label{eq:95}
\begin{split}
&O_{I, J}(\mathfrak a_{t}: t\in\II)=\Big(\sum_{j=0}^{J-1}\sup_{t\in\BB[I_j]\cap\II}|\mathfrak a_{t}-\mathfrak a_{I_j}|^2\Big)^{1/2}\\
&\hspace{1cm}\lesssim \Big(\sum_{t_l\in\proj_l(\II)}\sup_{\substack{(t_1,\ldots, t_{l-1},  t_{l+1},\ldots, t_k)\in\prod_{i\in[k]\setminus\{l\}}\proj_i(\II)\\(t_1,\ldots, t_{l-1}, t_l, t_{l+1},\ldots, t_k)\in\II}}|\mathfrak a_{(t_1,\ldots, t_{l-1}, t_l, t_{l+1},\ldots, t_k)}|^2\Big)^{1/2}.
\end{split}
\end{align}
  Inequality \eqref{eq:95} will be used in Section
\ref{sec:circle}. It is important to note that the parameter $t\in\II$
in the definition of oscillations and the sequence
$I\in\mathfrak S_J(\II)$ both take values in $\II$.

\item Let $k=1$, then for any $j_0, m\in\NN$ so that $j_0< 2^m$ and
any sequence of complex numbers $(\mathfrak a_k: k\in\NN)$, any
$J\in[2^m]$ and any $I\in\mathfrak S_J([j_0, 2^m))$, the
Rademacher--Menshov inequality for oscillations \eqref{eq:102} holds,
namely we have
\begin{align}
\label{eq:164}
\begin{split}
O_{I, J}(\mathfrak a_{j}: j_0\le j< 2^m)\le \sqrt{2}\sum_{i=0}^m\Big(\sum_{j=0}^{2^{m-i}-1}\big|\sum_{\substack{k\in U_{j}^i\\
U_{j}^i\subseteq [j_0, 2^m)}} (\mathfrak a_{k+1}-\mathfrak a_{k})\big|^2\Big)^{1/2},
\end{split}
\end{align}
where $U_j^i:=[j2^i, (j+1)2^i)$ for any $i, j\in\ZZ$. Inequality \eqref{eq:164}  follows essentially from \cite[Lemma 2.5,
p. 534]{MSZ2}. Inequality \eqref{eq:164} will be used in
Section \ref{sec:IW}. If we replace the left-hand side of \eqref{eq:164} with
$\sup_{ j_0\le m<n< 2^m}|\mathfrak a_{n}-\mathfrak a_{m}|$ the same
conclusion remains true. 

\item Assume that $(\mathfrak a_{t}: t\in\RR^k)\subseteq\CC$ be a
$k$-parameter family of measurable functions on $X$. Let
$\II\subseteq \RR$ and $|\II|\ge2$, then for every $p\in[1, \infty]$,
as in \cite[Proposition 2.6, p.2266]{MSW-survey}, we have
\begin{align}
\label{eq:135}
\qquad \qquad \norm[\big]{\sup_{t \in (\II\setminus\{\sup\II\})^k}\abs{\mathfrak a_{t}}}_{L^p(X)}
\le \sup_{t \in \II^k}\norm{\mathfrak a_{t}}_{L^p(X)}
+\sup_{J\in\ZZ_+}\sup_{I\in \mathfrak S_J(\II)}\norm[\big]{O_{\bar{I}, J}(\mathfrak a_{t}: t \in \II^k)}_{L^p(X)},
\end{align}
where $\bar{I}\in\mathfrak S_J(\II^k)$ is the diagonal sequence
corresponding to a sequence $I\in \mathfrak S_J(\II)$; that is ${\bar{I}} := ({\bar{I}}_i : i\in\NN_{\le J})$ with ${\bar{I}}_i = (I_i, \ldots, I_i) \in \II^k$ whenever $I=(I_i : i\in\NN_{\le J})$. 

\item 
A remarkable feature of the oscillation seminorms is that they imply
pointwise convergence, see \cite[Proposition 2.8, p.2267]{MSW-survey}. Let $(X, \calB(X), \mu)$ be a $\sigma$-finite measure space. For
$k\in\ZZ_+$ let $(\mathfrak a_{t}: t\in\NN^k)\subseteq\CC$ be a
$k$-parameter family of measurable functions on $X$. Suppose that
there is $p\in[1, \infty)$ and a constant $C_p\in\RR_+$ such that
\begin{align*}
\sup_{J\in\ZZ_+}\sup_{I\in \mathfrak S_J(\NN)}
\norm[\big]{O_{\bar{I}, J}(\mathfrak a_{t}: t \in \NN^k)}_{L^p(X)}\le C_p<\infty.
\end{align*}
Then the limit $\lim_{\min\{t_1,\ldots, t_k\}\to\infty}\mathfrak a_{(t_1,\ldots, t_k)}$
exists  $\mu$-almost everywhere on $X$.
\end{enumerate}

\section{Exponential sum estimates}\label{sec:exp}

In this section, we recall several exponential sum estimates from
\cite{BMSW}. To simplify the presentation, we will focus on the
two-parameter setting $k=2$, noting that our arguments can be adapted
to the multi-parameter context $k\ge 2$.

\subsection{Weyl's type estimates}
We begin with recalling the one-parameter estimate.
\begin{theorem}[Proposition 5.4, \cite{BMSW}]
\label{thm:weyl1par}
Let $P\in\RR[\rm m]$ be a polynomial of degree $d\in\ZZ_+$ such that $P(m):=c_dm^d+\ldots+c_1m$.
Then there exists a constant $C\in\RR_+$ such that for every $M\in\ZZ_+$ the following is true.
Suppose that for some $j\in[d]$ there are $a, q\in\ZZ$ with $(a,q)=1$ such that $1\le q\le M^j$ and  
$|c_j-a /q|\le{q^{-2}}$.
Then for certain $\tau(d)\in\ZZ_+$ one has
\begin{align}
\label{eq:w4'}
\Big|\sum_{m=1}^M \ex(P(m))\Big|\le
CM\log(2M)\bigg(\frac{1}{q}+\frac{1}{M}+\frac q{M^j}\bigg)^{\frac{1}{\tau(d)}}.
\end{align}
\end{theorem}

\begin{remark} \label{remark:weyl1par}
It follows from the above theorem that if $|c_j-a /q|\le{q^{-2}}$ is satisfied for some
$M^\beta\le q\le M^{j-\beta}$, with some $\beta\in(0,1)$, then
\begin{align*}
\Big|\sum_{m=1}^M \ex(P(m))\Big|\lesssim
M^{1-\beta/\tau(d)}(\log M)\lesssim M^{1-\beta/2\tau(d)}.
\end{align*}
\end{remark}
 Let $K_1, K_2\in\NN$, $M_1, M_2\in\ZZ_+$ satisfy $K_1< M_1$ and $K_2< M_2$. Let $Q\in\RR[\rm m_1, \rm m_2]$ be given 
and define the double exponential sums 
\begin{align}
\label{eq:58}
S_{K_1, M_1, K_2, M_2}(Q):=&\sum_{m_1=K_1+1}^{M_1}\sum_{m_2=K_2+1}^{M_2}\ex(Q(m_1, m_2)),\\
\label{eq:87}
S_{K_1, M_1, K_2, M_2}^1(Q):=&\sum_{m_1=K_1+1}^{M_1}\Big|\sum_{m_2=K_2+1}^{M_2}\ex(Q(m_1, m_2))\Big|,\\
\label{eq:213}
S_{K_1, M_1, K_2, M_2}^2(Q):=&\sum_{m_2=K_2+1}^{M_2}\Big|\sum_{m_1=K_1+1}^{M_1}\ex(Q(m_1, m_2))\Big|.
\end{align}
If $K_1=K_2=0$ we will abbreviate \eqref{eq:58}, \eqref{eq:87} and \eqref{eq:213} respectively to
$S_{M_1,  M_2}(Q), S_{M_1,  M_2}^1(Q)$ and $S_{M_1,  M_2}^2(Q)$. By the triangle inequality we have
\begin{align}
\label{eq:226}
|S_{K_1, M_1, K_2, M_2}(Q)|\le \min(S_{K_1, M_1, K_2, M_2}^1(Q),
S_{K_1, M_1, K_2, M_2}^2(Q)).
\end{align}
We will need the following counterpart of Weyl's inequality for double sums. 
\begin{theorem}[Proposition 5.22, \cite{BMSW}]
\label{thm:weyl2par}
Let $d_1, d_2\in\ZZ_+$ and $Q\in\RR[\rm m_1, \rm m_2]$ be such that 
\begin{align*}
Q(m_1, m_2):=\sum_{\gamma_1=0}^{d_1}\sum_{\gamma_2=0}^{d_2}
c_{\gamma_1, \gamma_2}m_1^{\gamma_1}m_2^{\gamma_2}, \quad \text{ and } \quad c_{0, 0}=0.
\end{align*}
Then there exists a constant $C\in\RR_+$ such that for all
$K_1, K_2\in\NN$, $M_1, M_2\in\ZZ_+$ satisfying $K_1\le M_1$ and
$K_2\le M_2$ the following holds.  Suppose that for some
$\rho_1\in [d_1]$ and $\rho_2\in [d_2]$ there are
$a_{\rho_1, \rho_2}\in\ZZ, q_{\rho_1, \rho_2}\in\ZZ_+$ such that
$(a_{\rho_1, \rho_2}, q_{\rho_1, \rho_2})=1$ and
\begin{align}
\label{eq:w6}
\Big|c_{\rho_1, \rho_2}-\frac{a_{\rho_1, \rho_2}}{q_{\rho_1, \rho_2}}\Big|\le\frac{1}{q_{\rho_1, \rho_2}^2}.
\end{align}
Set $k_i:=d_i(d_i+1)$ for $i\in[2]$, $M_{-} := \min(M_1^{\rho_1}, M_2^{\rho_2})$ and $M_{+} := \max(M_1^{\rho_1}, M_2^{\rho_2})$.
Then for $i\in [2]$,
\begin{align}
\label{eq:86}
S_{K_1, M_1, K_2, M_2}^i(Q)
&\le CM_1M_2\bigg( \frac{1}{M_{-}}+\frac{q_{\rho_1, \rho_2}\log q_{\rho_1, \rho_2}}{M_1^{\rho_1}M_2^{\rho_2}}+\frac{1}{q_{\rho_1, \rho_2}}+\frac{\log q_{\rho_1, \rho_2}}{M_{+}}\bigg)^{\frac{1}{4k_1k_2}}.
\end{align}
In view of \eqref{eq:226} the estimate \eqref{eq:86} holds for $|S_{K_1, M_1, K_2, M_2}(Q)|$ as well.
\end{theorem}
For future reference we set
\begin{equation}\label{eq:delta_Weyl}
\delta_{\rm{Weyl}}:=\min\Big(\frac{1}{2\tau(d_2)}, \frac{1}{4k_1k_2}\Big), 
\end{equation}
with $\tau(d_2)$ from \eqref{eq:w4'} and $k_i:=d_i(d_i+1)$ for $i\in[2]$.

\subsection{Estimates for complete exponential sums}
For $d_1, d_2\in\ZZ_+$ let
\begin{align}
\label{eq:2}
\Gamma:=\Gamma(d_1, d_2):=\NN_{\le d_1}\times \NN_{\le d_2}\setminus\{(0, 0)\}.
\end{align}
Let
$\emptyset\neq\Lambda\subseteq \Gamma$ and define the corresponding double complete
exponential sum by
\begin{align}
\label{eq:323}
G^\Lambda(a/q):=\frac{1}{q^2}\sum_{r_1=1}^q\sum_{r_2=1}^q\ex((r_1, r_2)^\Lambda\cdot{a/q}),
\qquad  a \in\ZZ^\Lambda, \quad q\in \ZZ_+.
\end{align}

We begin with recalling a well-known bound.

\begin{lemma}
\label{lem:31}
Let $d_1, d_2\in\ZZ_+$ be given and let
$\emptyset\neq\Lambda \subseteq \Gamma$ with $\Gamma=\Gamma(d_1, d_2)$
as in \eqref{eq:2}. Then there are $C_\Lambda\in\RR_+$ and
$\delta_{\Lambda}\in(0, 1)$ such that for every
$a = (a_{\gamma})_{\gamma \in \Lambda} \in \ZZ^{\Lambda}$ and
$q\in\ZZ_+$ satisfying
\begin{align*}
\gcd(\{a_{\gamma_1, \gamma_2}: (\gamma_1, \gamma_2)\in \Lambda\}\cup\{q\})=1,
\end{align*}
we have
\begin{align}
\label{eq:97}
|G^{\Lambda}({a/q})|\le C_\Lambda q^{-\delta_{\Lambda}}.
\end{align}
\end{lemma}

\begin{proof}
Define a new sequence $b=(b_\gamma)_{\gamma\in\Gamma}\in\ZZ^{\Gamma}$ by setting $b_\gamma:=a_\gamma$ if $\gamma \in \Lambda$, and
$b_\gamma:=0$ if $\gamma \in \Gamma\setminus\Lambda$. By the standard exponential sum estimates for complete
sums, there exists $\delta\in(0, 1)$ such that
\begin{align}
\label{eq:97'}
|G^{\Gamma}({b/q})|\lesssim_\Gamma q^{-\delta},
\end{align}
see for instance \cite[Chapter 2]{ACK}.  Alternatively, to prove
\eqref{eq:97'}, one can use Theorems 
\ref{thm:weyl1par} and \ref{thm:weyl2par}, and proceed as in \cite[Lemma 4.14]{MSZ3}.  Noting
that $G^{\Lambda}({a/q})=G^{\Gamma}({b/q})$, inequality
\eqref{eq:97} follows.
\end{proof}

We will need a counterpart of Lemma \ref{lem:31} for the partially complete exponential sums.
Fix $\emptyset\neq\Lambda \subseteq \Gamma$ as above. Partially complete exponential sums are defined by  
\begin{align}
\label{eq:GGamma2_def}
G^{\Lambda_2}_{m_1}(a/q)
:=\frac{1}{q}
\sum_{r_2=1}^{q}
\ex\big((m_1, r_2)^{\Lambda_2}\cdot a/q\big), \qquad a/q\in(\TT\cap\QQ)^{\Lambda_2}, \quad m_1\in\ZZ_+,
\end{align}
and
\begin{align}
\label{eq:GGamma2_def'}
G^{\Lambda_1}_{m_2}(a/q)
:=\frac{1}{q}
\sum_{r_1=1}^{q}
\ex\big((r_1, m_2)^{\Lambda_1}\cdot a/q\big), \qquad a/q\in(\TT\cap\QQ)^{\Lambda_1}, \quad m_2\in\ZZ_+,
\end{align}
where $\Lambda_1:=\Set{\gamma\in\Lambda}{\gamma_1\not=0}$ and $\Lambda_2:=\Set{\gamma\in\Lambda}{\gamma_2\not=0}$.

To estimate \eqref{eq:GGamma2_def} and \eqref{eq:GGamma2_def'} we need a simple arithmetic lemma.
\begin{lemma}\label{lem:simplified_q}
Let $K\in\ZZ_+$ and suppose that $a=(a_1, \dots, a_K)\in \ZZ^K$ and
$q\in\ZZ_+$ are such that $(a,q)=(a_1,\dots a_K,q)=1$. For each
$i\in[K]$ let $b_i\in \ZZ$ and $q_i\in\ZZ_+$ be such that
${b_i}/{q_i}={a_i}/{q}$ and $(b_i,q_i)=1$.
Then there exists $i_0\in [K]$ such that $q_{i_0}\ge q^{1/2^K}$.
\end{lemma}

\begin{proof}
If $K=1$ then there is nothing to prove, so assume that
$K\ge 2$. Denote $\kappa:={2^{-K}}$, and for $i\in[K]$ let
$d_i:=(a_i,q)$. Since $b_i={a_i}/{d_i}$ and $q_i={q}/{d_i}$, it
suffices to prove that for some $i_0\in[K]$ we have
$d_{i_0}\le q^{1-\kappa}$. Assume for a contradiction that for each
$i\in[K]$ one has $d_i>q^{1-\kappa}$.

For $j\in[K]$ let $D_j:=(d_1,\dots, d_j)$. By the assumption $(a,q)=1$, we have
$$
D_K=(d_1,\dots, d_K)=(a_1,\dots,a_K,q)=1.
$$
Furthermore, we have $(D_{K-1}, d_K)=D_K=1$. Thus, $D_{K-1}$ and $d_K$ are co-prime and both divide $q$. Consequently, the product  $D_{K-1}d_K$ also divides $q$. In particular,
$$
d_K D_{K-1}\le q.
$$
Since $d_K>q^{1-\kappa}$, it follows that $D_{K-1}<q^\kappa$. Suppose that for some $j\in[K-2]$, we have
\begin{align}
\label{eq:1}
D_{K-j} < q^{(1+2+2^2+\dots+2^{j-1})\kappa}.
\end{align}
By induction one can easily prove that \eqref{eq:1} holds with $j+1$
in place of $j$. Then taking $j=K-1$ in \eqref{eq:1}, we obtain
$$
q^{1-\kappa} < d_1=D_1 < q^{(1+2+2^2+\dots+2^{K-2})\kappa} < q^{2^{K-1}\kappa},
$$
which leads to the inequality $\kappa (2^{K-1}+1) > 1$, which cannot hold since $\kappa=1/2^K$. 
\end{proof}

The following result for  the partially complete exponential sums from \eqref{eq:GGamma2_def}  will be  needed later. 
\begin{proposition}
\label{prop:32}
Let $d_1, d_2\in\ZZ_+$ be given and recall $\Gamma=\Gamma(d_1, d_2)$
from \eqref{eq:2}.  Then there exist $C_\Gamma\in\RR_+$ and
$\delta\in(0, 1)$ such that for any $M_1\in\ZZ_+$ and
$a\in\ZZ^{\Gamma_2}, q\in\ZZ_+$ satisfying $(a, q)=1$ and
$q\le M_1^{1/\chi}$ for some $0<\chi<\frac{1}{10|\Gamma|}$, we have
\begin{align}
\label{eq:Gauss2}
\frac{1}{M_1}\sum_{m_1=1}^{M_1}|G^{\Gamma_2}_{m_1}(a/q)|\le C_\Gamma \, q^{-\delta}.
\end{align}
An analogous result holds for the partially complete exponential sums from \eqref{eq:GGamma2_def'}.
\end{proposition}

\begin{proof}
Since $\Gamma=\NN_{\le d_1}\times \NN_{\le d_2}\setminus\{(0, 0)\}$ as in \eqref{eq:2}, then $\Gamma_2=\mathbb N_{\le d_1}\times[d_2]$.
To prove \eqref{eq:Gauss2} we fix  $a/q\in\QQ$ with $a=(a_\gamma)_{\gamma\in\Gamma^2}\in\ZZ^{\Gamma_2}$ and $q\in\ZZ_+$ such that $(a, q)=1$. 
For each $\gamma\in\Gamma_2$ let $b_\gamma\in\ZZ$ and $q_\gamma\in\ZZ_+$ be such that ${b_\gamma}/{q_\gamma}={a_\gamma}/q$. 
Let $\mathfrak{q}:=q^{1/2^{|\Gamma|}}$,
by Lemma \ref{lem:simplified_q}  there exists $\gamma_0\in\Gamma_2$ such that 
\begin{equation}\label{eq:mathfrakq_gamma0}
q_{\gamma_0}\ge \mathfrak{q}.
\end{equation}
Define a sequence $(\rho_\gamma)_{\gamma\in\Gamma^2}\in\QQ^{\Gamma_2}$ by setting
$$
\rho_{\gamma}:=
\begin{cases}
\frac{1}{(10|\Gamma|)^{d_2+1}}, &\quad \textrm{if} \quad \gamma\in  \{(\gamma_1, \gamma_2)\in\Gamma_2: \gamma_1\neq 0\},
\\
\frac{1}{(10|\Gamma|)^{\gamma_2-1}}, &\quad \textrm{if} \quad \gamma=(0, \gamma_2)\in \{(\gamma_1, \gamma_2)\in\Gamma_2: \gamma_1=0\}.
\end{cases}
$$
Note that we have
\begin{equation}\label{eq:rho_conds}
(10|\Gamma|)^{d_2+1}\rho_{\gamma}=(10|\Gamma|)^{d_2-1}\rho_{(0,d_2)}=\ldots=10|\Gamma|\rho_{(0,2)} =\rho_{(0,1)} =   1
\end{equation}
for any $\gamma\in  \{(\gamma_1, \gamma_2)\in\Gamma_2: \gamma_1\neq 0\}$.
We now consider two cases.

\paragraph{\textbf{Case 1}} Assume first that $q_\gamma \ge \mathfrak{q}^{\rho_\gamma}$ for some $\gamma\in\{(\gamma_1, \gamma_2)\in\Gamma_2: \gamma_1\neq 0\}$.  Since $M_1\ge q^{\chi}$, we can appeal to inequality \eqref{eq:86} with $M_2=q$ to obtain for some $\delta\in(0, 1)$ that
\begin{align*}
\frac{1}{M_1}\sum_{m_1=1}^{M_1}|G^{\Gamma_2}_{m_1}(a/q)|\lesssim_\Gamma q^{-\delta}.
\end{align*}
Let us emphasize that the assumption $M_1\ge q^{\chi}$ was critical here, since if $M_1$ was small compared to $q$, then the decay given by \eqref{eq:86} would be insufficient.

\paragraph{\textbf{Case 2}}
To complete the proof of the proposition it remains to consider the case when $q_\gamma < \mathfrak{q}^{\rho_\gamma}$ for all $\gamma\in\{(\gamma_1, \gamma_2)\in\Gamma_2: \gamma_1\neq 0\}$. In this case, we  show that for some $\delta>0$ the estimate
$$
|G^{\Gamma_2}_{m_1}(a/q)|\lesssim q^{-\delta}
$$
holds uniformly in $m_1$. The latter follows by the application of  Weyl's inequality (see Theorem \ref{thm:weyl1par}), which gives
decay in terms of a power of $\mathfrak{q}$, and consequently also in
terms of a power of $q$.  This completes the proof of the proposition.
\end{proof}

Finally, let \(\delta_{\emptyset}:=\delta \in (0, 1)\) be the exponent from Proposition \ref{prop:32} for which inequality \eqref{eq:Gauss2} and its analogue for \eqref{eq:GGamma2_def'} hold. Let also $\delta_\Lambda\in(0, 1)$ be the exponent from Lemma \ref{lem:31}. For future reference, as in \eqref{eq:delta_Weyl},  we set
\begin{align}
\label{eq:12}
\delta_{\rm{Gauss}}:=\min_{\Lambda\subseteq\Gamma}\delta_\Lambda\in(0, 1).
\end{align}

\section{Multi-parameter Ionescu--Wainger theory}
\label{sec:IW}
One of the most important ingredients in our argument is the
Ionescu--Wainger multiplier theorem for the family of
\textit{canonical fractions} proved recently in \cite{KMPWW}. In this
section, we will prove its multi-parameter variants.  We begin with
fixing necessary notation and terminology.

\subsection{Ionescu--Wainger multiplier theorem}
For $d\in\ZZ_+$ and $N\ge1$ define $1$-periodic sets of \textit{canonical fractions} by
\begin{align}
\label{IWeq:372}
\mathcal{R}_{\le N}^d := \big\{{a}/{q}\in(\QQ\cap\TT)^d:  q \in [N] \text{ and } (a, q)=1\big\},
\end{align}
where $(a,d):=\gcd(a_1,\dots,a_d,q)$ is the greatest common divisor of $a_1,\dots,a_d,q$. 
We also let
\begin{align}
\label{eq:372}
\Sigma_{\le l}^d :=\mathcal{R}_{\le 2^l}^d,
\quad \text{ and } \quad
\Sigma_l^d := \Sigma_{\le l}^d \backslash \Sigma_{\le l-1}^d.
\end{align}
Then 
\begin{align}
\label{eq:373}
|\Sigma_{\le l}^d| \le 2^{(d+1)l}, \qquad \text{ and } \quad Q_{\le l}:=\lcm([2^l])\le 3^{2^l}.
\end{align}

We now recall the  Ionescu--Wainger
multiplier theorem for canonical fractions from \cite{KMPWW}.

\begin{theorem}
\label{thm:iw-semi}
Let $d\in\ZZ_+$ and $p\in[p_0',p_0]$ for some $p_0\in2\ZZ_+$.  There exist absolute constants
$C_{d, p_0}, {\textbf{C}}_{{\rm IW}}(p_0)\in\RR_+$, such that, for every $l\in\NN$, the following
holds.
Let $0<\varepsilon_l \le (2p_0 2^{l p_0})^{-1}$, and let
$\mathfrak m: \RR^d \to L(H_0,H_1)$ be a measurable function supported on
$\varepsilon_{l}[-1/2, 1/2)^d$,  with values in the space $L(H_{0},H_{1})$ of bounded linear
operators between separable Hilbert spaces $H_{0}$ and $H_{1}$. 
Then 
\begin{align}
\label{eq:376}
\|T_{\ZZ^d}^{\Sigma_{\le l}^d}[\mathfrak m]\|_{\ell^p(\ZZ^d;H_0)\to \ell^p(\ZZ^d;H_1)}
\le C_{d, p_0}
2^{{\textbf{C}}_{{\rm IW}}(p_0) \frac{l\log \log l}{\log l}} \|T_{\RR^d}[\mathfrak m]\|_{L^{p_0}(\RR^d;H_0)\to L^{p_0}(\RR^d;H_1)}.
\end{align}
\end{theorem}

Theorem~\ref{thm:iw-semi} transfers square function estimates from the
continuous to the discrete setting, which will be useful in the
following sections.  The norm in \eqref{eq:376}, unlike the support
hypothesis, is scale-invariant, in the sense that $\mathfrak m$ can be replaced by $\mathfrak m(A\cdot)$ for any invertible linear
transformation $A:\RR^d\to \RR^d$. As a corollary, we obtain the
following result.

\begin{corollary}
Under the assumptions of Theorem \ref{thm:iw-semi}, for any $\varepsilon\in(0, 1)$, one has
\begin{align}
\label{eq:376cor}
\|T_{\ZZ^d}^{\Sigma_{\le l}^d}[\mathfrak m]\|_{\ell^p(\ZZ^d;H_0)\to \ell^p(\ZZ^d;H_1)}
\le_{d,\varepsilon, p_0}
2^{\varepsilon l} \|T_{\RR^d}[\mathfrak m]\|_{L^{p_0}(\RR^d;H_0)\to L^{p_0}(\RR^d;H_1)}.
\end{align}
\end{corollary}

An important ingredient in the proof of Theorem \ref{thm:iw-semi} 
is the sampling principle of
Magyar--Stein--Wainger from \cite{MSW}, which we  recall for future reference.

\begin{proposition}
\label{prop:msw}
Let $d\in\ZZ_+$ be given. There exists an absolute constant $C_{\rm MSW}(d)\in\RR_+$ such that the following holds.
Let $p \in [1,\infty]$ and $q\in\ZZ_+$, and let
$B_1, B_2$ be finite-dimensional Banach spaces.  Let
$\mathfrak m : \RR^d \to L(B_1, B_2)$ be a bounded operator-valued
function supported on $q^{-1}[-1/2,1/2)^d$, then
\[
\|T_{\ZZ^d}^{q^{-1}[q]^d}[\mathfrak m]\|_{\ell^{p}(\ZZ^d;B_1)\to \ell^{p}(\ZZ^d;B_2)}\le
C_{\rm MSW}(d)\|T_{\RR^d}[\mathfrak m]\|_{L^{p}(\RR^d;B_1)\to L^{p}(\RR^d;B_2)}.
\]
\end{proposition}
We refer to \cite[Corollary 2.1, pp. 196]{MSW} for a proof.  A
generalization of Proposition \ref{prop:msw} to real interpolation
spaces can be found in \cite{MSZ1}.  We emphasize that $B_1$ and $B_2$
are general (finite dimensional) Banach spaces in Proposition
\ref{prop:msw}, in contrast to the Hilbert space--valued multipliers
appearing in Theorem \ref{thm:iw-semi} and so Proposition
\ref{prop:msw} includes maximal function formulations and can also
accommodate oscillation seminorms.

\subsection{One-parameter seminorm variant of the Ionescu--Wainger theorem} We fix $k\in\ZZ_+$ and 
let $\Lambda:=\{\lambda_1,\ldots,\lambda_k\}\subset\ZZ_+$ be a  set  of size $k$ of natural positive exponents,
 and consider the associated one-parameter family of dilations which for every $x\in\RR^k$, is defined by
\begin{align*}
(0,\infty)\ni t\mapsto t\circ x:=(t^{\lambda_1}x_1,\ldots,t^{\lambda_k}x_k)\in\RR^k.
\end{align*}
Let  $\Upsilon:=(\Upsilon_n:\RR^k\to\CC: n\in\NN)$ be a sequence of measurable
functions which define a sequence of positive operators in the sense that for every $n\in\NN$, one has
\begin{align}
\label{eq:377}
T_{\RR^k}[\Upsilon_n]f\ge0 \quad\text{if}\quad f\ge0.
\end{align}
Furthermore suppose there exist  $C_{\Upsilon}\in\RR_+$,
 $\delta_{\Upsilon}\in(0, 1)$ and $\tau\in (1, 2]$ such that for every $\xi\in\RR^k$ and
$n\in\NN$, one has
\begin{align}
\label{eq:378}
\abs{\Upsilon_n(\xi)}&\le
C_{\Upsilon}\min\big\{1, \abs{\tau^n\circ \xi}^{-\delta_{\Upsilon}}\big\},\\
\label{eq:379}
\abs{\Upsilon_n(\xi)-1}&\le
C_{\Upsilon}\min\big\{1, \abs{\tau^n\circ \xi}^{\delta_{\Upsilon}}\big\}.
\end{align}
The condition \eqref{eq:377} implies that the operator  $T_{\RR^k}[\Upsilon_n] f = f *\mu_n$ is a convolution
with a positive measure $\mu_n$ and condition \eqref{eq:379} implies that $\Upsilon_n(0) = 1$ and so each $\mu_n$ 
is a probability measure. Hence for every
$p\in[1, \infty]$, we have 
\begin{align}
\label{eq:380}
A_p^{\Upsilon}:=\sup_{n\in\NN}\|T_{\RR^k}[\Upsilon_n]\|_{L^p(\RR^k)\to L^p(\RR^k)}\le 1.
\end{align}
In this generality, $L^p$ estimates (with $1<p\le \infty$) for the maximal function 
$\sup_{n\in\NN}|T_{\RR^k}[\Upsilon_n] f(x)|$ were obtained in \cite{DR}
and corresponding $r$-variational and jump inequalities were established in
\cite{jsw} (see also \cite{MSZ2}). Here we extend these results further.

\begin{definition}[$\Gamma$-lifted extensions]
For $d\in\ZZ_+$ let $\Gamma=\{i_1,\ldots, i_k\}\subseteq [d]$ be a set of size $k\in[d]$. Let $\mathfrak m: \RR^k \to B$ be  a measurable function taking its values in a separable Banach space $B$. 
 We define a $\Gamma$-lifted extension of $\mathfrak m$  by setting
\begin{align}\label{eq:lift}
\mathfrak m^{\Gamma}(\xi):=\mathfrak m(\xi_{i_1},\ldots, \xi_{i_k}),  \quad \text{ for } \quad \xi=(\xi_1,\ldots, \xi_d)\in\RR^d.
\end{align}

\end{definition}

Our first main result is the following  one-parameter seminorm variant of Theorem
\ref{thm:iw-semi}.

\begin{theorem}
\label{thm:IW1}
Let $d\in\ZZ_+$ and $\Gamma\subseteq[d]$ be a set of size $k\in[d]$. Let $\Upsilon=(\Upsilon_n:\RR^k\to\CC: n\in\NN)$ be a sequence
of measurable functions satisfying conditions \eqref{eq:377},
\eqref{eq:378} and \eqref{eq:379}, and let
$\Upsilon^{\Gamma}:=(\Upsilon_n^{\Gamma}:\RR^d\to\CC: n\in\NN)$ be the
corresponding $\Gamma$-lifted sequence.  Let $p\in[p_0',p_0]$ for some
$p_0\in2\ZZ_+$ and let $\varepsilon\in(0, 1)$ be arbitrary. Then there
exists an absolute constant
$C=C(d, p,\tau, \varepsilon, \Gamma, A_p^{\Upsilon}, C_{\Upsilon})\in\RR_+$
such that for every integer $l\in\NN$ and $m\in\RR_+$ such that
$2^{-m}\le (2p_0 2^{l p_0})^{-1}$ the following holds. If
\begin{align}
\label{eq:167}
\supp \Upsilon_n\subseteq 2^{-m}[-1/2, 1/2)^k
\quad \text{ for all }\quad n\in\NN,
\end{align}
then for every
$f=(f_{\iota}:\iota\in\NN)\in \ell^p(\ZZ^d; \ell^2(\NN))$ one has
\begin{align}
\label{eq:381}
\sup_{J\in\ZZ_+}\sup_{I\in\mathfrak S_J(\NN)}
\norm[\bigg]{\Big(\sum_{\iota\in\NN}O_{I, J}\big(T_{\ZZ^d}^{\Sigma_{\le l}^d}\big[\Upsilon_n^{\Gamma}\eta_{\le m}^{\Gamma^c}\big]f_{\iota}:n\in\NN\big)^2\Big)^{1/2}}_{\ell^p(\ZZ^{d})}
\le C
2^{\varepsilon l} 
\|f\|_{\ell^p(\ZZ^{d}; \ell^2(\NN))}.
\end{align}
In particular, \eqref{eq:381} implies the maximal estimate
\begin{align}
\label{eq:382}
\norm[\bigg]{\Big(\sum_{\iota\in\NN}\sup_{n\in\NN}\big|T_{\ZZ^d}^{\Sigma_{\le l}^d}\big[\Upsilon_n^{\Gamma}\eta_{\le m}^{\Gamma^c}\big]f_{\iota}\big|^2\Big)^{1/2}}_{\ell^p(\ZZ^{d})}
\le C
2^{\varepsilon l} 
\|f\|_{\ell^p(\ZZ^{d}; \ell^2(\NN))}.
\end{align}
\end{theorem}

\begin{proof}
The proof proceeds in essentially the same way as the proof of \cite[Theorem 6.14]{BMSW}.
\end{proof}
\subsection{Multi-parameter seminorm variant of the Ionescu--Wainger
theorem for tensor products of bumps} We will
generalize Theorem \ref{thm:IW1} to the multi-parameter setting. We
begin with a multi-parameter seminorm variant of Theorem
\ref{thm:iw-semi} or more precisely Theorem \ref{thm:IW1}.
\begin{theorem}
\label{thm:IW1'}
Let $d\in\ZZ_+$ be given. For each $i\in[d]$, let  $\Upsilon^i:=(\Upsilon_{n_i}^i:\RR\to\CC: n_i\in\NN)$ be a sequence of measurable functions
satisfying conditions \eqref{eq:377},
\eqref{eq:378} and \eqref{eq:379} with $k=1$.
Let $\Upsilon:=\big(\prod_{i\in[d]}\Upsilon_{n_i}^{(i)}:\RR^d\to\CC: (n_1,\ldots, n_d)\in\NN^d\big)$ be a
multi-parameter sequence  such that for each $i\in[d]$, the function $\Upsilon_{n_i}^{(i)}:\RR^d\to\CC$ is defined as follows
\begin{align}
\label{eq:Upsi}
\Upsilon_{n_i}^{(i)}(\xi):=\Upsilon_{n_i}^i(\xi_i)
\quad \text{ for } \quad \xi=(\xi_1,\ldots, \xi_d)\in\RR^d.
\end{align}
Let $p\in[p_0',p_0]$ for some $p_0\in2\ZZ_+$ and let $\varepsilon\in(0, 1)$. Then there exists an
absolute constant
\begin{align*}
C=C(d,p, \tau,\varepsilon, A_p^{\Upsilon^1},\ldots, A_p^{\Upsilon^d}, C_{\Upsilon^1},\ldots, C_{\Upsilon^d})\in\RR_+
\end{align*} 
such
that for every integer $l\in\NN$ and $m\in\RR_+$ such that $2^{-m}\le (2p_0 2^{l p_0})^{-1}$
the following holds. If for each $i\in[d]$ the following support condition is satisfied
\begin{align}
\label{eq:44}
\supp \Upsilon_{n_i}^i\subseteq 2^{-m}[-1/2, 1/2),
\end{align}
then for every
$f=(f_{\iota}:\iota\in\NN)\in \ell^p(\ZZ^d; \ell^2(\NN))$, we have
\begin{align}
\label{eq:62}
\begin{gathered}
\sup_{J\in\ZZ_+}\sup_{I\in\mathfrak S_J(\NN^d)}
\norm[\bigg]{\Big(\sum_{\iota\in\NN}O_{I, J}\big(T_{\ZZ^d}^{\Sigma_{\le l}^d}\big[\prod_{i\in[d]}\Upsilon_{n_i}^{(i)}\big]f_{\iota}:(n_1,\ldots, n_d)\in\NN^d\big)^2\Big)^{1/2}}_{\ell^p(\ZZ^{d})}\\
\le C
2^{\varepsilon l} 
\|f\|_{\ell^p(\ZZ^{d}; \ell^2(\NN))}
\end{gathered}
\end{align}
In particular, \eqref{eq:62} implies the maximal estimate
\begin{align*}
\norm[\bigg]{\Big(\sum_{\iota\in\NN}\sup_{(n_1,\ldots, n_d)\in\NN^d}\big|T_{\ZZ^d}^{\Sigma_{\le l}^d}\big[\prod_{i\in[d]}\Upsilon_{n_i}^{(i)}\big]f_{\iota}\big|^2\Big)^{1/2}}_{\ell^p(\ZZ^{d})}
\le C
2^{\varepsilon l}
\|f\|_{\ell^p(\ZZ^{d}; \ell^2(\NN))}.
\end{align*}
\end{theorem}
We make some remarks about Theorem \ref{thm:IW1'}.
\begin{enumerate}[label*={\arabic*}.]
\item Theorem \ref{thm:IW1'} is a multi-parameter oscillation variant of the Ionescu--Wainger  theorem \cite{IW}. It has been formulated for the multipliers which have product structure, however it can be extended to non-product multipliers as well. An instance of such a situation will be detailed in the next subsection, see Theorem \ref{thm:osc_IW_2par}, for the multipliers that arise in our main result.  
\item In contrast to the one-parameter theory, it is not clear whether multi-parameter $r$-variational or jump counterparts of Theorem \ref{thm:IW1'} are available. As far as we know it is not even clear if there are useful multi-parameter definitions of $r$-variational or jump semi-norms. From this point of view the multi-parameter oscillation semi-norm is a crucial tool allowing us to handle pointwise convergence problems in the  multi-parameter setting.
\item Our proof can be adapted to the Euclidean setting  $\RR^d$ as well. Namely, for $d\in\ZZ_+$ and $i\in[d]$ let  $\Upsilon^i:=(\Upsilon_{n_i}^i:\RR\to\CC: n_i\in\ZZ)$ be a sequence of measurable functions
satisfying conditions \eqref{eq:377},
\eqref{eq:378} and \eqref{eq:379} with $d=1$. Then for
every $p\in(1, \infty)$, there exists a constant $C>0$ such that for
every $f=(f_{\iota}:\iota\in\NN)\in L^p(\RR^d; \ell^2(\NN))$ one has
\begin{align}
\label{eq:tensorbumpcont}
\begin{gathered}
\sup_{J\in\ZZ_+}\sup_{I\in\mathfrak S_J(\ZZ^d)}
\norm[\bigg]{\Big(\sum_{\iota\in\NN}O_{I, J}\big(T_{\RR^d}\big[\prod_{i\in[d]}\Upsilon_{n_i}^{(i)}\big]f_{\iota}:(n_1,\ldots, n_d)\in\ZZ^d\big)^2\Big)^{1/2}}_{L^p(\RR^{d})}\\
\le C
\|f\|_{L^p(\RR^{d}; \ell^2(\NN))},
\end{gathered}
\end{align}
where $\Upsilon=\big(\prod_{i\in[d]}\Upsilon_{n_i}^{(i)}:\RR^d\to\CC: (n_1,\ldots, n_d)\in\ZZ^d\big)$
is a multi-parameter sequence satisfying \eqref{eq:Upsi}. Inequality \eqref{eq:tensorbumpcont} can be also established for more general non-product multipliers, but we do not address this here.
\end{enumerate}

\begin{proof}[Proof of Theorem \ref{thm:IW1'}]
For each $l\in\NN$ define an integer
\begin{align}
\label{eq:383}
\kappa_l:=\big\lfloor \big(100 - \log_2\big({\delta_{\Upsilon}\log_2\tau}\big)\big)  (l+1) \big\rfloor+2.
\end{align}
We divide the proof into several steps to make the argument digestible.
\paragraph{\bf Step 1} We fix $J\in\ZZ_+$ and $I\in\mathfrak S_J(\NN^d)$ and observe that 
\begin{align*}
&O_{I, J}\big(T_{\ZZ^d}^{\Sigma_{\le l}^d}\big[\prod_{i\in[d]}\Upsilon_{n_i}^{(i)}\big]f_{\iota}:(n_1,\ldots, n_d)\in\NN^d\big)\\
&\qquad\lesssim
\sum_{s\in[d]}\Big(\sum_{j=0}^{J-1}\sup_{\substack{n_i\in\NN\\
i\in[d]\setminus\{s\}}} \sup_{I_j^s\le n_s<I_{j+1}^s}\big|T_{\ZZ^d}^{\Sigma_{\le l}^d}\big[\prod_{i\in[d]\setminus\{s\}}
\Upsilon_{n_i}^{(i)}(\Upsilon_{n_s}^{(s)}-\Upsilon_{I_{j}^s}^{(s)})\big]f_{\iota}\big|^2\Big)^{1/2}.
\end{align*}
This shows that the left-hand side of \eqref{eq:62} is dominated by the following expression
\begin{align*}
\sum_{s\in[d]}
\sup_{J\in\ZZ_+}\sup_{I\in\mathfrak S_J(\NN)}
\norm[\bigg]{\Big(\sum_{\iota\in\NN}\sum_{j=0}^{J-1}\sup_{\substack{n_i\in\NN\\
i\in[d]\setminus\{s\}}} \sup_{I_j\le n_s<I_{j+1}}
\big|T_{\ZZ^d}^{\Sigma_{\le l}^d}\big[\prod_{i\in[d]\setminus\{s\}}\Upsilon_{n_i}^{(i)}(\Upsilon_{n_s}^{(s)}-\Upsilon_{I_{j}}^{(s)})\big]f_{\iota}\big|^2\Big)^{1/2}}_{\ell^p(\ZZ^{d})}.
\end{align*}
We will show that for every $p\in(1, \infty)$,   every
$f=(f_{\iota}:\iota\in\NN)\in \ell^p(\ZZ^d; \ell^2(\NN))$, each $s\in[d]$ and for any $\varepsilon\in(0, 1)$, one has 
\begin{align}
\label{eq:65}
\nonumber&\sup_{J\in\ZZ_+}\sup_{I\in\mathfrak S_J(\NN)}
\norm[\bigg]{\Big(\sum_{\iota\in\NN}\sum_{j=0}^{J-1}\sup_{\substack{n_i\in\NN\\
i\in[d]\setminus\{s\}}} \sup_{I_j\le n_s<I_{j+1}}\big|T_{\ZZ^d}^{\Sigma_{\le l}^d}\big[\prod_{i\in[d]\setminus\{s\}}\Upsilon_{n_i}^{(i)}(\Upsilon_{n_s}^{(s)}-\Upsilon_{I_{j}}^{(s)})\big]f_{\iota}\big|^2\Big)^{1/2}}_{\ell^p(\ZZ^{d})}\\
&\hspace{6cm}\lesssim 2^{\varepsilon l} 
\|f\|_{\ell^p(\ZZ^{d}; \ell^2(\NN))}.
\end{align}
To prove \eqref{eq:65} we can assume that $s=1$. The proof for $s\in[d]\setminus\{1\}$ is the same. In the argument below $\varepsilon\in(0, 1)$ may change from occurrence to occurrence.
\paragraph{\bf Step 2}
In fact we shall prove a more general inequality that will imply \eqref{eq:65}. Our aim will be to show that for every $p\in(1, \infty)$,    every
$f=(f_{\iota}:\iota\in\NN)\in \ell^p(\ZZ^d; \ell^2(\NN))$,  and any $A_k:=[k]\setminus\{1\}\subseteq [d]\setminus\{1\}$ with $k\in[d]$ (we note that $A_1=\emptyset$),  one has
\begin{align}
\label{eq:78}
\nonumber&\sup_{J\in\ZZ_+}\sup_{I\in\mathfrak S_J(\NN)}
\norm[\bigg]{\Big(\sum_{\iota\in\NN}\sum_{j=0}^{J-1}\sup_{\substack{n_i\in\NN\\
i\in A_k}} \sup_{I_j\le n_1<I_{j+1}}\big|T_{\ZZ^d}^{\Sigma_{\le l}^d}\big[(\Upsilon_{n_1}^{(1)}-\Upsilon_{I_{j}}^{(1)})\Big(\prod_{i\in A_k}\Upsilon_{n_i}^{(i)}\Big)\eta_{\le m}^{A_k^c}\big]f_{\iota}\big|^2\Big)^{1/2}}_{\ell^p(\ZZ^{d})}\\
&\hspace{6cm}\lesssim 2^{\varepsilon l}
\|f\|_{\ell^p(\ZZ^{d}; \ell^2(\NN))}.
\end{align}
Here the complement is taken with respect to the ambient set $[d]\setminus\{1\}$. In particular $A_1^c = [d]\setminus\{1\}$. The estimate \eqref{eq:65} follows from \eqref{eq:78} by taking $k=d$.

The proof of \eqref{eq:78}  is by induction on $k\in[d]$. The base case $k=1$ immediately follows from Theorem \ref{thm:IW1}; see the oscillation inequality \eqref{eq:381}. Thus we can assume that \eqref{eq:78} is true for some  $k\in[d-1]$ and we will prove that it remains true for $k+1$. In our argument
it is essential that we are working in the vector-valued setting, otherwise it is not clear how to employ the induction.
We will prove
\begin{align}
\label{eq:67}
\begin{gathered}
\sup_{J\in\ZZ_+}\sup_{I\in\mathfrak S_J(\NN)}
\norm[\bigg]{\Big(\sum_{\iota\in\NN}\sum_{j=0}^{J-1}\sup_{I_j\le n_1<I_{j+1}}\sup_{\substack{n_i\in\NN\\
i\in A_{k+1}}} |\mathfrak a_{n_{i}; i\in A_{k+1}}^{n_1, I_j }(f_{\iota})|^2\Big)^{1/2}}_{\ell^p(\ZZ^{d})}
\lesssim 2^{\varepsilon l}
\|f\|_{\ell^p(\ZZ^{d}; \ell^2(\NN))},
\end{gathered}
\end{align}
where
\begin{align*}
\mathfrak a_{n_{i}; i\in A_k}^{n_1, I_j }(f_{\iota}):=T_{\ZZ^{d}}^{\Sigma_{\le l}^{d}}\big[(\Upsilon_{n_1}^{(1)}-\Upsilon_{I_{j}}^{(1)})\big(\prod_{i\in A_k}\Upsilon_{n_i}^{(i)}\big)\eta_{\le m}^{A_k^c}\big]f_{\iota}, \qquad k\in\NN.
\end{align*}
By the disjointness of supports (see condition \eqref{eq:44}) we have the following key formula
\begin{align}
\label{eq:79}
\mathfrak a_{n_{i}; i\in A_{k+1}}^{n_1, I_j}(f_{\iota})=
\mathfrak a_{n_{i}; i\in A_k}^{n_1, I_j }\big(T_{\ZZ^d}^{\Sigma_{\le l}^{d}}\big[\Upsilon_{n_{k+1}}^{(k+1)}\eta_{\le m-1}^{[d]\setminus\{k+1\}}\big] f_{\iota}\big).
\end{align}
As in the proof of Theorem \ref{thm:IW1} we split the supremum over $n_{k+1}$ into small ($0\le n_{k+1}< 2^{\kappa_l}$) and large ($n_{k+1}\ge 2^{\kappa_l}$) scales, where $\kappa_l$ is the quantity defined in \eqref{eq:383}.

\paragraph{\bf Step 3}
We first handle the small scale case and prove for any $\varepsilon\in(0, 1)$ that
\begin{align}
\label{eq:68}
\begin{gathered}
\sup_{J\in\ZZ_+}\sup_{I\in\mathfrak S_J(\NN)}
\norm[\bigg]{\Big(\sum_{\iota\in\NN}\sum_{j=0}^{J-1}\sup_{I_j\le n_1<I_{j+1}}\sup_{\substack{n_i\in\NN\\
i\in A_k}}\sup_{n_{k+1}\in \NN_{<2^{\kappa_l}}} |\mathfrak a_{n_{i}; i\in A_{k+1}}^{n_1, I_j }(f_{\iota})|^2\Big)^{1/2}}_{\ell^p(\ZZ^{d})}
\lesssim 2^{\varepsilon l}
\|f\|_{\ell^p(\ZZ^{d}; \ell^2(\NN))}.
\end{gathered}
\end{align}
Using the Rademacher--Menshov inequality \eqref{eq:164}  we obtain
\begin{align*}
\sup_{n_{k+1}\in\NN_{<2^{\kappa_l}}}|\mathfrak a_{n_{i}; i\in A_{k+1}}^{n_1, I_j}(f_{\iota})|\lesssim |\mathfrak a_{n_{i}; i\in A_k}^{n_1, I_j }(F_{\iota})|
+\sum_{v=0}^{\kappa_l}\Big(\sum_{u=0}^{2^{\kappa_l-v}-1}| \mathfrak a_{n_{i}; i\in A_k}^{n_1, I_j }(F_{\iota}^{u, v})|^2\Big)^{1/2},
\end{align*}
where
\begin{align*}
F_{\iota}:=&T_{\ZZ^d}^{\Sigma_{\le l}^d}\big[\Upsilon_{0}^{(k+1)}\eta_{\le m-1}^{[d]\setminus\{k+1\}}\big]f_{\iota},\\
F_{\iota}^{u, v}:=&T_{\ZZ^d}^{\Sigma_{\le l}^d}\big[\sum_{j\in U_{u}^{v}}(\Upsilon_{j+1}^{(k+1)}-\Upsilon_{j}^{(k+1)})\eta_{\le m-1}^{[d]\setminus\{k+1\}}\big]f_{\iota},
\end{align*}
and  $U_u^v:=[u2^v, (u+1)2^v)$.
Consequently, we obtain
\begin{gather}
\text{ LHS of \eqref{eq:68} }
\label{eq:70}\lesssim
\sup_{J\in\ZZ_+}\sup_{I\in\mathfrak S_J(\NN)}
\norm[\bigg]{\Big(\sum_{\iota\in\NN}\sum_{j=0}^{J-1}\sup_{I_j\le n_1<I_{j+1}}\sup_{\substack{n_i\in\NN\\
i\in A_k}} |\mathfrak a_{n_{i}; i\in A_k}^{n_1, I_j }(F_{\iota})|^2\Big)^{1/2}}_{\ell^p(\ZZ^{d})}\\
\label{eq:71}\qquad +\sum_{v=0}^{\kappa_l}
\sup_{J\in\ZZ_+}\sup_{I\in\mathfrak S_J(\NN)}
\norm[\bigg]{\Big(\sum_{\iota\in\NN}\sum_{u=0}^{2^{\kappa_l-v}-1}\sum_{j=0}^{J-1}\sup_{I_j\le n_1<I_{j+1}}\sup_{\substack{n_i\in\NN\\
i\in A_k}} |\mathfrak a_{n_{i}; i\in A_k}^{n_1, I_j }(F_{\iota}^{u, v})|^2\Big)^{1/2}}_{\ell^p(\ZZ^{d})}.
\end{gather}

By induction we see that
\begin{equation}
\label{eq:72}
\text{ \eqref{eq:70} }\lesssim 2^{\varepsilon l} \bigg\|\Big(\sum_{\iota\in\NN}|F_{\iota}|^2\Big)^{1/2}\bigg\|_{\ell^p(\ZZ^d)}
\lesssim 2^{\varepsilon l}
\|f\|_{\ell^p(\ZZ^{d}; \ell^2(\NN))},
\end{equation}
where the second inequality follows from Lemma \ref{lem:10} and Theorem \ref{thm:iw-semi} due to \eqref{eq:380}. 

By induction and the definition of $\kappa_l$ we also see that
\begin{align}
\label{eq:73}
\text{ \eqref{eq:71} }\lesssim 2^{\varepsilon l} \bigg\|\Big(\sum_{\iota\in\NN}\sum_{u=0}^{2^{\kappa_l-v}-1}|F_{\iota}^{u, v}|^2\Big)^{1/2}\bigg\|_{\ell^p(\ZZ^d)}.
\end{align}
We note that this is the point in the proof where the vector-valued nature of our induction hypothesis plays a crucial role.
Now it suffices to justify that
\begin{align}
\label{eq:74}
\bigg\|\Big(\sum_{\iota\in\NN}\sum_{u=0}^{2^{\kappa_l-v}-1}|F_{\iota}^{u, v}|^2\Big)^{1/2}\bigg\|_{\ell^p(\ZZ^d)}\lesssim 2^{\varepsilon l}
\|f\|_{\ell^p(\ZZ^{d}; \ell^2(\NN))}.
\end{align}
Indeed, collecting \eqref{eq:70}, \eqref{eq:71}, \eqref{eq:72}, \eqref{eq:73} and \eqref{eq:74} we obtain \eqref{eq:68} as claimed. To prove \eqref{eq:74} we apply Theorem \ref{thm:iw-semi}, which reduces the problem to showing that for every $p\in(1, \infty)$ and every $f=(f_{\iota}:\iota\in\NN)\in L^p(\RR^d; \ell^2(\NN))$ one has
\begin{align}
\label{eq:76}
\bigg\|\Big(\sum_{\iota\in\NN}\sum_{u=0}^{2^{\kappa_l-v}-1}\big|T_{\RR^d}\big[\sum_{m\in U_{u}^{v}}(\Upsilon_{m+1}^{(k+1)}-\Upsilon_{m}^{(k+1)})\eta_{\le m-1}^{[d]\setminus\{k+1\}}\big]f_{\iota}\big|^2\Big)^{1/2}\bigg\|_{L^p(\RR^d)}\lesssim 
\|f\|_{L^p(\RR^{d}; \ell^2(\NN))}.
\end{align}
Invoking Lemma \ref{lem:10}, the  inequality \eqref{eq:76} follows from
\begin{align*}
\sup_{(\omega_n:n\in\NN)\in\{-1, 1\}^{\NN}}\norm[\Big]{\sum_{n\in\NN}
\omega_nT_{\RR^d}[(\Upsilon_{{n+1}}^{(k+1)}-\Upsilon_{{n}}^{(k+1)})\eta_{\le m-1}^{[d]\setminus\{k+1\}}]f}_{L^p(\RR^d)}
\lesssim_p\norm{f}_{L^p(\RR^d)},
\end{align*}
which in turn is a consequence of \cite[Theorem B]{DR} by conditions 
\eqref{eq:377}, \eqref{eq:378},  \eqref{eq:379} and \eqref{eq:380}. 
\paragraph{\bf Step 4}
We now handle the large scale case and prove 
\begin{align}
\label{eq:81}
\begin{gathered}
\sup_{J\in\ZZ_+}\sup_{I\in\mathfrak S_J(\NN)}
\norm[\bigg]{\Big(\sum_{\iota\in\NN}\sum_{j=0}^{J-1}\sup_{I_j\le n_1<I_{j+1}}\sup_{\substack{n_i\in\NN\\
i\in A_k}}\sup_{n_{k+1}\in \NN_{\ge 2^{\kappa_l}}} |\mathfrak  b_{n_1,\ldots, n_d}|^2\Big)^{1/2}}_{\ell^p(\ZZ^{d})}
\lesssim 2^{\varepsilon l} 
\|f\|_{\ell^p(\ZZ^{d}; \ell^2(\NN))}
\end{gathered}
\end{align}
with $\mathfrak b_{n_1,\ldots, n_d}=\mathfrak a_{n_{i}; i\in A_{k+1}}^{n_1, I_j }(f_{\iota})$.
We define
\begin{align*}
G_{n_{k+1}}^{\iota, 1}:=&T_{\ZZ^d}^{\Sigma_{\le l}^d}\big[\Upsilon_{n_{k+1}}^{(k+1)}\eta_{\le m-1}^{[d]\setminus\{k+1\}}\big] f_{\iota},\\
G_{n_{k+1}}^{\iota, 2}:=&T_{\ZZ^d}^{\Sigma_{\le l}^d}\big[\Upsilon_{n_{k+1}}^{(k+1)}\eta_{\le 2^{10 l}}^{(k+1)}\eta_{\le m-1}^{[d]\setminus\{k+1\}}\big]f_{\iota}.
\end{align*}
Invoking \eqref{eq:79} we obtain
\begin{align*}
\mathfrak a_{n_{i}; i\in A_{k+1}}^{n_1, I_j}(f_{\iota})=\mathfrak a_{n_{i}; i\in A_k}^{n_1, I_j }(G_{n_{k+1}}^{\iota, 1}-G_{n_{k+1}}^{\iota, 2})+
\mathfrak a_{n_{i}; i\in A_k}^{n_1, I_j }(G_{n_{k+1}}^{\iota, 2}).
\end{align*}
It suffices to show \eqref{eq:81} with 
$\mathfrak b_{n_1,\ldots, n_d}=\mathfrak a_{n_{i}; i\in A_k}^{n_1, I_j }(G_{n_{k+1}}^{\iota, 1}-G_{n_{k+1}}^{\iota, 2})$ and $\mathfrak b_{n_1,\ldots, n_d}=\mathfrak a_{n_{i}; i\in A_k}^{n_1, I_j }(G_{n_{k+1}}^{\iota, 2})$.
\vskip 5pt
\paragraph{\bf Step 5}
To prove \eqref{eq:81} with $\mathfrak b_{n_1,\ldots, n_d}=\mathfrak a_{n_{i}; i\in A_k}^{n_1, I_j }(G_{n_{k+1}}^{\iota, 1}-G_{n_{k+1}}^{\iota, 2})$ it suffices to show that for every $p\in(1, \infty)$ there exists $\delta_p\in(0, 1)$ such that for any $n_{k+1}\ge 2^{\kappa_l}$ we have
\begin{align}
\label{eq:88tensor}
\begin{gathered}
\sup_{J\in\ZZ_+}\sup_{I\in\mathfrak S_J(\NN)}
\norm[\bigg]{\Big(\sum_{\iota\in\NN}\sum_{j=0}^{J-1}\sup_{I_j\le n_1<I_{j+1}}\sup_{\substack{n_i\in\NN\\
i\in A_k}} |\mathfrak a_{n_{i}; i\in A_k}^{n_1, I_j }(G_{n_{k+1}}^{\iota, 1}-G_{n_{k+1}}^{\iota, 2})|^2\Big)^{1/2}}_{\ell^p(\ZZ^{d})}\\
\lesssim \tau^{-\delta_pn_{k+1}}
\|f\|_{\ell^p(\ZZ^{d}; \ell^2(\NN))}.
\end{gathered}
\end{align}
By the inductive hypothesis we see that
\begin{align}
\label{eq:89}
\text{ LHS of \eqref{eq:88tensor} }\lesssim
\norm[\bigg]{\Big(\sum_{\iota\in\NN}\big|T_{\ZZ^d}^{\Sigma_{\le l}^d}\big[\Upsilon_{n_{k+1}}^{(k+1)}(1-\eta_{\le 2^{10 l}}^{(k+1)})\eta_{\le m-1}^{[d]\setminus\{k+1\}}\big]f_{\iota}\big|^2\Big)^{1/2}}_{\ell^p(\ZZ^{d})}.
\end{align}
To estimate the right-hand side of \eqref{eq:89} we use Lemma \ref{lem:10} and Theorem \ref{thm:iw-semi}, which reduce the problem to proving for every $p\in(1, \infty)$ and every $f\in L^p(\RR^d)$ that
\begin{align}
\label{eq:90}
\big\|T_{\RR^d}[\Upsilon_{n_{k+1}}^{(k+1)}(1-\eta_{\le 2^{10 l}}^{(k+1)})\eta_{\le m-1}^{[d]\setminus\{k+1\}}]f\big\|_{L^p(\RR^d)}
\lesssim \tau^{-\delta_pn_{k+1}}
\|f\|_{L^p(\RR^{d})}.
\end{align}
By simple interpolation and Plancherel's theorem, it suffices  to show
\begin{align*}
\big|\Upsilon_{n_{k+1}}^{(k+1)}(\xi)\big(1-\eta_{\le 2^{10 l}}^{(k+1)}(\xi)\big) \big|\lesssim \tau^{-\delta_{\Upsilon} n_{k+1}/2},
\end{align*}
for all $n_{k+1}\ge 2^{\kappa_l}$,
which follows from the definition of $\kappa_l$ and \eqref{eq:378}.
\vskip 5pt
\paragraph{\bf Step 6} We now prove \eqref{eq:81} with $\mathfrak b_{n_1,\ldots, n_d}=\mathfrak a_{n_{i}; i\in A_k}^{n_1, I_j }(G_{n_{k+1}}^{\iota, 2})$.  Recall the definition $Q_{\le l}$ from \eqref{eq:373}. Since $Q_{\le l}\le 2^{2^{\kappa_l}}$, we can write
\begin{align*}
T_{\ZZ^d}^{\Sigma_{\le l}^d}\big[\Upsilon_{n_{k+1}}^{(k+1)}\eta_{\le 2^{10 l}}^{(k+1)}\eta_{\le m-1}^{[d]\setminus\{k+1\}}\big]f_{\iota}
=T_{\ZZ}^{Q_{\le l}^{-1}[Q_{\le l}]}\big[\Upsilon_{n_{k+1}}^{(k+1)}\eta_{\le 2^{10 l}}^{(k+1)}\big]
T_{\ZZ^d}^{\Sigma_{\le l}^d}\big[\eta_{\le m-1}^{[d]}\big]f_{\iota}.
\end{align*}
Consequently,
\begin{align}
\label{eq:83}
\mathfrak a_{n_{i}; i\in A_k}^{n_1, I_j }(G_{n_{k+1}}^{\iota, 2})=
T_{\ZZ}^{Q_{\le l}^{-1}[Q_{\le l}]}\big[\Upsilon_{n_{k+1}}^{(k+1)}\eta_{\le 2^{10 l}}^{(k+1)}\big]
\mathfrak a_{n_{i}; i\in A_k}^{n_1, I_j }\big(T_{\ZZ^d}^{\Sigma_{\le l}^d}\big[\eta_{\le m-1}^{[d]}\big]f_{\iota}\big).
\end{align}
Furthermore, by basic properties of the Fourier transform for $x=(x_1,\ldots, x_d)\in\ZZ^d$ one has
\begin{align*}
T_{\ZZ}^{Q_{\le l}^{-1}[Q_{\le l}]}\big[\Upsilon_{n_{k+1}}^{(k+1)}\eta_{\le 2^{10 l}}^{(k+1)}\big](x)
&=
\begin{cases}
Q_{\le l}\phi_{n_{k+1}}(x_{k+1})\prod_{i\in[d]\setminus\{k+1\}}\ind{\{0\}}(x_i), & \text{ if } x_{k+1}\in Q_{\le l}\ZZ,\\
0, & \text{ if } x_{k+1}\not\in Q_{\le l}\ZZ
\end{cases}
\\
&=\varphi^{Q_{\le l}}_{n_{k+1}}(x),
\end{align*}
where 
$$
\varphi^{Q_{\le l}}_{n}(x):=Q_{\le l}\phi_{n}(x_{k+1})\ind{Q_{\le l}\ZZ}(x_{k+1})\prod_{i\in[d]\setminus\{k+1\}}\ind{\{0\}}(x_i), \qquad x\in\ZZ^d,
$$ 
with
$\phi_{n}(y):=\mathcal{F}^{-1}_{\ZZ}[\Upsilon_{n}^{(k+1)}\eta_{\le 2^{10 l}}^{(k+1)}](y)$ for $y\in\ZZ$. 

For $\ZZ^d\ni x=({\bm x}, x_{k+1})$, with ${\bm x}=(x_1,x_2,\dots,x_k,x_{k+2}, \dots, x_d)$ we define $f^{{\bm x}}_{\iota}:\ZZ\rightarrow \CC$ by setting
$$
f^{{\bm x}}_{\iota}(x_{k+1}):=f_{\iota}(x).
$$
Then for every $p\in(1, \infty)$ and every $f=(f_{\iota}:\iota\in\NN)\in \ell^p(\ZZ^d; \ell^2(\NN))$ we can write
\[
T_{\ZZ}^{Q_{\le l}^{-1}[Q_{\le l}]}\big[\Upsilon_{n_{k+1}}^{(k+1)}\eta_{\le 2^{10 l}}^{(k+1)}\big]f_{\iota}=\varphi^{Q_{\le l}}_{n_{k+1}}*_{\ZZ^d}f_{\iota},
\]
and consequently
\begin{align*}
&\bigg\|\Big(\sum_{\iota\in\NN}\sup_{n_{k+1}\in\NN}\big|T_{\ZZ}^{Q_{\le l}^{-1}[Q_{\le l}]}\big[\Upsilon_{n_{k+1}}^{(k+1)}\eta_{\le 2^{10 l}}^{(k+1)}\big]f_{\iota}\big|^2\Big)^{1/2}\bigg\|_{\ell^p(\ZZ^d)}^p
\\
&\quad\le\bigg\|\Big(\sum_{\iota\in\NN}\sup_{n_{k+1}\in\NN}(|\varphi^{Q_{\le l}}_{n_{k+1}}|*_{\ZZ^d}|f_{\iota}|)^2\Big)^{1/2}\bigg\|^p_{\ell^p(\ZZ^d)}
\\
&\quad=\sum_{{\bm x}\in\ZZ^{d-1}}\sum_{r_{k+1}\in[Q_{\le l}]}\sum_{j_{k+1}\in\ZZ}\Big(\sum_{\iota\in\NN}\sup_{n_{k+1}\in\NN}\big( |\phi^{Q_{\le l}}_{n_{k+1}}|*_{\ZZ} |f_{\iota}^{{\bm x},r_{k+1},Q_{\le l} }|(j_{k+1}) \big)^2\Big)^{p/2},
\end{align*}
with $\phi^{Q_{\le l}}_{n}(y):=Q_{\le l}\phi_{n}(Q_{\le l} y)$ and $f_{\iota}^{{\bm x}, r,Q_{\le l}}(y):=f^{{\bm x}}_{\iota}(r+Q_{\le l}y)$ for $y\in\ZZ$.

Observe that we have a pointwise domination
$$
\sup_{n_{k+1}\in\NN}|\phi^{Q_{\le l}}_{n_{k+1}}|*_\ZZ |g|(y)\lesssim \mathcal{M}_{HL}g(y), \qquad y\in\ZZ, \qquad g\in\ell^p(\ZZ),
$$
with $\mathcal{M}_{HL}$ denoting the classical Hardy--Littlewood maximal operator on $\ZZ$. The implicit constant in the above inequality does not depend on $Q_{\le l}$. By the Fefferman--Stein inequality we obtain
\begin{align}\nonumber
&\sum_{{\bm x}\in\ZZ^{d-1}}\sum_{r_{k+1}\in[Q_{\le l}]}\sum_{j_{k+1}\in\ZZ}\Big(\sum_{\iota\in\NN}\sup_{n_{k+1}\in\NN}\big( |\phi^{Q_{\le l}}_{n_{k+1}}|*_{\ZZ} |f_{\iota}^{{\bm x},r_{k+1},Q_{\le l} }|(j_{k+1}) \big)^2\Big)^{p/2}
\\ \label{eq:82}
&\quad\lesssim
\sum_{{\bm x}\in\ZZ^{d-1}}\sum_{r_{k+1}\in[Q_{\le l}]}\sum_{j_{k+1}\in\ZZ}\Big(\sum_{\iota\in\NN} |f_{\iota}^{{\bm x},r_{k+1},Q_{\le l} }(j_{k+1})|^2\Big)^{p/2}=\|f\|_{\ell^p(\ZZ^{d}; \ell^2(\NN))}^p.
\end{align}

Using the identity \eqref{eq:83}, the formula for the kernel
$T_{\ZZ}^{Q_{\le l}^{-1}[Q_{\le l}]}\big[\Upsilon_{n_{k+1}}^{(k+1)}\eta_{\le 2^{10 l}}^{(k+1)}\big]$,
inequality \eqref{eq:82} and the induction hypothesis we conclude
\begin{multline*}
\sup_{J\in\ZZ_+}\sup_{I\in\mathfrak S_J(\NN)}
\norm[\bigg]{\Big(\sum_{\iota\in\NN}\sum_{j=0}^{J-1}\sup_{I_j\le n_1<I_{j+1}}\sup_{\substack{n_i\in\NN\\
i\in A_k}}\sup_{n_{k+1}\in \NN_{\ge 2^{\kappa_l}}} |\mathfrak a_{n_{i}; i\in A_k}^{n_1, I_j }(G_{n_{k+1}}^{\iota, 2})|^2\Big)^{1/2}}_{\ell^p(\ZZ^{d})}\\
\lesssim\sup_{J\in\ZZ_+}\sup_{I\in\mathfrak S_J(\NN)}
\norm[\bigg]{\Big(\sum_{\iota\in\NN}\sum_{j=0}^{J-1}\sup_{I_j\le n_1<I_{j+1}}\sup_{\substack{n_i\in\NN\\
i\in A_k}} \big|\mathfrak a_{n_{i}; i\in A_k}^{n_1, I_j }\big(T_{\ZZ^d}^{\Sigma_{\le l}^d}\big[\eta_{\le m-1}^{[d]}\big]f_{\iota}\big)\big|^2\Big)^{1/2}}_{\ell^p(\ZZ^{d})}\\
\lesssim 2^{\varepsilon l}
\norm[\bigg]{\Big(\sum_{\iota\in\NN}\big|T_{\ZZ^d}^{\Sigma_{\le l}^d}\big[\eta_{\le m-1}^{[d]}\big]f_{\iota}\big|^2\Big)^{1/2}}_{\ell^p(\ZZ^{d})}.
\end{multline*}
Finally  Lemma \ref{lem:10} and Theorem \ref{thm:iw-semi} yield
\begin{align*}
\norm[\bigg]{\Big(\sum_{\iota\in\NN}\big|T_{\ZZ^d}^{\Sigma_{\le l}^d}\big[\eta_{\le m-1}^{[d]}\big]f_{\iota}\big|^2\Big)^{1/2}}_{\ell^p(\ZZ^{d})}\lesssim 2^{\varepsilon l}\|f\|_{\ell^p(\ZZ^{d}; \ell^2(\NN))}.
\end{align*}
Combining these two estimates we obtain the desired claim, completing the proof.
\end{proof}

\subsection{Multi-parameter seminorm variant of the Ionescu--Wainger
theorem for  the continuous multipliers}
Let $k\in\ZZ_+$ and a finite and nonempty set
$\Gamma\subseteq \NN^k\setminus\{{\bm 0}\}$ be given.   We will use Theorem \ref{thm:IW1'}
together with the square function techniques developed recently in
\cite{KLMP} to prove a multi--parameter oscillation seminorm
Ionescu--Wainger type estimate for the operator associated with the
multiplier
\begin{equation}\label{eq:PhiM1M2_def_kpar}
\mathfrak{m}_{M}^{\langle\Gamma\rangle}(\xi):=\frac{1}{(\tau-1)^k}\int_{[1, \tau]^k }\ex\big( \xi\cdot (M\otimes t)^{\Gamma} \big)dt, \qquad \xi\in\TT^\Gamma,\quad M\in\RR_+^k.
\end{equation}

For notation,
see Section \ref{section:2}. By Proposition \ref{thm:CCW},  there exist $C_{k, \Gamma}\in\RR_+$ and  $\delta_{\Gamma}\in(0, 1)$ such that
\begin{align}
\label{eq:9}
|\mathfrak{m}_{M}^{\langle\Gamma\rangle}(\xi)|&\le C_{k, \Gamma} \big(\max_{\gamma\in\Gamma}|M^{\gamma}\xi_{\gamma}| \big)^{-\delta_{\Gamma}},\\
\label{eq:11}
|\mathfrak{m}_{M}^{\langle\Gamma\rangle}(\xi)-1|&\le C_{k, \Gamma} \min\big(1, \max_{\gamma\in\Gamma}|M^{\gamma}\xi_{\gamma}| \big).
\end{align}
The van der Corput bound from \eqref{eq:9} and \eqref{eq:11} will be used here and throughout the paper.

We need some technical preliminaries that will allow us to deal with
the multi-parameter oscillations along subsets of a ``dyadic'' grid
$\DD_{\tau}^k$. Suppose that $\Gamma=\{\gamma_i: i\in[d]\}$ for some
$d\in\ZZ_+$ such that $|\Gamma|=d$,  then the polynomial $(t)^{\Gamma}$ can be further rewritten  as follows $(t)^{\Gamma}=(t^{\gamma_1},\ldots, t^{\gamma_d})$
for $t\in\RR^k$.

Each coordinate \( t^{\gamma_i} \) of \( (t)^{\Gamma} \) determines a certain
``direction'', which is uniquely defined by its exponents
\( \gamma_i = (\gamma_{i,1}, \ldots, \gamma_{i,k}) \). A naive thought
suggests that for each direction, one should perform an appropriate
Littlewood--Paley decomposition to control long variations with a
square function.

This idea is effective if there are \( d \)
non-degenerate directions but fails if the number of directions is
less than \( d \), which may happen. To overcome this difficulty, we
introduce an auxiliary parameter
\( r \coloneqq \mathrm{rank}(\mathbb{M}^{\Gamma}) \), where
\( \mathbb{M}^{\Gamma} \) is a matrix of size \( d \times k \) whose
\( i \)-th row \( R_i \) is the exponent
\( \gamma_i = (\gamma_{i,1}, \ldots, \gamma_{i,k}) \) of the monomial
\( t^{\gamma_i} \) from \( (t)^{\Gamma} \). There is a one-to-one correspondence between
these exponents and the ``directions'', and moreover the number of
non-degenerate directions is \( r \leq \min\{d, k\} \).

Now, choose a set $\Gamma_r:=\{\gamma_{i_1}, \dots, \gamma_{i_r}\} \subseteq \Gamma$ such that the
corresponding rows $R_{i_1}, \dots, R_{i_r}$ are linearly
independent. It is straightforward to check that the map
$s \mapsto (s)^{\Gamma_r}=(s^{\gamma_{i_1}}, \dots, s^{\gamma_{i_r}})$ is surjective. Moreover,
for each $s \in \RR^k_+$ the value of $(s)^{\Gamma}$ is completely determined
by the value of $(s)^{\Gamma_r}$, since each row in $\mathbb{M}^{\Gamma}$ can be
expressed as a linear combination of vectors
$R_{i_1}, \dots, R_{i_r}$. For each $n \in \ZZ^r$ we pick some
$s(n) \in \RR^k_+$ such that
$(s(n))^{\Gamma_r} = {\tau}^n=(\tau^{n_1}, \dots, \tau^{n_r}) \in \DD_{\tau}^r$ and
the long variations will be defined with the aid of the set
$\{ s(n) : n \in \ZZ^r\}$.

\begin{lemma}\label{lem:long_only}
Let $k\in \ZZ_+$ and $\Gamma\subseteq \NN^k\setminus\{{\bm 0}\}$ be given, and let  $\{a_s : s \in \RR^k_+\}\subseteq \CC$ be a family of numbers satisfying the following condition: $(s)^{\Gamma}=(t)^{\Gamma} \implies a_s = a_{t}$.  Then 
\begin{align}
\label{eq:7}
\{a_s : s \in \DD_\tau^k\} \subseteq \{a_{s(n)}  : n \in \ZZ^r\}.
\end{align}
\end{lemma}
This lemma will be applied with $a_M=F(\mathfrak{m}_M^{\langle\Gamma\rangle})$ for $M\in\RR_+^k$, where $F$ is a function that depends on $\mathfrak{m}_M^{\langle\Gamma\rangle}$.  
Since the coordinates of the polynomial $(t)^{\Gamma}$ from \eqref{eq:PhiM1M2_def_kpar} are monomials
they satisfy the following key identity
\begin{align}
\label{eq:6}
(s\otimes t)^{\Gamma}=(s)^{\Gamma}\otimes (t)^{\Gamma} \quad \text{ for any } \quad s, t\in\RR^k.
\end{align}
Using \eqref{eq:6}, we see that the condition from Lemma \ref{lem:long_only} holds.

\begin{proof}[Proof of Lemma \ref{lem:long_only}]
Take any $s\in\DD_\tau^k$. Since $(s)^{\Gamma}$ is completely determined by the value of $(s)^{\Gamma_r}$, so $a_s$ is in fact determined by $(s)^{\Gamma_r}$.  The mapping $s \mapsto (s)^{\Gamma_r}$ consists of monomials with all coefficients equal $1$, thus for $s \in \DD_\tau^k$ we clearly have 
$(s)^{\Gamma_r}=\tau^n$ for some $n\in\ZZ^r$. However, by the definition of $s(n)$ we also have  
$(s(n))^{\Gamma_r}=\tau^n$. It follows that $a_s=a_{s(n)}$, which proves \eqref{eq:7}. 
\end{proof}

We are ready to state the main result of the chapter.
\begin{theorem}\label{thm:osc_IW_2par}
Let $k\in \ZZ_+$ and $\Gamma\subseteq \NN^k\setminus\{{\bm 0}\}$ be given.
Let $p\in[p_0',p_0]$ for some $p_0\in2\ZZ_+$ and let
$\varepsilon\in(0, 1)$. Then there exists an absolute constant
$C=C(\Gamma,p_0, \tau, \varepsilon)\in\RR_+$ such that for
all integers $l\in\NN$ and $m\in\RR_+$ satisfying
$2^{-m}\le (2p_0 2^{l p_0})^{-1}$ the following holds
\begin{align}\label{eq:osc_IW_2par}
\sup_{J\in\ZZ_+}\sup_{I\in\mathfrak S_J(\DD_{\tau}^k) }\big\|O_{I, J}\big (T_{\ZZ^\Gamma}^{\Sigma^\Gamma_{\le l}}\big[ \mathfrak{m}_{M}^{\langle\Gamma\rangle}\eta^\Gamma_{\le m}\big]f: M \in\DD^k_{\tau}\big)\big\|_{\ell^p(\ZZ^\Gamma)}\le C 2^{\varepsilon l}\|f\|_{\ell^p(\ZZ^\Gamma)},
\end{align}
where $\mathfrak{m}_{M}^{\langle\Gamma\rangle}$ is the multiplier given by \eqref{eq:PhiM1M2_def_kpar}.
\end{theorem}

\begin{proof}
Fix $k\in \ZZ_+$ and
$\Gamma\subseteq \NN^k\setminus\{{\bm 0}\}$. Given a fixed Schwartz
function $\phi \colon \RR\rightarrow \RR_+$ such that
$\int_{\RR}\phi(x) \, \dif x=1$ we set
$\Upsilon \coloneqq \calF_\RR \phi$.  Then for $D\subseteq\Gamma$ let
\begin{equation*}
\mathfrak{m}_{M}^{\langle\Gamma\rangle, D}(\xi):=\mathfrak{m}_{M}^{\langle\Gamma\rangle}(\pi^D(\xi))\prod_{\gamma\in D}\Upsilon(M^\gamma \xi_\gamma),
\end{equation*}
where $\pi^D(\xi)\in\TT^\Gamma$ satisfies $(\pi^D(\xi))_\gamma := 0$ if $\gamma \in D$ and $(\pi^D(\xi))_\gamma := \xi_\gamma$ otherwise. In particular, $\mathfrak{m}_{M}^{\langle\Gamma\rangle, \emptyset}(\xi)=\mathfrak{m}_{M}^{\langle\Gamma\rangle}(\xi)$.  Now, we will divide the proof into several steps.

\paragraph{\bf Step 1} By the inclusion-exclusion formula, see \cite[Formula (4.3)]{KLMP}, we can  write
\begin{equation}\label{eq:deco4.3}
\mathfrak{m}_{M}^{\langle\Gamma\rangle}=\sum_{\emptyset\neq D\subseteq\Gamma}(-1)^{|D|+1}\mathfrak{m}^{\langle\Gamma\rangle, D}_{M}+\tilde{\mathfrak{m}}_{M}^{\langle\Gamma\rangle},
\end{equation}
where $\tilde{\mathfrak{m}}_{M}^{\langle\Gamma\rangle}:= \sum_{D \subseteq\Gamma} (-1)^{|D|} \mathfrak{m}^{\langle\Gamma\rangle, D}_{M}$. Each of the terms $\mathfrak{m}^{\langle\Gamma\rangle, D}_{M}$ in the sum on the right-hand side of \eqref{eq:deco4.3} can be further decomposed by writing
\begin{equation*}
\mathfrak{m}_{M}^{\langle\Gamma\rangle}(\pi^D(\xi))=\sum_{\emptyset\neq D_1\subseteq\Gamma\setminus D}(-1)^{|D_1|+1}\mathfrak{m}^{\langle\Gamma\rangle, D,D_1}_{M}(\xi)+\tilde{\mathfrak{m}}^{\langle\Gamma\rangle, D}_{M}(\xi),
\end{equation*}
with
$$
\mathfrak{m}^{\langle\Gamma\rangle, D,D_1}_{M}(\xi):=\mathfrak{m}_{M}^{\langle\Gamma\rangle}(\pi^{D\cup D_1}(\xi))\prod_{\gamma\in D_1}\Upsilon(M^\gamma\xi_\gamma),
$$
and
$$
\tilde{\mathfrak{m}}^{\langle\Gamma\rangle, D}_{M}:=\sum_{D_1\subseteq \Gamma\setminus D}(-1)^{|D_1|}\mathfrak{m}^{\langle\Gamma\rangle, D,D_1}_{M}.
$$
Consequently,  for each $\emptyset\neq D\subseteq \Gamma$, we have
\begin{equation*}
\mathfrak{m}_{M}^{\langle\Gamma\rangle, D}(\xi)=\sum_{\emptyset\neq D_1\subseteq\Gamma\setminus D}(-1)^{|D_1|+1}\mathfrak{m}_{M}^{\langle\Gamma\rangle}(\pi^{D\cup D_1}(\xi))\prod_{\gamma\in D\cup D_1}\Upsilon(M^\gamma \xi_\gamma)+\tilde{\mathfrak{m}}^{\langle\Gamma\rangle, D}_{M}(\xi)\prod_{\gamma\in D}\Upsilon(M^\gamma \xi_\gamma).
\end{equation*}
Plugging this to \eqref{eq:deco4.3}, we obtain
\begin{align*}
\mathfrak{m}_{M}^{\langle\Gamma\rangle}(\xi)&=\sum_{\emptyset\neq D\subseteq\Gamma}(-1)^{|D|+1}\sum_{\emptyset\neq D_1\subseteq\Gamma\setminus D}(-1)^{|D_1|+1}\mathfrak{m}_{M}^{\langle\Gamma\rangle}(\pi^{D\cup D_1}(\xi))\prod_{\gamma\in D\cup D_1}\Upsilon(M^\gamma \xi_\gamma)
\\
&\quad+\sum_{\emptyset\neq D\subseteq\Gamma}(-1)^{|D|+1}\tilde{\mathfrak{m}}^{\langle\Gamma\rangle, D}_{M}(\xi)\prod_{\gamma\in D}\Upsilon(M^\gamma \xi_\gamma)+\tilde{\mathfrak{m}}_{M}^{\langle\Gamma\rangle}(\xi).
\end{align*}
Next, we apply the above decomposition to $\mathfrak{m}_{M}^{\langle\Gamma\rangle}(\pi^{D\cup D_1}(\xi))$, then to $\mathfrak{m}_{M}^{\langle\Gamma\rangle}(\pi^{D\cup D_1\cup D_2}(\xi))$ with $D_2\subseteq\Gamma\setminus{D\cup D_1}$ and so on. After doing it sufficiently many times we obtain the representation of $\mathfrak{m}_{M}^{\langle\Gamma\rangle}$ as a finite linear combination of $
\prod_{\gamma\in\Gamma}\Upsilon(M^\gamma \xi_\gamma)
$
and terms of the form
$$
\tilde{\mathfrak{m}}^{\langle\Gamma\rangle, D \cup D_1\cup\dots\cup D_n}_{M}(\xi)\prod_{\gamma\in D\cup D_1\cup \dots \cup D_n}\Upsilon(M^\gamma \xi_\gamma),
$$
for some $D\subseteq\Gamma, D_1\subseteq\Gamma\setminus{D}, D_2\subseteq\Gamma\setminus{D\cup D_1}, \dots, D_n\subseteq \Gamma\setminus{D\cup D_1\cup\dots\cup D_{n-1}}$. 

\paragraph{\bf Step 2} Consequently, to prove \eqref{eq:osc_IW_2par} it suffices to show for any $D\subseteq\Gamma$ the following two estimates
\begin{align}\label{eq:osc_IW_2par_p1}
\sup_{J\in\ZZ_+}\sup_{I\in\mathfrak S_J(\DD_{\tau}^k) }\big\|O_{I, J}\big (T_{\ZZ^\Gamma}^{\Sigma^\Gamma_{\le l}}\big[\prod_{\gamma\in\Gamma}\Upsilon(M^\gamma \xi_\gamma)\eta^\Gamma_{\le m}(\xi)\big]f: M \in\DD^k_{\tau}\big)\big\|_{\ell^p(\ZZ^\Gamma)}\lesssim_{p, \tau, \varepsilon, \Gamma}2^{\varepsilon l}\|f\|_{\ell^p(\ZZ^\Gamma)},
\end{align}
and
\begin{align}\label{eq:osc_IW_2par_p2} 
\sup_{J\in\ZZ_+}\sup_{I\in\mathfrak S_J(\DD_{\tau}^k) }\big\|O_{I, J}\big (T_{\ZZ^\Gamma}^{\Sigma^\Gamma_{\le l}}\big[\tilde{\mathfrak{m}}^{\langle\Gamma\rangle, D}_{M}(\xi)\prod_{\gamma\in D}\Upsilon(M^\gamma \xi_\gamma)\big]f: M \in\DD^k_{\tau}\big)\big\|_{\ell^p(\ZZ^\Gamma)}
\lesssim_{p, \tau, \varepsilon, \Gamma}2^{\varepsilon l}\|f\|_{\ell^p(\ZZ^\Gamma)}.
\end{align}

\paragraph{\bf Step 3} We begin with showing \eqref{eq:osc_IW_2par_p1}. We fix $J\in\ZZ_+$ and $I\in\mathfrak S_J(\DD_{\tau}^k)$.  Note that if $M \in\DD^k_{\tau}$ then  $(M^\gamma)_{\gamma\in\Gamma}\in \DD^\Gamma_{\tau}$, thus
\begin{align*}
&O_{I, J}\big (T_{\ZZ^\Gamma}^{\Sigma^\Gamma_{\le l}}\big[\prod_{\gamma\in\Gamma}\Upsilon(M^\gamma \xi_\gamma)\eta^\Gamma_{\le m}(\xi)\big]f: M \in\DD^k_{\tau}\big)
\\
&\quad\le
O_{(I)^{\Gamma}, J}\big (T_{\ZZ^\Gamma}^{\Sigma^\Gamma_{\le l}}\big[\prod_{\gamma\in\Gamma}\Upsilon(u_\gamma \xi_\gamma)\eta^\Gamma_{\le m}(\xi)\big]f:(u_\gamma)_{\gamma\in\Gamma} \in\DD_{\tau}^\Gamma\big).
\end{align*} 
Thus, it remains to apply the estimate
\begin{align*}
\sup_{J\in\ZZ_+}\sup_{I\in\mathfrak S_J(\DD_{\tau}^k) }\big\|O_{I, J}\big (T_{\ZZ^\Gamma}^{\Sigma^\Gamma_{\le l}}\big[\prod_{\gamma\in\Gamma}\Upsilon(u_\gamma \xi_\gamma)\eta^\Gamma_{\le m}(\xi)\big]f:(u_\gamma)_{\gamma\in\Gamma} \in\DD_{\tau}^\Gamma\big)\big\|_{\ell^p(\ZZ^\Gamma)}\lesssim2^{\varepsilon l}\|f\|_{\ell^p(\ZZ^\Gamma)},
\end{align*}
which follows from Theorem \ref{thm:IW1'}. This completes the proof of \eqref{eq:osc_IW_2par_p1}.

\paragraph{\bf Step 4}  We turn to showing \eqref{eq:osc_IW_2par_p2}. Fix $J\in\ZZ_+, I\in\mathfrak S_J(\DD_{\tau}^k)$ and $j\in\mathbb{N}_{<J}$, and observe that
\begin{align*}
\sup_{M\in\mathbb B[I_j]\cap \DD_{\tau}^k}&\Big|T_{\ZZ^\Gamma}^{\Sigma^\Gamma_{\le l}}\big[\tilde{\mathfrak{m}}^{\langle\Gamma\rangle, D}_{M}(\xi) \prod_{\gamma\in D}\Upsilon(M^\gamma \xi_\gamma)\big]f-T_{\ZZ^\Gamma}^{\Sigma^\Gamma_{\le l}}\big[\tilde{\mathfrak{m}}^{\langle\Gamma\rangle, D}_{I_j}(\xi) \prod_{\gamma\in D}\Upsilon(I_j^\gamma\xi_\gamma)\big]f\Big|\\
&\le
\sup_{M\in\mathbb B[I_j]\cap \DD_{\tau}^k}\Big|T_{\ZZ^\Gamma}^{\Sigma^\Gamma_{\le l}}\big[\big(\tilde{\mathfrak{m}}^{\langle\Gamma\rangle, D}_{M}(\xi)-\tilde{\mathfrak{m}}^{\langle\Gamma\rangle, D}_{I_j}(\xi)\big)\prod_{\gamma\in D}\Upsilon(M^\gamma \xi_\gamma)\big]f\Big|\\
&+\sup_{M\in\mathbb B[I_j]\cap \DD_{\tau}^k}\Big|T_{\ZZ^\Gamma}^{\Sigma^\Gamma_{\le l}}\big[\tilde{\mathfrak{m}}^{\langle\Gamma\rangle, D}_{M}(\xi)\big(\prod_{\gamma\in D}\Upsilon(M^\gamma \xi_\gamma)- \prod_{\gamma\in D}\Upsilon(I_j^\gamma\xi_\gamma)\big)\big]f\Big|.
\end{align*}
To complete the proof of \eqref{eq:osc_IW_2par_p2} we need to show, for any $\varrho\in(0, 1)$, that
\begin{align}\label{eq:(1)} \nonumber
\sup_{J\in\ZZ_+}\sup_{I\in\mathfrak S_J(\ZZ^\Gamma)}&\norm[\Big]{\Big(\sum_{j=0}^{J-1}\sup_{M\in\mathbb B[I_j]\cap \DD_{\tau}^k}\sup_{u\in\NN^D}  \Big| T_{\ZZ^\Gamma}^{\Sigma^\Gamma_{\le l}}\big[(\tilde{\mathfrak{m}}^{\langle\Gamma\rangle, D}_{M}(\xi)-\tilde{\mathfrak{m}}^{\langle\Gamma\rangle, D}_{I_j}(\xi))\prod_{\gamma\in D}\Upsilon(u_\gamma \xi_\gamma)\big] f \Big|^2 \Big)^{1/2}}_{\ell^p(\ZZ^\Gamma)}
\\ 
&\qquad\lesssim 2^{\varrho l}\norm{f}_{\ell^p(\ZZ^\Gamma)},
\end{align}
and
\begin{align}\label{eq:(2)} \nonumber
\sup_{J\in\ZZ_+}\sup_{U\in\mathfrak S_J(\DD_{\tau}^D)}&\norm[\Big]{\Big(\sum_{j=0}^{J-1}\sup_{u\in\mathbb B[U_j]\cap \DD_{\tau}^D}\sup_{M\in\DD^k_\tau}  \Big| T_{\ZZ^\Gamma}^{\Sigma^\Gamma_{\le l}}\big[
\tilde{\mathfrak{m}}^{\langle\Gamma\rangle, D}_{M}(\xi)
\big(\prod_{\gamma\in D}\Upsilon(u_\gamma\xi_\gamma)
-\prod_{\gamma\in D}\Upsilon(U^j_\gamma\xi_\gamma)\big)\big]f \Big|^2 \Big)^{1/2}}_{\ell^p(\ZZ^\Gamma)}
\\ 
&\qquad\lesssim 2^{\varrho l}\norm{f}_{\ell^p(\ZZ^\Gamma)},
\end{align}
where $U_j=(U_\gamma^j)_{\gamma\in D}$ for $j\in\NN_{<J}$. 

\paragraph{\bf Step 5} We first treat \eqref{eq:(1)}. We begin with showing the vector-valued estimate
\begin{align}
\label{eq:tildedvecval}
\begin{gathered}
\sup_{J\in\ZZ_+}\sup_{I\in\mathfrak S_J(\ZZ^\Gamma)}
\norm[\bigg]{\Big(\sum_{\iota\in\NN}O_{I, J}\big(T_{\ZZ^\Gamma}^{\Sigma^\Gamma_{\le l}}\big[\tilde{\mathfrak{m}}^{\langle\Gamma\rangle, D}_M\eta^D_{\le m}\big]f_{\iota}: M\in\DD_\tau^k\big)^2\Big)^{1/2}}_{\ell^p(\ZZ^\Gamma)}
\lesssim
2^{\varepsilon l}\|f\|_{\ell^p(\ZZ^\Gamma; \ell^2(\NN))}.
\end{gathered}
\end{align}
Note that $\tilde{\mathfrak{m}}^{\langle\Gamma\rangle, D}_{M}$ is fully determined by the
mapping $M\mapsto (M)^{\Gamma_r}$, where $r=r(D)\in\ZZ_+$ is the
number of independent monomials in the family
$(t)^{\Gamma\setminus D}$.  Using 
\eqref{eq:7} and  \eqref{eq:91}, we obtain
\begin{align*}
\sup_{J\in\ZZ_+}\sup_{I\in\mathfrak S_J(\ZZ^\Gamma)}&
\norm[\bigg]{\Big(\sum_{\iota\in\NN}O_{I, J}\big(T_{\ZZ^\Gamma}^{\Sigma^\Gamma_{\le l}}\big[\tilde{\mathfrak{m}}^{\langle\Gamma\rangle, D}_M\eta^D_{\le m}\big]f_{\iota}: M\in\DD_\tau^k\big)^2\Big)^{1/2}}_{\ell^p(\ZZ^\Gamma)}
\\
&\quad\lesssim
\norm[\bigg]{\Big(\sum_{\iota\in\NN}\sum_{n\in\ZZ^r}\big|T_{\ZZ^\Gamma}^{\Sigma^\Gamma_{\le l}}\big[\tilde{\mathfrak{m}}^{\langle\Gamma\rangle, D}_{s(n)}\eta^D_{\le m}\big]f_{\iota}\big|^2\Big)^{1/2}}_{\ell^p(\ZZ^\Gamma)}.
\end{align*}
We shall show, for every $\varepsilon\in(0, 1)$, that
\begin{align}\label{eq:tildedsqf}
\norm[\bigg]{\Big(\sum_{\iota\in\NN}\sum_{n\in\ZZ^r}\big|T_{\ZZ^\Gamma}^{\Sigma^\Gamma_{\le l}}\big[\tilde{\mathfrak{m}}^{\langle\Gamma\rangle, D}_{s(n)}\eta^D_{\le m}\big]f_{\iota}\big|^2\Big)^{1/2}}_{\ell^p(\ZZ^\Gamma)}
\lesssim 2^{\varepsilon l}\|f\|_{\ell^p(\ZZ^\Gamma; \ell^2(\NN))}.
\end{align}
By Theorem \ref{thm:iw-semi}, to prove \eqref{eq:tildedsqf} it suffices to show a continuous estimate
\begin{align*}
\norm[\bigg]{\Big(\sum_{\iota\in\NN}\sum_{n\in\ZZ^r}\big|T_{\RR^\Gamma}\big[\tilde{\mathfrak{m}}^{\langle\Gamma\rangle, D}_{s(n)}\eta^D_{\le m}]f_{\iota}\big|^2\Big)^{1/2}}_{L^p(\RR^\Gamma)}\lesssim \norm[\Big]{\big(\sum_{\iota\in\NN}|f_{\iota}|^2\big)^{1/2}}_{L^p(\RR^\Gamma)}.
\end{align*}
In view of the Marcinkiewicz--Zygmund inequality, see Lemma
\ref{lem:10}, the above bound will follow from a scalar--valued
estimate
\begin{align}\label{eq:LPalaKLMP}
\sup_{(\omega_n)_{n\in\NN^r}\in \{-1,1\}^{\NN^r }}\norm[\big]{\sum_{n\in\ZZ^r}\omega_{n} H_{s(n)}f}_{L^p(\RR^\Gamma)}\lesssim \norm{f}_{L^p(\RR^\Gamma)},
\end{align}
with $H_s f:=T_{\RR^\Gamma}\big[\tilde{\mathfrak{m}}^{\langle\Gamma\rangle, D}_{s}\eta^D_{\le m}]f$ for $s\in \RR_+^k$.

\paragraph{\bf Step 6} To complete the proof of \eqref{eq:LPalaKLMP}, we will use the Littlewood--Paley theory. Following the notation from \cite{KLMP}, we define 
$$
R(n) \coloneqq  \{ x \in \RR^r_+ : \tau^n \preceq x \prec_{\rm s} \tau^{n + {\bf 1}}\}=\mathbb B[\tau^n],
\quad\text{ for } \quad n \in \ZZ^r
$$
and  let $S_n$ be the Littlewood--Paley projection corresponding to the cube $R(-n)$. In other words, we set $S_n f \coloneqq   T_{\RR^\Gamma} [\Xi_n] f$, where 
\[
\Xi_n ( \xi) \coloneqq  \ind{R(-n)} (|\xi_{i_1}|, \dots, |\xi_{i_r}|), \qquad \xi=(\xi_\gamma)_{\gamma\in\Gamma}\in\RR^\Gamma,
\]
with $i_1, \dots, i_r\in [d]$ denoting the indices corresponding to the maximal linearly independent set of monomials in $(t)^{\Gamma\setminus D}$. Then the identity
	\begin{align}\label{eq:PL1}
	\sum_{n \in \ZZ^r} S_{n}= \operatorname{Id}
	\end{align}
	holds in the strong operator topology on $L^2(\RR^\Gamma)$ and almost everywhere on $\RR^\Gamma$ for every $f\in L^2(\RR^\Gamma)$.   We also have
	\begin{align}\label{eq:PL2}
	\norm[\Big]{\Big(\sum_{n \in \ZZ^r} \abs{S_n f}^{2} \Big)^{\frac{1}{2}}}_{L^p(\RR^\Gamma)}
	\simeq
	\norm{f}_{L^p(\RR^\Gamma)}, \qquad f\in L^p(\RR^\Gamma).
	\end{align}
 Moreover, for each $m, n \in \ZZ^r$ and $\xi \in \RR^\Gamma$ such that $(\xi_{i_1}, \dots, \xi_{i_r}) \in R(-n-m)$ we have
\[
|\tilde{\mathfrak{m}}_{s(n)}^{\langle\Gamma\rangle, D}(\xi) | \lesssim \min \big\{ |\tau^{n_1} \xi_{i_1}|^{\pm \delta}, \dots, |\tau^{n_r} \xi_{i_r}|^{\pm  \delta} \big\} \lesssim \tau^{- \frac{ \delta |m|_1}{r}}
\]
for some $\delta\in(0, 1)$. This bound follows from \eqref{eq:9} and \eqref{eq:11} and the definition of $\tilde{\mathfrak{m}}_{s(n)}^{\langle\Gamma\rangle, D}(\xi)$. 
Using the above Littlewood--Paley projections we can write for $\omega_n\in\{-1,1\}$ that
\begin{align*}
\norm[\big]{\sum_{n\in\ZZ^r}\omega_{n} H_{s(n)}}_{L^p(\RR^\Gamma)}
&=\norm[\big]{\sum_{n\in\ZZ^r}\omega_{n} H_{s(n)}  \sum_{h\in\ZZ^r}  S_{n+h}S_{n+h} f}_{L^p(\RR^\Gamma)}\\
&\lesssim
\sum_{h\in\ZZ^r}\norm[\big]{\big(\sum_{n\in\ZZ^r}| H_{s(n)} S_{n+h} f|^2\big)^{1/2}}_{L^p(\RR^\Gamma)}.
\end{align*}
Finally, one can argue as in \cite[Section 3.3]{KLMP} to conclude
\begin{align*}
\sum_{h\in\ZZ^r}\norm[\big]{\big(\sum_{n\in\ZZ^r}| H_{s(n)} S_{n+h} f|^2\big)^{1/2}}_{L^p(\RR^\Gamma)}\lesssim \|f\|_{L^p(\RR^\Gamma)}.
\end{align*}
We refer to \cite[Inequality (3.11)]{KLMP} for more details. This concludes the proof of \eqref{eq:tildedvecval}.

\paragraph{\bf Step 7} To complete the proof of inequality \eqref{eq:(1)} we will show a stronger statement: for any $D,D'\subseteq \Gamma$, and all $J\in\ZZ_+$ and $I\in\mathfrak S_J(\ZZ^\Gamma)$ one has
\begin{align*}\label{eq:(1)} \nonumber
&\norm[\Big]{\Big(\sum_{\iota\in\NN}\sum_{j=0}^{J-1}\sup_{M\in\mathbb B[I_j]\cap\DD_{\tau}^k}\sup_{u\in\NN^{D'}}  \Big| T_{\ZZ^\Gamma}^{\Sigma^\Gamma_{\le l}}\big[(\tilde{\mathfrak{m}}^{\langle\Gamma\rangle, D}_{M}(\xi)-\tilde{\mathfrak{m}}^{\langle\Gamma\rangle, D}_{I_j}(\xi))\prod_{\gamma\in D'}\Upsilon(u_\gamma\xi_\gamma)\eta^\Gamma_{\le m}(\xi)\big]f_\iota \Big|^2 \Big)^{1/2}}_{\ell^p(\ZZ^\Gamma)}
\\
&\quad\lesssim 2^{\varrho l}\|f\|_{\ell^p(\ZZ^\Gamma; \ell^2(\NN))}.
\end{align*}
The above estimate can be proved by induction on the number of elements in $D'$. The base case $D'=\emptyset$ follows from inequality \eqref{eq:tildedvecval}. The rest of the proof is almost identical to the proof of inequality \eqref{eq:78}. We omit the details.

\paragraph{\bf Step 8} We pass to showing \eqref{eq:(2)}. By inequality \eqref{eq:91}, for any $J\in\ZZ_+$ and $U\in\mathfrak S_J(\NN^D)$, we have
\begin{align*}
&\norm[\Big]{\Big(\sum_{j=0}^{J-1}\sup_{u\in\mathbb B[U_j]\cap\DD_{\tau}^D}\sup_{M\in\DD^k_\tau}  \big| T_{\ZZ^\Gamma}^{\Sigma^\Gamma_{\le l}}\big[\tilde{\mathfrak{m}}^{\langle\Gamma\rangle, D}_{M}(\xi)\big(\prod_{\gamma\in D}\Upsilon(u_\gamma\xi_\gamma)
-\prod_{\gamma\in D}\Upsilon(U^j_\gamma\xi_\gamma)\big)\big]f) \big|^2 \Big)^{1/2}}_{\ell^p(\ZZ^\Gamma)}
\\
&\quad\le
\norm[\Big]{\Big(\sum_{n\in\NN^r}\sum_{j=0}^{J-1}\sup_{u\in\mathbb B[U_j]\cap\DD_{\tau}^D} \big| T_{\ZZ^\Gamma}^{\Sigma^\Gamma_{\le l}}\big[\tilde{\mathfrak{m}}^{\langle\Gamma\rangle, D}_{s(n)}(\xi)\big(\prod_{\gamma\in D}\Upsilon(u_\gamma\xi_\gamma)
-\prod_{\gamma\in D}\Upsilon(U^j_\gamma\xi_\gamma)\big)\big]f \big|^2 \Big)^{1/2}}_{\ell^p(\ZZ^\Gamma)}.
\end{align*}
Consider the following factorization
\begin{align*}
T_{\ZZ^\Gamma}^{\Sigma^\Gamma_{\le l}}\big[\tilde{\mathfrak{m}}^{\langle\Gamma\rangle, D}_{s(n)}(\xi)\big(\prod_{\gamma\in D}\Upsilon(u_\gamma\xi_\gamma)
-\prod_{\gamma\in D}\Upsilon(U^j_\gamma\xi_\gamma)\big)\big]f
=
T_{\ZZ^\Gamma}^{\Sigma^\Gamma_{\le l}}\big[
\eta^{\Gamma\setminus D}_{\le m}(\xi)\big(\prod_{\gamma\in D}\Upsilon(u_\gamma\xi_\gamma)
-\prod_{\gamma\in D}\Upsilon(U^j_\gamma\xi_\gamma)\big)\big]F_n,
\end{align*}
where
$$
F_n:=T_{\ZZ^\Gamma}^{\Sigma^\Gamma_{\le l}}\big[\tilde{\mathfrak{m}}^{\langle\Gamma\rangle, D}_{s(n)}\eta^D_{\le m}\big]f.
$$
Using the vector-valued oscillation inequality for the tensor product of bumps from Theorem \ref{thm:IW1'}, we obtain
\begin{align*}
&\norm[\Big]{\Big(\sum_{n\in\NN^r}\sum_{j=0}^{J-1}\sup_{u\in\mathbb B[U_j]\cap\DD_{\tau}^D}
\Big| T_{\ZZ^\Gamma}^{\Sigma^\Gamma_{\le l}}\big[
\big(\prod_{\gamma\in D}\Upsilon(u_\gamma\xi_\gamma)
-\prod_{\gamma\in D}\Upsilon(U^j_\gamma\xi_\gamma)\big)\big]F_n \Big|^2 \Big)^{1/2}}_{\ell^p(\ZZ^\Gamma)}
\\
&\quad\lesssim
2^{\varepsilon l}\norm[\big]{(\sum_{n\in\NN^r}|F_n|^2)^{1/2}}_{\ell^p(\ZZ^\Gamma)}.
\end{align*}
Finally, applying the estimate \eqref{eq:tildedsqf} (only the scalar-valued variant is needed here) we obtain
\begin{align*}
\norm[\big]{(\sum_{n\in\NN^r}|F_n|^2)^{1/2}}_{\ell^p(\ZZ^\Gamma)}
\lesssim
2^{\varepsilon l}\norm{f}_{\ell^p(\ZZ^\Gamma)}.
\end{align*}
That concludes the proof of \eqref{eq:(2)}, as well as \eqref{eq:osc_IW_2par_p2}. 
This completes the proof of the theorem.
\end{proof}

\section{Multi-parameter circle method: Proof of Theorem \ref{thm:osc}}
\label{sec:circle}
For ease of exposition, we only prove Theorem \ref{thm:osc} in the
two-parameter setting $k=2$. All our arguments are  adaptable to the general
multi-parameter setup at the expense of additional work. Fix $d_1, d_2\in\ZZ_+$, and as in Section \ref{sec:exp}, see \eqref{eq:2}, we let
\[
\Gamma:=\Gamma(d_1, d_2):=\NN_{\le d_1}\times \NN_{\le d_2}\setminus\{(0, 0)\},
\]
and we define four other sets
\begin{align}
\label{eq:47}
\begin{split}
\Gamma_1:=&\{\gamma\in \Gamma: \gamma_1\neq0\}=[d_1]\times\NN_{\le d_2}, \qquad \text{ and } \qquad \Gamma_1^c:=\Gamma\setminus\Gamma_1=\{0\}\times [d_2],\\
\Gamma_2:=&\{\gamma\in \Gamma: \gamma_2\neq0\}=\NN_{\le d_1}\times[d_2], \qquad \text{ and } \qquad
\Gamma_2^c:=\Gamma\setminus\Gamma_2=[d_1]\times\{0\}.
\end{split}
\end{align}
For every real number $K\ge1$ define
\begin{align}\label{eq:chiN_def}
\chi_K(x):=\frac{1}{|(K, \tau K]\cap\ZZ|}\ind{(K, \tau K]}(x), \qquad x\in \RR,
\end{align}
and for any real numbers $M_1, M_2\ge1$ and $\xi\in\RR^\Gamma$, define the multiplier 
\begin{align}\label{eq:mM1M2_def}
m_{M_1, M_2}^{\langle \Gamma \rangle}(\xi):=\sum_{m_1\in\ZZ}\sum_{m_2\in\ZZ}\chi_{M_1}(m_1)\chi_{M_2}(m_2)\ex\big({\xi}\cdot (m_1, m_2)^\Gamma\big).
\end{align}
From now on $\tau\in(1, 2]$ is fixed and we allow all the implied constants to depend on $\tau$. 
If suffices to show that for every $p\in(1, \infty)$, there is $C_{p, \tau, \Gamma}\in\RR_+$ such that for any  $f\in \ell^p(\ZZ^\Gamma)$, we have
\begin{equation}\label{eq:5goal}
\sup_{J\in\ZZ_+}\sup_{I\in\mathfrak S_J(\DD_{\tau}^2) }
\big\|O_{I, J}(T_{\ZZ^\Gamma}[m_{M_1,M_2}^{\langle \Gamma \rangle}]f: (M_1,M_2)\in\DD^2_{\tau})\big\|_{\ell^p(\ZZ^\Gamma)}
\le C_{p, \tau, \Gamma}\|f\|_{\ell^p(\ZZ^\Gamma)},
\end{equation}
since taking $P(t)=(t)^{\Gamma}$, we have $\tilde{A}_{M_1, M_2; \ZZ^{\Gamma}}^{P}f=T_{\ZZ^\Gamma}[m_{M_1,M_2}^{\langle \Gamma \rangle}]f$ with $\tilde{A}_{M_1, M_2; \ZZ^{\Gamma}}^P$ defined in \eqref{eq:3}.

It is not difficult  to see that once inequality \eqref{eq:5goal} is
established, Theorem \ref{thm:osc} readily follows.  By a simple
density argument, it suffices to prove inequality \eqref{eq:5goal}
only for compactly supported functions $f:\ZZ^\Gamma\to \CC$.  We now concentrate on proving
inequality \eqref{eq:5goal} for compactly supported functions, the
proof of which makes up the bulk of this paper.

\subsection{Preliminaries and approximating multipliers}
To prove inequality \eqref{eq:5goal}, we will need to construct an
approximating multiplier for which Theorem \ref{thm:osc_IW_2par} can
be applied. For this purpose, we develop a multi-parameter circle
method by adopting ideas from the classical approach.

For $\emptyset\neq\Lambda'\subseteq\Lambda\subseteq \Gamma$ and   an arithmetic function $G: (\QQ\cap\TT)^{\Lambda'}\rightarrow\CC$, a continuous multiplier  $m: \TT^\Lambda\rightarrow\CC$, 
a finite set of fractions $\Sigma\subset (\QQ\cap\TT)^{\Lambda'}$, and $\xi=(\xi_\gamma)_{\gamma\in\Lambda}\in\TT^\Lambda$, we define a multiplier
\begin{align}
\label{eq:15}
\big\llbracket G \mid  m\big\rrbracket_{\Sigma}(\xi)=\sum_{\theta\in\Sigma}G(\theta)m(\xi'-\theta, \xi''),
\end{align}
where $\xi=(\xi', \xi'')\in \TT^{\Lambda'}\times \TT^{\Lambda''}$ and $\Lambda''=\Lambda\setminus \Lambda'$.
Using this notation with appropriate $G, m$ and $\Sigma$ we will construct approximating multipliers.

\subsubsection{\bf\textit{Continuous multipliers}}
For $\emptyset\neq\Lambda\subseteq \Gamma$ the multiplier
$\mathfrak m_{M_1,M_2}^{\langle \Lambda \rangle}$ from \eqref{eq:PhiM1M2_def_kpar} is
a continuous counterpart of the multiplier $m_{M_1,M_2}^{\langle \Lambda \rangle}$ from \eqref{eq:mM1M2_def}
and it has the following form
\begin{equation}\label{eq:PhiM1M2_def}
\mathfrak m_{M_1,M_2}^{\langle \Lambda \rangle}(\xi):=\frac{1}{(\tau-1)^2}\int_{1}^{\tau}\int_{1}^{\tau}\ex\big( \xi\cdot (M_1t_1,M_2t_2)^{\Lambda} \big)dt_2dt_1, \qquad \xi\in\TT^\Lambda.
\end{equation}
By inequality \eqref{eq:9}, there exists $\delta_{\Lambda}\in(0, 1)$ such that
\begin{align}
\label{eq:10}
|\mathfrak{m}_{M}^{\langle \Lambda \rangle}(\xi)|\lesssim_{\Lambda} \big(\max_{\gamma\in\Lambda}|M^{\gamma}\xi_{\gamma}| \big)^{-\delta_{\Lambda}},\qquad M=(M_1, M_2),
\end{align}
and with the exponent $\delta_{\Lambda}\in(0, 1)$ satisfying \eqref{eq:10}, (similarly as in \eqref{eq:delta_Weyl} and \eqref{eq:12}), we set
\begin{align}
\label{eq:13}
\delta_{\rm{vdC}}:=\min_{\emptyset\neq\Lambda\subseteq \Gamma}\delta_{\Lambda}\in(0, 1).
\end{align}
In particular, inequality \eqref{eq:10} holds with $\delta_{\rm{vdC}}$ in place of $\delta_{\Lambda}$ uniformly for all 
$\emptyset\neq\Lambda\subseteq \Gamma$.

\subsubsection{\bf \textit{Logarithmic scales and rational fractions}} For $\emptyset\neq\Lambda\subseteq \Gamma$ and $(M_1,M_2) \in \RR_+^2$ and $K\in\RR_+$ we define two (logarithmic) quantities
\begin{align}
\label{eq:14}
\lo(K)=\Log K, \qquad \text{and } \qquad   \lo_{M_1,M_2}^{\Lambda}(K)=\big(\Log (M_1^{\gamma_1} M_2^{\gamma_2}/K) \big)_{\gamma\in\Lambda}\in\RR^\Lambda,
\end{align}
where $\Log x := \lfloor \log_2 x \rfloor$ for $x\in\mathbb{R}_+$.

As in Section \ref{sec:IW} for $N\ge1$ define $1$-periodic sets of \textit{canonical fractions} by
\begin{align}
\label{eq:16}
\mathcal{R}_{\le N}^\Lambda := \big\{{a}/{q}\in(\QQ\cap\TT)^\Lambda:  q \in [N] \text{ and } (a, q)=1\big\},
\end{align}
and for any $l\in\mathbb N$, we also let $\Sigma_{\le l}^\Lambda :=\mathcal{R}_{\le 2^l}^\Lambda$, and
$\Sigma_l^\Lambda := \Sigma_{\le l}^\Lambda \backslash \Sigma_{\le l-1}^\Lambda$.

\subsubsection{\bf \textit{Approximating multipliers}} For any real
numbers $M_1, M_2\ge 1$ and $u_1, \chi_1\in\RR_+$, using notation from \eqref{eq:15}, we define
an approximating multiplier for $m_{M_1,M_2}^{\langle \Gamma \rangle}$ by setting
\begin{align}
\label{eq:18}
\Phi_{M_1, M_2}^{\langle \Gamma \rangle}:=\Big \llbracket G^{\Gamma} \mid  \mathfrak m_{M_1,M_2}^{\langle \Gamma \rangle}\eta_{\le \lo_{M_1,M_2}^{\Gamma}(M_0^{\chi_1})}^{\Gamma}\Big\rrbracket_{\Sigma_{\le \lo(M_0^{u_1})}^{\Gamma}},
\end{align}
with $M_0:=\min(M_1, M_2)$, 
where $\eta^\Lambda_{\le n}$  is a cut-off function defined by \eqref{eq:789}, namely
\begin{align}\label{eq:etatilde_def}
\eta_{\le n}^{\Lambda}(\xi):=\prod_{\gamma\in\Lambda}\eta_{\le n_{\gamma}}(\xi_{\gamma}), \qquad \xi\in\TT^{\Lambda}, \quad n=(n_{\gamma})_{\gamma\in\Lambda}\in\ZZ^{\Lambda}, \quad \emptyset\neq\Lambda\subseteq \Gamma,
\end{align}
and $G^\Lambda$ is the complete exponential sum from
\eqref{eq:323}. By Proposition \ref{prop:32}, for all
$a/q\in (\TT\cap\QQ)^{\Lambda}$ such that $(a, q)=1$, we have
\begin{equation}\label{eq:Gauss}
    |G^\Lambda(a/q)| \lesssim q^{-\delta_{\textrm{Gauss}}},
\end{equation}
where $\delta_{\textrm{Gauss}}\in(0, 1)$ is the exponent from \eqref{eq:12}.

\subsubsection{\bf \textit{Choice of parameters}} 
To prove inequality \eqref{eq:5goal}, we fix $p\in(1,\infty)$ and choose $p_0\in2\ZZ$, an auxiliary interpolation parameter, such that $p\in(p_0',p_0)$.
The constants that will appear in our arguments are allowed to depend on $p$ and $p_0$. We also set an interpolation exponent
\begin{equation}\label{eq:thetap_def}
\theta_p:=\Big(\frac{1}{\min\{p,p'\}
}-\frac{1}{p_0}\Big) \Big(\frac{1}2-\frac{1}{p_0}
\Big)^{-1}\in(0,1).
\end{equation}

We further define
\begin{equation}\label{eq:uniform_delta}
\delta:=\min(\delta_{\textrm{Weyl}},  \delta_{\textrm{Gauss}}, \delta_{\textrm{vdC}})\in(0, 1),
\end{equation}
with the exponents $\delta_{\textrm{Weyl}}$, $\delta_{\textrm{Gauss}}$ and $\delta_{\textrm{vdC}}$  given by \eqref{eq:delta_Weyl}, \eqref{eq:12} and \eqref{eq:13}, respectively.

Let $\chi_1, u_1, \chi_2, u_2, \alpha\in\RR_+$  be positive constants such that
\begin{align}\label{eq:parameter_tuning}
\frac{1}{10^6|\Gamma|p_0}>\chi_1>4(|\Gamma|+1)u_1>\frac{10^3}{\delta}(|\Gamma|+1)^3\chi_2>\frac{10^3}{\delta}(|\Gamma|+1)^3 u_2>\frac{10^6}{\delta^2}(|\Gamma|+1)^6 \alpha,
\end{align}
where $\delta$ is given by \eqref{eq:uniform_delta}.
Moreover, let $\varrho\in(0, 1)$ be a small constant satisfying
\begin{equation}\label{eq:epsilon_def}
1/\varrho > 100(|\Gamma|+1+C_\Gamma)
\end{equation}
with $C_\Gamma\in\RR_+$ to be specified later. In fact, $C_\Gamma$ will be the constant arising in  \eqref{eq:weak_err}.

\subsection{Key $\ell^2(\ZZ^\Gamma)$ estimates}
Our aim is to prove that for every $M_1\in\DD_\tau$ and  $f\in \ell^2(\ZZ^\Gamma)$, we have
\begin{align}\label{eq:error_M1_dec_l2}
\big\| \sup_{M_2\in\DD_\tau:M_2\ge M_1}|T_{\ZZ^\Gamma}[ m_{M_1,M_2}^{\langle \Gamma \rangle}-\Phi_{M_1,M_2}^{\langle \Gamma \rangle}] f| \big\|_{\ell^2(\ZZ^\Gamma)}\lesssim M_1^{-\alpha\varrho/2}\|f\|_{\ell^2(\ZZ^\Gamma)}, 
\end{align}
with $\alpha\in(0, 1)$ as in \eqref{eq:parameter_tuning}, and $\varrho\in(0, 1)$ as in \eqref{eq:epsilon_def}. A similar inequality is true when we change the roles of $M_1$ with $M_2$. Therefore, we restrict our attention to \eqref{eq:error_M1_dec_l2}.

\subsubsection{\bf \textit{Large and small scale maximal estimates}}
In fact, inequality \eqref{eq:error_M1_dec_l2} is reduced to proving the following  large-scale maximal estimate 
\begin{align}\label{eq:l2_main_int}
    \big\|\sup_{M_2\in\DD_\tau: M_2^{\varrho} \le M_1 \le M_2} | T_{\ZZ^\Gamma}[m_{M_1,M_2}^{\langle \Gamma \rangle}-\Phi_{M_1,M_2}^{\langle \Gamma \rangle}]f| \big\|_{\ell^2(\ZZ^\Gamma)} \lesssim M_1^{-\alpha\varrho/2}\|f\|_{\ell^2(\ZZ^\Gamma)},
\end{align}
and  the small-scale maximal estimate 
\begin{align}\label{eq:l2_main_int_case2}
    \big\|\sup_{M_2\in \DD_\tau:M_1 \le M_2^\varrho} | T_{\ZZ^\Gamma}[ m_{M_1,M_2}^{\langle \Gamma \rangle}-\Phi_{M_1,M_2}^{\langle \Gamma \rangle}]f| \big\|_{\ell^2(\ZZ^\Gamma)} \lesssim M_1^{-\alpha}\|f\|_{\ell^2(\ZZ^\Gamma)}.
\end{align}

We now give a few comments about the proof of inequalities \eqref{eq:l2_main_int} and \eqref{eq:l2_main_int_case2}.

\begin{enumerate}[label*={\arabic*}.]
\item The proof of inequality \eqref{eq:l2_main_int} is fairly straightforward, as the multi-parameter (in fact, two-parameter) Weyl's inequality from Theorem \ref{thm:weyl2par} works efficiently in the large-scale regime \((M_2^{\varrho} \le M_1 \le M_2)\), which, in fact, is a regime where the scales are comparable upon taking logarithms. In this case, the contribution from minor arcs, as well as the approximation on major arcs, produces a decay with respect to the larger parameter \(M_2\), which is handled by simple square function arguments. The proof of inequality \eqref{eq:l2_main_int} then goes much the same way as in the one-parameter theory; we refer, for instance, to \cite[Section 7]{BMSW}.

\smallskip

\item The proof of inequality \eqref{eq:l2_main_int_case2} is quite intricate. Working in the small-scale regime \((M_1 \le M_2^{\varrho})\), the parameters are no longer comparable, and one has to apply the classical circle method iteratively, variable by variable, which is a technically demanding process. By applying Weyl's inequality with respect to the larger parameter \(M_2\), one must understand the behavior of the multiplier along polynomials whose coefficients are polynomials with variables corresponding to the smaller scale \(M_1\). Then we encounter a new phenomenon, which we call the ``major arc rigidity'', which can be formulated as the following dichotomy:
\smallskip
\begin{itemize}
\item[(i)] The coefficients of the polynomial corresponding to the larger scale \(M_2\) --- which are polynomials with variables corresponding to the smaller scale \(M_1\) --- belong to the major arcs, but this happens very rarely, and we gain a power saving with respect to the smaller scales \(M_1\). So the maximal and square function techniques are efficient in this case.

\smallskip

\item[(ii)] Otherwise, the original two-variable polynomial is
equidistributed in the sense that its frequencies on the Fourier
transform side are all well-structured and must belong to the
two-parameter major arcs. This also reveals that the multiplier
\(\Phi_{M_1,M_2}^{\langle \Gamma \rangle}\) is an approximating multiplier for
\(m_{M_1,M_2}^{\langle \Gamma \rangle}\). Once again, the maximal and square function
techniques are efficient and lead to the conclusion 
\eqref{eq:l2_main_int_case2}.
\end{itemize}

\smallskip

\item The proof of \eqref{eq:error_M1_dec_l2} follows if we show
inequalities \eqref{eq:l2_main_int} and
\eqref{eq:l2_main_int_case2}. This will be done in Sections
\ref{sec:l2easy} and \ref{sec:l2hard}, respectively. Assuming
momentarily \eqref{eq:error_M1_dec_l2}, we show how to complete the
proof of inequality \eqref{eq:5goal}, and consequently Theorem
\ref{thm:osc}.
\end{enumerate}

\subsection{Key $\ell^p(\ZZ^\Gamma)$ estimates}
Inequality \eqref{eq:5goal} will be proved once we show the error term estimate
\begin{align}\label{eq:osc_error}
\sup_{J\in\ZZ_+}\sup_{I\in\mathfrak S_J(\DD_{\tau}^2) }\big\|O_{I, J}(T_{\ZZ^\Gamma}[m_{M_1,M_2}^{\langle \Gamma \rangle}-\Phi_{M_1,M_2}^{\langle \Gamma \rangle}]f: (M_1,M_2) \in\DD^2_{\tau})\big\|_{\ell^p(\ZZ^\Gamma)}\lesssim_{p, \tau, \Gamma}\|f\|_{\ell^p(\ZZ^\Gamma)},
\end{align}
and the main term estimate
\begin{align}\label{eq:osc_period}
\sup_{J\in\ZZ_+}\sup_{I\in\mathfrak S_J(\DD_{\tau}^2) }\big\|O_{I, J}\big (T_{\ZZ^\Gamma}[\Phi_{M_1,M_2}^{\langle \Gamma \rangle}]f: (M_1,M_2) \in\DD^2_{\tau}\big)\big\|_{\ell^p(\ZZ^\Gamma)}\lesssim_{p, \tau, \Gamma}\|f\|_{\ell^p(\ZZ^\Gamma)}.
\end{align}
It suffices to prove inequalities  \eqref{eq:osc_error} and \eqref{eq:osc_period} with $\DD_{\tau, \le }^2:=\{(M_1, M_2)\in \DD_\tau^2: M_1\le M_2\}$ and $\DD_{\tau, \ge }^2:=\{(M_1, M_2)\in \DD_\tau^2: M_1\ge M_2\}$ in place of $\DD_{\tau}^2$. We  focus on $\DD_{\tau, \le }^2$, the other case is similar.
\subsubsection{\bf \textit{Proof of the error term estimate \eqref{eq:osc_error}} assuming \eqref{eq:osc_period}}
Using inequality \eqref{eq:95} with $\II=\DD_{\tau, \le }^2$, we have 
\begin{align*}
&\sup_{J\in\ZZ_+}\sup_{I\in\mathfrak S_J(\DD_{\tau, \le }^2) }\big\|O_{I, J}(T_{\ZZ^\Gamma}[ m_{M_1,M_2}^{\langle \Gamma \rangle}-\Phi^{\langle \Gamma \rangle}_{M_1,M_2}]f: (M_1,M_2) \in\DD_{\tau, \le }^2)\big\|_{\ell^p(\ZZ^\Gamma)}
\\
&\quad\lesssim\sum_{M_1\in\DD_\tau}\big\| \sup_{M_2\in\DD_\tau:M_2\ge M_1}|T_{\ZZ^\Gamma}[ m_{M_1,M_2}^{\langle \Gamma \rangle}-\Phi^{\langle \Gamma \rangle}_{M_1,M_2}]f| \big\|_{\ell^p(\ZZ^\Gamma)}.
\end{align*}
To finish the proof of \eqref{eq:osc_error} it suffices to show that for any $M_1\in\DD_\tau$,   one has
\begin{align}\label{eq:error_M1_dec_lp}
\big\| \sup_{M_2\in\DD_\tau:M_2\ge M_1}|T_{\ZZ^\Gamma}[ m_{M_1,M_2}^{\langle \Gamma \rangle}-\Phi^{\langle \Gamma \rangle}_{M_1,M_2}]f| \big\|_{\ell^p(\ZZ^\Gamma)}\lesssim M_1^{-\frac{1}{2}\alpha\varrho\theta_p}\|f\|_{\ell^p(\ZZ^\Gamma)}
\end{align}
with $\theta_p\in(0, 1)$ as in \eqref{eq:thetap_def}. It suffices to prove \eqref{eq:error_M1_dec_lp} 
for all  $M_1\ge C_{p_0}$, which are sufficiently large in terms of the auxiliary interpolation parameter $p_0\in2\ZZ$.
By the one-parameter maximal theory from \cite{MST2}, which gives the estimates independent of the coefficients of the underlying polynomial, we obtain 
\begin{align*}
\sup_{M_1\in\DD_\tau}\big\| \sup_{M_2\in\DD_\tau}|T_{\ZZ^\Gamma}[ m_{M_1,M_2}^{\langle \Gamma \rangle} ]f| \big\|_{\ell^u(\ZZ^\Gamma)}\lesssim \|f\|_{\ell^u(\ZZ^\Gamma)}, \qquad  u\in\{p_0,p_0'\}.
\end{align*}
Moreover, since the oscillation seminorm dominates a supremum, it follows by \eqref{eq:osc_period} that
\begin{align*}
\sup_{M_1\in\DD_\tau}\big\| \sup_{M_2\in\DD_\tau:M_2\ge M_1}|T_{\ZZ^\Gamma}[ \Phi^{\langle \Gamma \rangle}_{M_1,M_2} ]f| \big\|_{\ell^u(\ZZ^\Gamma)}\lesssim \|f\|_{\ell^u(\ZZ^\Gamma)}, \qquad u\in\{p_0,p_0'\}.
\end{align*}
Combining the above estimates we obtain
\begin{align*}
\sup_{M_1\in\DD_\tau}\big\|\sup_{M_2\in\DD_\tau:M_2\ge M_1}|T_{\ZZ^\Gamma}[ m_{M_1,M_2}^{\langle \Gamma \rangle}-\Phi^{\langle \Gamma \rangle}_{M_1,M_2}]f| \big\|_{\ell^u(\ZZ^\Gamma)}\lesssim \|f\|_{\ell^u(\ZZ^\Gamma)}, \qquad  u\in\{p_0,p_0'\}.
\end{align*}
Interpolating the above bound with \eqref{eq:error_M1_dec_l2} gives \eqref{eq:error_M1_dec_lp}, and implies \eqref{eq:osc_error}, given \eqref{eq:osc_period}.

\subsubsection{\bf \textit{Proof of the main term estimate \eqref{eq:osc_period}}}
Let 
\begin{equation}\label{eq:Ns_def}
N_s:=2^{\frac{s}{1000u_1}}.
\end{equation}
Taking $\delta_p:=\delta \theta_p$, where $\theta_p\in(0,1)$ is given by \eqref{eq:thetap_def}, 
it suffices to show
\begin{align}\label{eq:osc_period_s}
\sup_{J\in\ZZ_+}\sup_{I\in\mathfrak S_J(\DD_{\tau}^2) }\big\|O_{I, J}\big (T_{\ZZ^\Gamma}[\mathfrak h^{s}_{M_1,M_2}\mathfrak g_s]f: (M_1,M_2) \in\DD^{2, s}_{\tau, \le}\big)\big\|_{\ell^p(\ZZ^\Gamma)}\lesssim_{p, \tau, \Gamma}2^{-s \delta_p/2}\|f\|_{\ell^p(\ZZ^\Gamma)},
\end{align}
where $\DD^{2, s}_{\tau, \le}:=\{(M_1, M_2)\in \DD^{2}_{\tau, \le}: M_1\ge 2^{s/u_1}\}$, and
\begin{align*}
\mathfrak h^{s}_{M_1,M_2}:=&\Big\llbracket 1 \mid  \mathfrak m_{M_1,M_2}^{\langle \Gamma \rangle}\eta_{\le \lo_{M_1,M_2}^{\Gamma}(M_1^{\chi_1})}^{\Gamma}\eta_{\le \lo_{N_s,N_s}^{\Gamma}(N_s^{\chi_1})}^{\Gamma}\Big\rrbracket_{\Sigma_{s}^{\Gamma}},\\
\mathfrak g_s:=&\Big\llbracket G^{\Gamma} \mid  \eta_{\le \lo_{N_s,N_s}^{\Gamma}(N_s^{\chi_1})}^{\Gamma}\Big\rrbracket_{\Sigma_{s}^{\Gamma}}.
\end{align*}
Indeed, for $(M_1,M_2) \in \DD^{2}_{\tau, \le}$, we can write
\begin{align*}
\Phi_{M_1,M_2}^{\langle \Gamma \rangle}=\sum_{s=0}^{\infty} \Big\llbracket G^{\Gamma} \mid  \mathfrak m_{M_1,M_2}^{\langle \Gamma \rangle}\eta_{\le \lo_{M_1,M_2}^{\Gamma}(M_1^{\chi_1})}^{\Gamma}\Big\rrbracket_{\Sigma_{s}^{\Gamma}} \ind{\DD^{2, s}_{\tau, \le}}(M_1,M_2),
\end{align*}
and if $(M_1,M_2)\in \DD^{2, s}_{\tau, \le}$, then $M_1\ge N_s$, and we have a factorization
\[
\Big\llbracket G^{\Gamma} \mid  \mathfrak m_{M_1,M_2}^{\langle \Gamma \rangle}\eta_{\le \lo_{M_1,M_2}^{\Gamma}(M_1^{\chi_1})}^{\Gamma}\Big\rrbracket_{\Sigma_{s}^{\Gamma}}=\mathfrak h^{s}_{M_1,M_2}\mathfrak g_s.
\]
Note that for every $I\in\mathfrak S_J(\DD_{\tau}^2)$, there is at most one $i_0 \in \NN_{<J}$ for which the box ${\mathbb B}[I_{i_0}]$ intersects both 
$\DD^{2, s}_{\tau, \le}$ and its complement in $\DD^{2, s}_{\tau, \le}$. Therefore
\[
O_{I, J}\big (T_{\ZZ^\Gamma}[\mathfrak h^{s}_{M_1,M_2}\mathfrak g_s \ind{\DD^{2,s}_{\tau, \le}}]f: (M_1,M_2) \in\DD^{2}_{\tau, \le}\big) 
\]
\[
\le O_{I, J}\big (T_{\ZZ^\Gamma}[\mathfrak h^{s}_{M_1,M_2}\mathfrak g_s]f: (M_1,M_2) \in\DD^{2, s}_{\tau, \le}\big) + \sup_{(M_1,M_2) \in \DD^{2,s}_{\tau,\le}} \big| T_{\ZZ^\Gamma}[\mathfrak h^{s}_{M_1,M_2}\mathfrak g_s]f\bigr|
\]
and so by \eqref{eq:135}, we see that it suffices to prove \eqref{eq:osc_period_s}.

The proof of \eqref{eq:osc_period_s} will follow if we show that for any $\varepsilon\in(0, 1)$, we have 
\begin{align}\label{eq:osc_period_IW_s}
\sup_{J\in\ZZ_+}\sup_{I\in\mathfrak S_J(\DD_{\tau}^2) }\big\|O_{I, J}\big (T_{\ZZ^\Gamma}[ \mathfrak h^{s}_{M_1,M_2}]f: (M_1,M_2) \in\DD^{2, s}_{\tau, \le}\big)\big\|_{\ell^p(\ZZ^\Gamma)}\lesssim_{p, \tau, \Gamma,\varepsilon}2^{\varepsilon s}\|f\|_{\ell^p(\ZZ^\Gamma)}
\end{align}
and
\begin{align}\label{eq:Gauss_final_s}
\big\|T_{\ZZ^\Gamma}[\mathfrak g_s]f\big\|_{\ell^p(\ZZ^\Gamma)}\lesssim_{p, \tau, \Gamma,\varepsilon}2^{-\delta_p s}\|f\|_{\ell^p(\ZZ^\Gamma)}.
\end{align}
To prove inequality \eqref{eq:Gauss_final_s}, one can follow the ideas in the proof of \cite[Lemma 7.43]{BMSW}. We omit the details. To estimate \eqref{eq:osc_period_IW_s}, we introduce the multiplier
\begin{align*}
\tilde{\mathfrak h}^{s}_{M_1,M_2}:=\Big\llbracket 1 \mid  \mathfrak m_{M_1,M_2}^{\langle \Gamma \rangle}\eta_{\le \lo_{N_s,N_s}^{\Gamma}(N_s^{\chi_1})}^{\Gamma}\Big\rrbracket_{\Sigma_{s}^{\Gamma}}.
\end{align*}
We observe that, by Theorem \ref{thm:osc_IW_2par}, we can write
\begin{align}
\label{eq:17}
\sup_{J\in\ZZ_+}\sup_{I\in\mathfrak S_J(\DD_{\tau}^2) }\big\|O_{I, J}\big (T_{\ZZ^\Gamma}[ \tilde{\mathfrak h}^{s}_{M_1,M_2}]f: (M_1,M_2) \in\DD^{2, s}_{\tau, \le}\big)\big\|_{\ell^p(\ZZ^\Gamma)}\lesssim_{p, \tau, \Gamma,\varepsilon}2^{\varepsilon s}\|f\|_{\ell^p(\ZZ^\Gamma)},
\end{align}
whereas by Theorem \ref{thm:iw-semi} there exists $\alpha_p\in(0, 1)$ such that for every $M_1\in \DD^{2, s}_{\tau, \le}$,  we have
\begin{align}
\label{eq:19}
\bigg\|\Big(\sum_{\substack{M_2\in \DD_{\tau}\\M_2 \ge M_1}}\big|T_{\ZZ^\Gamma}[{\mathfrak h}^{s}_{M_1,M_2}- \tilde{\mathfrak h}^{s}_{M_1,M_2}]f\big|^2\Big)^{1/2}\bigg\|_{\ell^p(\ZZ^\Gamma)}\lesssim_{p, \tau, \Gamma,\varepsilon}2^{\varepsilon s}M_1^{-\alpha_p}\|f\|_{\ell^p(\ZZ^\Gamma)}.
\end{align}
To obtain inequality \eqref{eq:19} we can proceed as in the proof of \cite[Claim 7.55]{BMSW}, we omit the details.

\section{Multi-parameter circle method: Proof of inequality (\ref{eq:l2_main_int})}
\label{sec:l2easy}
The aim of this section is to prove the large-scale maximal estimate 
from \eqref{eq:l2_main_int}. We shall treat the
multiplier $m_{M_1,M_2}^{\langle \Gamma \rangle}$ defined in \eqref{eq:mM1M2_def} separately on the two-parameter minor
and major arcs. An efficient way to do this is to introduce suitable major-arc projection multipliers.  Namely, for $\emptyset\neq \Lambda\subseteq \Gamma$ and $u, \chi\in\RR_+$, we define a  major-arc projection multiplier corresponding to $\Lambda$ by
\begin{align}
\label{eq:20}
\Pi_{M_1, M_2}^{\Lambda, u, \chi}:=\Big\llbracket 1 \mid  \eta_{\le \lo_{M_1,M_2}^{\Lambda}(M_1^{\chi})}^{\Lambda}\Big\rrbracket_{\Sigma_{\le \lo(M_1^{u})}^{\Lambda}}.
\end{align}
The  minor-arc projection multiplier is defined by  $1-\Pi_{M_1, M_2}^{\Lambda, u, \chi}$. We begin with the minor-arc analysis.

\subsection{Minor arc estimate} 
Our goal will be to prove that for any fixed $M_1\in\DD_\tau$,  we have
\begin{align}\label{eq:minor_l2_main}
    \big\|\sup_{M_2\in\DD_\tau: M_2^{\varrho} \le M_1 \le M_2} | T_{\ZZ^\Gamma}[m_{M_1,M_2}^{\langle \Gamma \rangle}(1-\Pi^\Gamma_{M_1,M_2})]f| \big\|_{\ell^2(\ZZ^\Gamma)} \lesssim M_1^{-\alpha\varrho}\|f\|_{\ell^2(\ZZ^\Gamma)},
\end{align}
with $\alpha\in\RR_+$ as in \eqref{eq:parameter_tuning} and $\varrho\in\RR_+$ as in \eqref{eq:epsilon_def}, and with $\Pi^\Gamma_{M_1,M_2}:=\Pi^{\Gamma, u_1, \chi_1}_{M_1,M_2}$ defined in \eqref{eq:20}, where the parameters $u_1,\chi_1\in\RR_+$ have been specified in \eqref{eq:parameter_tuning}.

\begin{proof}[Proof of inequality \eqref{eq:minor_l2_main}]
It suffices to show that for all $M_2\in\DD_\tau$ and $M_2^{\varrho}\le M_1\le M_2$ we have 
\begin{equation*}
\big\|T_{\ZZ^\Gamma}[m_{M_1,M_2}^{\langle \Gamma \rangle}(1-\Pi^\Gamma_{M_1,M_2})]f \big\|_{\ell^2(\ZZ^\Gamma)}
\lesssim M_2^{-\alpha\varrho}\|f\|_{\ell^2(\ZZ^\Gamma)}.
\end{equation*}
By Plancherel's theorem this is equivalent to showing
\begin{equation}\label{eq:minor_l2}
     |m_{M_1,M_2}^{\langle \Gamma \rangle}(\xi)(1-\Pi^\Gamma_{M_1,M_2}(\xi))| \lesssim M_2^{-\alpha\varrho}, \qquad \xi\in\TT^\Gamma.
\end{equation}
Define a sequence $(\varepsilon_\gamma)_{\gamma\in\Gamma}$ by setting
$$
\varepsilon_{\gamma}:=
\frac{u_1}{2|\Gamma|}
\begin{cases}
\big(\frac{\delta_{\textrm{Weyl}}}{10|\Gamma|}\big)^{d_1+d_2+1}, &\quad \textrm{if} \quad \gamma\in  [d_1]\times[d_2],
\\
\big(\frac{\delta_{\textrm{Weyl}}}{10|\Gamma|}\big)^{\gamma_2-1}, &\quad \textrm{if} \quad \gamma=(0, \gamma_2)\in \{0\}\times[d_2],
\\
\big(\frac{\delta_{\textrm{Weyl}}}{10|\Gamma|}\big)^{\gamma_1-1}, &\quad \textrm{if} \quad \gamma=(\gamma_1,0)\in 
[d_1]\times\{0\},
\end{cases}
$$
with $u_1$ as in \eqref{eq:parameter_tuning} and
$\delta_{\textrm{Weyl}}$ as in \eqref{eq:delta_Weyl}. 

To prove \eqref{eq:minor_l2} we fix $\xi \in \mathbb{T}^{\Gamma}$. By Dirichlet’s principle for every $\gamma \in \Gamma$
 there exist integers $0 \le a_{\gamma} \le q_{\gamma}$ such that $(a_{\gamma},q_{\gamma})=1$, and  $1 \le q_{\gamma} \le M_1^{\gamma_1}M_2^{\gamma_2}M_1^{-\varepsilon_\gamma}$ and
 \begin{align*}
     \Big|\xi_{\gamma}-\frac{a_{\gamma}}{q_{\gamma}}\Big| \le \frac{M_1^{\varepsilon_\gamma}}{q_{\gamma}M_1^{\gamma_1}M_2^{\gamma_2}}\le \frac{1}{q_{\gamma}^2}.
 \end{align*}

\paragraph{\bf Step 1} If $1 \le q_{\gamma} \le M_1^{\varepsilon_\gamma}$ for all $\gamma\in \Gamma$. Then the left-hand side of \eqref{eq:minor_l2} vanishes.

\paragraph{\bf Step 2} Assume now that $ M_1^{\varepsilon_{\omega}} < q_{\omega}$ for some $\omega=(\omega_1, \omega_2)\in \Gamma$.  

If $\omega_1, \omega_2\ge 1$, then we can use the two-parameter Weyl's inequality, see Theorem \ref{thm:weyl2par}, and Plancherel's theorem to obtain
\begin{equation*}
     |m_{M_1,M_2}^{\langle \Gamma \rangle}(\xi)(1-\Pi^\Gamma_{M_1,M_2}(\xi))| \lesssim M_1^{-\alpha}\le M_2^{-\alpha\varrho}, \qquad \xi\in\TT^\Gamma.
\end{equation*}
where we have  used the fact that $M_2^{\varrho}\le M_1\le M_2$. We emphasize that the condition $M_2^{\varrho}\le M_1$ is crucial for getting sufficient decay from the application of Theorem \ref{thm:weyl2par}.

\paragraph{\bf Step 3} To complete the proof of the proposition it remains to consider the case when $q_\gamma \le M_1^{\varepsilon_\gamma}$ for all $\gamma\in\{(\gamma_1, \gamma_2)\in\Gamma: \gamma_1\neq 0, \gamma_2\neq0\}$ and $q_{\omega} > M_1^{\varepsilon_{\omega}}$ for some $\omega\in( \{0\}\times[d_2] ) \cup ([d_1]\times\{0\})$. Without the loss of generality assume that $\omega=(0,\omega_2)$ for some $\omega_2\in [d_2]$. The treatment of the case when $\omega$ is of the form $\omega=(\omega_1,0)$  for some $\omega_1\in [d_1]$ is analogous.

Recalling that $\Gamma_2=\{\gamma\in\Gamma: \gamma_2\neq0\}$, we write
\begin{align*}
|m_{M_1,M_2}^{\langle \Gamma \rangle}(\xi)| \le \sum_{m_1\in\ZZ}\chi_{M_1}(m_1)\Big| \sum_{m_2\in\ZZ}\chi_{M_2}(m_2)\ex\big(\sum_{\gamma\in\Gamma_2}\xi_{\gamma}(m_1, m_2)^{\gamma}\big)\Big|,
\end{align*}
and we show, uniformly in $m_1\in\supp\chi_{M_1}$, that
$$
\Big| \sum_{m_2\in\ZZ}\chi_{M_2}(m_2)\ex\big(\sum_{\gamma\in\Gamma_2}\xi_{\gamma} (m_1, m_2)^{\gamma}\big)\Big|\lesssim M_1^{-\alpha}.
$$
Let $\sigma:=\max\{    \gamma_2\in[d_2]: q_{(0,\gamma_2)} \ge M_1^{  \varepsilon_{(0,\gamma_2)}   }  \}$, and define
$$
Q:=\begin{cases}
\lcm\big(\{q_\gamma: \gamma\in  \{(\gamma_1, \gamma_2)\in\Gamma_2: \gamma_1\neq 0\} \} \big), &\quad \textrm{if} \quad \sigma=d_2,
\\
\lcm \big(\{q_\gamma: \gamma\in  \{(\gamma_1, \gamma_2)\in\Gamma_2: \gamma_1\neq 0\}\}
\\
 \qquad\qquad\cup \{ q_{(0,\sigma+1)}, q_{(0,\sigma+2)}, \dots, q_{(0, d_2)} \} \big), &\quad \textrm{if} \quad \sigma < d_2,
\end{cases}
$$
and note that 
\begin{equation}\label{eq:QqmaxM}
Q\le M_1^{    \varepsilon_{(0,\sigma+1)}+\dots+\varepsilon_{(0,d_2)}+ |\Gamma|\varepsilon_{(1,1)}       } 
\le  
M_1^{     \frac{\delta_{\textrm{Weyl}}}{5|\Gamma|}   \varepsilon_{(0,\sigma)    } }.
\end{equation}
Splitting the summation into classes modulo $Q$, we obtain 
\begin{align*}
\Big| \sum_{m_2\in\ZZ}\chi_{M_2}(m_2)\ex\big(\sum_{\gamma\in\Gamma_2}\xi_{\gamma} (m_1, m_2)^{\gamma}\big)\Big|
&\lesssim M_2^{-1}
\sum_{r=1}^Q \bigg| \sum_{n=U_2+1}^{V_2}  \ex\big(\sum_{\gamma\in\Gamma_2}\xi_{\gamma} (m_1, (Qn+r))^\gamma\big)\bigg|,
\end{align*}
where 
$U_2:=\lfloor \frac{M_2-r}{Q} \rfloor$ and $V_2:=\lfloor\frac{\tau M_2 -r}{Q}\rfloor$. It suffices to show that for a fixed $r$ one has 
\begin{align}\label{eq:bazur}
\Big|\sum_{n=U_2+1}^{V_2}  \ex\big(\sum_{\gamma\in\Gamma_2}\xi_{\gamma} (m_1, (Qn+r))^\gamma\big)\Big|\lesssim Q^{-1} M_1^{-\alpha}M_2.
\end{align}
Inequality \eqref{eq:bazur} follows by the summation by parts combined
with classical Weyl's inequality from Theorem \ref{thm:weyl1par},
which can be applied in view of Lemma \ref{lem:3} and condition
\eqref{eq:QqmaxM}.  This concludes the proof of \eqref{eq:minor_l2} as
desired, provided that $\alpha$  satisfies condition \eqref{eq:parameter_tuning}.
\end{proof}

\subsection{Major arc approximation}
Now, we approximate $m_{M_1,M_2}^{\langle \Gamma \rangle}$ on the major arcs.  Our goal is to show that for a fixed $M_1\in\DD_\tau$
one has
\begin{align}\label{eq:major_l2_main}
\big\|\sup_{M_2\in\DD_\tau:M_2^{\varrho} \le M_1 \le M_2} | T_{\ZZ^\Gamma}[ m_{M_1,M_2}^{\langle \Gamma \rangle}\Pi^\Gamma_{M_1,M_2}-\Phi_{M_1,M_2}^{\langle \Gamma \rangle}]f| \big\|_{\ell^2(\ZZ^\Gamma)} \lesssim M_1^{-\alpha\varrho/2}\|f\|_{\ell^2(\ZZ^\Gamma)},
\end{align}
with $\alpha\in\RR_+$ as in \eqref{eq:parameter_tuning} and
$\varrho\in\RR_+$ as in \eqref{eq:epsilon_def}, and with
$\Pi^\Gamma_{M_1,M_2}:=\Pi^{\Gamma, u_1, \chi_1}_{M_1,M_2}$ defined in
\eqref{eq:20}, where the parameters $u_1,\chi_1\in\RR_+$ have been
specified in \eqref{eq:parameter_tuning}.

Combining \eqref{eq:minor_l2_main} and \eqref{eq:major_l2_main} we obtain inequality \eqref{eq:l2_main_int} as claimed at the beginning of this section.

\begin{proof}[Proof of inequality \eqref{eq:major_l2_main}]
Inequality \eqref{eq:major_l2_main} will clearly follow from the bound
\begin{align}\label{eq:major_l2}
\|T_{\ZZ^\Gamma}[m_{M_1,M_2}^{\langle \Gamma \rangle}\Pi^\Gamma_{M_1,M_2}-\Phi_{M_1,M_2}^{\langle \Gamma \rangle}] f\|_{\ell^2(\ZZ^\Gamma)}\lesssim M_2^{-3\varrho/4}\|f\|_{\ell^2(\ZZ^\Gamma)}, \qquad M_2^\varrho \le M_1\le M_2.
\end{align}

To prove \eqref{eq:major_l2}, we use the mean value theorem to show that
\begin{align*}
\big|m_{M_1,M_2}^{\langle \Gamma \rangle}(\xi)\Pi^\Gamma_{M_1,M_2}(\xi)-\Phi_{M_1,M_2}^{\langle \Gamma \rangle}(\xi)\big|
 \lesssim M_1^{-3/4}\le M_2^{-3\varrho/4},
\end{align*}
so \eqref{eq:major_l2} follows by Plancherel's theorem, which consequently also implies \eqref{eq:major_l2_main}  as desired.
\end{proof}

\section{Multi-parameter circle method: Proof of inequality  (\ref{eq:l2_main_int_case2})}
\label{sec:l2hard}
The aim of this section is to prove the small-scale maximal estimate from \eqref{eq:l2_main_int_case2}, which is quite intricate and constitutes
the core of the paper. In contrast to \cite{BMSW}, we will apply the classical circle method iteratively, which will reveal a new phenomenon that we refer to as  ``major arc rigidity'', with Lemma \ref{lem:comb} revealing the key feature of the method. We begin by introducing further notation and terminology and by outlining the strategy of the proof of inequality \eqref{eq:l2_main_int_case2}.

We now prove the minor arc estimate with respect to $\Gamma_2$, see \eqref{eq:47}, which reads as follows
\begin{align}
\label{eq:48}
\big\|\sup_{M_2\in \DD_\tau:M_1 \le M_2^\varrho} | T_{\ZZ^\Gamma}[ (1-\Pi^{\Gamma_2}_{M_1,M_2}) m_{M_1,M_2}^{\langle \Gamma \rangle}] f | \big\|_{\ell^2(\ZZ^\Gamma)} \lesssim M_1^{-\alpha}\|f\|_{\ell^2(\ZZ^\Gamma)},
\end{align}
as well as the  major arc estimate with respect to $\Gamma_2$, which reads as follows 
\begin{align}
\label{eq:49}
\big\|\sup_{M_2\in \DD_\tau:M_1 \le M_2^\varrho} | T_{\ZZ^\Gamma}[ \Pi^{\Gamma_2}_{M_1,M_2} m_{M_1,M_2}^{\langle \Gamma \rangle}-\Phi_{M_1,M_2}^{\langle \Gamma \rangle}]f| \big\|_{\ell^2(\ZZ^\Gamma)} \lesssim M_1^{-\alpha}\|f\|_{\ell^2(\ZZ^\Gamma)},
\end{align}
with $\alpha\in\RR_+$ as in \eqref{eq:parameter_tuning} and $\varrho\in\RR_+$ as in \eqref{eq:epsilon_def}, and with $\Pi^{\Gamma_2}_{M_1,M_2}:=\Pi^{\Gamma_2, u_1, \chi_1}_{M_1,M_2}$ defined in \eqref{eq:20}, where the parameters $u_1,\chi_1\in\RR_+$ have been specified in \eqref{eq:parameter_tuning}.

Once inequalities \eqref{eq:48} and \eqref{eq:49} are established, the
triangle inequality then yields inequality
\eqref{eq:l2_main_int_case2}. To prove inequalities \eqref{eq:48} and
\eqref{eq:49}, we will apply the circle method in the
one-parameter setting. The key tool to handle inequality \eqref{eq:48}
is Lemma \ref{lem:comb}, which will reveal the ``major arc rigidity''
phenomenon. To handle inequality \eqref{eq:49}, we will delicately use
summation by parts, which will allow us to gradually reduce the matter
to the multiplier $\Phi_{M_1, M_2}^{\langle \Gamma \rangle}$. 

\subsection{Minor arc estimates: Proof of inequality \eqref{eq:48}}
For $\gamma_2\in[d_2]$, we let
\begin{gather}
\label{eq:P2gamma_def}
\xi^{\Gamma_2}_{\gamma_2}(m_1):=\sum_{\gamma_1=0}^{d_1} \xi_{(\gamma_1, \gamma_2)}m_1^{\gamma_1},
\qquad \text{ and } \qquad
\xi^{\Gamma_2^c}_{0}(m_1):=\sum_{\gamma_1=1}^{d_1} \xi_{(\gamma_1,0)}m_1^{\gamma_1}, \\
\label{eq:P2_def} \xi_{\Gamma_2}(m_1):=(\xi^{\Gamma_2}_{\gamma_2}(m_1))_{\gamma_2 \in [d_2]}\in\RR^{[d_2]}.
\end{gather}
Similarly, one can define
$\xi^{\Gamma_1^c}_{0}(m_2), \xi^{\Gamma_1}_{\gamma_1}(m_2)$ and
$\xi_{\Gamma_1}(m_2)$. We decompose
\begin{equation}\label{eq:variable_separate}
\xi\cdot (m_1, m_2)^{\Gamma}=\xi^{\Gamma_2^c}_{0}(m_1)+\sum_{\gamma_2 \in [d_2]}\xi^{\Gamma_2}_{\gamma_2}(m_1)m_2^{\gamma_2}
=\xi^{\Gamma_2^c}_{0}(m_1)+\xi_{\Gamma_2}(m_1) \cdot (m_2^{\gamma_2})_{\gamma_2 \in [d_2]}.
\end{equation}
Recalling that $\{x\}$ denotes the fractional part of $x\in\RR$, we set
$$
{\bm \{}\xi_{\Gamma_2}(m_1){\bm \}} := \big( \{\xi^{\Gamma_2}_{\gamma_2}(m_1)\} \big)_{\gamma_2\in[d_2]}\in\TT^{[d_2]}.
$$
We further write
\begin{align}
\label{eq:63}
m_{M_1,M_2}^{\langle \Gamma \rangle}=m_{M_1,M_2}^{\langle [d_2], \textrm{major}\rangle}+m_{M_1,M_2}^{\langle [d_2], \textrm{minor}\rangle},
\end{align}
where
\begin{align}
\label{eq:minor_major_split}
\begin{split}
m_{M_1,M_2}^{\langle [d_2], \textrm{major}\rangle}(\xi):=&\sum_{m_1\in\ZZ}\chi_{M_1}(m_1)
\Pi^{[d_2]}_{M_2}\big({\bm \{}\xi_{\Gamma_2}(m_1){\bm \}}\big)\sum_{m_2\in\ZZ}\chi_{M_2}(m_2)\ex(\xi\cdot (m_1, m_2)^{\Gamma}),\\
m_{M_1,M_2}^{\langle [d_2], \textrm{minor} \rangle}(\xi):=&\sum_{m_1\in\ZZ}\chi_{M_1}(m_1)\Big(1-\Pi^{[d_2]}_{M_2}\big({\bm \{}\xi_{\Gamma_2}(m_1){\bm \}}\big)\Big)
\sum_{m_2\in\ZZ}\chi_{M_2}(m_2)\ex(\xi\cdot (m_1, m_2)^{\Gamma}),
\end{split}
\end{align}
and $\Pi^{[d_2]}_{M_2}:=\Pi^{\Gamma_1^c, u_2, \chi_2}_{M_2, M_2}$ defined in \eqref{eq:20}, where  $u_2,\chi_2\in\RR_+$ have been specified in \eqref{eq:parameter_tuning}.

\subsubsection{\textbf{One-parameter minor arc estimates}}\label{step1} 
We first establish the minor arc estimate
\begin{align}\label{eq:minor_l2_main_case2}
\big\|\sup_{M_2\in \DD_\tau:M_1 \le M_2^\varrho} | T_{\ZZ^\Gamma} [(1-\Pi^{\Gamma_2}_{M_1,M_2}) m_{M_1,M_2}^{\langle [d_2], \textrm{minor}\rangle}] f | \big\|_{\ell^2(\ZZ^\Gamma)} \lesssim M_1^{-\alpha}\|f\|_{\ell^2(\ZZ^\Gamma)}.
\end{align}

\begin{proof}[Proof of inequality \eqref{eq:minor_l2_main_case2}]
It suffices to show that
\begin{align}
\label{eq:23}
\big\|T_{\ZZ^\Gamma} [(1-\Pi^{\Gamma_2}_{M_1,M_2}) m_{M_1,M_2}^{\langle [d_2], \textrm{minor}\rangle}] f \big\|_{\ell^2(\ZZ^\Gamma)} \lesssim M_2^{-\alpha}\|f\|_{\ell^2(\ZZ^\Gamma)}.
\end{align}
Since  $\varrho < 1$,   by a simple square function argument summing $\eqref{eq:23}$ over $M_2\in \DD_\tau$ such that $M_2\ge M_1$ we obtain \eqref{eq:minor_l2_main_case2}. By Plancherel's theorem, \eqref{eq:23} will follow if we  
show the pointwise estimate
\begin{align}\label{eq:minor_l2_case2}
|(1-\Pi^{\Gamma_2}_{M_1,M_2}(\xi))  m_{M_1,M_2}^{\langle [d_2], \textrm{minor}\rangle}(\xi)|\lesssim M_2^{-\alpha}, \qquad \xi\in\TT^\Gamma.
\end{align}
Using a one-parameter Weyl's inequality from Theorem \ref{thm:weyl1par} (see also Remark \ref{remark:weyl1par}) and \eqref{eq:parameter_tuning} we obtain \eqref{eq:minor_l2_case2}, which completes the proof of \eqref{eq:minor_l2_main_case2}.
\end{proof}

\subsubsection{\textbf{One-parameter major arc approximations}}\label{step2}  
We now establish the major arc estimate
\begin{align}
\label{eq:25}
\big\|\sup_{M_2\in \DD_\tau: M_1 \le M_2^\varrho}
| T_{\ZZ^\Gamma}[ (1-\Pi^{\Gamma_2}_{M_1,M_2}) (m_{M_1,M_2}^{\langle [d_2], \textrm{major}\rangle}-\Psi^{\langle [d_2]\rangle }_{M_1,M_2}) ] f | \big\|_{\ell^2(\ZZ^\Gamma)}
\lesssim M_1^{-\alpha}\|f\|_{\ell^2(\ZZ^\Gamma)},
\end{align}
where
\begin{align}
\label{eq:mtilde_def}
\Psi^{\langle [d_2]\rangle}_{M_1,M_2}(\xi):=&\sum_{m_1\in\ZZ}\chi_{M_1}(m_1) \ex(\xi_{0}^{\Gamma_2^c}(m_1))
\Big\llbracket G^{\Gamma_1^c}\mid \mathfrak{m}^{\langle \Gamma_1^c\rangle }_{M_2, M_2}\eta^{\Gamma_1^c}_{\le \lo_{M_2, M_2}^{\Gamma_1^c}(M_2^{\chi_2})}\Big\rrbracket_{\Sigma^{\Gamma_1^c}_{\le \lo(M_2^{u_2})}}({\bm\{}\xi_{\Gamma_2}(m_1){\bm \}}),
\end{align}
with the complete exponential sum as in \eqref{eq:323}, given by
\begin{align}
\label{eq:24}
G^{\Gamma_1^c}(a/q):=\frac{1}{q}\sum_{r=1}^q\ex\big(\sum_{\gamma_2 \in [d_2]} r^{\gamma_2}a_{(0, \gamma_2)}/q\big), \qquad a=(a_{(0, \gamma_2)})_{\gamma_2\in[d_2]}\in\ZZ^{\Gamma_1^c}, \quad q\in\ZZ_+,
\end{align}
and the multiplier  as in \eqref{eq:PhiM1M2_def_kpar}, given by 
\begin{equation}\label{eq:mN_def}
\mathfrak{m}^{\langle \Gamma_1^c\rangle}_{M_2, M_2}(\zeta):= \frac{1}{\tau-1}\int_1^\tau\ex\big( \sum_{\gamma_2 \in [d_2]}(M_2 t)^{\gamma_2}\zeta_{(0, \gamma_2)} \big)dt, \qquad \zeta=(\zeta_{(0, \gamma_2)} )_{\gamma_2\in [d_2]}\in\TT^{\Gamma_1^c}, \quad M_2\in\RR_+.
\end{equation}
Inequality \eqref{eq:25} will follow if we show that 
\begin{align}\label{eq:major_l2_main_case2}
    \big\|T_{\ZZ^\Gamma} [ (1-\Pi^{\Gamma_2}_{M_1,M_2}) (m_{M_1,M_2}^{\langle [d_2], \textrm{major}\rangle}-\Psi^{\langle [d_2]\rangle}_{M_1,M_2}) ] f  \big\|_{\ell^2(\ZZ^\Gamma)} \lesssim M_2^{-1/2}\|f\|_{\ell^2(\ZZ^\Gamma)}.
\end{align}

\begin{proof}[Proof of inequality \eqref{eq:major_l2_main_case2}]
By Plancherel's theorem, to prove inequality \eqref{eq:major_l2_main_case2} it suffices to show
\begin{equation}\label{eq:major_l2_case2}
\|m_{M_1,M_2}^{\langle [d_2], \textrm{major}\rangle}-\Psi^{\langle [d_2]\rangle }_{M_1,M_2}\|_{L^\infty(\TT^\Gamma)}\lesssim M_2^{-1/2}.
\end{equation}
This readily follows by the mean value theorem. That completes the proof of  \eqref{eq:major_l2_case2}, and consequently yields inequality  \eqref{eq:major_l2_main_case2} (and hence \eqref{eq:25}) as desired.
\end{proof}

\subsubsection{\textbf{Abstract maximal estimate}}\label{stepABSTRACT} 
In view of decomposition \eqref{eq:63} and inequalities  
\eqref{eq:minor_l2_main_case2} and \eqref{eq:25},
the proof of inequality \eqref{eq:48} is reduced to proving that
\begin{equation}\label{eq:maingoal}
\big\|\sup_{M_2\in \DD_\tau: M_1 \le M_2^\varrho} | T_{\ZZ^\Gamma} [(1-\Pi^{\Gamma_2}_{M_1,M_2}) \Psi^{\langle [d_2]\rangle}_{M_1,M_2}] f | \big\|_{\ell^2(\ZZ^\Gamma)} \lesssim M_1^{-\alpha}\|f\|_{\ell^2(\ZZ^\Gamma)}.
\end{equation}
To handle inequality \eqref{eq:maingoal}  we will need the following abstract lemma. 
\begin{lemma}\label{lem:padova}
Let $(d_{M_1,M_2})_{(M_1, M_2)\in\DD_{\tau}^2}$ be a family of bounded functions on $\TT^{\Gamma}$. Suppose that for every  $M_1\in\DD_{\tau}$ there exist constants ${\bf M}_{M_1}, {\bf S}_{M_1}\in\RR_+$ such that for all $f\in\ell^2(\ZZ^\Gamma)$ the following conditions are satisfied:
\begin{itemize}

\item[\bf (M)]\label{cond1}   The maximal-function estimate
\begin{equation}\label{eq:cond1}
\big\|\sup_{M_2\in \DD_\tau: M_1 \le M_2^\varrho} | T_{\ZZ^\Gamma} [ d_{M_1,M_2}] f | \big\|_{\ell^2(\ZZ^\Gamma)}
\le {\bf M}_{M_1}\|f\|_{\ell^2(\ZZ^\Gamma)}.
\end{equation}

\item[\bf (S)]\label{cond2} The square-function estimate
\begin{equation}\label{eq:cond2}
\bigg( \sum_{j=0}^{2^{m-i}-1} 
\Big\|T_{\ZZ^\Gamma} \big[
\sum_{\substack{k\in U_j^i \\ U_j^i\subseteq [j_\varrho,2^m)}} (d_{M_1, \tau^{k+1}}-d_{M_1, \tau^k}) 
\big] f   
\Big\|^2_{\ell^2(\ZZ^\Gamma)} \bigg)^{1/2}\le {\bf S}_{M_1}\|f\|_{\ell^2(\ZZ^\Gamma)},
\end{equation}
for any $i\in\NN_{\le m}$, where $U_j^i:=[j 2^i, (j+1) 2^i)$, and 
\begin{equation}\label{eq:mabstr_def}
m:=\big\lfloor 3 u_1 \log_2  M_1 \big\rfloor\quad \text{ and } \quad
j_\varrho:=j_\varrho(M_1):=\big\lfloor {\varrho}^{-1}\log_\tau M_1\big\rfloor.
\end{equation}
\end{itemize}
Then we have 
\begin{equation}\label{eq:lemgoal}
\big\|\sup_{M_2\in \DD_\tau: M_1 \le M_2^\varrho} | T_{\ZZ^\Gamma}[ \Pi^{\Gamma_2}_{M_1,M_2} d_{M_1,M_2}] f | \big\|_{\ell^2(\ZZ^\Gamma)} \lesssim m\big({\bf M}_{M_1}+{\bf S}_{M_1}\big)\|f\|_{\ell^2(\ZZ^\Gamma)}.
\end{equation}
In \eqref{eq:cond2}, we assume that the constant ${\bf S}_{M_1}$ is uniform for all $i\in\NN_{\le m}$.
\end{lemma}

\begin{proof}[Proof of Lemma \ref{lem:padova}]
The proof of inequality \eqref{eq:lemgoal} proceeds in essentially the same  way as the proof of Theorem \ref{thm:IW1'}. We split the supremum in \eqref{eq:lemgoal} into small-scale and large-scale regimes. The small-scale regime is handled using the Rademacher--Menshov inequality for the supremum, as in \eqref{eq:164}; see Step~3 in the proof of Theorem \ref{thm:IW1'}. The large-scale regime is handled by the sampling principle, as in Step~6 of the proof of Theorem \ref{thm:IW1'}. We leave the details to the reader.
\end{proof}

\subsubsection{\textbf{Change of scales of fractions}}\label{step4}
Recall that our goal now is to show \eqref{eq:maingoal}. We begin with changing the scale of fractions in $\Psi^{\langle [d_2]\rangle}_{M_1,M_2}$. More precisely, we will show that 
\begin{equation}\label{eq:step4goal}
    \big\|\sup_{M_2\in \DD_\tau: M_1 \le M_2^{\varrho}} | T_{\ZZ^\Gamma}\big[ (1-\Pi^{\Gamma_2}_{M_1,M_2})(\Psi^{\langle [d_2]\rangle}_{M_1,M_2}-\Psi^{\langle [d_2]; 1\rangle}_{M_1,M_2}) \big] f | \big\|_{\ell^2(\ZZ^\Gamma)} \lesssim M_1^{-\alpha}\|f\|_{\ell^2(\ZZ^\Gamma)}, 
\end{equation}
with $\Psi^{\langle [d_2]\rangle}_{M_1,M_2}$ as in \eqref{eq:mtilde_def} and
\begin{align}
\label{eq:mdoubletilde_def}
    \Psi^{\langle [d_2]; 1\rangle}_{M_1,M_2}(\xi)
:= \sum_{m_1\in\ZZ}\chi_{M_1}(m_1) \ex(\xi_{0}^{\Gamma_2^c}(m_1))
\Big\llbracket G^{\Gamma_1^c}\mid \mathfrak{m}^{\langle \Gamma_1^c\rangle }_{M_2, M_2}\eta^{\Gamma_1^c}_{\le \lo_{M_2, M_2}^{\Gamma_1^c}(M_2^{\chi_2})}\Big\rrbracket_{\Sigma^{\Gamma_1^c}_{\le \lo(M_1^{u_2})}}({\bm\{}\xi_{\Gamma_2}(m_1){\bm \}}).
\end{align}
The only difference between $\Psi^{\langle [d_2]\rangle}_{M_1,M_2}$ and $\Psi^{\langle [d_2]; 1\rangle}_{M_1,M_2}$ is that
the set $\Sigma^{\Gamma_1^c}_{\le \lo(M_2^{u_2})}$ is replaced with $\Sigma^{\Gamma_1^c}_{\le \lo(M_1^{u_2})}$.

We begin with  decomposing
\begin{align*}
\Psi^{\langle [d_2]\rangle}_{M_1,M_2}(\xi)-\Psi^{\langle [d_2]; 1\rangle}_{M_1,M_2}(\xi)
=
\sum_{\lo(M_1^{u_2}) < s \le \lo(M_2^{u_2})}
\Delta_{M_1,M_2}^{s}(\xi),
\end{align*}
with
\begin{align}\label{eq:hs_def}
\Delta_{M_1,M_2}^{s}(\xi):=\sum_{m_1\in\ZZ}\chi_{M_1}(m_1) \ex(\xi_{0}^{\Gamma_2^c}(m_1))
\Big\llbracket G^{\Gamma_1^c}\mid \mathfrak{m}^{\langle \Gamma_1^c\rangle }_{M_2, M_2}\eta^{\Gamma_1^c}_{\le \lo_{M_2, M_2}^{\Gamma_1^c}(M_2^{\chi_2})}\Big\rrbracket_{\Sigma^{\Gamma_1^c}_{s}}({\bm\{}\xi_{\Gamma_2}(m_1){\bm \}}).
\end{align}
Observe that \eqref{eq:step4goal} will follow if we show that for a fixed $s\in\NN$, one has
\begin{equation}
\label{eq:28}
\big\|\sup_{    \substack{M_2\in \DD_\tau: M_1 \le M_2^{\varrho} \\ \lo(M_2^{u_2})\ge s}    }  | T_{\ZZ^\Gamma}[ (1-\Pi^{\Gamma_2}_{M_1,M_2})\Delta_{M_1,M_2}^{s} ] f | \big\|_{\ell^2(\ZZ^\Gamma)} \lesssim m 2^{-\delta s/2}\|f\|_{\ell^2(\ZZ^\Gamma)}, 
\end{equation}
where $\delta\in(0, 1)$ appears in \eqref{eq:uniform_delta} and $m\in\ZZ_+$ is defined as in Lemma \ref{lem:padova}. Indeed, summing \eqref{eq:28} over $s>\lo(M_1^{u_2})$ and using \eqref{eq:parameter_tuning} we obtain inequality \eqref{eq:step4goal} as desired.

We now recall the definition of $N_s=2^{\frac{s}{1000u_1}}$ from \eqref{eq:Ns_def}.
Clearly, for $2^s\le M_2^{u_2}$ (and even for $2^s\le M_2^{u_1}$, which we will need later) one has $N_s\le M_2^{1/1000}$. On the other hand, the scale $N_s$ is large enough for the application of Ionescu--Wainger theorem,  see Theorem \ref{thm:IW1} with the family of fractions $\Sigma^{[d_2]}_{s}$. 
Note that for $2^s \le M_2^{u_2}$, we can write 
\begin{align*}
\eta^{\Gamma_1^c}_{\le \lo_{M_2, M_2}^{\Gamma_1^c}(M_2^{\chi_2})}(\zeta)=
\eta^{\Gamma_1^c}_{\le \lo_{M_2, M_2}^{\Gamma_1^c}(M_2^{\chi_2})}(\zeta)
\eta^{\Gamma_1^c}_{\le \lo_{N_s, N_s}^{\Gamma_1^c}(N_s^{\chi_2})}(\zeta)
\eta^{\Gamma_1^c}_{\le \lo_{N_s, N_s}^{\Gamma_1^c}(N_s^{\chi_2})}(\zeta), \qquad \zeta\in\TT^{\Gamma_1^c}.
\end{align*}
Thus, for a fixed $m_1\in(M_1,\tau M_1)$ we have a factorization
\begin{align*}
\Big\llbracket G^{\Gamma_1^c}\mid \mathfrak{m}^{\langle \Gamma_1^c\rangle }_{M_2, M_2}\eta^{\Gamma_1^c}_{\le \lo_{M_2, M_2}^{\Gamma_1^c}(M_2^{\chi_2})}\Big\rrbracket_{\Sigma^{\Gamma_1^c}_{s}}({\bm\{}\xi_{\Gamma_2}(m_1){\bm \}})
=\mathfrak g_{m_1}^s(\xi)\mathfrak h_{m_1, M_2}^s(\xi),
\end{align*}
where
\begin{align*}
\mathfrak g_{m_1}^s(\xi):=&\Big\llbracket G^{\Gamma_1^c}\mid \eta^{\Gamma_1^c}_{\le \lo_{N_s, N_s}^{\Gamma_1^c}(N_s^{\chi_2})}\Big\rrbracket_{\Sigma^{\Gamma_1^c}_{s}}({\bm\{}\xi_{\Gamma_2}(m_1){\bm \}}),\\
\mathfrak h_{m_1, M_2}^s(\xi):= &\Big\llbracket 1\mid \mathfrak{m}^{\langle \Gamma_1^c\rangle }_{M_2, M_2}\eta^{\Gamma_1^c}_{\le \lo_{M_2, M_2}^{\Gamma_1^c}(M_2^{\chi_2})}\eta^{\Gamma_1^c}_{\le \lo_{N_s, N_s}^{\Gamma_1^c}(N_s^{\chi_2})}\Big\rrbracket_{\Sigma^{\Gamma_1^c}_{\le s}}({\bm\{}\xi_{\Gamma_2}(m_1){\bm \}}).
\end{align*}
From now on we can assume that
\begin{align}
\label{eq:fac}
\Delta_{M_1,M_2}^{s}(\xi):=\sum_{m_1\in\ZZ} \chi_{M_1}(m_1) \ex(\xi_{0}^{\Gamma_2^c}(m_1))\mathfrak g_{m_1}^s(\xi)\mathfrak h_{m_1, M_2}^s(\xi).
\end{align}

In view of Lemma \ref{lem:padova}, to show \eqref{eq:28}, it suffices to
prove that $d_{M_1,M_2}=\Delta_{M_1,M_2}^{s}$ satisfies conditions
\eqref{eq:cond1} and \eqref{eq:cond2} with
${\bf M}_{M_1}, {\bf S}_{M_1}\simeq 2^{-\delta s/2}$. We verify each
of these conditions below.

\subsubsection{\textbf{Verification of \eqref{eq:cond1} for
$d_{M_1,M_2}=\Delta_{M_1,M_2}^{s}$}}\label{715}
We show that for a fixed $s\in\NN$, one has
\begin{align}\label{eq:goals}
\big\|\sup_{M_2\in \DD_\tau: M_1 \le M_2^{\varrho}} | T_{\ZZ^\Gamma}[  \Delta_{M_1,M_2}^{s} ] f | \big\|_{\ell^2(\ZZ^\Gamma)} \lesssim 2^{-\delta s/2}\|f\|_{\ell^2(\ZZ^\Gamma)},
\end{align}
with $\Delta_{M_1,M_2}^{s}$ as in \eqref{eq:fac}.
\begin{proof}[Proof of inequality \eqref{eq:goals}]
The proof of inequality \eqref{eq:goals} is technically intricate, so we will proceed in a few steps.

\paragraph{\bf Step 1}
Using \eqref{eq:fac} we can estimate
\begin{align*}
\big\|\sup_{M_2\in \DD_\tau: M_1 \le M_2^{\varrho}} | T_{\ZZ^\Gamma} [ \Delta_{M_1,M_2}^{s} ] f  | \big\|_{\ell^2(\ZZ^{\Gamma})} 
&\le
\sum_{m_1\in\ZZ}\chi_{M_1}(m_1) J^s(m_1)\big\| T_{\ZZ^\Gamma} [ \mathfrak g_{m_1}^s ]f \big\|_{\ell^2(\ZZ^{\Gamma})},
\end{align*}
where
\begin{align*}
J^s(m_1):=\sup_{\|g\|_{\ell^2(\ZZ^\Gamma)}\le 1}\big \|\sup_{M_2\in \DD_\tau: M_1 \le M_2^{\varrho}}|T_{\ZZ^\Gamma} [ \mathfrak h_{m_1,M_2}^{s} ]g|\big \|_{\ell^2(\ZZ^{\Gamma})}.
\end{align*}
Thus, we are reduced to showing that for a fixed $m_1\in[M_1,\tau M_1]$, one has
\begin{align}\label{eq:IWsup}
J^s(m_1)\lesssim 2^{\varepsilon s}
\end{align}
for arbitrarily small $\varepsilon\in(0, 1)$ (for our purposes $\varepsilon<\delta/2$ will suffice),
and
\begin{align}\label{eq:gdecay}
\big\| T_{\ZZ^\Gamma} [ \mathfrak g_{m_1}^s ]f \big\|_{\ell^2(\ZZ^{\Gamma})}\lesssim 2^{-\delta s}\|f\|_{\ell^2(\ZZ^{\Gamma})}.
\end{align}

\paragraph{\bf Step 2} We begin with showing \eqref{eq:IWsup}. Let
\begin{equation}\label{eq:mM2_def}
\mathfrak{n}^{\langle \Gamma_1^c\rangle }_{M_2}:=\mathfrak{m}^{\langle \Gamma_1^c\rangle }_{M_2, M_2}\eta^{\Gamma_1^c}_{\le \lo_{M_2, M_2}^{\Gamma_1^c}(M_2^{\chi_2})},
\end{equation}
with $\mathfrak{m}^{\langle \Gamma_1^c\rangle }_{M_2, M_2}$ as in \eqref{eq:mN_def}.
For further reference, by \eqref{eq:9}, we have 
\begin{equation}\label{eq:mNbd}
|\mathfrak{n}^{\langle \Gamma_1^c\rangle }_{M_2}(\zeta)|\lesssim |(M_2^{\gamma_2}\zeta_{(0, \gamma_2)})_{\gamma_2\in[d_2]}|^{-\delta_{\textrm{vdC}}},\qquad \zeta\in\TT^{\Gamma_1^c}.
\end{equation} 
Furthermore, by \eqref{eq:mNbd} and \eqref{eq:11}, we also have
\begin{equation}\label{eq:mNdiff}
|\mathfrak{n}^{\langle \Gamma_1^c\rangle }_{\tau M_2}(\zeta)-\mathfrak{n}^{\langle \Gamma_1^c\rangle }_{M_2}(\zeta)|
\lesssim\min( |(M_2^{\gamma_2}\zeta_{(0, \gamma_2)})_{\gamma_2\in[d_2]}|, |(M_2^{\gamma_2}\zeta_{(0, \gamma_2)})_{\gamma_2\in[d_2]}|^{-\delta_{\textrm{vdC}}} ).  
\end{equation}

For $x\in\ZZ^\Gamma$ we have
\begin{align*}
T_{\ZZ^\Gamma} [ \mathfrak h_{m_1,M_2}^{s} ]f(x)&=\int\limits_{\TT^{\Gamma}} 
\hat{f}(\xi) \Big\llbracket 1\mid \mathfrak{n}^{\langle \Gamma_1^c\rangle }_{M_2}\eta^{\Gamma_1^c}_{\le \lo_{N_s, N_s}^{\Gamma_1^c}(N_s^{\chi_2})}\Big\rrbracket_{\Sigma^{\Gamma_1^c}_{\le s}}({\bm\{}\xi_{\Gamma_2}(m_1){\bm \}}) \ex(\xi\cdot x) \, d\xi.
\end{align*}
We will apply the change variables $\xi\mapsto u$ given by the map
\begin{equation}\label{eq:chgvbl}
u_\gamma := 
\begin{cases}
\xi_{\gamma_2}^{\Gamma_2}(m_1), &\qquad \gamma=(0,\gamma_2)\in\Gamma_1^c,
\\
\xi_{\gamma}, &\qquad \gamma\in \Gamma_1.
\end{cases}
\end{equation}
Note that since for $\gamma=(0,\gamma_2)\in\Gamma_1^c$ we have $\xi_{\gamma_2}^{\Gamma_2}(m_1)=\xi_\gamma+\sum_{\gamma_1=1}^{d_1}\xi_{(\gamma_1,\gamma_2)}m_1^{\gamma_1}$, thus the Jacobian matrix of the above change of variables is upper triangular with all diagonal entries equal to $1$; hence the associated Jacobian determinant is $1$. Therefore, using the translation invariance of the Lebesgue measure on the torus, and setting
\begin{align*}
U_{m_1}^{[d_2]}(u):=\Big(u_{(0,\gamma_2)}-\sum\limits_{\gamma_1=1}^{d_1}m_1^{\gamma_1}u_{(\gamma_1,\gamma_2)} \Big)_{\gamma_2\in[d_2]},
\end{align*}
we obtain
\begin{align*}
T_{\ZZ^\Gamma} [ \mathfrak h_{m_1,M_2}^{s} ]f(x)&=
 \int\limits_{\TT^{\Gamma}}\Big\llbracket 1\mid \mathfrak{n}^{\langle \Gamma_1^c\rangle }_{M_2}\eta^{\Gamma_1^c}_{\le \lo_{N_s, N_s}^{\Gamma_1^c}(N_s^{\chi_2})}\Big\rrbracket_{\Sigma^{\Gamma_1^c}_{\le s}}(u'') 
\hat{f}\big( u', U_{m_1}^{[d_2]}(u)\big) 
\ex\big( u'\cdot x' + U_{m_1}^{[d_2]}(u)\cdot x''\big)du,
\end{align*}
where we wrote $u=(u', u'')$ with $u'\in\TT^{\Gamma_1}$ and $u''=( u_{(0,\gamma_2)})_{\gamma_2\in [d_2]}\in\TT^{\Gamma_1^c}$, and $x=(x', x'')$ with $x'\in\TT^{\Gamma_1}$ and $x''\in\TT^{\Gamma_1^c}$.
Therefore, upon regrouping the terms inside $\ex(\cdot)$, we obtain
\begin{align*}
T_{\ZZ^\Gamma} [ \mathfrak h_{m_1,M_2}^{s} ]f(x)&= 
 \int\limits_{\TT^{\Gamma}}\Big\llbracket 1\mid \mathfrak{n}^{\langle \Gamma_1^c\rangle }_{M_2}\eta^{\Gamma_1^c}_{\le \lo_{N_s, N_s}^{\Gamma_1^c}(N_s^{\chi_2})}\Big\rrbracket_{\Sigma^{\Gamma_1^c}_{\le s}}(u'') 
\hat{f}\big( u', U_{m_1}^{[d_2]}(u)\big) 
\\
&\qquad\quad\times \ex\Big(  \sum_{\gamma\in \Gamma_1^c\cup\Gamma_2^c} u_\gamma x_\gamma +  \sum_{\gamma\in [d_1]\times[d_2]} u_{\gamma} (x_\gamma- m_1^{\gamma_1} x_{(0,\gamma_2)})\Big)du.
\end{align*}
 Note that coordinates $(x_\gamma)_{\gamma\in[d_1]\times[d_2]}$ are shifted by integers $m_1^{\gamma_1}x_{(0,\gamma_2)}$ and these shifts can be disregarded when summing over $x\in\ZZ^\Gamma$ due to the translation invariance of the counting measure. Consequently, we obtain
\begin{align}
\nonumber
&\Big \| \sup_{M_2\in \DD_\tau: M_1 \le M_2^{\varrho}} |T_{\ZZ^\Gamma} [ \mathfrak h_{m_1,M_2}^{s} ]f| \Big \|^2_{\ell^2(\ZZ^\Gamma)} \\
&\label{eq:27}\quad=
\sum_{x\in\ZZ^\Gamma} \sup_{M_2\in \DD_\tau: M_1 \le M_2^{\varrho}} \Big | \int_{\TT^{\Gamma_1^c}}
 \ex(u''\cdot x'')\Big\llbracket 1\mid \mathfrak{n}^{\langle \Gamma_1^c\rangle }_{M_2}\eta^{\Gamma_1^c}_{\le \lo_{N_s, N_s}^{\Gamma_1^c}(N_s^{\chi_2})}\Big\rrbracket_{\Sigma^{\Gamma_1^c}_{\le s}}(u'') 
F_{x^1}(u'')du''
\Big|^2,
\end{align}
where
\begin{align}\label{eq:F_def}
F_{x'}(u'')&=
 \int_{\TT^{ \Gamma_1} }
\hat{f}\big( u', U_{m_1}^{[d_2]}(u)\big) 
\ex\big(  u'\cdot x' \big) du'.
\end{align}
Changing the variable $u$ back to $\xi$, see \eqref{eq:chgvbl}; that is, taking
\begin{equation}\label{eq:chgvblrev}
\xi_\gamma := 
\begin{cases}
u_\gamma-\sum_{\gamma_1=1}^{d_1}m_1^{\gamma_1}u_{(\gamma_1, \gamma_2)}, &\qquad \gamma\in\Gamma_1^c,
\\
u_{\gamma}, &\qquad \gamma\in \Gamma_1,
\end{cases}
\end{equation}
we obtain
\begin{align}
\label{eq:FinvFT}
\begin{split}
\big(\mathcal{F}^{-1}_{\TT^{ \Gamma_1^c  } }F_{x'}\big)(x'')
&= 
 \int_{\TT^{\Gamma}} \hat{f}\big( u', U_{m_1}^{[d_2]}(u)\big)\ex( u \cdot x )du
\\ 
&=\int_{\TT^\Gamma}\hat{f}(\xi)  \ex\Big(\xi'\cdot x'
+ 
\sum_{\gamma_2\in[d_2] }  \xi_{\gamma_2}^{\Gamma_2}(m_1) x_{(0,\gamma_2)}
\Big)d\xi
\\ 
&=f \big(   (x_\gamma)_{\gamma\in\Gamma\setminus{([d_1]\times [d_2])}},   (x_\gamma +m_1^{\gamma_1}x_{(0,\gamma_2)})_{\gamma\in [d_1]\times [d_2]}\big),
\end{split}
\end{align}
where as before we set  $\xi=(\xi', \xi'')$ with $\xi'\in\TT^{\Gamma_1 }$ and $\xi''=( u_{(0,\gamma_2)})_{\gamma_2\in [d_2]}\in\TT^{\Gamma_1^c}$.
Now, we  aim to apply the one-parameter Ionescu--Wainger theorem, Theorem \ref{thm:IW1}. However, since the convolution operator associated with the Fourier multiplier $\mathfrak{n}^{\langle \Gamma_1^c\rangle }_{M_2}$ is not positive in the sense of definition \eqref{eq:377}, we need to decompose our expression on the right-hand side of \eqref{eq:27}. We have
\begin{align*}
&\int_{\TT^{\Gamma_1^c}}
 \ex(u''\cdot x'')\Big\llbracket 1\mid \mathfrak{n}^{\langle \Gamma_1^c\rangle }_{M_2}\eta^{\Gamma_1^c}_{\le \lo_{N_s, N_s}^{\Gamma_1^c}(N_s^{\chi_2})}\Big\rrbracket_{\Sigma^{\Gamma_1^c}_{\le s}}(u'') 
F_{x'}(u'')du''
\\
&=\int_{\TT^{\Gamma_1^c}}
 \ex(u''\cdot x'')\Big\llbracket 1\mid \mathfrak{m}^{\langle \Gamma_1^c\rangle }_{M_2, M_2}\eta^{\Gamma_1^c}_{\le \lo_{N_s, N_s}^{\Gamma_1^c}(N_s^{\chi_2})}\Big\rrbracket_{\Sigma^{\Gamma_1^c}_{\le s}}(u'') 
F_{x'}(u'')du''
\\
&\quad
+\int_{\TT^{\Gamma_1^c}}
 \ex(u''\cdot x'')\Big\llbracket 1\mid \mathfrak{m}^{\langle \Gamma_1^c\rangle }_{M_2, M_2}\big(\eta^{\Gamma_1^c}_{\le \lo_{M_2, M_2}^{\Gamma_1^c}(M_2^{\chi_2})}-1\big)\eta^{\Gamma_1^c}_{\le \lo_{N_s, N_s}^{\Gamma_1^c}(N_s^{\chi_2})}\Big\rrbracket_{\Sigma^{\Gamma_1^c}_{\le s}}(u'') 
F_{x'}(u'')du''.
\end{align*}

To treat the first term above we notice that the multipliers
$\mathfrak{m}^{\langle \Gamma_1^c\rangle }_{M_2, M_2}$ give rise to a family of positive
operators. Thus, applying Theorem \ref{thm:IW1} (on
$\ZZ^{\Gamma_1^c}$), we obtain
\begin{align*}
\sum_{x\in\ZZ^\Gamma} \sup_{M_2\in \DD_\tau: M_1 \le M_2^{\varrho}} &\Big | \int_{\TT^{\Gamma_1^c}}
 \ex(u''\cdot x'')\Big\llbracket 1\mid \mathfrak{m}^{\langle \Gamma_1^c\rangle }_{M_2, M_2}\eta^{\Gamma_1^c}_{\le \lo_{N_s, N_s}^{\Gamma_1^c}(N_s^{\chi_2})}\Big\rrbracket_{\Sigma^{\Gamma_1^c}_{\le s}}(u'') 
F_{x'}(u'')du''
\Big|^2
\\
&\quad\lesssim 2^{2\varepsilon s}
\sum_{x'\in\ZZ^{\Gamma_1} } \sum_{ x''\in\ZZ^{\Gamma_1^c} }\big|\big(\mathcal{F}^{-1}_{\TT^{ \Gamma_1^c  } }F_{x'}\big)(x'')\big|^2
= 2^{2\varepsilon s} \|f\|^2_{\ell^2(\ZZ^\Gamma)},
\end{align*}
where in the last relation we used again the translation invariance of the counting measure.

It remains to show
\begin{align}\label{eq:rybki}
\sum_{x\in\ZZ^\Gamma} \sup_{M_2\in \DD_\tau: M_1 \le M_2^{\varrho}} &\Big|\mathcal{F}^{-1}_{\TT^{ \Gamma_1^c  } }\Big(\Big\llbracket 1\mid \mathfrak{m}^{\langle \Gamma_1^c\rangle }_{M_2, M_2}\big(\eta^{\Gamma_1^c}_{\le \lo_{M_2, M_2}^{\Gamma_1^c}(M_2^{\chi_2})}-1\big)\eta^{\Gamma_1^c}_{\le \lo_{N_s, N_s}^{\Gamma_1^c}(N_s^{\chi_2})}\Big\rrbracket_{\Sigma^{\Gamma_1^c}_{\le s}}
F_{x'}\Big)(x'')\Big|^2\\
\nonumber &\quad\lesssim 2^{2\varepsilon s} \|f\|^2_{\ell^2(\ZZ^\Gamma)}.
\end{align}
The argument is similar to the one above, but easier, because it relies only on orthogonality and does not require the Ionescu--Wainger theory. For $\zeta\in\TT^{\Gamma_1^c}$, notice that
\begin{equation*}
\Big|\eta^{\Gamma_1^c}_{\le \lo_{M_2, M_2}^{\Gamma_1^c}(M_2^{\chi_2})}(\zeta)-1\Big|\lesssim\min(1, |(M_2^{\gamma_2}\zeta_{(0, \gamma_2)})_{\gamma_2\in [d_2]}|),
\end{equation*}
which together with \eqref{eq:9} and \eqref{eq:uniform_delta} gives 
\begin{equation}\label{eq:etatilde_dec}
\Big|\mathfrak{m}^{\langle \Gamma_1^c\rangle }_{M_2, M_2}(\zeta)\big(\eta^{\Gamma_1^c}_{\le \lo_{M_2, M_2}^{\Gamma_1^c}(M_2^{\chi_2})}(\zeta)-1\big)\Big|\lesssim\min(|(M_2^{\gamma_2}\zeta_{(0, \gamma_2)})_{\gamma_2\in [d_2]}|^{-\delta_{\textrm{vdC}}}, |(M_2^{\gamma_2}\zeta_{(0, \gamma_2)})_{\gamma_2\in [d_2]}|).
\end{equation}
Using \eqref{eq:etatilde_dec} and the disjointness of supports of functions $\eta^{\Gamma_1^c}_{\le \lo_{M_2, M_2}^{\Gamma_1^c}(M_2^{\chi_2})}( u''-\theta )$ as $\theta$ varies over $\Sigma^{\Gamma_1^c}_{\le s}$, we obtain
\begin{align*}
\sum_{M_2\in \DD_\tau: M_1 \le M_2^{\varrho}}\Big|\Big\llbracket 1\mid \mathfrak{m}^{\langle \Gamma_1^c\rangle }_{M_2, M_2}\big(\eta^{\Gamma_1^c}_{\le \lo_{M_2, M_2}^{\Gamma_1^c}(M_2^{\chi_2})}-1\big)\eta^{\Gamma_1^c}_{\le \lo_{N_s, N_s}^{\Gamma_1^c}(N_s^{\chi_2})}\Big\rrbracket_{\Sigma^{\Gamma_1^c}_{\le s}}(u'')\Big|^2\lesssim 1.
\end{align*}
Now, dominating the supremum in \eqref{eq:rybki} by a square function, applying Plancherel's identity and using the last bound and \eqref{eq:FinvFT} with $F_{x'}$ as in \eqref{eq:F_def}, we conclude that
\begin{align*}
\sum_{x\in\ZZ^\Gamma} &\sup_{M_2\in \DD_\tau: M_1 \le M_2^{\varrho}} \Big|\mathcal{F}^{-1}_{\TT^{ \Gamma_1^c  } }\Big(\Big\llbracket 1\mid \mathfrak{m}^{\langle \Gamma_1^c\rangle }_{M_2, M_2}\big(\eta^{\Gamma_1^c}_{\le \lo_{M_2, M_2}^{\Gamma_1^c}(M_2^{\chi_2})}-1\big)\eta^{\Gamma_1^c}_{\le \lo_{N_s, N_s}^{\Gamma_1^c}(N_s^{\chi_2})}\Big\rrbracket_{\Sigma^{\Gamma_1^c}_{\le s}}
F_{x'}\Big)(x'')\Big|^2
\\
&\quad\lesssim
\sum_{x'\in\ZZ^{\Gamma_1}} 
\int_{\TT^{\Gamma_1^c}} 
|F_{x'} (u'')|^2 
du''
= 
\sum_{x'\in\ZZ^{\Gamma_1}} 
\bigg( 
\sum_{x''\in\ZZ^{\Gamma_1^c}} 
\Big| 
\big(\mathcal{F}^{-1}_{\TT^{\Gamma_1^c}}  F_{x'} \big)(x'') 
\Big|^2
\bigg)= 
\| f \|^2_{\ell^2(\ZZ^\Gamma)}.
\end{align*}
That completes the proof of \eqref{eq:rybki}, and consequently of \eqref{eq:IWsup}.

\paragraph{\bf Step 3} We turn to showing inequality \eqref{eq:gdecay}. 
By the change of variables \eqref{eq:chgvbl}, we have 
\begin{align*}
T_{\ZZ^\Gamma}[\mathfrak g_{m_1}^s] f(x)
&=
\int\limits_{\TT^{\Gamma}} \Big\llbracket G^{\Gamma_1^c}\mid \eta^{\Gamma_1^c}_{\le \lo_{N_s, N_s}^{\Gamma_1^c}(N_s^{\chi_2})}\Big\rrbracket_{\Sigma^{\Gamma_1^c}_{s}}(u'')
\hat{f}\big( u', U_{m_1}^{[d_2]}(u)\big)
\\
&\quad\times \ex\Big(  \sum_{\gamma\in\Gamma_1^c\cup\Gamma_2^c} u_\gamma x_\gamma +  \sum_{\gamma\in [d_1]\times[d_2]} u_{\gamma} (x_\gamma- m_1^{\gamma_1} x_{(0,\gamma_2)})\Big)du.
\end{align*}
Shifting the summation index, applying Plancherel's identity and using
the disjointness of supports of functions
$\eta^{\Gamma_1^c}_{\le \lo_{M_2, M_2}^{\Gamma_1^c}(M_2^{\chi_2})}( u''-\theta )$ as $\theta$ varies over $\Sigma^{\Gamma_1^c}_{\le s}$, we obtain, with
$F_{x''}$ as in \eqref{eq:F_def}, that
\begin{align*}
\|T_{\ZZ^\Gamma}[\mathfrak g_{m_1}^s] f\|^2_{\ell^2(\ZZ^\Gamma)}
&=
\sum_{x'\in\ZZ^{\Gamma_1}} 
\int_{\TT^{\Gamma_1^c}} 
|F_{x'}(u'')|^2 
\Big|\Big\llbracket G^{\Gamma_1^c}\mid \eta^{\Gamma_1^c}_{\le \lo_{N_s, N_s}^{\Gamma_1^c}(N_s^{\chi_2})}\Big\rrbracket_{\Sigma^{\Gamma_1^c}_{s}}(u'')\Big|^2du''.
\end{align*}
Finally, using the decay of Gaussian sums, see \eqref{eq:Gauss}, and Plancherel's theorem we obtain
\begin{align*}
\|T_{\ZZ^\Gamma}[\mathfrak g_{m_1}^s] f\|^2_{\ell^2(\ZZ^\Gamma)}
\lesssim 2^{-2\delta s} 
\sum_{x'\in\ZZ^{\Gamma_1}} 
\int_{\TT^{\Gamma_1^c}} 
|F_{x'} (u'')|^2du''
=
2^{-2\delta s}  \| f \|^2_{\ell^2(\ZZ^\Gamma)},
\end{align*}
so inequality \eqref{eq:gdecay} follows. This concludes the proof of inequality \eqref{eq:goals} as desired.
\end{proof}

\subsubsection{\textbf{Verification of \eqref{eq:cond2} for
$d_{M_1,M_2}=\Delta_{M_1,M_2}^{s}$}} We show that for a fixed $s\in\NN$, one has
\begin{align}
\label{eq:29}
\bigg( \sum_{j=0}^{2^{m-i}-1} 
\Big\|T_{\ZZ^\Gamma} \big[
\sum_{\substack{k\in U_j^i \\ U_j^i\subseteq [j_\varrho,2^m)}} (d_{M_1, \tau^{k+1}}-d_{M_1, \tau^k}) 
\big] f   
\Big\|^2_{\ell^2(\ZZ^\Gamma)} \bigg)^{1/2}\lesssim 2^{-s\delta/2}\|f\|_{\ell^2(\ZZ^\Gamma)},
\end{align}
uniformly in $i\in\mathbb{N}_{\le m}$, with $d_{M_1,M_2}=\Delta_{M_1,M_2}^{s}$ as in \eqref{eq:fac}.

\begin{proof}[Proof of inequality \eqref{eq:29}]
By Minkowski's inequality
\begin{align*} 
&\bigg( \sum_{j=0}^{2^{m-i}-1} 
\Big\|T_{\ZZ^\Gamma} \big[
\sum_{\substack{k\in U_j^i \\ U_j^i\subseteq [j_\varrho,2^m)}} (d_{M_1, \tau^{k+1}}-d_{M_1, \tau^k}) 
\big] f   
\Big\|^2_{\ell^2(\ZZ^\Gamma)} \bigg)^{1/2}
\\
&\quad\le
\sum_{m_1\in\ZZ} \chi_{M_1}(m_1) 
\bigg\| 
\Big( \sum_{\substack{j=0 \\ j: U_j^i\subseteq [j_\varrho, 2^m)}}^{2^{m-i}-1}   
\big|
T_{\ZZ^\Gamma}
[ 
\mathfrak h^s_{m_1,\tau^{(j+1)2^i}} - \mathfrak h^s_{m_1,\tau^{j 2^i}} 
] T_{\ZZ^\Gamma}
[\mathfrak g_{m_1}^s ]f
\big|^2   
\Big)^{1/2}  \bigg\|_{\ell^2(\ZZ^\Gamma)}.
\end{align*}
In view of inequality \eqref{eq:gdecay}, it suffices to show that  for any  $f\in \ell^2(\ZZ^\Gamma)$, we have 
\begin{equation}\label{eq:alphas2}
\bigg\| 
\Big( \sum_{j=0}^{2^{m-i}-1}   
\big|
T_{\ZZ^\Gamma}
[ 
\mathfrak h^s_{m_1,\tau^{(j+1)2^i}} - \mathfrak h^s_{m_1,\tau^{j 2^i}} 
] f
\big|^2   
\Big)^{1/2}  \bigg\|_{\ell^2(\ZZ^\Gamma)}\lesssim  \| f \|_{\ell^2(\ZZ^\Gamma)}.
\end{equation}

Using the change of variables \eqref{eq:chgvbl}, with the notation from \eqref{eq:mM2_def}, we obtain
\begin{align*}
&\bigg\| 
\Big( \sum_{j=0}^{2^{m-i}-1}   
\big|
T_{\ZZ^\Gamma}
[ 
\mathfrak h^s_{m_1,\tau^{(j+1)2^i}} - \mathfrak h^s_{m_1,\tau^{j 2^i}} 
] f
\big|^2   
\Big)^{1/2}  \bigg\|_{\ell^2(\ZZ^\Gamma)}^2=
\sum_{x'\in\ZZ^{\Gamma_1}}\sum_{x''\in\ZZ^{\Gamma_1^c}} G(x', x''),
\end{align*}
where
\begin{align*}
G(x', x''):=\sum_{j=0}^{2^{m-i}-1}\Big | \int_{\TT^{\Gamma_1^c}}
\ex(u''\cdot x'')\Big\llbracket 1\mid (\mathfrak{n}^{\langle\Gamma_1^c\rangle}_{\tau^{(j+1)2^i}}-\mathfrak{n}^{\langle\Gamma_1^c\rangle}_{\tau^{j2^i}})\eta^{\Gamma_1^c}_{\le \lo_{N_s, N_s}^{\Gamma_1^c}(N_s^{\chi_2})}\Big\rrbracket_{\Sigma^{\Gamma_1^c}_{\le s}}(u'')
F_{x'}(u'')du''
\Big|^2,
\end{align*}
with $F_{x'}$ as in \eqref{eq:F_def}. Using \eqref{eq:mNdiff} and a simple orthogonality argument, we conclude that
\begin{align*}
\sum_{x'\in\ZZ^{\Gamma_1}}\sum_{x''\in\ZZ^{\Gamma_1^c}}G(x', x'') 
\lesssim 
\sum_{x'\in\ZZ^{\Gamma_1}}\sum_{x''\in\ZZ^{\Gamma_1^c}}
\Big | \int_{\TT^{\Gamma_1^c}}
 \ex(u''\cdot x'') F_{x'}(u'')du''\Big|^2.
\end{align*}
Applying the inverse Fourier transform formula \eqref{eq:FinvFT}, we obtain \eqref{eq:alphas2} and 
\eqref{eq:29} follows.
\end{proof}

\subsubsection{\textbf{Change of scales in the cutoff function}} \label{step5}
In view of inequality \eqref{eq:step4goal} the proof of inequality \eqref{eq:maingoal} will be complete if we show that
\begin{equation}\label{eq:goal_2tilde}
\big\|\sup_{M_2\in \DD_\tau: M_1 \le M_2^{\varrho}} | T_{\ZZ^\Gamma}\big[ (1-\Pi^{\Gamma_2}_{M_1,M_2}) \Psi^{\langle [d_2]; 1\rangle}_{M_1,M_2} \big]f | \big\|_{\ell^2(\ZZ^\Gamma)} \lesssim M_1^{-\alpha}\|f\|_{\ell^2(\ZZ^\Gamma)},
\end{equation}
with $\Psi^{\langle [d_2]; 1\rangle}_{M_1,M_2}$ defined in \eqref{eq:mdoubletilde_def}.
In this step we will tighten the cutoff scales in $\Psi^{\langle [d_2]; 1\rangle}_{M_1,M_2}$. More precisely, we will show 
\begin{equation}\label{eq:step5goal}
    \big\|\sup_{M_2\in \DD_\tau: M_1 \le M_2^{\varrho}} | T_{\ZZ^\Gamma}\big[ (1-\Pi^{\Gamma_2}_{M_1,M_2})(\Psi^{\langle [d_2]; 1\rangle}_{M_1,M_2}-\Psi^{\langle [d_2]; 2\rangle}_{M_1,M_2}) \big] f | \big\|_{\ell^2(\ZZ^\Gamma)} \lesssim M_1^{-\alpha}\|f\|_{\ell^2(\ZZ^\Gamma)}, 
\end{equation}
with 
\begin{align}
\label{eq:mtripletilde_def}
    \Psi^{\langle [d_2]; 2\rangle}_{M_1,M_2}(\xi)
:= \sum_{m_1\in\ZZ}\chi_{M_1}(m_1) \ex(\xi_{0}^{\Gamma_2^c}(m_1))
\Big\llbracket G^{\Gamma_1^c}\mid \mathfrak{m}^{\langle \Gamma_1^c\rangle }_{M_2, M_2}\eta^{\Gamma_1^c}_{\le \lo_{M_2, M_2}^{\Gamma_1^c}(M_1^{\chi_2})}\Big\rrbracket_{\Sigma^{\Gamma_1^c}_{\le \lo(M_1^{u_2})}}({\bm\{}\xi_{\Gamma_2}(m_1){\bm \}}).
\end{align}
Observe that the only difference between $\Psi^{\langle [d_2]; 1\rangle}_{M_1,M_2}$ and $\Psi^{\langle [d_2]; 2\rangle}_{M_1,M_2}$ is that in the definition of the latter the cutoff function $\eta^{\Gamma_1^c}_{\le \lo_{M_2, M_2}^{\Gamma_1^c}(M_2^{\chi_2})}$ is replaced by $\eta^{\Gamma_1^c}_{\le \lo_{M_2, M_2}^{\Gamma_1^c}(M_1^{\chi_2})}$.

As before, we decompose 
\begin{align*}
\Psi^{\langle [d_2]; 2\rangle}_{M_1,M_2}(\xi)-\Psi^{\langle [d_2]; 1\rangle}_{M_1,M_2}(\xi)
=
\sum_{0\le  s \le \lo(M_1^{u_2})}
\Delta_{M_1,M_2}^{s}(\xi),
\end{align*}
with
\begin{align}\label{eq:30}
\Delta_{M_1,M_2}^{s}(\xi):=\sum_{m_1\in\ZZ} \chi_{M_1}(m_1) \ex(\xi_0^{\Gamma_2^c}(m_1))\Big\llbracket G^{\Gamma_1^c}\mid \mathfrak{m}^{\langle \Gamma_1^c\rangle }_{M_2, M_2}\tilde\eta^{\Gamma_1^c}_{\le \lo_{M_2, M_2}^{\Gamma_1^c}(M_1^{\chi_2})}\Big\rrbracket_{\Sigma^{\Gamma_1^c}_{s}}({\bm\{}\xi_{\Gamma_2}(m_1){\bm \}}),
\end{align}
where
\begin{align*}
\tilde\eta^{\Gamma_1^c}_{\le \lo_{M_2, M_2}^{\Gamma_1^c}(M_1^{\chi_2})}:=\eta^{\Gamma_1^c}_{\le \lo_{M_2, M_2}^{\Gamma_1^c}(M_1^{\chi_2})}-\eta^{\Gamma_1^c}_{\le \lo_{M_2, M_2}^{\Gamma_1^c}(M_2^{\chi_2})}.
\end{align*}
Observe that \eqref{eq:step5goal} will follow if we show that for a fixed $s\in\NN$ with $s \le \lo(M_1^{u_2})$,
one has
\begin{equation}
\label{eq:31}
\big\|\sup_{M_2\in \DD_\tau: M_1 \le M_2^{\varrho} }  | T_{\ZZ^\Gamma}[ (1-\Pi^{\Gamma_2}_{M_1,M_2})\Delta_{M_1,M_2}^{s} ] f | \big\|_{\ell^2(\ZZ^\Gamma)} \lesssim 2^{-\delta s}M_1^{-\alpha}\|f\|_{\ell^2(\ZZ^\Gamma)}, 
\end{equation}
where $\delta\in(0, 1)$ is the constant in \eqref{eq:uniform_delta}. Indeed, summing \eqref{eq:31} over $s\ge 0$ immediately yields inequality \eqref{eq:step5goal}.

Recalling the definition of $N_s=2^{\frac{s}{1000u_1}}$ from \eqref{eq:Ns_def}, we note that for $2^s \le M_1^{u_2}$, we can write 
\begin{align*}
\tilde\eta^{\Gamma_1^c}_{\le \lo_{M_2, M_2}^{\Gamma_1^c}(M_1^{\chi_2})}(\zeta)=
\tilde\eta^{\Gamma_1^c}_{\le \lo_{M_2, M_2}^{\Gamma_1^c}(M_1^{\chi_2})}(\zeta)
\eta^{\Gamma_1^c}_{\le \lo_{N_s, N_s}^{\Gamma_1^c}(N_s^{\chi_2})}(\zeta)
\eta^{\Gamma_1^c}_{\le \lo_{N_s, N_s}^{\Gamma_1^c}(N_s^{\chi_2})}(\zeta), \qquad \zeta\in\TT^{\Gamma_1^c}.
\end{align*}
Thus, for a fixed $m_1\in(M_1,\tau M_1)$ we have a factorization
\begin{align*}
\Big\llbracket G^{\Gamma_1^c}\mid \mathfrak{m}^{\langle \Gamma_1^c\rangle }_{M_2, M_2}\tilde\eta^{\Gamma_1^c}_{\le \lo_{M_2, M_2}^{\Gamma_1^c}(M_1^{\chi_2})}\Big\rrbracket_{\Sigma^{\Gamma_1^c}_{s}}({\bm\{}\xi_{\Gamma_2}(m_1){\bm \}})
=\mathfrak g_{m_1}^s(\xi)\mathfrak h_{m_1, M_2}^s(\xi),
\end{align*}
where
\begin{align*}
\mathfrak g_{m_1}^s(\xi):=&\Big\llbracket G^{\Gamma_1^c}\mid \eta^{\Gamma_1^c}_{\le \lo_{N_s, N_s}^{\Gamma_1^c}(N_s^{\chi_2})}\Big\rrbracket_{\Sigma^{\Gamma_1^c}_{s}}({\bm\{}\xi_{\Gamma_2}(m_1){\bm \}}),\\
\mathfrak h_{m_1, M_2}^s(\xi):= &\Big\llbracket 1\mid \mathfrak{m}^{\langle \Gamma_1^c\rangle }_{M_2, M_2}\tilde\eta^{\Gamma_1^c}_{\le \lo_{M_2, M_2}^{\Gamma_1^c}(M_1^{\chi_2})}\eta^{\Gamma_1^c}_{\le \lo_{N_s, N_s}^{\Gamma_1^c}(N_s^{\chi_2})}\Big\rrbracket_{\Sigma^{\Gamma_1^c}_{\le s}}({\bm\{}\xi_{\Gamma_2}(m_1){\bm \}}).
\end{align*}
From now on we can assume that
\begin{align}
\label{eq:32}
\Delta_{M_1,M_2}^{s}(\xi)=\sum_{m_1\in\ZZ} \chi_{M_1}(m_1) \ex(\xi_{0}^{\Gamma_2^c}(m_1))\mathfrak g_{m_1}^s(\xi)\mathfrak h_{m_1, M_2}^s(\xi).
\end{align}

In view of Lemma \ref{lem:padova}, to show \eqref{eq:31}, it suffices to
prove that $d_{M_1,M_2}=\Delta_{M_1,M_2}^{s}$ satisfies conditions
\eqref{eq:cond1} and \eqref{eq:cond2} with
${\bf M}_{M_1}, {\bf S}_{M_1}\simeq 2^{-\delta s}M_1^{-2\alpha}$. We verify each
of these conditions below.

\subsubsection{\textbf{Verification of \eqref{eq:cond1} for
$d_{M_1,M_2}=\Delta_{M_1,M_2}^{s}$}}
We show that for a fixed $s\in\NN$, one has
\begin{align}\label{eq:goal_tilde}
\big\|\sup_{M_2\in \DD_\tau: M_1 \le M_2^{\varrho}} | T_{\ZZ^\Gamma}[  \Delta_{M_1,M_2}^{s} ] f | \big\|_{\ell^2(\ZZ^\Gamma)} \lesssim 2^{-\delta s}M_1^{-2\alpha}\|f\|_{\ell^2(\ZZ^\Gamma)},
\end{align}
with $\Delta_{M_1,M_2}^{s}$ defined in \eqref{eq:32}.

\begin{proof}[Proof of inequality \eqref{eq:goal_tilde}]
Using \eqref{eq:fac} we can estimate
\begin{align*}
\big\|\sup_{M_2\in \DD_\tau: M_1 \le M_2^{\varrho}} | T_{\ZZ^\Gamma} [ \Delta_{M_1,M_2}^{s} ] f  | \big\|_{\ell^2(\ZZ^{\Gamma})} 
&\le
\sum_{m_1\in\ZZ}\chi_{M_1}(m_1) J^s(m_1)\big\| T_{\ZZ^\Gamma} [ \mathfrak g_{m_1}^s ]f \big\|_{\ell^2(\ZZ^{\Gamma})},
\end{align*}
where
\begin{align*}
J^s(m_1):=\sup_{\|g\|_{\ell^2(\ZZ^\Gamma)}\le 1}\big \|\sup_{M_2\in \DD_\tau: M_1 \le M_2^{\varrho}}|T_{\ZZ^\Gamma} [ \mathfrak h_{m_1,M_2}^{s} ]g|\big \|_{\ell^2(\ZZ^{\Gamma})}.
\end{align*}
Invoking inequality \eqref{eq:gdecay}, we are reduced to showing that for a fixed $m_1\in[M_1,\tau M_1]$, one has
\begin{align}
\label{eq:33}
J^s(m_1)\lesssim M_1^{-2\alpha}.
\end{align}
 Inequality \eqref{eq:gdecay} combined with
estimate \eqref{eq:33}  yield inequality \eqref{eq:goal_tilde}.

To prove \eqref{eq:33}, we will proceed in much the same way as in the proof of inequality \eqref{eq:rybki}. We first observe that for any $\zeta\in\TT^{\Gamma_1^c}$, we have 
\begin{align}
\label{eq:34}
\big|\mathfrak{m}^{\langle \Gamma_1^c\rangle }_{M_2, M_2}\tilde\eta^{\Gamma_1^c}_{\le \lo_{M_2, M_2}^{\Gamma_1^c}(M_1^{\chi_2})}(\zeta)\big|\lesssim M_1^{-\delta \chi_2/2}\min\big(|(M_2^{\gamma_2}\zeta_{(0, \gamma_2)})_{\gamma_2\in [d_2]}|, |(M_2^{\gamma_2}\zeta_{(0, \gamma_2)})_{\gamma_2\in [d_2]}|^{-1}\big)^{\delta/2}.
\end{align}
Then by the change of variables \eqref{eq:chgvbl}, dominating the  supremum by a square function and using Plancherel's theorem, combined with estimate \eqref{eq:34}, we deduce \eqref{eq:33} as desired.
\end{proof}

\subsubsection{\textbf{Verification of \eqref{eq:cond2} for
$d_{M_1,M_2}=\Delta_{M_1,M_2}^{s}$}} We show that for a fixed $s\in\NN$, one has
\begin{align}
\label{eq:35}
\bigg( \sum_{j=0}^{2^{m-i}-1} 
\Big\|T_{\ZZ^\Gamma} \big[
\sum_{\substack{k\in U_j^i \\ U_j^i\subseteq [j_\varrho,2^m)}} (d_{M_1, \tau^{k+1}}-d_{M_1, \tau^k}) 
\big] f   
\Big\|^2_{\ell^2(\ZZ^\Gamma)} \bigg)^{1/2}\lesssim 2^{-s\delta} M_1^{-2\alpha} \|f\|_{\ell^2(\ZZ^\Gamma)},
\end{align}
uniformly in $i\in\mathbb{N}_{\le m}$, with $d_{M_1,M_2}=\Delta_{M_1,M_2}^{s}$ as in \eqref{eq:32}.
\begin{proof}[Proof of inequality \eqref{eq:35}]
As in the proof of inequality \eqref{eq:29}, if suffices to prove that
\begin{align}
\label{eq:36}
\bigg\| 
\Big( \sum_{j=0}^{2^{m-i}-1}   
\big|
T_{\ZZ^\Gamma}
[ 
\mathfrak h^s_{m_1,\tau^{(j+1)2^i}} - \mathfrak h^s_{m_1,\tau^{j 2^i}} 
] f
\big|^2   
\Big)^{1/2}  \bigg\|_{\ell^2(\ZZ^\Gamma)}\lesssim M_1^{-2\alpha} \| f \|_{\ell^2(\ZZ^\Gamma)}.
\end{align}
For this purpose, applying the change of variables \eqref{eq:chgvbl}, with the notation from \eqref{eq:mM2_def}, we obtain
\begin{align*}
\bigg\| 
\Big( \sum_{j=0}^{2^{m-i}-1}   
\big|
T_{\ZZ^\Gamma}
[ 
\mathfrak h^s_{m_1,\tau^{(j+1)2^i}} - \mathfrak h^s_{m_1,\tau^{j 2^i}} 
] f
\big|^2   
\Big)^{1/2}  \bigg\|_{\ell^2(\ZZ^\Gamma)}^2=
\sum_{x'\in\ZZ^{\Gamma_1}}\sum_{x''\in\ZZ^{\Gamma_1^c}} G(x', x''),
\end{align*}
where
\begin{align*}
G(x', x''):=\sum_{j=0}^{2^{m-i}-1}\Big | \int_{\TT^{\Gamma_1^c}}
\ex(u''\cdot x'')\llbracket 1\mid (\mathfrak{n}^{\langle \Gamma_1^c\rangle}_{\tau^{(j+1)2^i}}-
\mathfrak{n}^{\langle \Gamma_1^c\rangle}_{\tau^{j2^i}})\eta^{\Gamma_1^c}_{\le \lo_{N_s, N_s}^{\Gamma_1^c}(N_s^{\chi_2})}\rrbracket_{\Sigma^{\Gamma_1^c}_{\le s}}(u'')
F_{x'}(u'')du''
\Big|^2,
\end{align*}
with $F_{x'}$ as in \eqref{eq:F_def}, and
\begin{align*}
\mathfrak{n}^{\langle \Gamma_1^c\rangle }_{M_2}:=\mathfrak{m}^{\langle \Gamma_1^c\rangle }_{M_2, M_2}\tilde\eta^{\Gamma_1^c}_{\le \lo_{M_2, M_2}^{\Gamma_1^c}(M_1^{\chi_2})}.
\end{align*}
By a simple orthogonality argument exploiting \eqref{eq:34} and formula \eqref{eq:FinvFT}, we conclude that
\begin{align*}
\sum_{x'\in\ZZ^{\Gamma_1}}\sum_{x''\in\ZZ^{\Gamma_1^c}} G(x', x'')
\lesssim 
 M_1^{-4\alpha}\|f\|_{\ell^2(\ZZ^\Gamma)}^2.
 \end{align*}
 This implies inequality \eqref{eq:36} and concludes the verification of \eqref{eq:cond2}.
\end{proof}

\subsubsection{\textbf{Major-arc rigidity phenomena}}\label{step6}
In view of inequality \eqref{eq:step5goal} the proof of inequality
\eqref{eq:goal_2tilde} will be complete if we show that
\begin{equation}\label{eq:goal_mm_2old}
\big\|\sup_{M_2\in \DD_\tau: M_1 \le M_2^{\varrho}} | T_{\ZZ^\Gamma}\big[ (1-\Pi^{\Gamma_2}_{M_1,M_2}) \Psi^{\langle [d_2]; 2\rangle}_{M_1,M_2} \big]f | \big\|_{\ell^2(\ZZ^\Gamma)} \lesssim M_1^{-\alpha}\|f\|_{\ell^2(\ZZ^\Gamma)},
\end{equation}
with $\Psi^{\langle [d_2]; 2\rangle}_{M_1,M_2}$ defined in \eqref{eq:mtripletilde_def}. 
Setting
\begin{align}
\label{eq:43}
\mathfrak h_{m_1, M_2}(\xi):=
\Big\llbracket G^{\Gamma_1^c}\mid \mathfrak{n}^{\langle \Gamma_1^c\rangle }_{M_2}\eta^{\Gamma_1^c}_{\le \lo_{M_1, M_1}^{\Gamma_1^c}(M_1^{\chi_2})}\Big\rrbracket_{\Sigma^{\Gamma_1^c}_{\le \lo(M_1^{u_2})}}({\bm\{}\xi_{\Gamma_2}(m_1){\bm \}}),
\end{align}
with
\begin{align}
\label{eq:42}
\mathfrak{n}^{\langle \Gamma_1^c\rangle }_{M_2}:=\mathfrak{m}^{\langle \Gamma_1^c\rangle }_{M_2, M_2}\eta^{\Gamma_1^c}_{\le \lo_{M_2, M_2}^{\Gamma_1^c}(M_1^{\chi_2})},
\end{align}
we can write
\begin{align*}
(1-\Pi^{\Gamma_2}_{M_1,M_2}(\xi)) \Psi^{\langle [d_2];2\rangle}_{M_1,M_2}(\xi)
&= (1-\Pi^{\Gamma_2}_{M_1,M_2}(\xi)) \sum_{m_1\in\ZZ} \chi_{M_1}(m_1) \ex(\xi_{0}^{\Gamma_2^c}(m_1))\mathfrak h_{m_1, M_2}(\xi),
\end{align*}
since \
$\eta^{\Gamma_1^c}_{\le \lo_{M_2, M_2}^{\Gamma_1^c}(M_1^{\chi_2})}=
\eta^{\Gamma_1^c}_{\le \lo_{M_2, M_2}^{\Gamma_1^c}(M_1^{\chi_2})}
\eta^{\Gamma_1^c}_{\le \lo_{M_1, M_1}^{\Gamma_1^c}(M_1^{\chi_2})}$. Setting
\begin{equation*}
    A_{M_1,M_2}^{m_1}:=\bigcup_{\theta \in \Sigma^{\Gamma_1^c}_{\le \lo(M_1^{u_2})}}A^{m_1}_{M_1,M_2}(\theta),
\end{equation*}
with 
\begin{equation*}
A^{m_1}_{M_1,M_2}(\theta):=
\big\{\xi \in \mathbb{T}^{\Gamma}:  (1-\Pi^{\Gamma_2}_{M_1,M_2}(\xi)) \neq 0 \quad\textrm{and}\quad \eta^{\Gamma_1^c}_{\le \lo_{M_2, M_2}^{\Gamma_1^c}(M_1^{\chi_2})}({\bm\{}\xi_{\Gamma_2}(m_1){\bm \}}-\theta) \neq 0  \big\},
\end{equation*}
we can further write that
\begin{align*}
(1-\Pi^{\Gamma_2}_{M_1,M_2}(\xi)) \Psi^{\langle [d_2];2\rangle}_{M_1,M_2}(\xi)
&=(1-\Pi^{\Gamma_2}_{M_1,M_2}(\xi)) \sum_{m_1\in\ZZ} \chi_{M_1}(m_1) \ex(\xi_{0}^{\Gamma_2^c}(m_1)) \mathbbm{1}_{A_{M_1,M_2}^{m_1}}(\xi)\mathfrak h_{m_1, M_2}(\xi).
\end{align*}

We now replace $A_{M_1,M_2}^{m_1}$ with a larger set independent of $M_2$. To this end we note that for any $M_2\in \DD_\tau$ satisfying $M_2\ge M_1^{1/\varrho}$ one can write
\begin{align}
\label{eq:39}
&A^{m_1}_{M_1,M_2}(\theta)\subseteq A^{m_1}_{M_1}(\theta):= \bigcup_{ N \in \DD_\tau: N\ge M_1^{1/\varrho} } A^{m_1}_{M_1,N}(\theta), \quad \text{ for } \quad \theta \in \Sigma^{\Gamma_1^c}_{\le \lo(M_1^{u_2})}.
\end{align}
Consequently, letting 
\begin{equation}\label{eq:AM1m1_def}
A_{M_1}^{m_1}:=\bigcup_{\theta \in \Sigma^{\Gamma_1^c}_{\le \lo(M_1^{u_2})}}A^{m_1}_{M_1}(\theta),
\end{equation}
we see, by \eqref{eq:39} and \eqref{eq:AM1m1_def},  that it suffices to prove inequality\eqref{eq:goal_mm_2old} with
\begin{align}
\label{eq:45}
(1-\Pi^{\Gamma_2}_{M_1,M_2} )\Psi^{\langle [d_2];2\rangle}_{M_1,M_2}
=(1-\Pi^{\Gamma_2}_{M_1,M_2})\Delta_{M_1, M_2},
\end{align}
where
\begin{align}
\label{eq:46}
\Delta_{M_1, M_2}(\xi):=\sum_{m_1\in\ZZ} \chi_{M_1}(m_1) \ex(\xi_{0}^{\Gamma_2^c}(m_1))\mathbbm{1}_{A_{M_1}^{m_1}}(\xi)
\mathfrak h_{m_1, M_2}(\xi).
\end{align}    

The key ingredient to prove inequality \eqref{eq:goal_mm_2old} will be the following  counting lemma.

\begin{lemma}\label{lem:comb}
The estimate
\begin{equation}
\label{eq:40}
\sum_{m_1\in\ZZ} \chi_{M_1}(m_1) \mathbbm{1}_{A_{M_1}^{m_1}(\theta)}(\xi) \lesssim M_1^{-u_1/2(|\Gamma|+1)^2}
\end{equation}
holds uniformly in $\theta \in \Sigma^{\Gamma_1^c}_{\le \lo(M_1^{u_2})}$ and $\xi\in\TT^\Gamma$. Consequently, in view of \eqref{eq:AM1m1_def}, we also have
\begin{equation}\label{eq:m1averagedec}
\sum_{m_1\in\ZZ} \chi_{M_1}(m_1) \mathbbm{1}_{A_{M_1}^{m_1}}(\xi) \lesssim M_1^{-4\alpha}, \qquad \xi\in\TT^{\Gamma}.
\end{equation}

\end{lemma}

\begin{proof}
Inequality \eqref{eq:40} implies \eqref{eq:m1averagedec}, in view of \eqref{eq:parameter_tuning} as $|\Sigma^{\Gamma_1^c}_{\le \lo(M_1^{u_2})}|\le M_1^{u_2(|\Gamma|+1)}$.
It suffices to prove inequality \eqref{eq:40}. We will proceed in a few steps.

\paragraph{\bf Step 1}
Fix $\theta=a/q \in \Sigma^{\Gamma_1^c}_{\le \lo(M_1^{u_2})}$. Let $\kappa=u_1/2(|\Gamma|+1)^2$ and suppose that for some $\xi\in\TT^\Gamma$ there exist $K$
points $m_1<m_2< \dots <m_K$ in the interval $(M_1,\tau M_1)$ such that 
\begin{equation}\label{eq:xicond}
\xi \in A^{m_j}_{M_1}(a/q), \qquad j\in [K].
\end{equation}
Our goal is to show that $K \lesssim M_1^{1-\kappa}$.

Observe that for each $j\in[K]$ there exists $N_j\ge M_1^{1/\varrho}$ such that 
\begin{align*}
\begin{cases}
\eta^{\Gamma_1^c}_{\le \lo_{N_j, N_j}^{\Gamma_1^c}(M_1^{\chi_2})}({\bm\{}\xi_{\Gamma_2}(m_j){\bm \}}-a/q) \neq 0 
\\
1-\Pi^{\Gamma_2}_{M_1,N_j}(\xi) \neq 0.
\end{cases}
\end{align*}
The second condition above is independent of $m_j$, thus letting $N_0:=\min(N_j: j\in[K])$, we obtain
\begin{align*}
\begin{cases}
\eta^{\Gamma_1^c}_{\le \lo_{N_0, N_0}^{\Gamma_1^c}(M_1^{\chi_2})}({\bm\{}\xi_{\Gamma_2}(m_j){\bm \}}-a/q) \neq 0 , \qquad j\in[K],
\\
1-\Pi^{\Gamma_2}_{M_1,N_0}(\xi) \neq 0.
\end{cases}
\end{align*}
The above conditions imply
\begin{align}\label{eq:Nzero}
\begin{cases}
|\{\xi_{\gamma_2}^{\Gamma_2}(m_j)\}-a_{\gamma_2}/q|\le {M_1^{\chi_2}}{N_0^{-\gamma_2}}, \qquad\gamma_2\in[d_2], \quad j\in[K],
\\ 
1-\Pi^{\Gamma_2}_{M_1,N_0}(\xi) \neq 0.
\end{cases} 
\end{align}
In particular, for each $\gamma_2\in[d_2]$ and each $j\in[K]$ there exists $k_{j,\gamma_2}\in\ZZ$ such that  
\begin{equation}\label{eq:Nzerok}
|\xi_{\gamma_2}^{\Gamma_2}(m_j)-k_{j,\gamma_2}/q|\le {M_1^{\chi_2}}{N_0^{-\gamma_2}}.
\end{equation}
It follows that for each $\gamma_2\in[d_2]$ the system of equations 
\begin{align}\label{eq:syst_mtx}
\begin{bmatrix} 
    1 & m_1 & \dots  & m_1^{d_1}\\
    1 & m_2 & \dots  & m_2^{d_1}\\
    \vdots & \vdots &\ddots & \vdots\\
    1 & m_K &\dots  & m_{K}^{d_1} 
    \end{bmatrix}
\begin{bmatrix} 
    \xi_{(0,\gamma_2)}\\
    \xi_{(1,\gamma_2)}\\
    \vdots \\
    \xi_{(d_1,\gamma_2)} 
    \end{bmatrix}
=
\begin{bmatrix} 
    k_{1,\gamma_2}/q\\
    k_{2,\gamma_2}/q\\
    \vdots \\
    k_{K,\gamma_2}/q 
    \end{bmatrix}
\end{align}
is satisfied up to the error of order $O\big({M_1^{\chi_2}}{N_0^{-\gamma_2}}\big)$. In particular, for any choice of distinct indices $i_0, i_1,\dots, i_{d_1}\in [K]$ the vector $\xi$ satisfies, up to a small error, a square system
\begin{equation}\label{eq:Vsys}
W_{i_0, i_1,\dots, i_{d_1}}
\begin{bmatrix} 
    \xi_{(0,\gamma_2)}\\
    \xi_{(1,\gamma_2)}\\
    \vdots \\
    \xi_{(d_1,\gamma_2)} 
    \end{bmatrix}
=
\begin{bmatrix} 
    k_{i_1,\gamma_2}/q\\
    k_{i_2,\gamma_2}/q\\
    \vdots \\
    k_{i_{d_1},\gamma_2}/q 
    \end{bmatrix},
\end{equation}
with the Vandermonde matrix given by
\[
W_{i_0, i_1,\dots, i_{d_1}}:= \begin{bmatrix} 
    1 & m_{i_0} & \dots  & m_{i_0}^{d_1}\\
    1 & m_{i_1} & \dots  & m_{i_1}^{d_1}\\
    \vdots & \vdots &\ddots & \vdots\\
    1 & m_{i_{d_1}} &\dots  & m_{i_{d_1}}^{d_1} 
    \end{bmatrix}.
\]
It is a well-known fact from linear algebra that the determinant of $W_{i_0, i_1,\dots, i_{d_1}}$ is 
\begin{align*}
\det (W_{i_0, i_1,\dots, i_{d_1}}) = \prod_{0\le j<k\le d_1}(m_{i_j}-m_{i_k}).
\end{align*}
Therefore, by multiplying \eqref{eq:Vsys} by the inverse of $W_{i_0, i_1,\dots, i_{d_1}}$, we see that for each $\gamma_2\in[d_2]$  the coordinates of the vector  $(\xi_{(0,\gamma_2)}, \xi_{(1,\gamma_2)},\dots, \xi_{(d_1,\gamma_2)})$ satisfy, up to a small error, the following relation
\begin{align}\label{eq:Zqprod}
\xi_{(0,\gamma_2)}, \xi_{(1,\gamma_2)},\dots, \xi_{(d_1,\gamma_2)}\in\Big(q\prod_{0\le j<k\le d_1}(m_{i_j}-m_{i_k})\Big)^{-1}\ZZ.
\end{align}
Note that upon multiplying by $W^{-1}_{i_0, i_1,\dots, i_{d_1}}$ the initial error terms of order $O\big({M_1^{\chi_2}}{N_0^{-\gamma_2}}\big)$ increase only by a factor of $M_1^{C_\Gamma}$, with a constant $C_\Gamma\in\RR_+$ depending on $\Gamma$. This is acceptable, since $\varrho\in(0, 1)$ in \eqref{eq:epsilon_def} can be chosen to be  very small, and consequently $N_0^{\gamma_2}\ge M_1^{1/\varrho}$ dominates the numerator in $O\big({M_1^{\chi_2+C_\Gamma}}{N_0^{-\gamma_2}}\big)$.  Since \eqref{eq:Zqprod} holds for any choice of $(d_1+1)$-tuples $(i_0,i_1,\dots, i_{d_1})$, then for each $\gamma_2\in[d_2]$, up to a small error, it follows that
\begin{align}\label{eq:Zqprodint}
\xi_{(0,\gamma_2)}, \xi_{(1,\gamma_2)},\dots, \xi_{(d_1,\gamma_2)}\in\bigcap_{   \substack{(i_0,i_1,\dots, i_{d_1}) \\ 0\le i_0<i_1<\dots<i_{d_1}\le K }  }\Big(q\prod_{0\le j<k\le d_1}(m_{i_j}-m_{i_k})\Big)^{-1}\ZZ.
\end{align}
Let $D$ denote the greatest common divisor of products $\prod_{0\le j<k\le d_1}(m_{i_j}-m_{i_k})$ over all possible choices of $(d_1+1)$-tuples $(i_0,i_1,\dots, i_{d_1})$, that is
$$
D:=\gcd_{   \substack{(i_0,i_1,\dots, i_{d_1}) \\ 0\le i_0<i_1<\dots<i_{d_1}\le K }  }\Big(\prod_{0\le j<k\le d_1}(m_{i_j}-m_{i_k})\Big).
$$
It follows from \eqref{eq:Zqprodint} that for each $\gamma_2\in[d_2]$, again up to a small error, one has
\begin{align*}
\xi_{(0,\gamma_2)}, \xi_{(1,\gamma_2)},\dots, \xi_{(d_1,\gamma_2)}\in(qD)^{-1}\ZZ.
\end{align*}

\paragraph{\bf Step 2}
Assume that $qD\ge M_1^{u_1}$. We will show that in this case one can
select from $m_1,\dots, m_K$ a subsequence with large gaps between its
elements, which will lead to an upper bound for $K$.

To see that this is indeed the case, consider first the
$(d_1+1)$-tuple of indexes
$(i_0,i_1,\dots, i_{d_1})=(1,2\dots, d_1+1)$. Note that the product
$\prod_{0\le j<k\le d_1}(m_{i_j}-m_{i_k}) \ge D$ and consists of 
$(d_1+1)^2$ numbers.
Therefore, there exist
$j, k\in [d_1+1]$ such that $j\neq k$ and $m_k-m_j \ge D^{1/(d_1+1)^2}$. Then
$m_{d_1+1}-m_1\ge m_k-m_j\ge D^{1/(d_1+1)^2}$.

Next, consider  the $(d_1+1)$-tuple
$(i_0,i_1,\dots, i_{d_1})=(d_1+1,d_1+2,\dots, 2d_1+1)$. Arguing in the
analogous way as before we can show that
$m_{2d_1+1}-m_{d_1+1}\ge D^{1/(d_1+1)^2}$. Continuing this procedure
we obtain a sequence $m_1<m_{d_1+1}<m_{2d_1+1}<\dots$ consisting of
roughly $K/(d_1+1)$ numbers in $(M_1,\tau M_1)$ which are separated by
$D^{1/(d_1+1)^2}$. The sum of gaps between elements of this sequence
cannot exceed the length of the interval $(M_1,\tau M_1)$, which leads
to the estimate
\begin{align*}
\frac{K}{(d_1+1)}D^{1/(d_1+1)^2}\lesssim M_1. 
\end{align*}
Since $q\le M_1^{u_2}<  M_1^{u_1/2}$, it follows that 
$$
K\lesssim M_1 D^{-1/(d_1+1)^2}\le M_1 D^{-1/(|\Gamma|+1)^2}\le  M_1^{1-u_1/(|\Gamma|+1)^2}q^{1/(|\Gamma|+1)^2}<M_1^{1-\kappa},
$$
which proves the lemma in the case $qD\ge M_1^{u_1}$.

\paragraph{\bf Step 3}
Assume now $qD< M_1^{u_1}$. It suffices to show that  either $K\lesssim M_1^{1-\kappa}$ or $\xi\notin\supp\big( 1-\Pi^{\Gamma_2}_{M_1,N_0} \big)$. If $K\lesssim M_1^{1-\kappa}$, then there is nothing to do. So we can assume that $K\gtrsim M_1^{1-\kappa}$.

Observe that the previous argument shows that for $\TT^\Gamma\ni\xi=(\xi',\xi'')\in \TT^{\Gamma_2}\times\TT^{\Gamma_2^c}$,  satisfying \eqref{eq:xicond}, we have $\xi'\in\big((qD)^{-1}\ZZ\big)^{\Gamma_2}$ up to a small error. More precisely, there exists a constant $C_\Gamma\in\RR_+$, depending only on $\Gamma$ and integers $(\mathfrak{a}_\gamma)_{\gamma\in\Gamma_2}\subseteq[qD]^{\Gamma_2}$  such that
\begin{equation}\label{eq:weak_err}
\Big|\xi_\gamma-\frac{\mathfrak{a}_\gamma}{qD}\Big| \le {M_1^{C_\Gamma+\chi_2}}{N_0^{-\gamma_2}}, \qquad \gamma\in\Gamma_2.
\end{equation}  
The factor $M_1^{C_\Gamma}$ appears because the original error term in \eqref{eq:Nzerok} is amplified by a factor of order $M_1^{C_\Gamma}$ through the process of solving the system \eqref{eq:syst_mtx}.

 Note that the above condition shows $\xi'$ is near fractions with a small denominator $qD$, but this alone does not guarantee that $\xi$ lies outside $\supp (1-\Pi^{\Gamma_2}_{M_1,N_0})$. To see that this estimate is, in fact, self-improving, we need to combine it with the original system \eqref{eq:Nzero}. The details are as follows.

\paragraph{\bf Step 4} Fix $\gamma_2\in[d_2]$ and $j\in[K]$. By \eqref{eq:weak_err} and \eqref{eq:Nzerok} there exist 
$\zeta_{(\gamma_1,\gamma_2)}=O\big({M_1^{C_\Gamma+\chi_2}}{N_0^{-\gamma_2}}\big)$ for $\gamma_1\in\NN_{\le d_1}$ and $\rho_{j, \gamma_2}=O\big({M_1^{\chi_2}}{N_0^{-\gamma_2}}\big)$ such that 
\begin{align} \label{eq:syst_main}
\begin{cases}
\xi_{(\gamma_1,\gamma_2)}={\mathfrak{a}_{(\gamma_1,\gamma_2)}}/{(qD)}+\zeta_{(\gamma_1,\gamma_2)},\qquad \gamma_1\in\NN_{\le d_1},
\\
\xi_{\gamma_2}^{\Gamma_2}(m_j)=k_{j,\gamma_2}/q+\rho_{j, \gamma_2}.
\end{cases}
\end{align}
 By \eqref{eq:syst_main} we have
\begin{align}\label{eq:pre-poly}
k_{j,\gamma_2}/q+\rho_{j, \gamma_2}=\xi_{\gamma_2}^{\Gamma_2}(m_j)
=\sum_{\gamma_1=0}^{d_1}\xi_{(\gamma_1,\gamma_2)}m_j^{\gamma_1}
=\frac{1}{qD}\sum_{\gamma_1=0}^{d_1}\mathfrak{a}_{(\gamma_1,\gamma_2)}m_j^{\gamma_1}+\sum_{\gamma_1=0}^{d_1}\zeta_{(\gamma_1,\gamma_2)}m_j^{\gamma_1}.
\end{align}
Moving all the error terms to the right side, we obtain
\begin{align}\label{eq:fracvserr}
k_{j,\gamma_2}/q-\frac{1}{qD}\sum_{\gamma_1=0}^{d_1}\mathfrak{a}_{(\gamma_1,\gamma_2)}m_j^{\gamma_1}=\sum_{\gamma_1=0}^{d_1}\zeta_{(\gamma_1,\gamma_2)}m_j^{\gamma_1}-\rho_{j, \gamma_2}=O\big(  { M_1^{d_1+C_{\Gamma}+1} }{N_0^{-\gamma_2}} \big).
\end{align}
Note that the right-hand side is very small since $N_0\ge M_1^{1/\varrho}$. On the other hand, we have
\begin{align*}
k_{j,\gamma_2}/q-\frac{1}{qD}\sum_{\gamma_1=0}^{d_1}\mathfrak{a}_{(\gamma_1,\gamma_2)}m_j^{\gamma_1}\in (qD)^{-1}\ZZ.
\end{align*}
Since $qD\lesssim M_1^{2u_1}$ the equation \eqref{eq:fracvserr} forces
$$
k_{j,\gamma_2}/q-\frac{1}{qD}\sum_{\gamma_1=0}^{d_1}\mathfrak{a}_{(\gamma_1,\gamma_2)}m_j^{\gamma_1}=0
$$
and then \eqref{eq:pre-poly} reduces to
$$
\sum_{\gamma_1=0}^{d_1}\zeta_{(\gamma_1,\gamma_2)}m_j^{\gamma_1}=\rho_{j, \gamma_2}.
$$
Applying the above reasoning to each $j$ and $\gamma_2$ we obtain a system
\begin{align*}
Q_{\gamma_2}(m_j):=\sum_{\gamma_1=0}^{d_1}\zeta_{(\gamma_1,\gamma_2)}m_j^{\gamma_1}=\rho_{j, \gamma_2}, \qquad j\in[K], \quad \gamma_2\in[d_2],
\end{align*}
which implies 
\begin{align*}
|Q_{\gamma_2}(m_j)|\le {M_1^{\chi_2}}{N_0^{-\gamma_2}}, \qquad j\in[K], \quad \gamma_2\in[d_2].
\end{align*}
We shall show that for a fixed $\gamma_2\in[d_2]$ the above system can hold for a large number of $m_j$'s only if all the coefficients of the polynomial on the left--hand side are very small.

\paragraph{\bf Step 5} We begin with observing that there exist $K^{(0)}\ge \log M_1$ elements $m_1^{(0)}<\dots<m_{K^{(0)}}^{(0)}$ among $m_1, \dots, m_K$ which are $(K/\log M_1)$-separated, that is
$$
|m_{j+1}^{(0)}-m_j^{(0)}|\ge K/\log M_1, \qquad j\in[K^{(0)}-1].
$$ 
Then, by the mean value theorem, for each $j\in[K^{(0)}]$ there exists $\tilde m_j^{(1)}$ such that
$$
|Q_{\gamma_2}'(\tilde m_j^{(1)})| =  \frac{|Q_{\gamma_2}(m_{j+1}^{(0)})-Q_{\gamma_2}(m_j^{(0)})|} {|m_{j+1}^{(0)}-m_j^{(0)}|}\le 2\frac{M_1^{\chi_2}\log M_1}{K N_0^{\gamma_2}}.
$$
Omitting every other element in the (increasing) sequence $(\tilde m_j^{(1)})_{j\in[K^{(0)}]}$ we obtain a subsequence consisting of $K^{(1)}\ge \frac{1}{2}\log M_1$ numbers
$m_1^{(1)}<\dots<m_{K^{(1)}}^{(1)}$ which are again $(K/\log M_1)$-separated and satisfy
$$
|Q_{\gamma_2}'(m_j^{(1)})| \le 2\frac{M_1^{\chi_2}\log M_1}{K N_0^{\gamma_2}}.
$$
Using the mean value theorem we obtain an increasing sequence of
$K^{(1)}$ numbers $(\tilde m_j^{(2)})_{j\in[K^{(1)}]}$ satisfying
$$
|Q_{\gamma_2}^{(2)}(\tilde m_j^{(2)})| =  \frac{|Q_{\gamma_2}'(m_{j+1}^{(1)})-Q_{\gamma_2}'(m_j^{(1)})|} {|m_{j+1}^{(1)}-m_j^{(1)}|}\le 4\frac{M_1^{\chi_2}(\log M_1)^2}{K^2 N_0^{\gamma_2}}.
$$
Then, as before, we can leave out every other element in the sequence $\tilde m_1^{(2)}< \dots < \tilde m_{K^{(1)}}^{(2)}$ obtaining a sequence of  $K^{(2)}\ge \frac{1}{4}\log M_1$ numbers $(m_j^{(2)})_{j\in[K^{(2)}]}$ which are 
$(K/\log M_1)$-separated and satisfy
$$
|Q_{\gamma_2}^{(2)}(m_j^{(2)})| \le 4\frac{M_1^{\chi_2}(\log M_1)^2}{K^2 N_0^{\gamma_2}}.
$$
Iterating the above procedure $d_1$ times we obtain an increasing sequence $(m_j^{(d_1)})_{_{j\in[K^{(d_1)}]}}$ consisting of $K^{(d_1)}\ge \frac{1}{2^{d_1}}\log M_1$ numbers satisfying $(K/\log M_1)$-separation condition and such that
$$
|Q_{\gamma_2}^{(d_1)}( m_j^{(d_1)} )| \le {2^{d_1} }\frac{M_1^{\chi_2}(\log M_1)^{d_1}}{K^{d_1} N_0^{\gamma_2}}.
$$
Note that 
\begin{align*}
Q_{\gamma_2}^{(d_1)}(m)&=\partial_m^{d_1}\Big( \sum_{k=0}^{d_1} \zeta_{(k, \gamma_2)}m^k\Big)=d_1! \zeta_{(d_1, \gamma_2)}.
\end{align*}
Recalling that $K\gtrsim M_1^{1-\kappa}$ and using \eqref{eq:parameter_tuning}, we obtain
$$
|\zeta_{(d_1, \gamma_2)}|\lesssim\frac{M_1^{\chi_2}(\log M_1)^{d_1}}{K^{d_1} N_0^{\gamma_2}}\lesssim\frac{M_1^{\chi_2+d_1\kappa}(\log M_1)^{d_1}}{M_1^{d_1} N_0^{\gamma_2}}
\lesssim\frac{ M_1^{\chi_1/2} }{  M_1^{d_1} N_0^{\gamma_2}  }.
$$
Now observe that the above estimate together with the bound
$$
|Q_{\gamma_2}^{(d_1-1)}( m_j^{(d_1-1)} )| \le {2^{d_1-1} }\frac{M_1^{\chi_2}(\log M_1)^{d_1-1}}{K^{d_1-1} N_0^{\gamma_2}}
$$
give
$$
|\zeta_{(d_1-1, \gamma_2)}|\lesssim\frac{ M_1^{\chi_1/2} }{  M_1^{d_1-1} N_0^{\gamma_2}  }.
$$
Iterating this procedure until we descend to the index $\gamma_1=0$, we obtain
$$
|\zeta_{(\gamma_1, \gamma_2)}|\lesssim\frac{ M_1^{\chi_1/2} }{  M_1^{\gamma_1} N_0^{\gamma_2}  }, \qquad \gamma_1\in \NN_{\le d_1}, \quad \gamma_2\in[d_2].
$$
Combining this with \eqref{eq:syst_main}, we obtain
\begin{equation*}
\Big|\xi_\gamma-\frac{\mathfrak{a}_\gamma}{qD}\Big| \lesssim \frac{ M_1^{\chi_1/2} }{  M_1^{\gamma_1} N_0^{\gamma_2}  }, \qquad \gamma\in\Gamma_2.
\end{equation*}  
Since $qD< M_1^{u_1}$ we have ${\mathfrak{a}}/{(qD)} \in \Sigma^{\Gamma_2}_{\le \lo(M_1^{u_1})}$ and consequently $1-\Pi^{\Gamma_2}_{M_1,N_0}(\xi)=0$ as desired.
\end{proof}

To prove inequality \eqref{eq:goal_mm_2old}, by Lemma \ref{lem:padova},  it is enough to check that  $d_{M_1,M_2}=\Delta_{M_1,M_2}$ with $\Delta_{M_1,M_2}$ from \eqref{eq:45} and \eqref{eq:46} satisfies conditions \eqref{eq:cond1} and \eqref{eq:cond2} with ${\bf M}_{M_1}, {\bf S}_{M_1}\simeq M_1^{-2\alpha}$. 

\subsubsection{\textbf{Verification of \eqref{eq:cond1} for $d_{M_1,M_2}=\Delta_{M_1,M_2}$}}
We aim at showing
\begin{equation}\label{eq:goal_mm_2_part1_main}
\big\|\sup_{M_2\in \DD_\tau: M_1 \le M_2^{\varrho}} |T_{\ZZ^\Gamma}[d_{M_1,M_2}] f |  \big\|_{\ell^2(\ZZ^\Gamma)}  \lesssim M_1^{-2\alpha}\|f\|_{\ell^2(\ZZ^\Gamma)},
\end{equation}
with  $d_{M_1,M_2}=\Delta_{M_1,M_2}$ as in \eqref{eq:46}.
\begin{proof}[Proof of inequality \eqref{eq:goal_mm_2_part1_main}]
By \eqref{eq:46}, observe that
\begin{align*}
&\big\|\sup_{M_2\in \DD_\tau: M_1 \le M_2^{\varrho}} |T_{\ZZ^\Gamma}[\Delta_{M_1,M_2}] f |  \big\|_{\ell^2(\ZZ^\Gamma)}
\le\sum_{m_1\in\ZZ} \chi_{M_1}(m_1)\big\|\sup_{M_2\in \DD_\tau: M_1 \le M_2^{\varrho}}
\big| T_{\ZZ^\Gamma}\big[\mathfrak h_{m_1, M_2}\mathbbm{1}_{A_{M_1}^{m_1}} \big] f \big| \big\|_{\ell^2(\ZZ^\Gamma)}.
\end{align*}
Changing the variables and applying the Ionescu--Wainger multiplier
theorem as in the proof of inequality \eqref{eq:goals}, see Section
\ref{715}, we may conclude for any $\varepsilon\in(0, 1)$ that
\begin{align*}
\big\|\sup_{M_2\in \DD_\tau: M_1 \le M_2^{\varrho}}
\big| T_{\ZZ^\Gamma}\big[\mathfrak h_{m_1, M_2}\mathbbm{1}_{A_{M_1}^{m_1}} \big] f \big| \big\|_{\ell^2(\ZZ^\Gamma)}
 \lesssim M_1^{\varepsilon}
\big\| T_{\ZZ^\Gamma} [ \mathbbm{1}_{A_{M_1}^{m_1}}] f  \big\|_{\ell^2(\ZZ^\Gamma)}.
\end{align*}
Combining these two estimates, and using the Cauchy--Schwarz inequality followed by the Plancherel
theorem, the triangle inequality, Lemma \ref{lem:comb} and condition \eqref{eq:parameter_tuning}, we can
further write
\begin{align*}
&\big\|\sup_{M_2\in \DD_\tau: M_1 \le M_2^{\varrho}} |T_{\ZZ^\Gamma}[\Delta_{M_1,M_2}] f |  \big\|_{\ell^2(\ZZ^\Gamma)}
\\
&\quad \lesssim 
M_1^\varepsilon \sum_{\theta \in \Sigma^{\Gamma_1^c}_{\le \lo(M_1^{u_2})}}  \Big( \sum_{m_1\in\ZZ} \chi_{M_1}(m_1) \int_{\TT^\Gamma} \mathbbm{1}_{A_{M_1}^{m_1}(\theta) }(\xi) |\hat{f}(\xi)|^2d\xi     \Big)^{1/2} \lesssim  M_1^{-2\alpha} \|f\|_{\ell^2(\ZZ^\Gamma)}.
\end{align*}
This completes the proof of inequality \eqref{eq:goal_mm_2_part1_main}. 
\end{proof}

\subsubsection{\textbf{Verification of \eqref{eq:cond2} for $d_{M_1,M_2}=\Delta_{M_1,M_2}$}} We prove that the following inequality
\begin{align}
\label{eq:41}
\Big(\sum_{j=0}^{2^{m-i}-1} \big\| T_{\ZZ^\Gamma} \big[  
\sum_{\substack{k\in U_j^i \\ U_j^i\subseteq [j_\varrho,2^m)}} (d_{M_1, \tau^{k+1}}-d_{M_1, \tau^k})
\big] f    
\big\|^2_{\ell^2(\ZZ^\Gamma)}\Big)^{1/2}
\lesssim M_1^{-2\alpha}\| f \|_{\ell^2(\ZZ^\Gamma)}
\end{align} 
holds uniformly for every $i\in\mathbb{N}_{\le m}$ with
$d_{M_1,M_2}=\Delta_{M_1,M_2}$ as in \eqref{eq:46}. The argument will
be similar to the arguments verifying \eqref{eq:cond2} for the
multipliers in the earlier sections.

\begin{proof}[Proof of inequality \eqref{eq:41}]
Arguing in the same way as in the proof of inequality \eqref{eq:29}, by \eqref{eq:43}, we see that  
\begin{equation*}
\Big\| 
\Big( \sum_{\substack{j=0 \\ j: U_j^i\subseteq [j_\varrho, 2^m)}}^{2^{m-i}-1}   
\big|
T_{\ZZ^\Gamma} [ 
\mathfrak h_{m_1,\tau^{(j+1)2^i}} - \mathfrak h_{m_1,\tau^{j 2^i}} ] f
\big|^2   
\Big)^{1/2}  \Big\|_{\ell^2(\ZZ^\Gamma)}\lesssim \|f\|_{\ell^2(\ZZ^\Gamma)},
\end{equation*}
holds uniformly in $i\in\mathbb{N}_{\le m}$ and $m_1\in(M_1,\tau M_1)$. Thus, by \eqref{eq:46}, we have
\begin{align*} 
&\sum_{j=0}^{2^{m-i}-1} \Big\| T_{\ZZ^\Gamma} \big[  
\sum_{\substack{k\in U_j^i \\ U_j^i\subseteq [j_\varrho,2^m)}} (\Delta_{M_1, \tau^{k+1}}-\Delta_{M_1, \tau^k})
\big] f    
\Big\|^2_{\ell^2(\ZZ^\Gamma)}
\\
&\qquad\lesssim
\sum_{m_1\in\ZZ} \chi_{M_1}(m_1) \big\|\mathbbm{1}_{A_{M_1}^{m_1}} \hat{f}\big\|^2_{L^2(\TT^\Gamma)}\lesssim M_1^{-4\alpha}\| f \|^2_{\ell^2(\ZZ^\Gamma)},
\end{align*}
where the last estimate follows by \eqref{eq:m1averagedec}. This yields inequality \eqref{eq:41} as desired.
\end{proof}

Lemma \ref{lem:padova}  yields inequality \eqref{eq:goal_mm_2old} and consequently inequality \eqref{eq:48} follows.

\subsection{Major arc estimates: Proof of inequality \eqref{eq:49}}\label{step8} 
Writing  $\TT^\Gamma\ni\xi=(\xi',\xi'')$ with $\xi'\in\TT^{\Gamma_2}$ and $\xi''\in\TT^{\Gamma_2^c}$, (see \eqref{eq:47} for the definitions of $\Gamma_2$ and $\Gamma_2^c$), we consider a multiplier
\begin{align}\label{eq:hGamma2_def}
\Phi_{M_1,M_2}^{\langle \Gamma_2; u, \chi\rangle}
&:=\sum_{m_1\in\ZZ}\chi_{M_1}(m_1)
\Big\llbracket G^{\Gamma_2}_{m_1} \mid \mathfrak m^{\langle \Gamma_2^c, \Gamma_2\rangle}_{m_1, M_2} \eta^{\Gamma_2}_{\le \lo_{M_1, M_2}^{\Gamma_2}(M_1^{\chi})}\Big\rrbracket_{\Sigma^{\Gamma_2}_{\le \lo(M_1^{u})}},
\end{align}
where $G^{\Gamma^2}_{m_1}$ is the partially complete exponential sum defined in \eqref{eq:GGamma2_def}
and
\begin{align}\label{eq:PhiGamma2_def}
\mathfrak m^{\langle \Gamma_2^c, \Gamma_2\rangle}_{m_1,M_2}(\xi):
=\frac{1}{\tau-1}\int_1^\tau
\ex\big(\xi\cdot(m_1, M_2y_2)^\Gamma\big)d y_2, \qquad \xi\in\TT^\Gamma.
\end{align}
Using the polynomial $\xi_{0}^{\Gamma_2^c}(m_1)$ from \eqref{eq:P2gamma_def}, note that 
 the multiplier $\mathfrak m^{\langle \Gamma_2^c, \Gamma_2\rangle}_{m_1,M_2}$ can be rewritten as follows
\begin{align*}
\mathfrak m^{\langle \Gamma_2^c, \Gamma_2\rangle}_{m_1,M_2}(\xi)
=\ex(\xi_{0}^{\Gamma_2^c}(m_1))\frac{1}{\tau-1}\int_1^\tau
\ex\big(\xi'\cdot (m_1, M_2y_2)^{\Gamma_2}\big)d y_2, \qquad \xi\in\TT^\Gamma.
\end{align*}
We also note that by \eqref{eq:Gauss2}, with $\delta_{\textrm{Gauss}}\in\RR_+$ as in \eqref{eq:12}, one has
\begin{equation}
\label{eq:55}
\sum_{m_1\in\ZZ}\chi_{M_1}(m_1)|G^{\Gamma_2}_{m_1}(a/q)|\lesssim q^{-\delta_{\textrm{Gauss}}},\qquad a/q\in(\TT\cap\QQ)^{\Gamma_2}, 
\end{equation}
uniformly in $q\le M_1^{10|\Gamma|}$, see Proposition \ref{prop:32}.

A one-parameter major arc approximation, similar as in the proof of inequality  \eqref{eq:major_l2_case2}, leads to
\begin{equation}\label{eq:hFTapprox}
\big\|\Pi^{\Gamma_2}_{M_1,M_2}m_{M_1,M_2}^{\langle \Gamma\rangle }-\Phi_{M_1,M_2}^{\langle \Gamma_2; u_1, \chi_1\rangle}\big\|_{L^\infty(\TT^\Gamma)}\lesssim M_2^{-1/4},
\end{equation}
with
$\Pi^{\Gamma_2}_{M_1,M_2}:=\Pi^{\Gamma_2, u_1, \chi_1}_{M_1,M_2}$ defined in \eqref{eq:20}, and  $u_1,\chi_1\in\RR_+$  specified in \eqref{eq:parameter_tuning}.
Dominating the supremum by the sum, applying Parseval's identity and using  \eqref{eq:hFTapprox}, we obtain that
\begin{equation}
\label{eq:50}
\big\|\sup_{M_2\in \DD_\tau: M_1 \le M_2^{\varrho}} | T_{\ZZ^\Gamma}[\Pi^{\Gamma_2}_{M_1,M_2}m_{M_1,M_2}^{\langle \Gamma\rangle}-\Phi_{M_1,M_2}^{\langle\Gamma_2; u_1, \chi_1\rangle}] f |  \big\|_{\ell^2(\ZZ^\Gamma)} \lesssim M_1^{-1/4}\|f\|_{\ell^2(\ZZ^\Gamma)}.
\end{equation}
The estimate \eqref{eq:50} reduces the proof of inequality \eqref{eq:49} to the following bound
\begin{align}
\label{eq:51}
\big\|\sup_{M_2\in \DD_\tau: M_1 \le M_2^{\varrho}} | T_{\ZZ^\Gamma}[\Phi_{M_1,M_2}^{\langle \Gamma_2; u_1, \chi_1\rangle}-\Phi_{M_1,M_2}^{\langle\Gamma\rangle}] f |  \big\|_{\ell^2(\ZZ^\Gamma)} \lesssim M_1^{-\alpha}\|f\|_{\ell^2(\ZZ^\Gamma)},
\end{align}
with $\alpha\in\RR_+$ as in \eqref{eq:parameter_tuning} and $\varrho\in\RR_+$ as in \eqref{eq:epsilon_def}. To handle inequality \eqref{eq:51} we will further perform a one-parameter major/minor arc decomposition with the aid of the projection
$\Pi^{\Gamma}_{M_1}:=\Pi^{\Gamma, u_1, \chi_1}_{M_1,M_1}$ defined in \eqref{eq:20} with  $u_1,\chi_1\in\RR_+$  as in \eqref{eq:parameter_tuning}. More precisely, we prove the minor arc inequality
\begin{equation}
\label{eq:52}
\big\|\sup_{M_2\in \DD_\tau: M_1 \le M_2^{\varrho}} | T_{\ZZ^\Gamma}[(1-\Pi^\Gamma_{M_1})\Phi_{M_1,M_2}^{\langle \Gamma_2; u_1, \chi_1\rangle}  ] f |  \big\|_{\ell^2(\ZZ^\Gamma)}\lesssim M_1^{-\alpha} \|f\|_{\ell^2(\ZZ^\Gamma)},
\end{equation}
as well as the corresponding major arc inequality
\begin{equation}
\label{eq:53}
\big\|\sup_{M_2\in \DD_\tau: M_1 \le M_2^{\varrho}} | T_{\ZZ^\Gamma}[\Pi^\Gamma_{M_1}\Phi_{M_1,M_2}^{\langle \Gamma_2; u_1, \chi_1\rangle}  -\Phi_{M_1,M_2}^{\langle \Gamma\rangle} ] f |  \big\|_{\ell^2(\ZZ^\Gamma)}\lesssim M_1^{-\alpha} \|f\|_{\ell^2(\ZZ^\Gamma)},
\end{equation}
Once inequalities \eqref{eq:52} and \eqref{eq:53} are proved, the triangle inequality yields \eqref{eq:51} immediately, and consequently inequality \eqref{eq:49} follows as desired. We focus on proving  inequalities \eqref{eq:52} and \eqref{eq:53}.

\subsubsection{\textbf{One-parameter minor arc estimates: Proof of inequality \eqref{eq:52}}}
We begin with the following reduction
\begin{equation}
\label{eq:54}
\big\|\sup_{M_2\in \DD_\tau: M_1 \le M_2^{\varrho}} | T_{\ZZ^\Gamma}[(1-\Pi^\Gamma_{M_1})(\Phi_{M_1,M_2}^{\langle \Gamma_2; u_1, \chi_1\rangle}-\Phi_{M_1,M_2}^{\langle \Gamma_2; u_2, \chi_2\rangle} ) ] f |  \big\|_{\ell^2(\ZZ^\Gamma)}\lesssim M_1^{-\alpha} \|f\|_{\ell^2(\ZZ^\Gamma)},
\end{equation}
which follows by observing that
\[
\Phi_{M_1,M_2}^{\langle \Gamma_2; u_1, \chi_1\rangle}-\Phi_{M_1,M_2}^{\langle \Gamma_2; u_2, \chi_2\rangle}=(\Phi_{M_1,M_2}^{\langle \Gamma_2; u_1, \chi_1\rangle}-\Phi_{M_1,M_2}^{\langle \Gamma_2; u_2, \chi_1\rangle})+(\Phi_{M_1,M_2}^{\langle\Gamma_2; u_2, \chi_1\rangle}-\Phi_{M_1,M_2}^{\langle\Gamma_2; u_2, \chi_2\rangle}).
\]

The first term produces the decay in \eqref{eq:54} by inequality \eqref{eq:55} if $a/q\in \Sigma^{\Gamma_2}_{\le \lo(M_1^{u_1})}\setminus \Sigma^{\Gamma_2}_{\le \lo(M_1^{u_2})}$. Indeed, one can write
\begin{align*}
\Phi_{M_1,M_2}^{\langle \Gamma_2; u_1, \chi_1\rangle}-\Phi_{M_1,M_2}^{\langle \Gamma_2; u_2, \chi_1\rangle}=\sum_{m_1\in\ZZ}\chi_{M_1}(m_1)\mathfrak h_{m_1, M_2}\mathfrak g_{m_1},
\end{align*}
where
\begin{align*}
\mathfrak h_{m_1, M_2}:=&\Big\llbracket 1 \mid \mathfrak m^{\langle \Gamma_2^c, \Gamma_2\rangle}_{m_1, M_2} \eta^{\Gamma_2}_{\le \lo_{M_1, M_2}^{\Gamma_2}(M_1^{\chi_1})}\Big\rrbracket_{\Sigma^{\Gamma_2}_{\le \lo(M_1^{u_1})}},\\
\mathfrak g_{m_1}:=&\Big\llbracket G^{\Gamma_2}_{m_1} \mid  \eta^{\Gamma_2}_{\le \lo_{M_1, M_1}^{\Gamma_2}(M_1^{\chi_1})}\Big\rrbracket_{\Sigma^{\Gamma_2}_{\le \lo(M_1^{u_1})}\setminus \Sigma^{\Gamma_2}_{\le \lo(M_1^{u_2})}}.
\end{align*}
By the Ionescu--Wainger theorem (see Theorem \ref{thm:IW1}) for arbitrary $\varepsilon\in(0, 1)$, we have
\begin{align}
\label{eq:64}
\sup_{m_1\in[M_1, \tau M_1]}\big\|\sup_{M_2\in \DD_\tau: M_1 \le M_2^{\varrho}} | T_{\ZZ^\Gamma}[(1-\Pi^\Gamma_{M_1}) \mathfrak h_{m_1, M_2} ] f |  \big\|_{\ell^2(\ZZ^\Gamma)}\lesssim M_1^{\varepsilon} \|f\|_{\ell^2(\ZZ^\Gamma)}.
\end{align}
Moreover, by the Cauchy--Schwarz inequality and inequality \eqref{eq:55}, we obtain
\begin{align}
\label{eq:66}
\sum_{m_1\in\ZZ}\chi_{M_1}(m_1)\| T_{\ZZ^\Gamma}[(1-\Pi^\Gamma_{M_1})\mathfrak g_{m_1}] f   \|_{\ell^2(\ZZ^\Gamma)}\lesssim M_1^{-\delta/2} \|f\|_{\ell^2(\ZZ^\Gamma)},
\end{align}
with $\delta\in(0, 1)$ as in \eqref{eq:uniform_delta}. Taking $\varepsilon<\delta/4$  and combining \eqref{eq:64} with \eqref{eq:66}, we readily obtain inequality \eqref{eq:54} with $\Phi_{M_1,M_2}^{\langle \Gamma_2; u_1, \chi_1\rangle}-\Phi_{M_1,M_2}^{\langle \Gamma_2; u_2, \chi_1\rangle}$ in place of $\Phi_{M_1,M_2}^{\langle \Gamma_2; u_1, \chi_1\rangle}-\Phi_{M_1,M_2}^{\langle \Gamma_2; u_2, \chi_2\rangle}$ in view of \eqref{eq:parameter_tuning}.

The second term can be handled by a simple square function estimate as in the proof of inequality \eqref{eq:rybki}.
Namely, we have 
\begin{align}
\label{eq:75}
\Phi_{M_1,M_2}^{\langle\Gamma_2; u_2, \chi_1\rangle}-\Phi_{M_1,M_2}^{\langle\Gamma_2; u_2, \chi_2\rangle}=\sum_{m_1\in\ZZ}\chi_{M_1}(m_1)\mathfrak h_{m_1, M_2},
\end{align}
where
\begin{align*}
\mathfrak h_{m_1, M_2}:=&\Big\llbracket G^{\Gamma_2}_{m_1} \mid \mathfrak m^{\langle \Gamma_2^c, \Gamma_2\rangle}_{m_1, M_2} \big(\eta^{\Gamma_2}_{\le \lo_{M_1, M_2}^{\Gamma_2}(M_1^{\chi_1})}-\eta^{\Gamma_2}_{\le \lo_{M_1, M_2}^{\Gamma_2}(M_1^{\chi_2})}\big)\eta^{\Gamma_2}_{\le \lo_{M_1, M_1}^{\Gamma_2}(M_1^{\chi_1})}\Big\rrbracket_{\Sigma^{\Gamma_2}_{\le \lo(M_1^{u_2})}}.
\end{align*}
Then one can show that
\begin{align}
\label{eq:77}
\sup_{m_1\in[M_1, \tau M_1]}\bigg\|\Big(\sum_{M_2\in\DD_{\tau}}|T_{\ZZ^\Gamma}[(1-\Pi^\Gamma_{M_1})\mathfrak h_{m_1, M_2}] f |^2\Big)^{1/2}\bigg\|_{\ell^2(\ZZ^\Gamma)}\lesssim M_1^{-\alpha} \|f\|_{\ell^2(\ZZ^\Gamma)}.
\end{align}
Using this bound we easily obtain inequality \eqref{eq:54} with $\Phi_{M_1,M_2}^{\langle\Gamma_2; u_2, \chi_1\rangle}-\Phi_{M_1,M_2}^{\langle\Gamma_2; u_2, \chi_2\rangle}$ in place $\Phi_{M_1,M_2}^{\langle \Gamma_2; u_1, \chi_1\rangle}-\Phi_{M_1,M_2}^{\langle \Gamma_2; u_2, \chi_2\rangle}$. The estimate \eqref{eq:77} follows, by a simple square function argument, from the bound
\begin{align*}
&\big|\mathfrak m^{\langle \Gamma_2^c, \Gamma_2\rangle}_{m_1, M_2}(\zeta) \big(\eta^{\Gamma_2}_{\le \lo_{M_1, M_2}^{\Gamma_2}(M_1^{\chi_1})}(\zeta)-\eta^{\Gamma_2}_{\le \lo_{M_1, M_2}^{\Gamma_2}(M_1^{\chi_2})}(\zeta)\big)\big|\\
&\qquad \lesssim M_1^{-\delta \chi_2/2}\min\big(|(M_1^{\gamma_1}M_2^{\gamma_2}\zeta_{\gamma})_{\gamma\in \Gamma_2}|, |(M_1^{\gamma_1}M_2^{\gamma_2}\zeta_{\gamma})_{\gamma\in \Gamma_2}|^{-1}\big)^{\delta /2}.
\end{align*}
The latter bound can be derived from van der Corput's type estimate  for the multiplier $\mathfrak m^{\langle \Gamma_2^c, \Gamma_2\rangle}_{m_1,M_2}$ like in \eqref{eq:9}, the fact that $m_1\in[M_1, \tau M_1]$, and conditions \eqref{eq:parameter_tuning}. Hence inequality \eqref{eq:54} follows.

In view of \eqref{eq:54}, it suffices to prove that
\begin{equation}
\label{eq:56}
\big\|\sup_{M_2\in \DD_\tau: M_1 \le M_2^{\varrho}} | T_{\ZZ^\Gamma}[(1-\Pi^\Gamma_{M_1})\Phi_{M_1,M_2}^{\langle \Gamma_2; u_2, \chi_2\rangle}  ] f |  \big\|_{\ell^2(\ZZ^\Gamma)}\lesssim M_1^{-\alpha} \|f\|_{\ell^2(\ZZ^\Gamma)}.
\end{equation}

\begin{proof}[Proof of inequality \eqref{eq:56}]
In view of \eqref{eq:parameter_tuning}, there is a sequence
$(\beta_{\gamma}:\gamma\in\Gamma_2^c)\subseteq \RR_+$ 
such that
\begin{align}
\label{eq:110}
u_2+\sum_{\gamma\in\Gamma_2^c}\beta_{\gamma}<  u_1,
\qquad \text{ and } \qquad
(|\Gamma|+2)\chi_2+\sum_{\gamma\in\Gamma_2^c\setminus\{e_1\}}\beta_{\gamma}
<\frac{1}{4}\beta_{e_1},
\end{align}
where $e_1=(1,0)$,
and such that for every $\gamma\in\Gamma_2^c\setminus\{e_1\}$, we also have
\begin{align}
\label{eq:u0cond}
\frac{1}{\delta}(|\Gamma|+2)\chi_2<\frac{1}{4}\beta_{\gamma}.
\end{align}

The proof of inequality \eqref{eq:56} is fairly intricate, so we will proceed in four steps.

\paragraph{\bf Step 1}
For every $\theta'=(a_{\gamma}/q)_{\gamma\in\Gamma_2}\in \Sigma^{\Gamma_2}_{\le \lo(M_1^{u_2})}$, it suffices to show that
\begin{align}
\label{eq:59}
\big\|\sup_{M_2\in \DD_\tau: M_1 \le M_2^{\varrho}} | T_{\ZZ^\Gamma}[(1-\Pi^\Gamma_{M_1})\Phi_{M_1,M_2}^{\langle \Gamma_2; u_2, \chi_2; \theta'\rangle}  ] f |  \big\|_{\ell^2(\ZZ^\Gamma)}\lesssim \big(M_1^{-\beta_{e_1}/2}+ M_1^{-\beta_{\gamma_0}\delta_{\rm Weyl}/2}\big) \|f\|_{\ell^2(\ZZ^\Gamma)},
\end{align}
where $\beta_{{\gamma_0}} := \min(\beta_{\gamma}: \gamma \in \Gamma_2^c\setminus\{e_1\})
$ and
\begin{align*}
\Phi_{M_1,M_2}^{\langle \Gamma_2; u_2, \chi_2;\theta'\rangle}
&:=\sum_{m_1\in\ZZ}\chi_{M_1}(m_1)
\Big\llbracket G^{\Gamma_2}_{m_1} \mid \mathfrak m^{\langle \Gamma_2^c, \Gamma_2\rangle}_{m_1, M_2} \eta^{\Gamma_2}_{\le \lo_{M_1, M_2}^{\Gamma_2}(M_1^{\chi_2})}\Big\rrbracket_{\{\theta'\}}, \qquad \theta'\in\mathbb Q^{\Gamma_2}.
\end{align*}
Since $|\Sigma^{\Gamma_2}_{\le \lo(M_1^{u_2})}|\lesssim M_1^{u_2(|\Gamma|+1)}$, then by the triangle inequality and estimate \eqref{eq:59} combined with conditions \eqref{eq:110} and \eqref{eq:u0cond}, we readily obtain inequality \eqref{eq:56} as desired. 

\paragraph{\bf Step 2}
Fix $\theta'=(a_{\gamma}/q)_{\gamma\in\Gamma_2}\in \Sigma^{\Gamma_2}_{\le \lo(M_1^{u_2})}$ and  $\xi\in\supp \Phi_{M_1,M_2}^{\langle \Gamma_2; u_2, \chi_2; \theta'\rangle}$. Then for every $\gamma\in\Gamma_2$, we have 
$$
|\xi_{\gamma}-a_{\gamma}/q|
\le M_1^{-\gamma_1}M_2^{-\gamma_2}   M_1^{\chi_2}\le
M_1^{ -\gamma_1-\gamma_2+\chi_1}.
$$
Moreover, by the Dirichlet's principle, for every
$\gamma\in\Gamma_2^c$ there exist integers $0\le a_{\gamma}\le q_{\gamma}$
such that $(a_{\gamma}, q_{\gamma})=1$, and
$1\le q_{\gamma}\le M_1^{\gamma_1+\gamma_2-\beta_\gamma}$
and
\[
|\xi_{\gamma}-a_{\gamma}/q_{\gamma}|
\le
q_{\gamma}^{-1} M_1^{ -\gamma_1-\gamma_2+\beta_\gamma}
\le
q_{\gamma}^{-2}.
\]

We will show that there exists a coordinate $\gamma\in\Gamma_2^c$ such that $\xi_\gamma$ has a diophantine approximation with a relatively large denominator, more precisely $q_{\gamma}>  M_1^{\beta_\gamma}$. Suppose for a contradiction that $1\le q_{\gamma}\le  M_1^{\beta_\gamma}$ for every
$\gamma\in\Gamma_2^c$ and let
$q_0:= \lcm(\{q_{\gamma}: \gamma\in\Gamma_2^c\}\cup\{q\})$. Note that $1\le q\le  M_1^{u_2}$, since $(a_{\gamma}/q)_{\gamma\in\Gamma_2}\in\Sigma^{\Gamma_2}_{\le  \lo(M_1^{u_2})}$. Moreover, we
have assumed that $q_\gamma\le  M_1^{\beta_\gamma}$ for $\gamma\in\Gamma_2^c$, thus 
by \eqref{eq:110} we see that $q_0\le  M_1^{u_1}$.
We define $b \in \ZZ^\Gamma$ by setting 
\begin{align*}
b_\gamma/q_0
:=
\begin{cases}
a_\gamma/q, &\qquad \gamma\in\Gamma_2.\\
a_\gamma/q_\gamma, &\qquad \gamma\in\Gamma_2^c.
\end{cases}
\end{align*}
Then
$(b_\gamma/q_0)_{\gamma\in\Gamma}\in\Sigma^\Gamma_{\le \lo(M_1^{u_1})}$
and for every $\gamma\in\Gamma$ we have
\[
|\xi_{\gamma}-b_\gamma/q_0|
\le
M_1^{ -\gamma_1-\gamma_2}   M_1^{\chi_1}.
\]
This shows that
$1-\Pi^\Gamma_{M_1}(\xi)=0$, which is impossible for $\xi\in\supp \Phi_{M_1,M_2}^{\langle \Gamma_2; u_2, \chi_2; \theta'\rangle}$. Hence, we can assume that
$ M_1^{\beta_\gamma}\le q_{\gamma}\le M_1^{\gamma_1+\gamma_2-\beta_\gamma}$
for some $\gamma\in\Gamma_2^c$. We now distinguish two cases.

\paragraph{\bf Step 3} Assume that 
$ M_1^{\beta_{e_1}}\le q_{e_1}\le M_1^{1-\beta_{e_1}}$
and $q_{\gamma}\le  M_1^{\beta_\gamma}$ for all
$\gamma\in\Gamma_2^c\setminus\{e_1\}$.  We shall use Lemma \ref{lem:3}.
Take
$Q:=\lcm{(\{q_{\gamma}:\gamma\in\Gamma_2^c\setminus\{e_1\}\}\cup\{q\})}$. We will show that not only $\xi_{e_1}$ but also $Q\xi_{e_1}$ has a diophantine approximation by a fraction with a relatively large denominator. To this end observe that by
\eqref{eq:110} we have $Q\le  M_1^{\beta_{e_1}/4}$.
Since
\begin{align*}
|\xi_{e_1}-a_{e_1}/q_{e_1}|\le1/q_{e_1}^2, 
\end{align*}
for some integers $0\le a_{e_1}\le q_{e_1}$ such that
$(a_{e_1}, q_{e_1})=1$ and
$ M_1^{\beta_{e_1}}\le q_{e_1}\le M_1^{1-\beta_{e_1}}$,  Lemma \ref{lem:3} applies and we obtain
\begin{align}
\label{eq:95Q}
\abs{Q\xi_{e_1}-a_{e_1}'/q_{e_1}'}
\le\frac{ 1}{2q_{e_1}'M_1^{1-\beta_{e_1}}}\le \frac{1}{(q_{e_1}')^2},
\end{align}
for some integers $0\le a_{e_1}'\le q_{e_1}'$ such that $(a_{e_1}', q_{e_1}')=1$ and 
\[
\frac{1}{2Q}  M_1^{\beta_{e_1}}\le q_{e_1}'\le 2M_1^{1-\beta_{e_1}}.
\]
That in turn implies that $q_{e_1}'\ge \frac{1}{2Q} M_1^{\beta_{e_1}}\ge \frac{1}{2} M_1^{\beta_{e_1}/2}$.

Our argument will rely on an application of the summation by parts formula \eqref{eq:sbpformula} to the sum over $m_1$ in the definition of $\Phi_{M_1,M_2}^{\langle \Gamma_2; u_2, \chi_2; \theta'\rangle}$. We begin with eliminating the dependence on $m_1$ in the partially complete exponential sums $G^{\Gamma_2}_{m_1}$. That will be accomplished by splitting the summation into classes modulo $Q$.  We take
$U_1=\big\lfloor\frac{M_1-r_1}{Q}\big\rfloor$ and
$V_1=\big\lfloor\frac{\tau M_1-r_1}{Q}\big\rfloor$ and observe 
that
\begin{align*}
\Phi_{M_1,M_2}^{\langle \Gamma_2; u_2, \chi_2; \theta'\rangle}=\frac{1}{|(M_1,\tau M_1]\cap\ZZ|}\sum_{r_1=1}^{Q}
\Big\llbracket G^{\Gamma_2}_{r_1} \mid \sum_{m_1=U_1+1}^{V_1}\mathfrak m^{\langle \Gamma_2^c, \Gamma_2\rangle}_{Qm_1+r_1, M_2} \eta^{\Gamma_2}_{\le \lo_{M_1, M_2}^{\Gamma_2}(M_1^{\chi_2})}  \Big\rrbracket_{\{\theta'\}}.
\end{align*}

Let $\theta:=(a_{\gamma}/q_{\gamma})_{\gamma\in\Gamma_2^c\setminus\{e_1\}}\in\QQ^{\Gamma_2^c\setminus\{e_1\}}$ and $\theta'':=(0, \theta)\in\QQ\times \QQ^{\Gamma_2^c\setminus\{e_1\}}$. Then we can write

\begin{align*}
\sum_{m_1=U_1+1}^{V_1}
\mathfrak m^{\langle \Gamma_2^c, \Gamma_2\rangle }_{Qm_1+r_1, M_2}(\xi'-\theta',\xi'')=&E_{r_1}(\theta'')
\sum_{U_1< m_1\le V_1}
\ex\big((Qm_1+r_1)\xi_{e_1}\big)F_{m_1, M_2}^{r_1}(\xi'-\theta',  \xi''-\theta''),
\end{align*}
with $E_{r_1}(\theta''):=\ex\big(\theta\cdot  (r_1, 1)^{\Gamma_2^c\setminus\{e_1\}}\big)$,
and 
\begin{align*}
F_{m_1, M_2}^{r_1}(\xi', \xi''):=
\ex\Big(\sum_{\gamma\in \Gamma_2^c\setminus\{e_1\}}\xi_{\gamma}''(Qm_1+r_1)^{\gamma}\Big)
\mathfrak m^{\langle \emptyset, \Gamma_2;0\rangle }_{Qm_1+r_1,M_2}
(\xi'),
\end{align*}
where for $y_1\in \RR$, we set 
\begin{align}\label{eq:PsiGamma2zero}
\mathfrak m^{\langle \emptyset, \Gamma_2;\gamma\rangle }_{y_1,M_2}(\zeta):
=\frac{1}{\tau-1}\int_1^\tau  
\ex\big(\zeta\cdot(y_1,M_2y_2)^{\Gamma_2}\big)
y_2^{\gamma_2}d y_2,  \qquad \zeta\in\TT^{\Gamma_2}, \quad \gamma\in\Gamma_2\cup\{0\}.
\end{align}

Using the summation by parts formula \eqref{eq:sbpformula}, we obtain
\begin{align*}
\sum_{U_1< m_1\le V_1}&
\ex\big((Qm_1+r_1)\xi_{e_1}\big)F_{m_1, M_2}^{r_1}(\xi'-\theta',  \xi''-\theta'')\\
&=S_{V_1}(\xi)F_{V_1+1, M_2}^{r_1}(\xi'-\theta',  \xi''-\theta'')
-\sum_{U_1< m_1\le V_1}
S_{m_1}(\xi)G_{m_1, M_2}^{r_1}(\xi'-\theta',  \xi''-\theta''),
\end{align*}
where for an integer $M>U_1$, we let
\begin{align*}
S_M(\xi):=\sum_{k=U_1+1}^{M} \ex((Qk+r_1)\xi_{e_1}), \quad \text{ and }\quad 
G_{m_1, M_2}^{r_1}(\xi):=\int_{m_1}^{m_1+1}\partial_{y_1}
F_{y_1, M_2}^{r_1}(\xi)d y_1,\qquad \xi\in\TT^{\Gamma}.
\end{align*}

Thus, we obtain
\begin{equation}\label{eq:ab_def}
\Phi_{M_1,M_2}^{\langle \Gamma_2; u_2, \chi_2; \theta'\rangle}
={\mathfrak a}_{M_1,M_2}^{\langle \Gamma_2; \theta', \theta''\rangle}+
{\mathfrak b}_{M_1,M_2}^{\langle \Gamma_2; \theta', \theta''\rangle},
\end{equation}
where
\begin{align*}
{\mathfrak a}_{M_1,M_2}^{\langle \Gamma_2; \theta', \theta''\rangle}:=
-\frac{1}{|(M_1,\tau M_1]\cap\ZZ|}\sum_{r_1=1}^{Q}
\Big\llbracket G^{\Gamma_2}_{r_1}E_{r_1} \mid \sum_{m_1=U_1+1}^{V_1}S_{m_1}G_{m_1, M_2}^{r_1} \eta^{\Gamma_2}_{\le \lo_{M_1, M_2}^{\Gamma_2}(M_1^{\chi_2})}  \Big\rrbracket_{\{\theta'\}\times\{\theta''\}},
\end{align*}
and
\begin{align*}
{\mathfrak b}_{M_1,M_2}^{\langle \Gamma_2; \theta', \theta''\rangle}:=
\frac{1}{|(M_1,\tau M_1]\cap\ZZ|}\sum_{r_1=1}^{Q}
\Big\llbracket G^{\Gamma_2}_{r_1}E_{r_1} \mid S_{V_1}F_{V_1+1, M_2}^{r_1} \eta^{\Gamma_2}_{\le \lo_{M_1, M_2}^{\Gamma_2}(M_1^{\chi_2})}  \Big\rrbracket_{\{\theta'\}\times\{\theta''\}},
\end{align*}
By a direct computation, we obtain
\begin{align*}
\partial_{y_1}F_{y_1, M_2}^{r_1}
=F_{y_1, M_2}^{r_1; 1}+F_{y_1, M_2}^{r_1; 2},
\end{align*}
with
\begin{align*}
F_{y_1, M_2}^{r_1; 1}(\xi):=&\sum_{\gamma\in\Gamma_2^c\setminus\{e_1\}} 2\pi {\bm i}
\gamma_1Q(Qy_1+r_1)^{\gamma_1-1} \xi_{\gamma}''\ F_{y_1, M_2}^{r_1}(\xi)
=
E_{y_1}^{r_1}(\xi'')\mathfrak m^{\langle \emptyset, \Gamma^2;0\rangle}_{Qy_1+r_1,M_2}
(\xi'),
\end{align*}
where
\begin{align*}
E_{y_1}^{r_1}(\xi'')&:=\sum_{\gamma\in\Gamma_2^c\setminus\{e_1\}}2\pi {\bm i}
\gamma_1Q
(Qy_1+r_1)^{\gamma_1-1}\xi_{\gamma}''\ \ex\Big(\sum_{\gamma\in \Gamma^1\setminus\{e_1\}} \xi_{\gamma}''(Qy_1+r_1)^{\gamma}
\Big), 
\end{align*}
and
\begin{align*}
F_{y_1, M_2}^{r_1; 2}(\xi):=
\sum_{\gamma\in\Gamma_2}
E_{y_1}^{r_1, \gamma}(\xi'')
\big(M_1^{\gamma_1-\chi_2}M_2^{\gamma_2}
\xi_{\gamma}'\big)
\mathfrak m^{\langle \emptyset, \Gamma^2;\gamma\rangle}_{Qy_1+r_1,M_2}
(\xi'),
\end{align*}
where
\begin{align*}
E_{y_1}^{r_1, \gamma}(\xi''):=
2\pi{\bm i}\gamma_1Q(Qy_1+r_1)^{\gamma_1-1}
M_1^{\chi_2-\gamma_1}
\ \ex\Big(\sum_{\gamma\in \Gamma^1\setminus\{e_1\}} \xi_{\gamma}''(Qy_1+r_1)^{\gamma}
\Big),  \qquad \gamma\in\Gamma_2.
\end{align*}
It is important that neither  $E_{y_1}^{r_1}$ nor  $E_{y_1}^{r_1, \gamma}$ depends on $M_2$.
Define
\begin{align*}
\mathfrak n^{\langle \emptyset, \Gamma_2;\gamma\rangle}_{y_1,M_2}(\zeta):=
\begin{cases}
\eta^{\Gamma_2}_{\le \lo_{M_1, M_2}^{\Gamma_2}(M_1^{\chi_2})}(\zeta)\mathfrak m^{\langle \emptyset, \Gamma_2;0\rangle}_{y_1,M_2}(\zeta) & \text{ if } \gamma=0,\\
(M_1^{\gamma_1-\chi_2}M_2^{\gamma_2}
\zeta_{\gamma})\eta^{\Gamma_2}_{\le \lo_{M_1, M_2}^{\Gamma_2}(M_1^{\chi_2})}(\zeta)\mathfrak m^{\langle \emptyset, \Gamma_2;\gamma\rangle}_{y_1,M_2}(\zeta)
& \text{ if } \gamma\in\Gamma_2.
\end{cases}
\end{align*}
Now by Proposition \ref{prop:msw},  we obtain
\begin{align}
\label{eq:maxLambda}
\sup_{r_1\in [Q]}\sup_{y_1\in\ZZ}\big\|\sup_{M_2\in \DD_\tau:M_1 \le M_2^\varrho}
|T_{\ZZ^\Gamma}[\mathfrak n^{\langle \emptyset, \Gamma_2;\gamma\rangle}_{Qy_1+r_1,M_2}
(\cdot - \theta')]f|\big\|_{\ell^2(\ZZ^\Gamma)}
\lesssim \norm{f}_{\ell^2(\ZZ^\Gamma)}.
\end{align}
In view of \eqref{eq:ab_def} and using \eqref{eq:maxLambda} we will prove that
for every $M_1\in\DD_\tau$, we have
\begin{align}
\label{eq:57}
\big\|\sup_{M_2\in \DD_\tau: M_1 \le M_2^{\varrho}} | T_{\ZZ^\Gamma}[(1-\Pi^\Gamma_{M_1})({\mathfrak a}_{M_1,M_2}^{\langle\Gamma_2; \theta', \theta''\rangle}+
{\mathfrak b}_{M_1,M_2}^{\langle\Gamma_2; \theta', \theta''\rangle})  ] f |  \big\|_{\ell^2(\ZZ^\Gamma)}\lesssim M_1^{-\beta_{e_1}/2}
\|f\|_{\ell^2(\ZZ^\Gamma)}
\end{align}
which in turn gives \eqref{eq:59} (by \eqref{eq:ab_def}) in this case.
Indeed, note that for every $r_1\in[Q]$, $m_1\in[U_1,V_1]$ and $m_1\le y_1\le m_1+1$ we have
\begin{align}
\label{eq:115}
|E_{r_1}(\theta'')|\le 1,
\end{align}
and
\begin{align}
\label{eq:116}
|S_{m_1}(\xi_{e_1})|
\lesssim \|Q\xi_{e_1}\|^{-1}
\lesssim q'_{e_1}\lesssim M_1^{1-\beta_{e_1}},
\end{align}
since by \eqref{eq:95Q} and the fact that $2\le \frac{1}{2} M_1^{\beta_{e_1}/2}\le q_{e_1}'\le M_1^{1-\beta_{e_1}}$ we can bound from below
$$
\|Q\xi_{e_1}\|\ge 1/q'_{e_1}-1/(q'_{e_1})^2\ge 1/(2q'_{e_1}).
$$
Moreover, 
\begin{align}
\label{eq:117}
|E_{y_1}^{r_1}(\xi''-\theta'')|
\lesssim Q  \max_{\gamma\in\Gamma_2^c\setminus\{e_1\}}M_1^{\beta_{\gamma}}M_1^{-1},
\end{align}
and for every $\gamma\in\Gamma_2$, we have  
\begin{align}
\label{eq:118}
|E_{y_1}^{r_1, \gamma}(\xi''-\theta'')|
\lesssim Q  M_1^{\chi_2}M_1^{-1}.
\end{align}
In view of \eqref{eq:maxLambda} and Plancherel's theorem, inequalities \eqref{eq:115}-\eqref{eq:118}, and the second condition in \eqref{eq:110},  we obtain
\begin{align*}
\big\|\sup_{M_2\in \DD_\tau: M_1 \le M_2^{\varrho}} | T_{\ZZ^\Gamma}[(1-\Pi^\Gamma_{M_1})({\mathfrak a}_{M_1,M_2}^{\langle\Gamma_2; \theta', \theta''\rangle}+
{\mathfrak b}_{M_1,M_2}^{\langle\Gamma_2; \theta', \theta''\rangle})  ] f |  \big\|_{\ell^2(\ZZ^\Gamma)}
\lesssim  M_1^{-\beta_{e_1}/2}\|f\|_{\ell^2(\ZZ^\Gamma)},
\end{align*}
since $Q\le  M_1^{\beta_{e_1}/4}$.
This completes the proof of inequality \eqref{eq:57}.

\paragraph{\bf Step 4} Assume that
$ M_1^{\beta_{\gamma_0}}\le q_{{\gamma_0}}\le M_1^{\gamma_0-\beta_{\gamma_0}}$
for some $\gamma_0\in\Gamma_2^c\setminus\{e_1\}$. Take
$U_1=\lfloor  M_1 \rfloor $ and
$V_1=\lfloor \tau M_1 \rfloor$. For $q\in\ZZ_+$ and
$r\in[q]$ define the set
$P_{r,q}:=\{m\in\ZZ: m\equiv r\pmod{q}\}$ and note  that
\begin{align*}
\ind{P_{r,q}}(m)=\frac{1}{q}\sum_{l=1}^q\ex({(m-r)l/q}), \qquad m\in\ZZ,
\end{align*}
and the multiplier $\Phi_{M_1,M_2}^{\langle\Gamma_2; u_2, \chi_2; \theta'\rangle}$ can be rewritten as 
\begin{align*}
\frac{1}{|(M_1,\tau M_1]\cap\ZZ|q}\sum_{l_1=1}^{q}\sum_{r_1=1}^{q}
\Big\llbracket G^{\Gamma_2}_{r_1}\ex({-r_1l_1/q}) \mid \eta^{\Gamma_2}_{\le \lo_{M_1, M_2}^{\Gamma_2}(M_1^{\chi_2})}\sum_{m_1=U_1+1}^{V_1}\ex({-m_1l_1/q})
\mathfrak m^{\langle\Gamma_2^c, \Gamma_2\rangle}_{m_1, M_2}  \Big\rrbracket_{\{\theta'\}}.
\end{align*}
By the summation by parts formula \eqref{eq:sbpformula}, note that
\begin{align*}
\sum_{m_1=U_1+1}^{V_1}&
\ex({m_1l_1/q})\mathfrak m^{\langle \Gamma_2^c, \Gamma_2\rangle}_{m_1, M_2}(\xi'-\theta',\xi'')=
\sum_{m_1=U_1+1}^{V_1}\ex({m_1l_1/q})\ex\big(\xi''\cdot (m_1,1)^{\Gamma_2^c} \big)\ \mathfrak m^{\langle \emptyset, \Gamma_2; 0\rangle}_{m_1, M_2}(\xi'-\theta')\\
&=S^{l_1/q}_{V_1}(\xi'')\mathfrak m^{\langle\emptyset, \Gamma_2; 0\rangle}_{V_1+1, M_2}(\xi'-\theta')-
\sum_{m_1=U_1+1}^{V_1}S^{l_1/q}_{m_1}(\xi'')\int_{m_1}^{m_1+1}\partial_{y_1}\mathfrak m^{\emptyset, \Gamma_2; 0}_{y_1, M_2}(\xi'-\theta')dy_1,
\end{align*}
where
\begin{equation*}
S^{l_1/q}_{m_1}(\xi'')
:=\sum_{k= U_1+1}^{m_1}
\ex({m_1l_1/q})\ex\big(\xi''\cdot (k,1)^{\Gamma_2^c} \big), \qquad  \xi''\in\TT^{\Gamma_2^c}.
\end{equation*}

Now note that
\begin{align*}
\int_{m_1}^{m_1+1}\partial_{y_1}\mathfrak m^{\langle\emptyset, \Gamma_2; 0\rangle}_{y_1, M_2}(\xi')dy_1=
\sum_{\gamma\in\Gamma_2}2\pi{\bm i}\gamma_1 (M_1^{\gamma_1-\chi_2}M_2^{\gamma_2} \xi_\gamma)\int_{m_1}^{m_1+1} M_1^{\chi_2-\gamma_1}y_1^{\gamma_1-1}\mathfrak m^{\langle\emptyset, \Gamma_2; \gamma\rangle}_{y_1, M_2}(\xi')d y_1.
\end{align*}
This leads to the decomposition
\begin{align}
\label{eq:22}
\Phi_{M_1,M_2}^{\langle\Gamma_2; u_2, \chi_2; \theta'\rangle}
={\mathfrak a}_{M_1,M_2}^{\langle\Gamma_2; \theta'\rangle}+
{\mathfrak b}_{M_1,M_2}^{\langle\Gamma_2; \theta'\rangle},
\end{align}
with
\begin{align*}
{\mathfrak a}_{M_1,M_2}^{\langle\Gamma_2; \theta'\rangle}:=
\frac{1}{|(M_1,\tau M_1]\cap\ZZ|q}\sum_{l_1=1}^{q}\sum_{r_1=1}^{q}
\Big\llbracket G^{\Gamma_2}_{r_1}\ex({-r_1l_1/q}) \mid S^{l_1/q}_{V_1}\mathfrak n^{\langle\emptyset, \Gamma_2; 0\rangle}_{V_1+1, M_2}  \Big\rrbracket_{\{\theta'\}},
\end{align*}
and 
\begin{align*}
{\mathfrak b}_{M_1,M_2}^{\langle\Gamma_2; \theta'\rangle}:=
\frac{1}{|(M_1,\tau M_1]\cap\ZZ|q}\sum_{l_1=1}^{q}\sum_{r_1=1}^{q}
\Big\llbracket G^{\Gamma_2}_{r_1}\ex({-r_1l_1/q}) \mid \sum_{m_1=U_1+1}^{V_1}S^{l_1/q}_{m_1}G_{m_1, M_2}  \Big\rrbracket_{\{\theta'\}},
\end{align*}
where
\begin{align*}
G_{m_1, M_2}(\xi'):=\sum_{\gamma\in\Gamma_2}2\pi{\bm i}\gamma_1 \int_{m_1}^{m_1+1} M_1^{\chi_2-\gamma_1}y_1^{\gamma_1-1}\mathfrak n^{\langle\emptyset, \Gamma_2; \gamma\rangle}_{y_1, M_2}(\xi')d y_1.
\end{align*}
Since $ M_1^{\beta_{\gamma_0}}\le q_{{\gamma_0}}\le M_1^{\gamma_0-\beta_{\gamma_0}}$
for some $\gamma_0\in\Gamma_2^c\setminus\{e_1\}$, we can apply the one-parameter Weyl's inequality, see Theorem \ref{thm:weyl1par}, and get
\begin{align}
\label{eq:121}
\abs{S^{l_1/q}_m( \xi'')}\lesssim M_1^{1-\beta_{\gamma_0}\delta_{\textrm{Weyl}}}, \qquad m\in [M_1,\tau M_1],
\end{align}
uniformly in $q$ and $l_1\in[q]$. In view of decomposition \eqref{eq:22}, taking into account \eqref{eq:maxLambda}, \eqref{eq:121} and \eqref{eq:u0cond}, we obtain
\begin{align*}
\big\|\sup_{M_2\in \DD_\tau:M_1 \le M_2^\varrho}
|T_{\ZZ^\Gamma}\big[ (1-\Pi^\Gamma_{M_1})({\mathfrak a}_{M_1,M_2}^{\langle\Gamma_2; \theta'\rangle}+
{\mathfrak b}_{M_1,M_2}^{\langle\Gamma_2; \theta'\rangle}) \big] f|\big\|_{\ell^2(\ZZ^\Gamma)}
\lesssim  M_1^{-\beta_{\gamma_0}\delta_{\textrm{Weyl}}/2}\|f\|_{\ell^2(\ZZ^\Gamma)},
\end{align*}
since  $1\le q\le  M_1^{u_2}$. That concludes the proof of inequality \eqref{eq:59} and hence \eqref{eq:52}.
\end{proof}

\subsubsection{\textbf{One-parameter major arc approximations: Proof of inequality \eqref{eq:53}}}
The final stage of the proof is to establish inequality
\eqref{eq:53}. Once this is proved, inequality
\eqref{eq:l2_main_int_case2} follows.

\begin{proof}[Proof of inequality \eqref{eq:53}]
Note that, with  $\TT^\Gamma\ni\xi=(\xi',\xi'')\in \TT^{\Gamma_2}\times\TT^{\Gamma_2^c}$, we have 
\begin{align*}
\Pi^\Gamma_{M_1}(\xi)\Phi_{M_1,M_2}^{\langle\Gamma_2; u_1, \chi_1\rangle} (\xi)
&=
\sum_{m_1\in\ZZ}\chi_{M_1}(m_1)
\sum_{\theta'\in\Sigma^{\Gamma_2}_{\le  \lo(M_1^{u_1})}}\sum_{\substack{\omega\in\Sigma^{\Gamma}_{\le  \lo(M_1^{u_1})}\\\omega'=\theta'}} G^{\Gamma_2}_{m_1}(\theta')
\\
&\quad\times\mathfrak m^{\langle\Gamma_2^c, \Gamma_2\rangle}_{m_1,M_2}
(\xi'-\theta',\xi'')\eta^{\Gamma_2}_{\le \lo_{M_1, M_2}^{\Gamma_2}(M_1^{\chi_1})}(\xi'-\theta')
\eta^{\Gamma}_{\le \lo_{M_1, M_1}^{\Gamma}(M_1^{\chi_1})}(\xi-\omega),
\end{align*}
which leads to the formula
\begin{align*}
\Pi^\Gamma_{M_1}(\xi)\Phi_{M_1,M_2}^{\langle\Gamma_2; u_1, \chi_1\rangle} (\xi)
&=
\sum_{m_1\in\ZZ}\chi_{M_1}(m_1)
\sum_{\theta\in\Sigma^{\Gamma}_{\le  \lo(M_1^{u_1})}} G^{\Gamma_2}_{m_1}(\theta')
\mathfrak m^{\langle\Gamma_2^c, \Gamma_2\rangle}_{m_1,M_2}
(\xi'-\theta',\xi'')
\eta^{\Gamma}_{\le \lo_{M_1, M_2}^{\Gamma}(M_1^{\chi_1})}(\xi-\theta).
\end{align*}
Fix $\theta=(a_{\gamma}/q)_{\gamma\in\Gamma}\in \Sigma^{\Gamma}_{\le  \lo(M_1^{u_1})}$, and 
let
$U_1=\frac{M_1-r_1}{q}$ and
$V_1=\frac{\tau M_1 -r_1}{q}$. Then
\begin{align*}
&\sum_{m_1\in\ZZ}\chi_{M_1}(m_1)
G^{\Gamma_2}_{m_1}(\theta')
\mathfrak m^{\langle\Gamma_2^c, \Gamma_2\rangle}_{m_1,M_2}
(\xi'-\theta',\xi'')
\\
&\quad=\frac{1}{|(M_1,\tau M_1]\cap\ZZ|}\sum_{r_1=1}^qG^{\Gamma_2}_{r_1}(\theta')\sum_{\floor{U_1}<m_1\le \floor{V_1}}
\mathfrak m^{\langle\Gamma_2^c, \Gamma_2\rangle}_{qm_1+r_1,M_2}
(\xi'-\theta',\xi'').
\end{align*}
Next, suppose that $\theta''=(a_{\gamma}/q)_{\gamma\in\Gamma_2^c}\in\QQ^{\Gamma_2^c}$ is such that
\[
\abs{\xi_{\gamma}-\theta_{\gamma}''}
\le M_1^{-\gamma_1} M_1^{\chi_1},\qquad \gamma\in\Gamma_2^c.
\] 
Then we can write
\begin{align*}
\mathfrak m^{\langle\Gamma_2^c, \Gamma_2\rangle}_{qm_1+r_1,M_2}
(\xi'-\theta',\xi'')=\ex\big(\theta''\cdot(r_1, 1)^{\Gamma_2^c}\big)\mathfrak m^{\langle\Gamma_2^c, \Gamma_2\rangle}_{qm_1+r_1,M_2}
(\xi-\theta).
\end{align*}
Consequently, we have 
\begin{align*}
&\sum_{m_1\in\ZZ}\chi_{M_1}(m_1)
G^{\Gamma_2}_{m_1}(\theta')
\mathfrak m^{\langle\Gamma_2^c, \Gamma_2\rangle}_{m_1,M_2}
(\xi'-\theta',\xi'')\\
&\quad=\frac{1}{|(M_1,\tau M_1]\cap\ZZ|}\sum_{r_1=1}^qG^{\Gamma_2}_{r_1}(\theta')\ex\big(\theta''\cdot(r_1, 1)^{\Gamma_2^c}\big)\sum_{\floor{U_1}<m_1\le \floor{V_1}}
\mathfrak m^{\langle \Gamma_2^c, \Gamma_2\rangle}_{qm_1+r_1,M_2}
(\xi-\theta).
\end{align*}
Now, we will approximate the last sum over $m_1$ by an integral. The details go as follows. 
\begin{align*}
\sum_{\floor{U_1}<m_1\le \floor{V_1}}
\mathfrak m^{\langle\Gamma_2^c, \Gamma_2\rangle}_{qm_1+r_1,M_2}
&=\frac{M_1(\tau-1)}{q}\mathfrak m^{\langle\Gamma_2^c, \Gamma_2\rangle}_{M_1,M_2}
+
\mathfrak m^{\langle\Gamma_2^c, \Gamma_2; r_1, 1\rangle}_{M_1,M_2}
+
\mathfrak m^{\langle\Gamma_2^c, \Gamma_2; r_1, 2\rangle}_{M_1,M_2},
\end{align*}
where $\mathfrak m_{M_1,M_2}^{\langle \Gamma\rangle}$ is given by \eqref{eq:PhiM1M2_def},  and 
\begin{align}\label{eq:PhiGamma21_def}
\mathfrak m^{\langle\Gamma_2^c, \Gamma_2; r_1, 1\rangle}_{M_1,M_2}
(\zeta)
:=\sum_{\floor{U_1}<m_1\le \floor{V_1}}
\int_{m_1-1}^{m_1}\int_{y_1}^{m_1}\Big(\partial_{x_1}
\mathfrak m^{\langle\Gamma_2^c, \Gamma_2\rangle}_{qx_1+r_1,M_2}
(\zeta)\Big)
d x_1d y_1, \qquad \zeta\in\TT^\Gamma,
\end{align}
and
\begin{align}\label{eq:PhiGamma22_def}
&\mathfrak m^{\langle\Gamma_2^c, \Gamma_2; r_1, 2\rangle}_{M_1,M_2}
(\zeta):=\bigg(\int_{\floor{U_1}}^{U_1}-\int_{\floor{V_1}}^{V_1}\bigg)
\mathfrak m^{\langle\Gamma_2^c, \Gamma_2\rangle}_{qy_1+r_1,M_2}
(\zeta)d y_1, \qquad \zeta\in\TT^\Gamma.
\end{align}
Taking into account these definitions and the identity 
$$
\sum_{r_1=1}^qG^{\Gamma_2}_{r_1}(\theta')\ex\big(\theta''\cdot (r_1, 1)^{\Gamma_2^c}\big)=qG^\Gamma(\theta),
$$
we conclude that
\begin{align*}
\Pi^\Gamma_{M_1}(\xi)\Phi_{M_1,M_2}^{\langle\Gamma_2; u_1, \chi_1\rangle} (\xi)
=\Phi_{M_1,M_2}^{\langle\Gamma\rangle}(\xi)+\sum_{l\in[3]}\Phi_{M_1,M_2}^{\langle\Gamma; l\rangle}(\xi),
\end{align*}
where for $l\in[2]$, we have
\begin{align}\label{eq:mathfrakgl_def}
\Phi_{M_1,M_2}^{\langle\Gamma; l\rangle}:=\frac{1}{|(M_1,\tau M_1]\cap\ZZ|}
\sum_{a/q\in\Sigma^\Gamma_{\le  \lo(M_1^{u_1})}}
\sum_{r_1=1}^q\Big\llbracket
G^{\Gamma_2}_{r_1}E_{r_1}\mid \mathfrak m^{\langle\Gamma_2^c, \Gamma_2; r_1, l\rangle}_{M_1,M_2}
\eta^{\Gamma}_{\le \lo_{M_1, M_2}^{\Gamma}(M_1^{\chi_1})}\Big\rrbracket_{\{a/q\}},
\end{align}
and $\Phi_{M_1,M_2}^{\langle\Gamma; 3\rangle}:=w_{M_1}\Phi_{M_1,M_2}^{\langle\Gamma\rangle}$, where
\begin{align*}
E_{r_1}(\theta''):=\ex\big(\theta''\cdot(r_1, 1)^{\Gamma_2^c}\big),\quad \text{ and } \quad
 w_{M_1}:=\bigg(\frac{M_1(\tau-1)}{|(M_1,\tau M_1]\cap\ZZ|}-1\bigg).
\end{align*}
The proof of \eqref{eq:53} is reduced to showing that for every $l\in[3]$, we have
\begin{align}
\label{eq:60}
\big\|\sup_{M_2\in \DD_\tau: M_1 \le M_2^{\varrho}}
|T_{\ZZ^\Gamma}[ \Phi_{M_1,M_2}^{\langle\Gamma; l\rangle} ] f|\big\|_{\ell^2(\ZZ^{\Gamma})}
\lesssim M_1^{-1/2}\|f\|_{\ell^2(\ZZ^{\Gamma})}.
\end{align}
The proof of inequality \eqref{eq:60} for $l=3$ immediately follows from Theorem \ref{thm:iw-semi}, since $|w_{M_1}|\lesssim M_1^{-1}$.

For $l\in[2]$ we will proceed much in the same way as in the proof of inequality \eqref{eq:56}.  Note that
\begin{align*}
\partial_{x_1}\mathfrak m^{\langle\Gamma_2^c, \Gamma_2\rangle}_{qx_1+r_1,M_2}(\xi)
=\sum_{\gamma\in\Gamma}
E^\gamma(r_1,x_1)
\big(M_1^{\gamma_1-\chi_1}M_2^{\gamma_2}\xi_{\gamma}\big)
\mathfrak m^{\langle\Gamma_2^c, \Gamma_2; \gamma\rangle}_{qx_1+r_1,M_2}
(\xi),
\end{align*}
where for $x_1\in\RR$, we have $E^\gamma(r_1,x_1):=
2\pi{\bm i}\gamma_1q(qx_1+r_1)^{\gamma_1-1}
M_1^{\chi_1-\gamma_1}$,
and
\begin{align*}
\mathfrak m^{\langle\Gamma_2^c, \Gamma_2; \gamma\rangle}_{x_1,M_2}(\zeta):
=\frac{1}{\tau-1}\int_1^\tau
\ex\big(\zeta\cdot (x_1,M_2y_2)^\Gamma\big)
y_2^{\gamma_2}d y_2, \qquad \zeta\in\TT^\Gamma, \qquad \gamma\in\Gamma\cup\{0\}.
\end{align*}
Define
\begin{align*}
\mathfrak n^{\langle\Gamma_2^c, \Gamma_2;\gamma\rangle}_{x_1,M_2}(\zeta):=
\begin{cases}
\eta^{\Gamma}_{\le \lo_{M_1, M_2}^{\Gamma}(M_1^{\chi_1})}(\zeta)\mathfrak m^{\langle\Gamma_2^c, \Gamma_2;0\rangle}_{x_1,M_2}(\zeta) & \text{ if } \gamma=0,\\
(M_1^{\gamma_1-\chi_2}M_2^{\gamma_2}
\zeta_{\gamma})\eta^{\Gamma}_{\le \lo_{M_1, M_2}^{\Gamma}(M_1^{\chi_1})}(\zeta)\mathfrak m^{\langle\Gamma_2^c, \Gamma_2;\gamma\rangle}_{x_1,M_2}(\zeta)
& \text{ if } \gamma\in\Gamma.
\end{cases}
\end{align*}
Now by Proposition \ref{prop:msw},  we obtain
\begin{align}
\label{eq:61}
\sup_{\theta\in \Sigma^\Gamma_{\le  \lo(M_1^{u_1})}}\sup_{x_1\in[M_1,\tau M_1]}\big\|\sup_{M_2\in \DD_\tau:M_1 \le M_2^\varrho}
|T_{\ZZ^\Gamma}[\mathfrak n^{\langle\Gamma_2^c, \Gamma_2;\gamma\rangle}_{x_1,M_2}
(\cdot - \theta)]f|\big\|_{\ell^2(\ZZ^\Gamma)}
\lesssim \norm{f}_{\ell^2(\ZZ^\Gamma)}.
\end{align}
To finish the proof of \eqref{eq:60}, we note that $E^\gamma$ does not depend on $M_2$ and 
\begin{align}
\label{eq:129}
|E^\gamma(r_1,x_1)|
\lesssim q  M_1^{\chi_1-1}.
\end{align}
Therefore, in view of inequality \eqref{eq:61} and  \eqref{eq:129} (for $l=1$), and condition \eqref{eq:parameter_tuning}, we obtain
\begin{align*}
\big\|\sup_{M_2\in \DD_\tau: M_1 \le M_2^{\varrho}}
|T_{\ZZ^\Gamma}[ \Phi_{M_1,M_2}^{\langle\Gamma; l\rangle}] f|\big\|_{\ell^2(\ZZ^\Gamma)}
\lesssim |\Sigma_{\le  \lo(M_1^{u_1})}|
M_1^{\chi_1+u_1}M_1^{-1}\|f\|_{\ell^2(\ZZ^\Gamma)}
\lesssim M_1^{-1/2}\|f\|_{\ell^2(\ZZ^\Gamma)}.
\end{align*}
This concludes the proof of \eqref{eq:60} and hence \eqref{eq:53}, and consequently inequality \eqref{eq:l2_main_int_case2} follows as desired.
\end{proof}


\end{document}